\documentclass[11pt,a4paper,reqno]{amsart}
\usepackage[applemac]{inputenc}
\usepackage[T1]{fontenc}
\usepackage{amsmath}
\usepackage{amsthm}
\usepackage{amsfonts}
\usepackage{amssymb}
\usepackage{graphicx}
\usepackage{amsbsy}
\usepackage{mathrsfs}
\usepackage{bbm}
\newtheorem{theorem}{Theorem}[section]
\newtheorem{lemma}[theorem]{Lemma}

\newtheorem{proposition}[theorem]{Proposition}

\newtheorem{definition}[theorem]{Definition}
\theoremstyle{remark}
\newtheorem{remark}[theorem]{\it \bf{Remark}\/}

\numberwithin{equation}{section}
\catcode`@=11
\def\section{\@startsection{section}{1}%
  \z@{1.5\linespacing\@plus\linespacing}{.5\linespacing}%
  {\normalfont\bfseries\large\centering}}
\catcode`@=12
\newcommand{\be}{\begin{equation}}
\newcommand{\ee}{\end{equation}}
\newcommand{\bea}{\begin{eqnarray}}
\newcommand{\eea}{\end{eqnarray}}
\newcommand{\bee}{\begin{eqnarray*}}
\newcommand{\eee}{\end{eqnarray*}}

\def\pa{\partial}

\def\RR{\mathbb{R}}

\def\Pt{\tilde{P}}

\def\fref#1{{\rm (\ref{#1})}}

\def\G{{\Gamma}}

\catcode`@=11
\def\supess{\mathop{\operator@font Sup\,ess}}
\catcode`@=12

\def\bt{\tilde{b}}

\def\RR{\mathbb{R}}

\def\e{\varepsilon}

\def\fref#1{{\rm (\ref{#1})}}

\def\R2+{\RR ^2_+}

\def\lsl{\frac{\lambda_s}{\lambda}}

\def\pa{\partial}

\def\lim{\mathop{\rm lim}}

\def\sup{\mathop{\rm sup}}

\def\l{\lambda}
\def\bt{\tilde{b}}

\def\log{{\rm log}}

\def\et{\tilde{\e}}

\def\lsl{\frac{\lambda_s}{\lambda}}

\def\Psih{\hat{\Psi}}

\def\deg{{\rm deg}}

\def\qbt{\tilde{Q}_{b,a}}

\def\Vt{\widetilde{V}}

\def\St{\tilde{S}}

\def\dkp{\delta_{k_+}}
\def\dk{\dkp}
\def\dkm{\delta_{k_-}}

\def\pa{\partial}

\def\et{\tilde{\e}}

\def\pa{\partial}

\def\Psit{\tilde{\Psi}}

\def\Mod{{\rm Mod}}

\def\St{\tilde{S}}
\def\Phit{\tilde{\Phi}}
\def\matchal{\mathcal}
\def\L{\mathcal L}
\def\Lt{\widetilde{\L}}
\def\dkm{\delta_{k_-}}
\def\Psih{\hat{\Psi}}
\def\Id{\rm{Id}}
\def\zetat{\tilde{\zeta}}
\def\Modt{\widetilde{\Mod}}
\def\at{\tilde{a}}
\def\dep{\delta_p}

\title[]{Type II blow up for the energy supercritical NLS}
\author[F. Merle]{Frank Merle}
\address{IHES and Universit\'e de Cergy Pontoise, France}
\email{frank.merle@math.u-cergy.fr}
\author[P. Rapha\"el]{Pierre Rapha\"el}
\address{Laboratoire J.A. Dieudonn\'e, Universit\'e de Nice Sophia Antipolis\\
 et Institut Universitaire de France}
\email{praphael@unice.fr}
\author[I. Rodnianski]{Igor Rodnianski}
\address{Department of Mathematics, MIT, USA\newline
\mbox{\,\,\,\,\,\,\,}Department of Mathematics, Princeton University, USA}
\email{irod@math.mit.edu}

\begin{document}
\maketitle

\begin{abstract}
We consider the energy super critical nonlinear Schr\"odinger equation 
$$i\pa_tu+\Delta u+u|u|^{p-1}=0$$ in large dimensions $d\geq 11$ with spherically symmetric data. For all $p>p(d)$ large enough, 
in particular in the super critical regime $$s_c=\frac d2-\frac{2}{p-1}>1,$$ 
we construct a family of $\mathcal C^{\infty}$ finite time blow up solutions which become singular via concentration of a universal profile 
$$u(t,x)\sim \frac{1}{\l(t)^{\frac 2{p-1}}}Q\left(\frac{r}{\l(t)}\right)e^{i\gamma(t)}$$ with the so called type II quantized blow up rates: $$\l(t)\sim c_u(T-t)^{\frac \ell\alpha}, \ \ \ell\in \Bbb N^*, \ \ 2\ell>\alpha=\alpha(d,p).$$ The essential feature of these solutions is that all norms below scaling remain bounded $$\limsup_{t\uparrow T}\|\nabla^su(t)\|_{L^2}<+\infty\ \ \mbox{for}\ \ 0\leq s<s_c.$$ Our analysis fully revisits the construction of type II blow up solutions for the corresponding heat equation in \cite{Velas}, \cite{Mizo1}, which was 
done using maximum principle techniques following \cite{MaM2}. Instead we develop a robust energy method, in continuation of the works in the energy critical case \cite{RaphRod}, \cite{MRR}, \cite{RSc1}, \cite{RSc2} and the $L^2$ critical case \cite{MMaR1}. This shades a new light on the essential role played by the solitary wave and its tail in the type II blow up mechanism, and the universality of the corresponding singularity formation in {\it both} energy critical and super critical regimes.
\end{abstract}

%%%%%%%%%%%%%%%%%%%%%%%%%%%%%%%%%%%%%%%%%%%%%%%%%%%%%%%%%%
%%%%%%%%%%%%%%%%%%%%%%%%%%%%%%%%%%%%%%%%%%

\section{Introduction}

%%%%%%%%%%%%%%%%%%%%%%%%%%%%%%%%%%%%%%%%%%%%%%%%%%

%%%%%%%%%%%%%%%%%%%%%%%%%%%%%%%%%%%%%%%%%%%%%%%%%%

\subsection{The NLS problem}

%%%%%%%%%%%%%%%%%%%%%%%%%%%%%%%%%%%%%%%%%%%%%%%%%%

In this paper we study the focusing nonlinear Schr\"odinger equation:
\be
\label{nls}
(NLS)\ \ \left\{\begin{array}{ll}i\pa_tu+\Delta u+u|u|^{p-1}, \\u_{|t=0}=u_0\end{array}\right.\ \ (t,x)\in \Bbb R^+\times \Bbb R^d, \ \ u(t,x)\in \Bbb C.
\ee
This canonical dissipative model conserves the total energy and mass:
\bea
\label{bvjeeboe}
&&E(u(t))=\frac 12\int |\nabla u|^2-\frac{1}{p+1}\int |u|^{p+1}=E(u_0),\\
\label{conmass}
&&\int |u(t)|^2=\int|u_0|^2.
\eea
The scaling symmetry $u_\l(t,x)=\l^{\frac2{p-1}}u(\l^2 t,\l x)$ for $\l>0$ is an isometry of the homogeneous Sobolev critical space $$\|u_\l(t,\cdot)\|_{\dot{H}^{s_c}}=\|u(\l^2 t,\cdot)\|_{\dot{H}^{s_c}} \ \ \mbox{for}\ \ s_c=\frac d 2-\frac {2}{p-1}.$$ We focus on the {\it energy critical and super critical models}: $$s_c\geq 1\ \ \mbox{i.e.}\ \ p\geq 2^*-1=\frac{d+2}{d-2}, \ \ d\ge 3.$$ 
These problems are locally well posed in $H^{s_c}$ and if the nonlinearity is analytic $$p=2q+1, \ \ q\in \Bbb N^*,$$ then the flow propagates Sobolev regularity and there holds the blow up criterion: $$T<+\infty\ \  \mbox{implies}\ \ \lim_{t\uparrow T}\|u(t)\|_{H^s}=+\infty \ \ \mbox{for}\ \ s>s_c.$$ 

%%%%%%%%%%%%%%%%%%%%%%%%%%%%%%%%%%%%%%%%%%%%
%%%%%%%%%%%%%%%%%%%%%%%%%%%%%%%%%%%%%%%%%%%%

\subsection{Type I and type II blow up for the heat equation}

%%%%%%%%%%%%%%%%%%%%%%%%%%%%%%%%%%%%%%%%%%%%
%%%%%%%%%%%%%%%%%%%%%%%%%%%%%%%%%%%%%%%%%%%%

Singularity formation is better understood for the scalar nonlinear heat equation 
\be
\label{NLH}
(NLH)\ \ \left\{\begin{array}{ll}\pa_tu=\Delta u+u^p, \\u_{t=0}=u_0\end{array}\right.\ \ (t,x)\in \Bbb R^+\times \Bbb R^d
\ee
in dimension $d\geq 3$, in particular in the radial setting where maximum principle techniques apply.  
In particular, one can construct time-dependent Lyapunov functionals, based on counting the number of spatial intersections between two solutions. Let us very briefly recall some of the main known facts on singularity formation for \fref{NLH} in the energy critical and super critical range $$p> 2^*-1, \ \ s_c> 1.$$ The basic object at the heart of the analysis is the self-similar profile. Let us look for solutions to \fref{NLH} of the explicit form 
\be
\label{develeotpemt}
u(t,x)=\frac{1}{\l(t)^{\frac 2{p-1}}}Q_b\left(\frac{r}{\l(t)}\right)
\ee 
where $\l(t)$ is given by the exact {\it self similar-scaling}: 
\be
\label{loiaionoie}
\l(t)=\sqrt{b(T-t)}, \ \ b=1.
\ee 
$Q_b$ is then a solution elliptic stationary self-similar equation: 
\be
\label{eqseflsim}
\Delta Q_b-b\Lambda Q_b+Q_b^p=, \ \ \Lambda=\frac{2}{p-1}+y\cdot\nabla, \ \ b=1.
\ee
Spherically symmetric solutions of  \fref{eqseflsim} are completely classified. There are two fundamental objects: the {\it regular at the origin} constant self-similar solution 
\be
\label{regularselfismil}
Q_1\equiv \kappa_p, \ \ \kappa_p=\left(\frac{2}{p-1}\right)^{\frac1{p-1}},
\ee
and the {\it singular at the origin homogeneous self-similar solution}:
\be
\label{selfsimilarsolutions}
R(r)=\frac{c_{\infty}}{r^{\frac{2}{p-1}}}, \ \ c_\infty=\left[\frac{2}{p-1}\left(d-2-\frac 2{p-1}\right)\right]^{\frac 2{p-1}}.
\ee
\noindent\underline{Type I blow up}: The regular constant self-similar solution \fref{regularselfismil} generates a {\it stable} blow up dynamics of so called type I with universal blow up rate given by: 
\be
\label{typeone}
\|u(t)\|_{L^{\infty}}\sim \frac{1}{(T-t)^{\frac 1{p-1}}},
\ee
consistent with \fref{develeotpemt}, \fref{loiaionoie}. The existence and stability of this object can be proved using spectral techniques and energy methods, \cite{GK1}, \cite{GK2}, \cite{GK3}, \cite{MZDuke}, \cite{BK}. In fact. this blow up regime exists for all $p$ and is not specific to the energy supercritical range. A related analysis has been recently successfully propagated to the case of the wave equation, \cite{Donn}.  

 In the regime $2^*-1<p<p_{JL}$ 
 there exists another class of regular solutions, decaying at $\infty$, to the self-similar equation \fref{eqseflsim}  which give rise to the type I unstable blow up\footnote{this 
 corresponds to a threshold regime between global solutions and the stable type I blow up dynamics.}, \cite{Lepin}, \cite{MaM1}. Here, 
 $p_{JL}$ if the Joseph-Lundgren exponent given by
 \be
\label{exponentpjl}
p>p_{JL}=\left\{\begin{array}{ll} +\infty\ \ \mbox{for}\ \ d\leq 10,\\
1+\frac{4}{d-4-2\sqrt{d-1}}\ \ \mbox{for}\ \ d\geq 11.
\end{array}\right.
\ee

\noindent\underline{Type II blow up}: In the 1992 unpublished manuscript by Herrero and Velasquez, announced in \cite{Velas}, 
proposed a different type of blow up mechanism for $p>p_{JL}$, based on a threshold structure of the spectrum of the linearized operator, close to \fref{selfsimilarsolutions}, 
\be
\label{deflinearized}
H_R=-\Delta +\Lambda -\frac{pc_{\infty}^{p-1}}{r^2}
\ee 
The spectrum of $H_R$ turns out to be completely explicit in suitable weighted spaces. The authors describe a singularity formation in which 
\be
\label{estlinfty}
\|u(t)\|_{L^{\infty}}\sim \frac{1}{(T-t)^{\frac{2\alpha\ell}{p-1}}}, \ \ \ell\in \Bbb N^*, \ \ 2\alpha\ell>1
\ee
where $\alpha$ is the phenomenological number \fref{defalpha}. The blow up bubble corresponds, in self-similar renormalized variables,
\be
\label{selfsim}
u(t,x)=\frac{1}{\l(t)^{\frac 2{p-1}}}v(s,z), \ \ z=\frac{r}{\l(t)}, \ \ \l(t)=\sqrt{T-t},\ \ \frac{ds}{dt}=\frac{1}{\l^2(t)},
\ee 
to a profile generated by the singular self-similar solution $R$:
\be
\label{laedignorderexapnsion}
v(s,z)=R(z)+e^{-\l_j s}\psi_j(z)+\rm{lot}
\ee 
where $\l_j$ is the $j$-th, $j=j(\ell)$, strictly positive eigenvalue with eigenvector $\psi_j$ of the linearized operator $H_R$. 
The decomposition \fref{laedignorderexapnsion} is singular at the origin and, in particular, does not readily
 imply the $L^{\infty}$ control \fref{estlinfty}. It is merely designed to capture the behavior of the solution tail, 
while the leading order of the solution near the origin is given by a renormalized smooth radial solitary wave $Q(r)$ solving 
$$\Delta Q+Q^p=0, \ \ Q(0)=1.$$ 

The situation was clarified in the series of works by Matano and Merle \cite{MaM1,MaM2} through the proof of two fundamental theorems in the radial setting:
\begin{itemize}
\item For $2^*-1<p<p_{JL}$, only type I \fref{typeone} occurs, with both stable and threshold regimes.
\item For $p>p_{JL}$, type II occurs as a threshold dynamics between type I and global existence. This requires in particular $d\geq 11$, and yields an indirect proof of the existence of type II blow up solutions.
\end{itemize}
We emphasize that an essential tool in the analysis in \cite{MaM1,MaM2} was a construction of a Lyapunov functional based on {\it the precise counting of intersections of a solution with the singular self-similar solution $R$}. 
This feature strongly anchors the analysis to the radial setting and to the use of tools reliant on the maximum principle.\\
Following that, using the maximum principle tools developed in \cite{MaM1,MaM2},  Mizoguchi, in \cite{Mizo1, Mizo2}, has been able  to rigorously implement the formal construction of \cite{Velas} to prove both the existence of solutions with blow up speed \fref{estlinfty} and to give a complete classification of radial type II blow up solutions\footnote{in a suitable class.}. The difficulty here is that the decomposition \fref{laedignorderexapnsion} is fundamentally singular both at infinity, where all terms have infinite energy, and at the origin, where both $R$ and $\psi_j$ are singular\footnote{without an obvious cancellation.}. The whole analysis consists in deriving \fref{laedignorderexapnsion}, first in some weak local $L^2$ sense, and then propagating this weak control to the $L^{\infty}$ topology in a self-similar window 
\be
\label{selfsimialrwindow}
\frac{1}{A(t)}<z<A(t), \ \ \lim_{t\to T}A(t)=+\infty.
\ee
The maximum principle tools developed in \cite{MaM1,MaM2} are once again essential in this analysis and not at all amenable 
to an extension of these results to a problem like NLS, or even the non-radial heat equation.

%%%%%%%%%%%%%%%%%%%%%%%%%%%%%%%%%%%%%%%%%%%%
%%%%%%%%%%%%%%%%%%%%%%%%%%%%%%%%%%%%%%%%%%%%

\subsection{Critical blow up problems}

%%%%%%%%%%%%%%%%%%%%%%%%%%%%%%%%%%%%%%%%%%%%
%%%%%%%%%%%%%%%%%%%%%%%%%%%%%%%%%%%%%%%%%%%%
 
The past ten years has seen remarkable progress on the question of singularity formation for {\it critical problems}, where the scaling symmetry meets a conservation law. For \fref{NLH}, this corresponds to the case $p=2^*-1$. Interestingly enough, even maximum principle techniques were not able to address this case, and despite some formal predictions \cite{FHV}, the rigorous derivation of type II blow up solutions has remained open until very recently.\\
 A new intuition based on Liouville classification theorem and a new set of {\it energy type techniques} have led to the pioneering blow up results on the mass critical (gKdV) \cite{MarM1}, \cite{M1}, \cite{MarM2}, to the new classification results of blow up dynamics near the ground state for the mass critical NLS \cite{MR1}, \cite{MR3}, \cite{MR4}, and more recently to a complete classification of the flow near the ground state for the (gKdV) \cite{MMaR1}, \cite{MMaR2}, \cite{MMaR3}. Energy critical models have also been a source of important progress in connection with the two dimensional critical geometric equations: the wave maps, the Schr\"odinger maps and the parabolic harmonic heat flow, \cite{RS}, \cite{KST}, \cite{NT2}, \cite{RaphRod}, \cite{MRR}, \cite{RSc1}, \cite{RSc2}. New fundamental tools have been developed for the construction of energy critical blow up 
 solutions, in settings where even an existence of singular dynamics had been mostly unknown, and for the analysis of their stability/finite codimensional instability. A {\it continuum} of blow up rates were constructed in \cite{KST} for the wave map problem, and in \cite{MMaR1} for gKdV, while for the parabolic heat flow, a {\it discrete} sequence of blow up regimes was rigorously obtained in \cite{RSc2}, in agreement with  the formal predictions in \cite{heatflow}.  In the setting of the nonlinear heat equation \fref{NLH}, these techniques have led to the first construction of type II blow up solutions in the energy critical case $p=3,d=4$, \cite{Schweyer}.\\

In all these works, the blow up profile is not given by a stationary self-similar solution to \fref{eqseflsim}, but rather by a soliton, i.e. a {\it smooth stationary or time periodic solution} to the original PDE, for example for the (NLS) equation: 
\be
\label{soliton}
u(t,x)=Q(x)e^{it}, \ \ \Delta Q+Q^p=0.
\ee
The blow up solution then corresponds to a decomposition $$u(t,x)=\frac{1}{\l(t)^{\frac 2{p-1}}}v(s,y)e^{i\gamma(t)}, \ \ y=\frac{x}{\l(t)}, \ \ \frac{ds}{dt}=\frac1{\l^2(t)},$$ with 
\be
\label{decompov}
v(s,y)=Q(y)+\e(s,y), \ \ |\e|\ll 1.
\ee
The blow up rate $\l(t)$ is {\it never} given by the self-similar speed \fref{selfsim}, but by its suitable deformations. The {\it ground state which is a smooth stationary solution}, as opposed to the singular self-similar solution \fref{selfsimilarsolutions}, turns out, after renormalization,  to be the universal attractor of the flow in a suitable topology:
\be
\label{cobervegce}
\lim_{t\uparrow T}\|\nabla^s\e(t)\|_{L^2}=0 \ \ \mbox{for} \ \ s>s_c.
\ee
A robust general strategy for the construction of blow up solutions in the critical regimes emerged from the works \cite{RaphRod}, \cite{MRR}, \cite{RSc1}, \cite{RSc2}, \cite{RSc3}, \cite{MMaR1} and relies on a two step procedure:
\begin{itemize}
 \item Construction of a suitable approximate blow up profile through iterated resolutions of elliptic equations. The "tail computation" allows one to derive formally the blow up speed from the behavior of the tail of a profile at infinity. An essential algebraic fact for the analysis is the asymptotic behavior 
 \be
 \label{asymptoticq}
 Q(r)\sim \frac{1}{r^{c(d)}}
 \ee
 The parameter $c(d)$ drives the derivation of the blow up law (and the possibility of a blow up with $Q$ profile).
 \item Implementation of an energy method to control the full flow via the derivation of "Lyapunov" functionals involving {\it super critical Sobolev norms} adapted to the linearized flow near the ground state, {\it which do not rely on spectral estimates} and may therefore be easily adapted to  various settings\footnote{for example, nonlocal non self-adjoint operators as in \cite{RSc3}, or quasilinear problems in \cite{MRR}.}.
 \end{itemize}
 
 %%%%%%%%%%%%%%%%%%%%%%%%%%%%%%%%%%%%%%%%%%%%%%%%%%

\subsection{Super critical numerology}

%%%%%%%%%%%%%%%%%%%%%%%%%%%%%%%%%%%%%%%%%%%%%%%%%%

Let us now come back to the super critical problem $s_c>1$ and discuss some essential algebraic facts. 
The problem $$\Delta Q+Q^p=0$$ 
admits a one parameter family of smooth 
spherically symmetric solitary waves solutions with the asymptotic behavior
\be
\label{extensionrq}
Q(r)\sim R(r)=\frac{c_{\infty}}{r^{\frac{2}{p-1}}}\ \ \mbox{as}\ \ r\to +\infty,
\ee
 with $c_{\infty}$ given by \fref{selfsimilarsolutions}. A well known characterization of the Joseph-Lundgren exponent \fref{exponentpjl} is given 
 through the positivity of the linearized operator closed to $Q$, see for example \cite{KaS}. Indeed, let $$L_+=-\Delta -pQ^{p-1},$$ then:
 \begin{itemize}
 \item for $2^*-1<p<p_{JL}$, $L_+$ has a non positive eigenvalue with well localized eigenvector;
 \item for $p>p_{JL}$, $L_+$ is strictly lower bounded by the Hardy potential 
 \be
 \label{positivityhh}
 L_+>-\Delta-\frac{(d-2)^2}{4r^2}>0.
 \ee 
 \end{itemize}
  The proof of \fref{positivityhh} relies on a Sturm-Liouville oscillation argument and is related to the asymptotic expansion 
  \be
  \label{expansionq}
  Q(r)=\frac{c_{\infty}}{r^{\frac{2}{p-1}}}+\frac{c_1}{r^{\gamma}}+o\left(\frac1{r^{\gamma}}\right), \ \ c_1\neq 0,
  \ee 
  where 
 \be
\label{defgamma}
\left\{\begin{array}{ll}\gamma=\frac12(d-2-\sqrt{{\rm Discr}})>0, \ \ {\rm Discr}=(d-2)^2-4pc_{\infty}^{p-1}>0\\
p>p_{JL}\ \  \mbox{iff}\ \ {\rm Discr}>0.
\end{array}\right.
\ee
We introduce the phenomenological number 
\be
\label{defalpha}
\alpha=\gamma-\frac{2}{p-1}, \ \ \alpha>2\ \ \mbox{for}\  p>p_{JL},
\ee
 see  Appendix \ref{numero}.
 %%%%%%%%%%%%%%%%%%%%%%%%%%%%%%%%%%%%%%%%%%%%%%%%%%

\subsection{Statement of the result}

%%%%%%%%%%%%%%%%%%%%%%%%%%%%%%%%%%%%%%%%%%%%%%%%%%

Our main claim in this paper is that the asymptotics \fref{expansionq} for $p>p_{JL}$, replaces the expansion \fref{asymptoticq}
in the critical case, are {\it perfectly suitable} for the implementation of the strategy for a construction of a blow bubble solution with profile $Q$. 
The resulting blow up mechanism is type II energy super critical:

\begin{theorem}[Type II blow up for the super critical NLS equation]
\label{thmmain}
Let $d\geq 11$. Let $\alpha$ be given by \fref{defalpha} and assume:
\be
\label{plarge} \left\{\begin{array}{llll}\ p=2q+1,\ \ q\in \Bbb N^*,\\
 p> p_{JL},\\
{\rm Discr}>4\end{array}\right.
\ee 
and
\be
\label{conditionalpha}
\frac \alpha 2\notin\Bbb N, \ \ \frac 12+\frac 12\left(\frac d2-\gamma\right)\notin \Bbb N, \ \ \frac 12+\frac 12\left(\frac d2-\frac 2{p-1}\right)\notin \Bbb N.
\ee
Fix an integer 
\be
\label{codnionrionfeell}
 \ell\in \Bbb N^* \ \ \mbox{with}\ \ \ell>\frac \alpha 2,
 \ee
 and an arbitrary large Sobolev exponent $$s^+\in \Bbb N, \ \ s_+\geq s(\ell)\to +\infty\ \ \mbox{as}\ \ \ell\to +\infty.$$ Then there exists a radially symmetric initial data $u_0(r)\in H^{s_+}(\Bbb R^d,\Bbb C)$ such that the corresponding solution to \fref{nls} blows up in finite time $0<T<+\infty$ via concentration of the soliton profile:
\be
\label{concnenergy}
u(t,r)=\frac{1}{\l(t)^{\frac 2{p-1}}}(Q+\e)\left(\frac{r}{\l(t)}\right)e^{i\gamma(t)}
\ee
with:\\

\noindent{\em (i) Blow up speed}: 
\be
\label{Pexciitedlaw}
\l(t)=c(u_0)(1+o_{t\uparrow T}(1))(T-t)^{\frac{\ell}{\alpha}}, \ \ c(u_0)>0;
\ee
\noindent{\em (ii) Stabilization of the phase}: 
\be
\label{convergncephase}
\gamma(t)\to \gamma(T)\in \Bbb R\ \ \mbox{as}\ \ t\to T;
\ee
\noindent{\em (iii) Asymptotic stability above scaling}: 
\be 
\label{concnenergybis}
\lim_{t\uparrow T}\|\nabla^s\e(t,\cdot)\|_{L^2}=0  \ \ \mbox{for all}\ \ s_c<s\leq s_+;
\ee
\noindent{\em (iv) Boundedness below scaling}:
\be
\label{boundedbelow}
\limsup_{t\uparrow T}\|u(t)\|_{H^s}<+\infty  \ \ \mbox{for all}\ \ 0\leq s<s_c;
\ee
\noindent{\em (v) Behavior of the critical norm}:
\be
\label{beahviorcriitcalnrom}
\|u(t)\|_{\dot{H}^{s_c}}=\left[c_\infty\sqrt{\frac{\ell}{\alpha}}+o_{t\uparrow T}(1)\right]\sqrt{|\log(T-t)|}.
\ee
\end{theorem}

\noindent {\it Comments on Theorem \ref{thmmain}}\\

\noindent{\it 1. On the assumptions on $p$}. The assumption \fref{conditionalpha} is generic but technical and avoids the appearance of logarithmic losses in the sequence of weighted Hardy inequalities which we will use to close our energy estimates. Unlike the situation in the critical case \cite{RaphRod}, \cite{MRR}, we claim that these logarithms are irrelevant in our setting and in this sense the assumption \fref{conditionalpha} could be removed. Regarding the assumption \fref{plarge}, 
${\rm Discr}>4$ is automatic for $d\geq 13$ and $p\geq 3$, or for p large enough in dimensions $d=11,12$. This assumption is relevant only for the asymptotic 
development of the solitary wave \fref{develpopemtn}, and allows for a simple decoupling of the remainder terms. We however claim that it is not essential and we could treat the case $\rm Discr >0$ along similar lines. Finally, the assumption $p=2q+1$ makes the nonlinearity analytic, and in particular allows us to estimate the solution in any homogeneous Sobolev norm $\dot{H}^s$. Given $\ell$ as in the statement of Theorem \ref{thmmain}, we need to control $\dot{H}^{s(\ell)}$ norm of the solution with $$\lim_{\ell\to _\infty} s(\ell)=+\infty.$$ 
Hence,  a $\mathcal C^{\infty}$ regularity of the nonlinearity is required for a statement which holds true for all $\ell$ large enough. 
However, for a given $\ell$, a blow up solution satisfying \fref{Pexciitedlaw} can be constructed for any $p\geq p(\ell)$ large enough using the techniques 
of this paper. Yet, as presented, our analysis does not cover non smooth nonlinearities near the $p_{JL}$ exponent.\\

\noindent{\it 2. On the role of the topology}. We stress that the structure of the blow up solution \fref{concnenergy}, \fref{concnenergybis} is {\it exactly} the same as 
the one obtained in the energy critical case \fref{cobervegce}, see in particular \cite{RaphRod}, \cite{MRR}, \cite{RSc1}. This is quite unexpected and reveals the essential role payed by the topology in which the deformation of the ground state is measured. 

For example, the structure of $Q$ and a theorem from \cite{Planchonetal} ensures that $e^{-itH_Q}$ enjoys standard Strichartz estimates, and hence we expect that $Q$ is stable and in fact asymptotically stable as a solution to \fref{nls} with respect to  {\it well localized} perturbations. 

This was proved using sup-sub solutions for the nonlinear heat equation in \cite{chinese}. A related  phenomenon is the global existence proof by Bejenaru, Tataru \cite{BT} for the energy critical Schr\"odinger map in the vicinity of the ground state harmonic map. However,  since $Q$ has infinite energy from \fref{expansionq}, if the perturbation is well localized then this kind of flow corresponds to {\it infinite energy} solutions. We should also mention here a very recent result of Krieger, Schlag \cite{KSc} on 
a global existence of certain solutions to a supercritical septic wave equation in dimension three, arising from the data with an {\it infinite} scale invariant norm.

On the contrary, the full initial data of Theorem \ref{thmmain} can be taken to be even {\it compactly supported} (and, of course, smooth). 
This means  that the initial perturbation $\e$ to $Q$ must possess 
a far away tail to cancel the slow decay of $Q$ at infinity, and hence ceases to belong to standard spaces in which decay is usually measured. These considerations
necessitate  the need to work with homogeneous high Sobolev norms for which $Q$ has a finite contribution and for which the decomposition \fref{concnenergy} makes complete sense. Let us also note  another unexpected feature: the subcritical conservation laws play essentially {\it no role}  in our analysis. In fact, the whole analysis takes place in the super critical algebra $\dot{H}^{\sigma}\cap\dot{H}^{s_+}$ with $$s_c<\sigma<\frac d2\ll s_+$$ and whether the full solution is or is not of finite energy or mass is irrelevant in the blow up regime under consideration.\\

\noindent{\it 3. On the role of the decay of the ground state}. The {\it tail computation}, initiated in the critical case, allows one to compute explicitly the expected rates of type II blow up directly from the asymptotic expansion of the ground state at spatial infinity, see the strategy of the proof below. It is therefore essential to recall that if $$
 Q(r)\sim \frac{1}{r^{c(d,p)}},\ \ p\geq 2^*-1,$$ then the mapping $$p\to c(d,p)\ \ \mbox{is discontinuous at}\ \ p=2^*-1$$ 
For the heat equation this explains why type II blow up holds in the critical case $p=2^*-1$, \cite{RSc1}, \cite{Schweyer}, ceases to exist for $2^*-1<p<p_{JL}$, \cite{MaM1}, and then exists again for $p>p_{JL}.$\\
 
 \noindent{\it 4. On the manifold construction}. The statement of Theorem \ref{thmmain} can be made more precise. Let $\ell\in \Bbb N^*$ satisfying \fref{codnionrionfeell}, $s_+\gg 1$, then our initial data is of the form 
 \be
 \label{vbjebbeibebe}
 u_0=Q_{b(0),a(0)}+\e_0
 \ee 
 where $Q_{b,a}$ is a deformation of a ground state $Q$ and
 $$a=(a_1,\dots,a_{L_-}), \ \ b=(b_1,\dots,b_{L_+}), \ \ s_+\sim 2L_+\sim 2L_-$$ correspond to possible unstable directions  of the flow in the $\dot{H^{s_+}}$ topology 
 in a suitable neighborhood of $Q$. 
 Fix a low Sobolev exponent $$s_c<\sigma<\frac d2,$$ we will show that for all $\e_0\in \dot{H}^{\sigma}\cap H^{s_+}$ small enough in this topology, for all $(b_{\ell+1}(0),\dots,b_{L_+}(0))\times (a_{k_\ell+1}(0),\dots,a_{L_-}(0)$ small enough, there exists a choice of unstable directions $$(b_2(0),\dots,b_{\ell}(0))\times (a_1(0),\dots,a_{k_\ell}(0))$$ such that the solution arising from initial data \fref{vbjebbeibebe} satisfies the conclusions of Theorem \ref{thmmain}. Here, $k_\ell$ is given by \fref{defkell}. 
 This implies that our blow up solutions are constructed for a  codimension $\ell-1+k_\ell>0$ manifold of initial data.  
 Let us insist that our class of initial data includes in particular {\it compactly supported $C^{\infty}$ initial data}. 
 As is now standard in the field, this manifold is constructed as a $\matchal C^0$ manifold using a soft Brouwer type fixed point argument. 
 This provides a precise count of the number of directions of instability in this type II blow up regime. Constructing a local Lipschitz manifold would require proving an appropriate local uniqueness statement. The recent analysis \cite{SzefKrieger} clearly suggests that once the existence is shown, by a Brouwer type argument, and  with a 
 strong decay on the solution -- as is the case in the setting of Theorem \ref{thmmain} -- then local uniqueness can be obtained by rerunning the machinery on the difference of two solutions, see also \cite{RSz}, \cite{MMaR1}.\\
 
  \noindent{\it 5. On quantization of blow up rates}. The quantization of blow up rates \fref{Pexciitedlaw} is the same as the one obtained in the case of the heat equation through a complete classification theorem in \cite{Mizo2}, see also \cite{RSc2}. In dispersive settings, a continuum of blow up rates can be constructed, \cite{KST}, but they correspond to solutions propagating from non-regular data and are therefore {\it never} $H^{\infty}$. See \cite{MMaR3} for the study of related phenomena. We expect that the quantized rates \fref{Pexciitedlaw} are the building blocks to classify type II blow up of smooth solutions near the ground state for \fref{nls}.\\
  
 \noindent{\it 6. Comparison with the heat equation}. Observe that \fref{concnenergy}, \fref{Pexciitedlaw}, \fref{concnenergybis} imply the rate of blow up
 $$\|u(t)\|_{L^{\infty}}\sim \frac{1}{\l^(t)^{\frac 2{p-1}}}\sim \frac{1}{(T-t)^{\frac{2\alpha \ell}{p-1}}}$$ which, according to \fref{estlinfty}, is the same as 
 for the nonlinear heat equation. Let us however stress that  the decomposition \fref{concnenergy} centered on the solitary wave looks very different from the decomposition \fref{laedignorderexapnsion} centered on the singular self-similar solution. In fact, we claim that the sharp description of the blow up behind \fref{concnenergy} {\it implies a quantized version} of the decomposition \fref{laedignorderexapnsion} in self-similar variables, see the Strategy of the proof below. In other words, our analysis covers, 
 with {\it one set of estimates} relying only on energy methods, both the self-similar zone and the zone near the singular point. This is a substantial clarification of the analysis of type II blow up.\\
 
 \noindent{\it 7. Other super critical blow up for NLS}.
  In the setting of the energy super critical NLS equation, the sole other example of a blow up phenomenon that we are aware of is the construction of {\it standing ring} blow up solutions for the focusing quintic model $p=5$ in all dimensions $d\geq 2$, \cite{Raphduke}, \cite{Raphring}. These solutions emerge from smooth well localized radial data and concentrate on the sphere $r=1$. The behavior of Sobolev norms is very different, in particular for these ring solutions $$\lim_{t\uparrow T}\|u(t)\|_{\dot{H}^s}=+\infty\ \ \mbox{for all}\ \ s>0,$$ which implies that these blow up solutions are very much connected to the mass conservation law. Theorem \ref{thmmain} gives the first result of type II blow up for the energy super critical NLS which, following \cite{MaM1}, \cite{MaM2}, should be understood as a singular regime where according to \fref{boundedbelow}, {\it all norms below scaling remain bounded}.\\
  
 Our approach can be extended to the heat and wave equations, and the radial assumption can be removed. The case of the wave equation will be treated in \cite{collot}.\\

\noindent{\bf Notations}: We collect the main algebraic notations and facts which are used throughout the paper.\\
\noindent\underline{Super critical numerolgy}: Given $d\geq 11$, $p>p_{JL},$ we let:
$$\gamma=\frac12(d-2-\sqrt{{\rm Discr}})>0, \ \ {\rm Discr}=(d-2)^2-4pc_{\infty}^{p-1}>0$$
and 
$$\alpha=\gamma-\frac{2}{p-1}>2,$$
see Appendix \ref{numero}.  We define\footnote{where we recall the definition of the entire part: $E(x)\leq x<E(x)+1$, $E(x)\in \Bbb Z$.}:
\bea
\label{defkzero}
&&\left\{\begin{array}{ll}k_+={\rm E}\left[\frac 12+\frac 12\left(\frac d2-\gamma\right)\right]\geq 1,\\
 \frac 12+\frac 12\left(\frac d2-\gamma\right)=k_++\delta_{k_+}, \ \ 0\leq \delta_{k_+}<1.\end{array}\right.\\
&&
\label{defkone}
\left\{\begin{array}{ll}k_-={\rm E}\left[\frac 12+\frac 12\left(\frac d2-\frac 2{p-1}\right)\right]>1,\\
\frac12+\frac 12\left(\frac d2-\frac 2{p-1}\right)=k_-+\delta_{k_-}, \ \ 0\leq \delta_{k_-}<1.\end{array}\right.
\eea
so that from \fref{conditionalpha}:
$$0<\delta_\pm<1.$$ 
We let \be
\label{defdeltazero}
\delta_p=\max\{\delta_+,\delta_-\}, \ \ 0<\delta_p<1,
\ee
and
\be
\label{defdektak}
\Delta k=k_--k_+\geq 1
\ee
from \fref{defalpha}. We will use the relations 
\be
\label{dgammkrealtion}
\left\{\begin{array}{ll}d-2\gamma-4k_+=4\dk-2\\
d-\frac{4}{p-1}-4k_-=4\dkm-2,\\
\frac{\alpha}{2}-\Delta k=\dkm-\dk.
\end{array}\right.
\ee
We let \be
\label{defkell}
\ell-\frac\alpha 2=k_{\ell}+\delta_{\ell},\ \ k_{\ell}\in \Bbb N, \ \ 0<\delta_{\ell}<1 
\ee
from \fref{conditionalpha}.\\
\noindent\underline{Notations for the analysis}: 
Given a large integer $L_+\gg1$, we let: 
\be
\label{deflmoins}
L_-=L_+-\Delta k
\ee
and define the Sobolev exponent:
\be
\label{calculimportant}
s_+=2k_++2L_++1.
\ee
We define the generator $\Lambda$ of a scaling symmetry  
$$\Lambda u=\frac2{p-1}u+y\cdot\nabla u.$$
Given $b_1>0$, we define:
\be
\label{defbnotbone}
B_0=\frac{1}{\sqrt{b_1}}, \ \ B_1=B_0^{1+\eta}
\ee
where 
\be
\label{defetal}
\eta=\frac{\eta_0}{L_+}, \ \ 0<\eta_0 \ll1.
\ee
We denote:
\bee
&&\mathcal B_d(\Bbb R)=\{x=(x_1,\dots,x_d)\in \Bbb R^d, \ \ \sum_{i=1}^dx_i^2\leq R^2\},\\
&&\mathcal S_d(\Bbb R)=\{x=(x_1,\dots,x_d)\in \Bbb R^d, \ \ \sum_{i=1}^dx_i^2= R^2\}.
\eee
We let the matrix 
\be
\label{defj}
J=\left(\begin{array}{ll}0&-1\\1&0\end{array}\right), \ \ J^2=-\Id=-\left(\begin{array}{ll}1&0\\0&1\end{array}\right).
\ee
For real vectors: $$u=\left|\begin{array}{ll}u_1\\ u_2\end{array}\right., \ \ v=\left|\begin{array}{ll}v_1\\v_2\end{array}\right., \ \ (u,v)=u_1v_1+u_2v_2$$ and for complex valued functions: $$(f,g)=\Re\left(\int_{\Bbb R^d} f\overline{g}\right).$$
The nonlinearity  $$f(u)=u|u|^{p-1}.$$ We define the sequence of iterated derivatives $$D^ku=\left|\begin{array}{ll} \Delta^m u\ \ \mbox{for}\ \ k=2m\\ \pa_y\Delta^m u\ \ \mbox{for}\ \ k=2m+1.\end{array}\right.$$
We let $\chi$ be a smooth radially symmetric cut-off function 
\be
\label{sahepchi}
\chi(x)=\left\{\begin{array}{ll} 1\ \ \mbox{for}\  |x|\leq 1\\ 0\ \ \mbox{for}\ \ |x|\geq 2.
\end{array}\right.
\ee
\noindent\underline{Linearized operator}. Given $\e\in \Bbb C$, we identify 
\be
\label{representatione}
\e=\left|\begin{array}{ll}\Re(\e)\\\Im(\e)\end{array}\right..
\ee 
Near $Q$ the linearization of \fref{nls} generates a linear operator $mathal L$, given in complex variables by 
$$\mathcal L\e=-\Delta\e-\frac{p+1}{2}Q^{p-1}\e-\frac{p-1}{2}Q^{p-1}\overline{\e}$$ or, equivalently, in terms of \fref{representatione}: $$\mathcal L=\left(\begin{array}{ll} L_+&0\\  0 & L_-\end{array}\right)$$ where $$L_+=-\Delta-pQ^{p-1}, \ \ L_-=-\Delta-Q^{p-1}.$$  We let the potentials
\be
\label{wlpusmoinus}
W_+=pQ^{p-1}, \ \ W_-=Q^{p-1},
\ee
and introduce the matrix  operator  \be
\label{defltilde}
\Lt=-J\L=\left(\begin{array}{ll} 0&L_-\\-L_+&0\end{array}
\right),
\ee
adapted to the linearized flow of \fref{nls}  near $Q$
$$i\pa_s\e=\L\e\ \ \mbox{i.e.}\ \ \pa_s\e=\Lt\e.$$  Observe that 
\be
\label{adjointltilde}
\Lt^*=\left(\begin{array}{ll} 0&-L_+\\L_-&0\end{array}\right)=J\Lt J, \ \ (J\Lt)^*=J\Lt.
\ee

%%%%%%%%%%%%%%%%%%%%%%%%%%%
%%%%%%%%%%%%%%%%%%%%%%%%%%%

\subsection{Strategy of the proof}

%%%%%%%%%%%%%%%%%%%%%%%%%%%
%%%%%%%%%%%%%%%%%%%%%%%%%%%

We now give a brief description of the proof of Theorem \ref{thmmain}. We keep the notations and the strategy close to the ones of the critical case, see in particular \cite{RSc2},
with the intent to show the deep unity of the analysis. In what follows, we pick $$\ell\in \Bbb N^*, \ \ \ell>\frac{\alpha}{2}$$ associated with the blow up speed 
\fref{Pexciitedlaw}, and another integer 
$$L_+\gg \ell, \ \ L_-=L_+-\Delta k,$$ related to the regularity of the solution and the construction of suitable Lyapunov functionals.\\

\noindent \underline{\emph{(i) Renormalized flow and iterated resonances}}. Let us look for a modulated solution $u(t,r)$ of \fref{NLH} in the modulated form: 
\be
\label{renofineo}
u(t,r)=v(s,y)e^{i\gamma}, \ \ y=\frac{r}{\lambda(t)},\ \ \frac{ds}{dt}=\frac{1}{\lambda^2(t)}
\ee which leads to the renormalized flow: 
\be
\label{eqselfsimilar}
\pa_sv-i\Delta v+b_1\Lambda v+ia_1v-iv|v|^{p-1}=0, \ \ b_1=-\lsl, \ \ a_1=\gamma_s.
\ee
Assuming that the leading par of the solution is given by the ground state profile\footnote{this is a theorem for type II blow up in the radial case, \cite{MaM1}.}, the remaining linear part of the flow is governed by the matrix Schr\"odinger operator $$
\Lt=\left(\begin{array}{ll} 0&L_-\\-L_+&0\end{array}
\right).
$$ The scaling and phase invariances of the problem induces an explicit resonance\footnote{This is not an eigenvalue since neither $Q$ nor $\Lambda Q$ decay 
sufficiently fast at infinity. In particular, $\Lambda Q\not\in L^2$.}: $$\Lt\left|\begin{array}{ll} \Lambda Q\\Q\end{array}\right.=0.$$ Each component behaves differently at infinity: $$Q\sim \frac{c_{\infty}}{y^{\frac2{p-1}}}$$ and there holds the {\it fundamental} cancellation of the tail at infinity:
\be
\label{slowedecay}
 \Lambda Q\sim \frac{c}{y^{\gamma}}\ \ \mbox{as}\ \ y\to \infty\ \ \mbox{with}\ \ \gamma=\alpha+\frac2{p-1}>2+\frac{2}{p-1}.
 \ee
 We already see here the appearance of the condition $p>p_{JL}$: for $2^*-1<p<p_{JL}$, the asymptotic \fref{slowedecay} is {\it false} and would instead include oscillations\footnote{a simple way of seeing this is to remark that $\gamma$ given by \fref{defgamma} is complex valued.}, see for example \cite{chinese}.\\ 
 We may now compute the kernel of the powers of $\Lt$ through the iterative scheme 
 \be
 \label{deft}
 \Lt \Phi_{k+1,+}=\Phi_{k,+}, \ \ \Phi_{0,+}=\left|\begin{array}{ll} \Lambda Q\\0\end{array}\right., \ \ \Lt \Phi_{k+1,-}=\Phi_{k,-}, \ \ \Phi_{0,-}=\left|\begin{array}{ll} 0\\Q\end{array}\right.
 \ee 
which display a non trivial tail at infinity:
 \be
 \label{tailtk}
 J^k\Phi_{k,+}\sim \left|\begin{array}{ll} c_{k,+}y^{2k-\gamma}\\0\end{array}\right.,  \ \  J^k\Phi_{k,-}\sim \left|\begin{array}{ll} 0\\c_{k,-}y^{2k-\frac{2}{p-1}}\end{array}\right. \ \mbox{for}\ \ y\gg 1.
 \ee 
 Note in passing that the positivity of $L_+$ is equivalent to $$\Lambda Q>0$$ and implies with $L_-Q=0$ the factorization 
 \be
 \label{cbejibjeibeibiebe}
 L_\pm=A_\pm^*A_\pm,\ \ A_+=-\pa_y+\pa_y(\log \Lambda Q), \ \ A_-=-\pa_y+\pa_y(\log Q)
 \ee which simplifies the resolution of $\Lt u=f$ in the radial sector.\\
 
\noindent \underline{\emph{(ii) Tail dynamics}}. 
 We now implement the approach developed in \cite{RaphRod}, \cite{MRR}, \cite{RSc2} and claim that $(\Phi_{k,\pm})_{k\geq 1}$ correspond to unstable directions which can be excited in a universal way to produce the type II blow up solutions. To see this, let us look for a slowly modulated solution to \fref{eqselfsimilar} of the form $v(s,y)=Q_{b(s),a(s)}(y)$ with 
 \be
 \label{defqbintroab}
 b=(b_1,\dots,b_{L_+}),\ \ a=(a_1,\dots,a_{L_-})
 \ee
 \be
 \label{defqbintro}
  \ \ Q_{b,a}=Q(y)+\sum_{k=1}^{L_+}b_k\Phi_{k,+}(y)+\sum_{k=1}^{L_-}a_k\Phi_{k,-}(y)+\sum_{k=2}^{L_\pm+2}S_{k,\pm}(y,a,b)
 \ee
 where we expect the a priori bounds 
 \be
 \label{bjvbbvebebebi}
 b_k\sim b_1^k, \ \ |a_k|\leq b_1^{k+\frac{\alpha}{2}},
 \ee and 
 the improved decay estimates $$|S_{k,+}(y)|\lesssim b_1^ky^{2(k-1)-\gamma}, \ \ |S_{k,-}(y)|\lesssim b_1^{k+\frac\alpha 2}y^{2(k-1)-\frac{2}{p-1}},$$ so that $S_k$ is in some sense homogeneous of degree $k$ in $b_1$ but decays better than $\Phi_k$. The key point is that {\it this improved decay is possible for a specific regime of the  universal dynamical system driving the modes $(b_k)_{1\leq i\leq L_+}\times (a_k)_{1\leq k\leq L_-}$}: this is the tail computation. In particular, the improved decay \fref{defqbintroab} for the $a_k$ parameters is in agreement with the worst decay \fref{tailtk} of $\Phi_{k,-}$, and we bootstrap a regime where the influence of the $a$ terms -i.e. the phase- is of lower order.\\
 Let us now illustrate the tail dynamics. We inject the decomposition \fref{defqbintro} into \fref{eqselfsimilar} and choose the law, i.e. ODE, for $((a_k)_s,(b_k)_s)$ which cancels the leading order term at infinity:\\
  
 \noindent \underline{$O(b_1)$}. We cannot adjust the law of $b_1$ for the first term and obtain from \fref{eqselfsimilar} the equation $$b_1\left(\Lt \Phi_{1,+}-\left|\begin{array}{ll}\Lambda Q\\0\end{array}\right.\right)=0, \ \ \Phi_{1,+}\sim \left|\begin{array}{ll} 0\\ \frac{c_{1,+}}{y^{\gamma-2}}\end{array}\right.\ \ \mbox{as}\ \ y\to +\infty.$$
  \underline{$O(a_1)$}. We similarly cannot adjust the law of $a_1$ for the first term and obtain from \fref{eqselfsimilar} the equation $$a_1\left(\Lt \Phi_{1,-}-\left|\begin{array}{ll}0\\Q\end{array}\right.\right)=0, \ \ \Phi_{1,-}\sim \left|\begin{array}{ll} \frac{c_{1,-}}{y^{\frac{2}{p-1}-2}}\\0\end{array}\right.\ \ \mbox{as}\ \ y\to +\infty.$$
 
 \noindent \underline{$O(b_1^2,b_2)$}. We consider the imaginary part and obtain $$(b_1)_s\Phi_{1,+}+b^2_1\Lambda  \Phi_{1,+}-b_2\Lt \Phi_{2,+}-\Lt S_{2,+}=b_1^2NL_1(\Phi_{1,+},Q)+{\rm lot}$$  where $NL_1(T_1,Q)$ corresponds to nonlinear interaction terms, while the lower order terms come from neglecting some
 additional contributions which arise after the use of the a priori bounds \fref{bjvbbvebebebi}. When considering the far away tail \fref{tailtk}, we have for $y$ large, $$\Lambda \Phi_{1,+}\sim \left(\frac2{p-1}-(\gamma-2)\right)\Phi_{1,+}=\left(2-\alpha\right) \Phi_{1,+}, \ \ \Lt \Phi_{2,+}=\Phi_{1,+}$$ and thus $$(b_1)_s\Phi_{1,+}+b^2_1\Lambda  \Phi_{1,+}-b_2\Lt \Phi_{2,+}\sim ((b_1)_s+\left(2-\alpha\right)b_1^2-b_2)\Phi_{1,+},$$ and hence the leading order growth for $y$ large is cancelled by the choice 
 $$ (b_1)_s+\left(2-\alpha\right)b_1^2-b_2=0.
 $$
  We then solve for $$\Lt S_{2,+}=b^2_1(\Lambda\Phi_{1,+}-(2-\alpha)\Phi_{1,+})-NL(\Phi_{1,+},Q)$$ and check that the far away tail is improved: $$|S_{2,+}|\ll b_1^2 y^{2-\gamma}\ \ \mbox{for}\ \ y\gg1.$$
  
  \noindent \underline{$O(b_1a_1,a_2)$}. We now consider the real part  and obtain to leading order $$(a_1)_s\Phi_{1,-}+a_1b_1\Lambda  \Phi_{1,-}-a_2\Lt \Phi_{2,-}-\Lt S_{2,-}=a_1b_1NL_1(\Phi_{1,+},Q)+{\rm lot}.$$  When considering the far away tail \fref{tailtk}, we have for $y$ large, $$\Lambda \Phi_{1,-}\sim \left[\frac2{p-1}-\left(\frac{2}{p-1}-2\right)\right]\Phi_{1,-}=2\Phi_{1,-}, \ \ \Lt \Phi_{2,-}=\Phi_{1,-}$$ and thus $$(a_1)_s\Phi_{1,-}+b_1a_1\Lambda  \Phi_{1,-}-a_2\Lt \Phi_{2,-}\sim ((a_1)_s+2b_1a_1-a_2)\Phi_{1,-},$$ and hence the leading order growth for $y$ large is cancelled by the choice 
 $$
 (a_1)_s+2b_1a_1-a_2=0.
 $$
  We then solve for $$\Lt S_{2,-}=a_1b_1(\Lambda\Phi_{1,-}-2\Phi_{1,-})-NL(\Phi_{1,-},Q)$$ and check that the far away tail is improved: $$|S_{2,-}|\ll a_1b_1 y^{-\frac{2}{p-1}}\ \ \mbox{for}\ \ y\gg1.$$
  
  \noindent \underline{$O(b_1^{k+1},b_{k+1})$}. At the $k$-th iteration, we obtain an elliptic equation of the form: $$(b_k)_s\Phi_{k,+}+b_1b_k\Lambda  \Phi_{k,+}-b_{k+1}\Lt \Phi_{k,+}-\Lt S_{k+1,+}=b_1^{k+1}NL_k(\Phi_{1,+},\dots, \Phi_{k,+},Q)+\ \ \mbox{lot}.$$ We have from \fref{tailtk} for tails: $$\Lambda \Phi_{k,+}\sim (2k-\alpha)\Phi_{k,+}$$ and therefore:
  $$(b_{k})_s\Phi_{k,+}+b_1b_{k}\Lambda \Phi_{k,+}-b_{k+1}\Lt\Phi_{k+1}\sim ((b_k)_s+(2k-\alpha)b_1b_k-b_{k+1})\Phi_{k,+}.$$ The cancellation of the leading order growth occurs for $$
  (b_{k})_s+(2k-\alpha)b_1b_{k}-b_{k+1}=0.$$ We then solve for the remaining $S_{k+1,+}$ term and check that $S_{k+1,+}\lesssim b_1^{k+1}y^{2k-\gamma}$ for y large.\\
  
    \noindent \underline{$O(b_1a_k,a_{k+1})$}. We obtain along similar lines: $$(a_k)_s\Phi_{k,-}+b_1a_k\Lambda  \Phi_{k,-}-a_{k+1}\Lt \Phi_{k,-}-\Lt S_{k+1,-}=b_1^{k}a_1NL_k(\Phi_{1,-},\dots, \Phi_{k,-},Q)+\ \ \mbox{lot}.$$ We have from \fref{tailtk} for tails: $$\Lambda \Phi_{k,-}\sim 2k\Phi_{k,-}$$ and therefore:
  $$(a_{k})_s\Phi_{k,-}+b_1a_{k}\Lambda \Phi_{k,-}-a_{k+1}\Lt\Phi_{k+1}\sim ((a_k)_s+2kb_1a_k-a_{k+1})\Phi_{k,-}.$$ The cancellation of the leading order growth occurs for $$
  (a_{k})_s+2kb_1a_{k}-a_{k+1}=0.$$ We then solve for the remaining $S_{k+1,-}$ term and check that $S_{k+1,-}\lesssim b_1^{k+1}y^{2k-\frac{2}{p-1}}$ for y large. Note that we neglected here further nonlinear terms in $a$ since $a$ will turn out to be lower order in the regime\footnote{for example $|a_1b_1|\sim b_1^{2+\frac{\alpha}{2}}$ but $a_1^2\lesssim b_1^{2+\alpha}$.} \fref{bjvbbvebebebi}.\\

 \noindent \underline{\emph{(iii) The universal system of ODE's}}. The above approach leads to the universal system of ODE's which we stop after the $(L_+)$-th iterate:
 \be
\label{systdynfundintro}
\left\{\begin{array}{llll}(b_k)_s+\left(2k-\alpha\right)b_1b_k-b_{k+1}=0,  \ \ 1\leq k\leq L_+, \ \ b_{L_++1}\equiv 0,\\
(a_k)_s+2kb_1a_k-a_{k+1}=0,  \ \ 1\leq k\leq L_-, \ \ a_{L_-+1}\equiv 0,\\
-\lsl=b_1,\ \ \gamma_s=a_1,\\
\frac{ds}{dt}=\frac1{\l^2}.
\end{array}\right.
\ee
Unlike the critical case, there is no further logarithmic correction to take into account. The system \fref{systdynfundintro} can be solved in a closed form, and a set of explicit solutions is given by 
\be
\label{covbeibveibvieb}
\left\{\begin{array}{ll} b_j^{e}(s)=\frac{c_j}{s^j} \ \ 1\leq j\leq L_+\\ a_j^e(s)=0, \ \ 1\leq j\leq L_-
\end{array}\right.,  \ \ s>s_0>0,
\ee
 where $$\left\{\begin{array}{lll}c_1=\frac{\ell}{2\ell-\alpha},\\
c_{j+1}=-\frac{\alpha(\ell-j)}{2l-\alpha}c_j, \ \ 1\leq j\leq \ell-1,\\
c_j=0, \ \ j\geq \ell+1
\end{array}\right.,\ \ \ell\in \Bbb N^*, \ \ \ell>\frac{\alpha}{2}.
$$
In the original time variable $t$, this produces $\lambda(t)$ vanishing in finite (blow up) time $T$ with: $$\l(t)\sim (T-t)^{\frac \ell\alpha}.$$ Moreover, the linearized flow of \fref{systdynfundintro} near this solution is explicit and displays $\ell-1$ unstable directions in $b$ and $k_\ell$ unstable directions in $a$, see Lemma \ref{lemmalinear} and Lemma \ref{lemmaak}. Note that $\ell>\frac{\alpha}{2}>1$ and hence type II is always unstable\footnote{On the contrary, the energy critical case treated in \cite{RSc1}, \cite{RSc2} formally corresponds to $\alpha=1$, and hence $\ell=1$ is admissible and generates a {\it stable} type II regime.}.\\

\noindent \underline{\emph{(iv). Decomposition of the flow and modulation equations}}. Let then the approximate solution $Q_{b,a}$ be given by \fref{defqbintro} which by construction generates an approximate solution to the renomalized flow \fref{eqselfsimilar}:
$$\Psi = \pa_sQ_{b,a} -i\Delta Q_{b,a}+ b_1\Lambda Q_{b,a}+ia_1Q_{b,a}-Q_{b,a}|Q_{b,a}|^{p-1}=\rm{Mod}(t)+O(b_1^{2L_++2})$$ where the modulation equation term is
roughly of the form:
$$\rm{Mod}(t) = \sum_{k=1}^{L_+}\left[(b_k)_s+(2k-\alpha)b_1b_k-b_{k+1}\right]\Phi_{k,+}+\sum_{k=1}^{L_-}\left[(a_k)_s+2kb_1a_k-a_{k+1}\right]\Phi_{k,-}.$$ We localize $Q_{b,a}$ in the zone $y\leq B_1$ to avoid the irrelevant growing tails for $y\gg \frac{1}{\sqrt{b_1}}$. We then pick initial data of the form $$u_0(y)=Q_{b,a}(y)+\e_0(y), \ \ \|\e_0(y)\|\ll 1$$ in some suitable sense and with $(b(0),a(0))$ chosen to be close to the date for the exact solution \fref{covbeibveibvieb}. By a standard modulation argument, we dynamically introduce a modulated decomposition of the flow 
\bea
\label{decomntoro}
\nonumber u(t,r)&=&(Q_{b(t),a(t)}+\e)\left(t,\frac{r}{\lambda(t)}\right)e^{i\gamma(t)}\\
&=&\left[(Q_{b(t),a(t)})\left(t,\frac{r}{\lambda(t)}\right)+w(t,r)\right]e^{i\gamma(t)}
\eea
where the $L_++L_-+2$ modulation parameters $(b(t),\l(t),a(t),\gamma(t))$ are chosen in order to manufacture the orthogonality conditions: 
\be
\label{vnekoenono}
(\e(t),\Lt^k\Phi_{M,+})=0, \ \ 0\leq k\leq L_+, \ \ (\e(t),\Lt^k\Phi_{M,-})=0, \ \ 0\leq k\leq L_-.
\ee
Here $\Phi_{M,\pm}(y)$ are some fixed directions depending on a large constant $M$, generating an approximation of the kernel of the powers of $\Lt$, see section \ref{sectionsetup}. This orthogonal decomposition, which for each fixed time t directly follows from the implicit function theorem, now allows us to compute the modulation equations governing the parameters $(b(t),\l(t),a(t),\gamma(t))$. The $Q_{b,a}$ construction produces the expected modulation equations\footnote{see Lemma \ref{modulationequations}.}:
\bea
\label{cnbecbnoenoe}
\nonumber &&\left|\lsl+b_1\right|+|\gamma_s-a_1|+\sum_{i=1}^{L_+}\left|(b_i)_s+(2i-\alpha)b_1b_i-b_{i+1}\right|+\sum_{i=1}^{L_-}\left|(a_i)_s+2ib_1a_i-a_{i+1}\right|\\
&\lesssim & \|\e\|_{loc}+b_1^{L_++\frac32}
\eea where $\|\e\|_{loc}$ measures a {\it spatially localized} norm of the radiation $\e$.\\

 \noindent \underline{\emph{(v). The mixed energy/Morawetz estimate.}} According to \fref{cnbecbnoenoe}, we need to show now that local norms of $\e$ are under control and do not perturb the dynamical system \fref{systdynfundintro}. This is achieved via a high order mixed energy/Morawetz type estimates which in particular provide control
 of the high order Sobolev norms adapted to the linear flow and based on the powers of the linear operator $\Lt$. In turn, the orthogonality conditions \fref{vnekoenono} are sharp enough to ensure the Hardy type coercivity of the {\it iterated} matrix operator:
 $$\mathcal E_{s_+}=(J\Lt\Lt^{k_+L_+}\e,\Lt^{k_++L_+}\e)\gtrsim  \int|\nabla^{s_+}\e|^2+\int \frac{|\e|^2}{1+y^{2s_+}}$$ where $s_+$ is given by \fref{calculimportant}. Here the factorization \fref{cbejibjeibeibiebe} will help simplify the argument. As stated above we can dynamically control this norm thanks to an energy estimate {\it seen on the linearized equation in original variables}, i.e., by working with $w$ in \fref{decomntoro} and not $\e$. This strategy was initiated in \cite{RS}, \cite{RaphRod}, \cite{MRR}, \cite{RSc2}. The outcome is an estimate of the form 
 \be
 \label{vnknvornror}
 \frac{d}{ds}\left\{\frac{\matchal E_{s_+}+b_1\matchal M}{\l^{2(s_+-s_c)}}\right\}\lesssim \frac{b_1^{2L_++1+\delta(d,p)}}{\l^{2(s_+-s_c)}}, \ \ \delta(d,p)>0
 \ee where the right hand side is controlled by the size of the error in the construction of the approximate blow up profile. Here $\mathcal M$ corresponds to an additional Morawetz type term needed to control $L^2$ terms sharply localized on the soliton core. A remarkable algebraic fact is that the corresponding virial type quadratic form is coercive thanks to the fact that $L_->L_+>0$ in $\dot{H}^1$, see \fref{positivityh}. Hence the estimate \fref{vnknvornror} belongs to the class of mixed energy/Morawetz estimates 
 introduced in \cite{RaphRod}, which have been particularly efficient in blow up settings, see in particular \cite{MMaR1}, and which completely avoids the use of spectral tools. We integrate \fref{vnknvornror} in time using the smallness $$b_1|\matchal M|\leq \frac 1{10}\matchal E_{s_+}$$ to estimate in the regime $b_1\sim b_1^e$ given by \fref{covbeibveibvieb}: 
 \be
\label{lossylog}
\int|\nabla^{s_+}\e|^2+\int \frac{|\e|^2}{1+y^{2s_+}}\lesssim \mathcal E_{s_+}\lesssim b_1^{2L_++\delta(d,p)}, \ \ \delta(d,p)>0,
\ee
which is good enough to control local norms in $\e$ and close the modulation equations \fref{cnbecbnoenoe}.\\

 \noindent \underline{\emph{(vi). Control of the nonlinear term and low Sobolev norms.}}
The control of high Sobolev norms alone is however {\it not enough} to control the {\it nonlinear term} and we need a low Sobolev estimate. The bounds 
following from the conservation laws would be too weak at this point, and we will need the fundamental observation that $$s_c=\frac d2-\frac2{p-1}<\frac d2\ll s_+,$$ while 
$\dot{H}^{\frac d2}$ almost embeds into $L^{\infty}$, and hence the space $$\dot{H^{\sigma}}\cap \dot{H}^{s_+},\ \ s_c<\sigma<\frac d2<s_+$$ is an algebra. To close the nonlinear term, it therefore suffices to close an estimate for the low Sobolev norm $\|\nabla^{\sigma}\e\|^2_{L^2}$ for some $s_c<\sigma<\frac d2.$ Let us insist that it is essential that this norm is {\it above scaling}, any norm of $\e$ below scaling  blows up. We then exhibit an energetic Lyapunov functional with the dynamical estimate: 
$$\frac{d}{ds}\left\{\frac{\|\nabla^{\sigma}\e\|_{L^2}^2}{\l^{2(\sigma-s_c)}}\right\}\lesssim \frac{b_1}{\l^{2(\sigma-s_c)}}\left[b_1^{\delta(d,p)}\|\nabla^{\sigma}\e\|_{L^2}^2+b_1^{\sigma-s_c+\delta(d,p)}\right]$$
which upon integration in time yields a bound 
$$ \|\nabla^{\sigma}\e\|_{L^2}^2\lesssim b_1^{\sigma-s_c+\delta(d,p)}, \ \ \delta(d,p)>0$$
which is enough to control of the nonlinear term.\\

\noindent \underline{\emph{(vii). Construction of the $\matchal C^0$ manifold.}}
The above scheme designs a bootstrap regime which traps blow up solutions with speed \fref{Pexciitedlaw}. According to Lemma \ref{lemmalinear}, Lemma \ref{lemmaak}, such a regime displays $k_\ell+\ell-1>0$ unstable modes and one therefore needs to build the associated stable manifold. We do this in a classical way using a Brouwer fixed point type argument as in \cite{CMM}, and the proof of Theorem \ref{thmmain} follows.\\

 \noindent \underline{\emph{(viii). Relation with the decomposition \fref{laedignorderexapnsion}.}} Let us conclude this introduction by making a link between the above construction and the decomposition of previously known type II blow up solutions for the heat equation \fref{laedignorderexapnsion}. For this, let us consider the two changes of variables: 
 \be
 \label{decomputfioef}
 u(t,x)=\frac{1}{\l^{\frac 2{p-1}}}v(s,y)e^{i\gamma(t)}=\frac{1}{\mu^{\frac 2{p-1}}}V(\tau,z)e^{i\gamma(t)}
 \ee
 with$$\left\{\begin{array}{ll} y=\frac{x}{\l(t)}, \ \ \frac{ds}{dt}=\frac{1}{\l^2}, \ \ \l(t)=(T-t)^{\frac{\ell}{\alpha}}\\  z=\frac{x}{\mu(t)}, \ \ \frac{d\tau}{dt}=\frac{1}{\mu^2}, \ \ \mu=\sqrt{T-t}\end{array}\right.$$ where the second decomposition corresponds to the self-similar variables \fref{laedignorderexapnsion} in the approach of Herrero-Velasquez:
 \be
 \label{nebvneovnoneo}
 V(\tau,z)=R(z)+e^{-\l_j \tau}\psi_j(z)+\rm{lot}
 \ee where $\l_j$ is the $j$-th, $j=j(\ell)$, strictly positive eigenvalue with eigenvector $\psi_j$ of the linearized operator $H_R$: $$H_R=-\Delta -i\Lambda -\frac{pc_{\infty}^{p-1}}{r^2}.$$ We now show how our construction and estimates for the renormalized $v$ imply the decomposition \fref{nebvneovnoneo} 
 in the far field in renormalized variables.\\
 We compute $$b_1\sim -\l\l_t\sim (T-t)^{\frac{2\ell}{\alpha}-1}$$ and thus $$z=\frac{\l}{\mu}y=(T-t)^{\frac{\ell}{\alpha}-\frac 12}z\sim \sqrt{b_1}y.$$ We now estimate the leading order term in the decomposition \fref{defqbintro} in the zone $$z\geq 1\ \ \mbox{i.e.}\ \ y\geq B_0=\frac{1}{\sqrt{b_1}}$$  by neglecting:
 \begin{itemize}
 \item  the $a$ terms which are lower order, see \fref{controlunstable}, \fref{improvebounderrorsak};
 \item  the $S$ terms which decay better and hence are lower order for $z\geq 1$;
 \item the $b_k$ terms for $k\geq \ell+1$ which are the stable modes and also turn out to be lower order, see \fref{improvebounderrors}.
 \end{itemize}
Using $$b_k\sim b_k^e\sim \frac{1}{s^k}\sim b_1^k$$ 
this gives the far away development: $$Q_{b,a}\sim Q+\sum_{k=1}^{\ell}b_k\Phi_{k,+}(y)+{\rm lot}= R+ \sum_{k=1}^{\ell}c_kb_1^ki^ky^{2k-\gamma}+{\rm lot}=R(y)+b_1^{\frac\gamma2}\sum_{k=1}^{\ell}c_ki^kz^{2k-\gamma}+{\rm lot},$$  and hence using \fref{decomputfioef} and the fact that $R$ is homogeneous: $$V(\tau,z)=\left(\frac{\mu}{\l}\right)^{\frac 2{p-1}}\left[R(y)+b_1^{\frac\gamma2}\sum_{k=1}^{\ell}c_ki^kz^k+{\rm lot}\right](z)=R(z)+b_1^{\frac\gamma 2}\left(\frac{\mu}{\l}\right)^{\frac 2{p-1}}\left[\sum_{k=1}^{\ell}c_ki^kz^{2k-\gamma}\right]+{\rm lot}.$$ We now compute $$b_1^{\frac\gamma2}\left(\frac{\mu}{\l}\right)^{\frac 2{p-1}}\sim\frac{(T-t)^{\frac\gamma 2\left[\frac{2\ell}{\alpha}-1\right]}}{(T-t)^{\frac1{p-1}\left[\frac{2\ell}{\alpha}-1\right]}}= e^{-\l_{\ell}\tau}, \ \ \l_\ell=\ell-\frac{\alpha}{2},$$ and obtain 
the leading order decomposition in the {\it far away} zone: $$V(\tau,z)=R(z)+e^{-\l_\ell \tau} \psi_\ell(z)+{\rm lot}$$ with $$\psi_\ell(z)=\sum_{k=1}^{\ell}c_ki^kz^{2k-\gamma}, \ \ \l_\ell=\ell-\frac{\alpha}{2}.$$ Now a simple computation, see Appendix \ref{sectionapprnedix}, reveals that $(\l_\ell,\psi_{\ell})$ is an eignevalue-eigenvector
pair for the linearized operator close to the singular self similar solution $R$. The exact same computation can be done for the heat equation, and the conclusion is the following: the {\it singular} decomposition \fref{laedignorderexapnsion} in self similar variables is exactly the long range expansion $y\geq \frac{1}{\sqrt{b_1}}$ of the {\it regular} decomposition \fref{decomntoro} in the regime \fref{Pexciitedlaw}.\\

This paper is organized as follows. In section \ref{linearrr}, we collect the main linear properties on the linearized matrix operator $\Lt$ and its iterates. In section \ref{sectiontwo}, we construct the approximate self-similar solutions $Q_{b,a}$ and obtain sharp estimates on the error term $\Psi$. We also exhibit an explicit solution to the dynamical system \fref{systdynfundintro} and show that it possesses $(\ell+k_\ell-1)$ directions of instability. In section \ref{sectionboot}, we set up the bootstrap argument, Proposition \ref{bootstrap}. In section \ref{sectionmonoton}, we construct the main Lyapunov functionals which rely on a mixed energy/Morawetz computation. In section \ref{sectionfour} we close the bootstrap bounds and build the $\mathcal C^0$ manifold of data satisfying the conclusions of Theorem \ref{thmmain}.\\

\noindent {\bf Acknowledgments}.  Part of this work was completed while P.R. was visiting the MIT, Boston, and the Institut du Non Lineaire, Nice, and he would like to thank both institutions for their kind hospitality. P.R and F.M were supported by the senior ERC grant BLOWDISOL. I.R. was supported in part by the NSF grant DMS-1001500.

%%%%%%%%%%%%%%%%%%%%%%%%%%%
%%%%%%%%%%%%%%%%%%%%%%%%%%%

%%%%%%%%%%%%%%%%%%%%%%%%%%%
%%%%%%%%%%%%%%%%%%%%%%%%%%%

\section{The linearized Hamiltonian and its iterates}
\label{linearrr}
%%%%%%%%%%%%%%%%%%%%%%%%%%%
%%%%%%%%%%%%%%%%%%%%%%%%%%%

We collect in this section the main properties of the linearized Hamiltonian close to $Q$, which are at the heart of both the construction of the approximate blow up profile and the derivation of coercivity properties required for the high Sobolev energy estimates.

%%%%%%%%%%%%%%%%%%%%%%%%%%%

\subsection{The matrix operator}

%%%%%%%%%%%%%%%%%%%%%%%%%%%

By a standard argument, all smooth radially symmetric solutions to 
\be\label{eqaoltions}
\Delta \phi+\phi^p=0
\ee are dilates of a given normalized ground state profile $$\phi(r)=\l^{\frac 2{p-1}}Q(\l r), \ \ \left\{\begin{array}{ll} \Delta Q+Q^p=0\\Q(0)=1\end{array}\right. .$$ 
Let us recall the following Lemma which follows directly from the results in \cite{chinese}, \cite{KaS}:

\begin{lemma}[Structure of the ground state and positivity of $L_\pm$]
\label{lemmasoliton}
Let $p>p_{JL}$, then:\\
{\em (i) Development of the solitary wave profile for $y\geq 1$}: there holds 
\be
\label{develpopemtn}
\forall k\geq 0, \ \ \pa_y^kQ=\pa_y^k\left[R+\frac{a_1}{y^{\gamma}}\right]+O\left(\frac{1}{y^{\gamma+g +k}}\right),\  \ a_1\neq 0,\ \ g>2
\ee
where $R$ is given by \fref{selfsimilarsolutions}.\\
{\em (ii) Degeneracy}: 
\be
\label{estkeydegen}
\Lambda Q=\frac{c}{y^{\gamma}}+O\left(\frac{1}{y^{\gamma+g}}\right) \ \ \mbox{as}\ \ y\to +\infty, \ \ c\neq 0.
\ee
{\em (iii) Positivity of $L_\pm$}: 
\be
\label{positivityh}
L_->L_+> -\Delta +\frac{1}{|y|^2}\left[c_p-\frac{(d-2)^2}{4}\right]>0
\ee
 for some $c_p>0$.\\
 {\em (iv) Positivity of $\Lambda Q$}: 
 \be
 \label{envnoevne}
 \Lambda Q>0.
 \ee
\end{lemma}

\begin{proof}[Proof of Lemma \ref{lemmasoliton}] The positivity \fref{positivityh} for $p>p_{JL}$ and the associated pointwise lower bound follows from a non trivial Sturm-Liouville oscillation argument, see \cite{KaS}. Now from \cite{chinese}, Thm 2.5, there holds the asymptotic expansion for $p>p_{JL}$ and $y\gg1$: 
\be
\label{ceonvnovneo}
Q(r)=\frac{c_{\infty}}{y^{\frac 2{p-1}}}+\frac{a_1}{y^{\gamma}}+O\left(\frac{1}{y^{\gamma+\alpha}}+\frac{1}{y^{\gamma_2}}\right)
\ee 
where $$\gamma_2=\frac{d-2+\sqrt{{\rm Discr}}}2.$$ We recall that $\alpha>2$ and from \fref{plarge}: $$\gamma_2-\gamma=\sqrt{\rm Discr}>2$$ and thus
\be
\label{developppemtnzero}
Q=R+\frac{a_1}{y^{\gamma}}+O\left(\frac{1}{y^{\gamma+g}}\right)\ \ g=\min\{\alpha,\sqrt{\rm Discr}\}>2.
\ee
The fact that the development \fref{developppemtnzero}  propagates to higher derivatives is now a simple consequence of the $Q$ equation \fref{eqaoltions}, this is left to the reader, and \fref{estkeydegen} follows. We finally claim that $a_1\neq 0$. Indeed, otherwise from \fref{ceonvnovneo}: $$\Lambda Q=O\left(\frac{1}{y^{\gamma+\alpha}}+\frac{1}{y^{\gamma_2}}\right),$$ and then the bounds 
\bee
&&d-3-2\gamma_2=-1-\sqrt{\rm Discr}<-1\\
&&d-3-2\gamma-2\alpha=-1+\sqrt{\rm Discr}-2\alpha=-1+\frac{4}{p-1}-(d-2)<-1
\eee
imply
\be
\label{cbeboebbvoeboevb}
\int |\nabla \Lambda Q|^2+\int \frac{|\Lambda Q|^2}{y^2}\lesssim \int (1+y^{d-1-2-2\gamma-2\alpha}+y^{d-1-2-2\gamma_2})dy<+\infty.
\ee
By scaling invariance, $$L_+\Lambda Q=0.$$ Fix a sufficiently large$R$ and let $\chi_R(y)$ be a smooth cut-off function, equal to one for $0\le y\le R$. We have 
$$
L_+ (\chi_R \Lambda Q)\lesssim \left (\frac {|\nabla \Lambda Q|}y + \frac {|\Lambda Q|}{y^2}\right) {\bf 1}_{y\ge R},
$$
which, combined with \eqref{cbeboebbvoeboevb}, implies
$$
\int L_+ (\chi_R \Lambda Q)\cdot (\chi_R \Lambda Q)\lesssim \frac 1{R^\eta}
$$
for some strictly positive $\eta$. On the other hand,  by strict positivity \fref{positivityh} of $L_+$,
$$
\int L_+ (\chi_R \Lambda Q)\cdot (\chi_R \Lambda Q)\ge c \int \frac {(\chi_R \Lambda Q)^2}{y^2}\ge C
$$
for some positive constant $C$ {\it independent} of $R$, which follows since $\Lambda Q$ does not vanish identically on any open set. Contradiction.

 \end{proof}
%%%%%%%%%%%%%%%%%%%%%%%%%%%%%%%%%%%%%%%%

\subsection{Factorization of $L_\pm$} 

%%%%%%%%%%%%%%%%%%%%%%%%%%%%%%%%%%%%%%%%

The positivity \fref{positivityh} implies\footnote{see \cite{MRR} for a similar structure.} the factorization of $L_\pm$.

\begin{lemma}[Factorization of $L_\pm$]
\label{facotrh}
Let 
\be
\label{defpotetnial}
V_+=\pa_y(\log (\Lambda Q)), \ \ V_-=\pa_y(\log Q)
\ee 
and the first order operators $$A_\pm u=-\pa_yu+V_\pm u, \ \ A_\pm^*u=\frac{1}{y^{d-1}}\pa_y(y^{d-1}u)+V_\pm u,$$ then $$L_\pm=A_\pm^*A_\pm.$$ \end{lemma}

\begin{remark} The adjoint operators $A_{\pm}^*$ are defined with respect to the Lebesgue measure $$\int_{y>0}(Au)vy^{d-1}dy=\int_{y>0}u(A^*v)y^{d-1}dy.$$ 
\end{remark}

We collect the following estimate on $V_\pm$ which follow from \fref{develpopemtn}:
\bea
 \label{behavoir}
&&V_+=\frac{\pa_y(\Lambda Q)}{\Lambda Q}=\left\{\begin{array}{ll} O(1)\ \ \mbox{as}\ \ y\to 0\\ -\frac{\gamma}y+O\left(\frac1{y^3}\right)\ \ \mbox{as}\ \ y\to +\infty\end{array}\right.,\\
 \label{behavoirvminus}
&&V_-=\frac{\pa_y Q}{ Q}=\left\{\begin{array}{ll} O(1)\ \ \mbox{as}\ \ y\to 0\\ -\frac{2}{(p-1)y}+O\left(\frac1{y^3}\right)\ \ \mbox{as}\ \ y\to +\infty\end{array}\right.,\\
\label{behaviorvzero}
&&Q^{p-1}=\left\{\begin{array}{ll} O(1)\ \ \mbox{as}\ \ y\to 0\\ \frac{c_\infty^{p-1}}{y^2}+O\left(\frac1{y^4}\right)\ \ \mbox{as}\ \ y\to +\infty.\end{array}\right.
\eea
We also estimate from \fref{develpopemtn} with the notations \fref{wlpusmoinus}: for $y\geq 1$,
\be
\label{behaviorvzeroW}
\pa^j_yW_\pm=O\left(\frac{1}{1+y^{2+j}}\right), \ \ j\geq 0.
\ee

%%%%%%%%%%%%%%%%%%%%%%%%%%%
%%%%%%%%%%%%%%%%%%%%%%%%%%%

\subsection{Inverting $L_+$}

%%%%%%%%%%%%%%%%%%%%%%%%%%%
%%%%%%%%%%%%%%%%%%%%%%%%%%%

We rewrite 
\be
\label{rewriteau}
A_+u=-\Lambda Q\pa_y\left(\frac{u}{\Lambda Q}\right), \ \ A_+^*u=\frac{1}{y^{d-1}\Lambda Q}\pa_y(y^{d-1}\Lambda Q)
\ee and hence the kernels of $A,A^*$ are explicit: 
\be
\label{efinitei}
\left\{\begin{array}{ll}
A_+u=0\ \ \mbox{on}\ \  \ \ \mbox{iff}\ \  u\in \mbox{Span}(\Lambda Q),\\
A_+^*u=0\ \ \mbox{on}\ \ \ \ \mbox{iff}\ \ u\in \mbox{Span}\left(\frac{1}{y^{d-1}\Lambda Q}\right).
\end{array}\right.
\ee
Hence
\be
\label{kernelh}
L_+u=0\ \ \mbox{on} \ \  \ \mbox{iff}\ \ u\in \mbox{Span}(\Lambda Q,\Gamma)
\ee 
with 
\be
\label{Gammaplus} \Gamma_+(y)=\Lambda Q\int_1^y\frac{dx}{x^{d-1}(\Lambda Q(x))^2}
\ee
which satisfies the Wronskian relation 
\be
\label{wronskiangamma}
\Gamma_+'(\Lambda Q)-\Gamma_+(\Lambda Q)'=\frac{1}{y^{d-1}}.
\ee
We observe the behavior 
\be
\label{behviorgmmaorigin}
\Gamma_+\sim \frac{c}{y^{d-2}}\ \ \mbox{as}\ \ y\to 0, \ \ c\neq 0.
\ee 
Moreover, from \fref{estkeydegen}: $$  \int_1^{+\infty}\frac{dx}{x^{d-1}(\Lambda Q(x))^2}\lesssim \int_1^{+\infty}\frac{dx}{x^{d-1-2\gamma}}<+\infty$$ where we used from \fref{defgamma} $d-1-2\gamma=1+\sqrt{{\rm Discr}}>1$. This implies: $$\Gamma_+\sim \frac{c}{y^{\gamma}}\ \ \mbox{as}\ \ y\to +\infty.$$
The explicit knowledge of the Green's functions allows us to introduce the formal inverse
\be
\label{definvesr}
L_+^{-1}f=-\G_+(y)\int_0^y f\Lambda Qx^{d-1}dx +\Lambda Q(y)\int_{0}^y f\G_+ x^{d-1}dx.
\ee 
The factorization of $L_+$ allows to us to compute $L_+^{-1}$ in an elementary two step process\footnote{this will avoid tracking cancellations in the formula \fref{definvesr} induced by the Wronskian relation \fref{wronskiangamma} when estimating the growth of $L_+^{-1} f$.}:

\begin{lemma}[Inversion of $L_+$]
\label{lemmainversion}
Let f be a $\matchal C^{\infty}$ radially symmetric function and $u=L_+^{-1}f$ be given by \fref{definvesr}, then 
\be
\label{invrsionauf}
A_+u=\frac{1}{y^{d-1}\Lambda Q}\int_0^yf\Lambda Qx^{d-1}dx, \ \ u=-\Lambda Q\int_0^y\frac{A_+u}{\Lambda Q}dx.
\ee
\end{lemma}

\begin{proof}[Proof of Lemma \ref{lemmainversion}]
We compute from \fref{wronskiangamma} $$A_+\Gamma_+=-\Gamma'_++\frac{(\Lambda Q)'}{\Lambda Q}\Gamma_+=-\frac{1}{y^{d-1}\Lambda Q}.$$ We therefore apply $A_+$ to \fref{definvesr} and compute using the cancellation $A_+(\Lambda Q)=0$: 
\be
\label{invrsionaufbis}
A_+u=\frac{1}{y^{d-1}\Lambda Q}\int_0^yf\Lambda Qx^{d-1}dx.
\ee 
Hence from \fref{rewriteau}: $$u=-\Lambda Q\int_0^y\frac{A_+u}{\Lambda Q}dx+c_u\Lambda Q.$$ We now estimate at the origin using the formula \fref{invrsionaufbis}, \fref{definvesr} and the behavior \fref{behviorgmmaorigin}: $$|A_+u|\lesssim y, \ \ |u|\lesssim y^2, \ \ \Lambda Q\sim c\neq 0$$ and thus $c_u=0.$ 
\end{proof}

%%%%%%%%%%%%%%%%%%%%%%%%%%%
%%%%%%%%%%%%%%%%%%%%%%%%%%%

\subsection{Inverting $L_-$}

%%%%%%%%%%%%%%%%%%%%%%%%%%%
%%%%%%%%%%%%%%%%%%%%%%%%%%%

We rewrite 
\be
\label{rewriteaumoins}
A_-u=-Q\pa_y\left(\frac{u}{Q}\right), \ \ A_-^*u=\frac{1}{y^{d-1}Q}\pa_y(y^{d-1}Qu)
\ee and hence the kernels of $A_-,A_-^*$ are explicit: 
\be
\label{efiniteimoins}
\left\{\begin{array}{ll} A_-u=0\ \ \mbox{on}\ \ \ \ \mbox{iff}\ \  u\in \mbox{Span}( Q)\\
A_-^*u=0\ \ \mbox{on}\ \ \ \ \ \mbox{iff}\ \ u\in \mbox{Span}\left(\frac{1}{y^{d-1} Q}\right).
\end{array}\right.
\ee
Hence 
\be
\label{kernelhmoins}
L_-u=0\ \ \mbox{on} \ \ \ \  \mbox{iff}\ \ u\in \mbox{Span}(Q,\Gamma_-)
\ee 
with 
\be
\label{Gammamoins} \Gamma_-(y)= Q\int_1^y\frac{dx}{x^{d-1}(Q(x))^2}
\ee
which satisfies the Wronskian relation 
\be
\label{wronskiangammamoins}
\Gamma'_- Q-\Gamma_- Q'=\frac{1}{y^{d-1}}.
\ee
We observe the behavior 
\be
\label{behviorgmmaoriginmoins}
\Gamma_-\sim \frac{c}{y^{d-2}}\ \ \mbox{as}\ \ y\to 0.
\ee 
Moreover, from \fref{estkeydegen}: $$  \int_1^{+\infty}\frac{dx}{x^{d-1}Q(x)^2}\lesssim \int_1^{+\infty}\frac{dx}{x^{d-1-\frac 4{p-1}}}<+\infty$$ where we used from \fref{defgamma} $d-1-\frac 4{p-1}>d-1-2\gamma>1$. This implies: $$\Gamma_-\sim \frac{c}{y^{\frac{2}{p-1}}}\ \ \mbox{as}\ \ y\to +\infty.$$
The explicit knowledge of the Green's functions allows us to introduce the formal inverse
$$(A_-^*)^{-1}f=\frac{1}{y^{d-1} Q}\int_0^yf Qx^{d-1}dx$$ and 
\be
\label{formalinversemoins}
L_-^{-1}f=\left\{\begin{array}{ll} Q\int_y^{+\infty}\frac{(A_-^*)^{-1}f}{Q}dx\ \ \mbox{if}\ \ \int_0^{+\infty}\left|\frac{(A_-^*)^{-1}f}{Q}\right|dx<+\infty,\\
-Q\int_0^y\frac{(A_-^*)^{-1}f}{Q}dx\ \ \mbox{otherwise}.
\end{array}
\right.
\ee

\begin{lemma}[Inversion of $L_-$]
\label{lemmainversionmoins}
Let f be a $\matchal C^{\infty}$ radially symmetric function and $u=L_-^{-1}f$ be given by \fref{formalinversemoins}, then 
\be
\label{invrsionaufminus}
L_-u=f, \ \ A_-u=\frac{1}{y^{d-1} Q}\int_0^yf Qx^{d-1}dx=(A_-^*)^{-1}f.
\ee
\end{lemma}

\begin{proof}[Proof of Lemma \ref{lemmainversionmoins}]  From \fref{rewriteaumoins}, \fref{formalinversemoins}: 
\bee
&&A_-u=-Q\pa_y\left(\frac u Q\right)=(A_-^*)^{-1}f=\frac{1}{y^{d-1} Q}\int_0^yf Qx^{d-1}dx\\
&& L_-u=A_-^*A_-u=\frac{1}{y^{d-1}Q}\pa_y\left(y^{d-1}QA_-u\right)=f
\eee
and \fref{invrsionauf} is proved.
\end{proof}

The definitions \fref{defltilde}, \fref{definvesr}, \fref{formalinversemoins} lead to the formal inverse of $\Lt$:
\be
\label{inversionltilde}
\Lt^{-1}=\left(\begin{array}{ll} 0&-(L_+)^{-1}\\(L_-)^{-1}&0\end{array}\right).
\ee

%%%%%%%%%%%%%%%%%%%%%%%%%%%
%%%%%%%%%%%%%%%%%%%%%%%%%%%

\subsection{Admissible functions}

%%%%%%%%%%%%%%%%%%%%%%%%%%%
%%%%%%%%%%%%%%%%%%%%%%%%%%%

We define a class of admissible functions which display a suitable behavior at infinity:
 
\begin{definition}[Admissible functions]
\label{defadmissible}
{\em 1. Scalar functions}: We say a radially symmetric $f\in \matchal C^{\infty}(\Bbb R^d, \Bbb R)$ is admissible of degree $(j,\pm)\in \Bbb R\times\{-,+\} $ if $f$ and its derivatives admit the bounds: for $y\ge 1$,
\be
\label{taylorexpansionoriginveoneonkv}
\forall k\geq 0, \ \ \left|\pa_y^k f(y) \right|\lesssim \left\{\begin{array}{ll}y^{2j-\gamma-k}\ \ \mbox{for}\ \ (j,+)\\
y^{2p-\frac{2}{j-1}-k}\ \ \mbox{for}\ \ (j,-)
\end{array}\right.
\ee
{\em 2. Vector valued functions}: We say a radially symmetric $\matchal C^{\infty}(\Bbb R^d,\Bbb R^2)$  complex valued function is admissible of degree $(p_1,p_2)\in \Bbb R\times \Bbb R $  if $f$ and its derivatives admit a bound: for $y\ge 1$,
\be
\label{taylorexpansionoriginbis}
\forall k\geq 0, \ \ \left|\pa_y^k\Re f(y) \right|\lesssim y^{2p_1-\gamma-k}, \ \ \left|\pa_y^k\Im f(y) \right|\lesssim y^{2p_2-\frac 2{p-1}-k}.
\ee
\end{definition}

$\Lt$ naturally acts on the class of admissible functions in the following way:

\begin{lemma}[Action of $\Lt,\Lt^{-1}$ on admissible functions]
\label{lemmapropinverse}
Let $f$ be an admissible function of degree $(p_1,p_2)\in \Bbb N^2$, then:\\
\noindent (i) $\Lambda f$ is admissible of degree $(p_1,p_2)$.\\
\noindent (ii) $J\Lt f$ is admissible of degree $(p_1-1,p_2-1)$.\\
\noindent (iii) $\Lt^{-1}(Jf)$ is admissible of degree $(p_1+1,p_2+1)$.\\
\noindent (iv) $J\Lt^{-1}f$ is admissible of degree $(p_1+1,p_2+1)$.
\end{lemma}

\begin{proof}[Proof of Lemma \ref{lemmapropinverse}]{\it Proof of (i)}. This is a direct consequence of \fref{taylorexpansionoriginbis}.\\
\noindent {\it Proof of (ii)}.  Let $f$ be admissible of degree $(p_1,p_2)$. Then $\Lt f$ is a smooth radially symmetric function. For $y\geq 2$, using \fref{defltilde}, 
the decay \fref{behaviorvzeroW} and a simple application of the Leibniz rule imply: for $y\geq 1$,
$$|\pa_y^k\Re(\Lt f)|=|\pa_y^k(L_-\Im f)|\lesssim y^{2p_2-\frac 2{p-1}-2-k},\ \ |\pa_y^k\Im(\Lt f)|=|\pa_y^k(L_+\Re f)|\lesssim y^{2p_1-\gamma-2-k},$$
and (ii) follows.\\
\noindent{\it Proof of (iii)}. We compute from \fref{inversionltilde}:
$$\Lt^{-1}J=\left(\begin{array}{ll} -(L_+)^{-1}&0\\0&(-L_-)^{-1}\end{array}\right).$$Let then $(p_1,p_2)\in \Bbb N^2$, f be admissible of degree $(p_1,p_2)$  and let us show that $u=\Lt^{-1}Jf$ is admissible of degree $(p_1+1,p_2+1)$. Near the origin, $u$ is bounded from \fref{definvesr}, \fref{formalinversemoins}, and hence from $\Lt u =Jf$, $u$ is a smooth radially symmetric function by standard elliptic regularity estimates. Moreover:
$$\Re u=-(L_+)^{-1}\Re f, \ \ \Im u=-(L_-)^{-1}\Im f.$$
\noindent {\it Inversion of $L_+$}: For $y\geq 1$, we use the lower bound from \fref{defgamma} $$d-2-2\gamma=\sqrt{\rm Discr}>0$$ to estimate from \fref{invrsionauf}:
\bea
\label{oneivneonvoe}
\nonumber A_+\Re u&=&-\frac{1}{y^{d-1}\Lambda Q}\int_0^y(\Re f)\Lambda Qx^{d-1}dx=O\left( \frac{1}{y^{d-1-\gamma}}\int_0^y x^{2p_1-2\gamma+d-1}dx\right)\\
&=& O(y^{2p_1+1-\gamma}),
\eea
$$
\Re u=  -\Lambda Q\int_0^y\frac{A_+\Re u}{\Lambda Q}dx=O\left(y^{-\gamma}\int_0^yx^{2p_1+1-\gamma+\gamma}dx\right)=O(y^{2p_1+2-\gamma}).
$$
We conclude from \fref{oneivneonvoe}, \fref{behavoir} $$|\pa_y\Re u|\lesssim y^{2p_1+1-\gamma},\ \ |\pa_y^2\Re u|\lesssim y^{2p_1-\gamma},$$ and then the bound $$|\pa_y^k  \Re u|\lesssim y^{2(p_1+1)-\gamma-k}, \ \ k\geq 0, \ \ y\ge 1$$ easily follows by induction by taking radial derivatives of the relation $L_+(\Re u)=-\Re f.$\\
\noindent {\it Inversion of $L_-$}: Using $$d-2-\frac{4}{p-1}>d-2-2\gamma>0,$$ we estimate from \fref{invrsionaufminus}:
\bea
\label{ceonvebeonoevnoe}
\nonumber A_-\Im u&=&(A^*_-)^{-1}f=-\frac{1}{y^{d-1} Q}\int_0^y(\Im f) Qx^{d-1}dx\\
&=& O\left( \frac{1}{y^{d-1-\frac 2{p-1}}}\int_0^y x^{2p_2-\frac{4}{p-1}+d-1}dx\right)= O(y^{2p_2+1-\frac{2}{p-1}}).
\eea
We now distinguish cases. If $\int_0^{+\infty}\left|\frac{(A_-^*)^{-1}\Im f}{Q}\right|dx<+\infty,$ then from \fref{formalinversemoins}:
$$|\Im u|=\left|Q\int_y^{+\infty}\frac{(A_-^*)^{-1}\Im f}{Q}dx\right|\lesssim y^{-\frac 2{p-1}}\lesssim y^{2(p_2+1)-\frac 2{p-1}},$$ and otherwise from $p_2\geq 0$ and \fref{ceonvebeonoevnoe}:
$$|\Im u|\lesssim \left|Q\int_0^{y}\frac{(A_-^*)^{-1}\Im f}{Q}dx\right|\lesssim y^{-\frac 2{p-1}}\int_0^y x^{2p_2+1}dx\lesssim y^{2(p_2+1)-\frac{2}{p-1}}.$$ 
This implies from \fref{ceonvebeonoevnoe}, \fref{behavoirvminus}:
 $$|\pa_y\Im u|\lesssim y^{2p_2+1-\frac2{p-1}},\ \ |\pa_y^2\Im u|\lesssim y^{2p_2-\frac{2}{p-1}},$$  and then again a simple induction argument by differentiation of the relation $L_-\Im u=-\Im f$ ensures the bound: 
 $$|\pa_y^k  \Im u|\lesssim y^{2(p_2+1)-\frac{2}{p-1}-k}, \ \ k\geq 0, \ \ y\ge 1.$$
  
\noindent{\it Proof of (iv)}. We compute from \fref{inversionltilde}:
$$J\Lt^{-1}=\left(\begin{array}{ll} -(L_-)^{-1}&0\\0&(-L_+)^{-1}\end{array}\right).$$ Let then $(p_1,p_2)\in\Bbb N^2$, f admissible of degree $(p_1,p_2)$  and let us show that $u=J\Lt^{-1}f$ is admissible of degree $(p_2+1,p_1+1)$. From  \fref{definvesr}, \fref{formalinversemoins}, u is radially symmetric and bounded near the origin, and hence from $\Lt u =Jf$, $u$ is a smooth radially symmetric function by standard elliptic regularity estimates. Moreover:
$$\Re u=-(L_-)^{-1}\Re f, \ \ \Im u=-(L_+)^{-1}\Im f.$$
\noindent {\it Inversion of $L_+$}: For $y\geq 1$, we use the lower bound from \fref{defgamma} 
\be
\label{cneonceonoe}
d-2-\frac2{p-1}-\gamma>d-2-2\gamma>0
\ee to estimate from \fref{invrsionauf}:
\bee
\nonumber A_+\Im u&=&-\frac{1}{y^{d-1}\Lambda Q}\int_0^y(\Im f)\Lambda Qx^{d-1}dx=O\left( \frac{1}{y^{d-1-\gamma}}\int_0^y x^{2p_2-\frac{2}{p-1}-\gamma+d-1}dx\right)\\
&=& O(y^{2p_2+1-\frac 2{p-1}}),
\eee
and then using $\gamma>\frac 2{p-1}$ again:
$$
\Im u=  -\Lambda Q\int_0^y\frac{A_+\Im u}{\Lambda Q}dx=O\left(y^{-\gamma}\int_0^yx^{2p_2+1-\frac{2}{p-1}+\gamma}dx\right)=O(y^{2p_2+2-\frac2{p-1}})
$$
and we easily conclude as above: $$|\pa_y^k  \Im u|\lesssim y^{2(p_2+1)-\frac{2}{p-1}-k}, \ \ k\geq 0, \ \ y\ge 1.$$ 
\noindent {\it Inversion of $L_-$}: Using \fref{cneonceonoe}, we estimate from \fref{invrsionaufminus}:
\bea
\label{cnjvnkevnneone}
\nonumber A_-\Re u&=&-\frac{1}{y^{d-1} Q}\int_0^y(\Re f) Qx^{d-1}dx=O\left( \frac{1}{y^{d-1-\frac 2{p-1}}}\int_0^y x^{2p_1-\gamma-\frac{2}{p-1}+d-1}dx\right)\\
&=& O(y^{2p_1+1-\gamma}).
\eea
We now distinguish cases. If $2p_1+1-\gamma+\frac2{p-1}<-1$,
then from \fref{invrsionaufminus}: $$\int_0^{+\infty}\left|\frac{(A_-^*)^{-1}\Re f}{Q}\right|dx=\int_0^{+\infty}\left|\frac{A_-u}{Q}dx\right|\lesssim \int_0^{+\infty} (1+x^{2p_1+1-\gamma+\frac2{p-1}})dx<+\infty$$ and thus from \fref{formalinversemoins}:
\bee
|\Re u|&=&\left|Q\int_y^{+\infty}\frac{(A_-^*)^{-1}\Re f}{Q}dx\right|\lesssim y^{-\frac 2{p-1}}\int_y^{+\infty} x^{2p_1+1-\gamma+\frac2{p-1}}dx\\
& \lesssim & y^{2p_1+2-\gamma}.
\eee
Otherwise, $2p_1+1-\gamma+\frac2{p-1}\geq -1$, but then using $\frac\alpha 2\notin \Bbb N$ from \fref{conditionalpha}: 
\be
\label{cniofhjwpjwp}
2p_1+1-\gamma+\frac2{p-1}=2p_1+1-\alpha> -1.
\ee 
Then either $\int_0^{+\infty}\left|\frac{(A_-^*)^{-1}\Re f}{Q}\right|dx<+\infty$ in which case:
$$|\Re u|=\left|Q\int_y^{+\infty}\frac{(A_-^*)^{-1}\Re f}{Q}dx\right|\lesssim y^{-\frac 2{p-1}}\lesssim y^{2(p_1+1)-\gamma}$$  where we used \fref{cniofhjwpjwp} in the last step, or otherwise from \fref{invrsionaufminus}, \fref{cnjvnkevnneone}: 
$$|\Re u|\lesssim \left|Q\int_0^{y}\frac{(A_-^*)^{-1}\Re f}{Q}dx\right|\lesssim y^{-\frac 2{p-1}}\int_0^y x^{2p_1+1-\gamma+\frac 2{p-1}}dx\lesssim y^{2(p_1+1)-\gamma}.$$ 
We then easily conclude as above: 
$$|\pa_y^k  \Re u|\lesssim y^{2(p_1+1)-\gamma-k}, \ \ k\geq 0, \ \ y\ge 1.$$ 
\end{proof}

%%%%%%%%%%%%%%%%%%%%%%%%%%%%%%%%%%%

\subsection{Generators of the kernel of $\Lt^i$}

%%%%%%%%%%%%%%%%%%%%%%%%%%%%%%%%%%%

Let us give an explicit example of admissible functions which will be essential for the analysis. 

\begin{lemma}[Generators of the kernel of $\Lt^i$]
\label{lemmaradiation}
(i) Let 
\be
\label{deftk}
\Phi_i=\Lt^{-i}\left|\begin{array}{ll} \Lambda Q\\ Q\end{array}\right.,\ \ i\geq 0
\ee
then $J^i\Phi_i$ is admissible of degree $(i,i)$.\\
(ii) Let the sequence 
\be
\label{defthetai}
\Psi_i=\Lambda \Phi_i-J^{-i}D_iJ^i\Phi_i, \ \ i\geq 1,\ \ D_i=\left(\begin{array}{ll}2i-\alpha&0\\0&2i\end{array}\right),
\ee
then $J^i\Psi_i$ is admissible of degree $(i-1,i-1)$.
\end{lemma}

\begin{remark} Equivalently, let the directions
\be
\label{defphiplus}
\Phi_{i,+}=\Lt^{-i}\Phi_{0,+},  \ \ \Phi_{0,+}=\left|\begin{array}{ll} \Lambda Q\\ 0\end{array}\right.,\ \ i\geq 0
\ee
\be
\label{defphiminus}
\Phi_{i,-}=\Lt^{-i}\Phi_{0,-},  \ \ \Phi_{0,-}=\left|\begin{array}{ll}0 \\ Q\end{array}\right.,\ \ i\geq 0.
\ee
A simple computation ensures $$J^{-i}D_iJ^i=\left\{\begin{array}{ll} D_i\ \ \mbox{for}\ \ i=2k\\ \left(\begin{array}{ll} 2i&0\\0&2i-\alpha\end{array}\right)\ \ \mbox{for}\ \ i=2k+1\end{array}\right.,$$  and thus $$\Psi_i=\Psi_{i,+}+\Psi_{i,-}$$ with:
\bea
\label{defpsiplus}&&\Psi_{i,+}=\Lambda \Phi_{i,+}-J^{-i}D_iJ^i\Phi_{i,+}= \Lambda\Phi_{i,+}-(2i-\alpha)\Phi_{i,+}\\
\label{defpsiminus}&&\Psi_{i,-}=\Lambda \Phi_{i,-}-J^{-i}D_iJ^i\Phi_{i,-}=\Lambda\Phi_{i,-}-2i\Phi_{i,-}.
\eea
and $J^i\Psi_{i,+}$ is real valued of degree $(i-1,+)$, and $J^i\Psi_{i,-}$ is imaginary of degree $(i-1,-)$.
\end{remark}

\begin{proof}[Proof of Lemma \ref{lemmaradiation}] {\em Proof of (i)}. $\Phi_0$ is admissible of degree $(0,0)$ from \fref{develpopemtn}. We now proceed by induction, assume the claim for $i$ and prove for $i+1$. By definition, $\Phi_{i+1}=\Lt^{-1} \Phi_{i}$. For $i=2k$, we have by induction: $$J^{i}\Phi_i=J^{2k}\Phi_{2k}=(-1)^{k}\Phi_{2k}$$ is admissible of degree $(2k,2k)$ and hence from Lemma \ref{lemmapropinverse} (iv), $$J^{i+1}\Phi_{i+1}=(-1)^kJ\Lt^{-1} \Phi_{i}$$ is admissible of degree $(i+1,i+1)$.  For $i=2k+1$, we have by induction: $$J^{i}\Phi_i=J^{2k+1}\Phi_{2k+1}=(-1)^{k}J\Phi_{2k+1}$$ is admissible of degree $(2k+1,2k+1)$ and hence from Lemma \ref{lemmapropinverse} (iii), $$J^{i+1}\Phi_{i+1}=(-1)^{k+1}\Lt^{-1} \Phi_{2k+1}=(-1)^{k}\Lt^{-1}(J J\Phi_{2k+1})$$ is admissible of degree $(i+1,i+1)$.\\
{\em Proof of (ii)}. We claim a more precise control of $J^i\Phi_i$ for $y\geq 1$: 
\be
\label{decopmohii}
\forall k\geq0, \ \ \forall i\geq 1, \ \ \left|\pa_y^k\left(J^i\Phi_i-\left|\begin{array}{ll}c_{1,i}y^{2i-\gamma}\\c_{2,i}y^{2i-\frac{2}{p-1}}\end{array}\right.\right)\right|\lesssim \left|\begin{array}{ll}c_{1,i}y^{2(i-1)-\gamma-k}\\c_{2,i}y^{2(i-1)-\frac{2}{p-1}-k}\end{array}\right..
\ee
Assume \fref{decopmohii}, then $\Psi_i$ is radially symmetric and satisfies the bound from \fref{defthetai}: for $y\geq 1$,
\bee
J^i\Psi_i&=&(\Lambda-D_i) J\Phi_i=(\Lambda-D_i)\left|\begin{array}{ll} c_{1,i}y^{2i-\gamma}\\ c_{2,i} y^{2i-\frac{2}{p-1}}\end{array}\right.+O\left(\left|\begin{array}{ll}c_{1,i}y^{2(i-1)-\gamma}\\c_{2,i}y^{2(i-1)}\end{array}\right.\right)\\
& = & O\left(\left|\begin{array}{ll}c_{1,i}y^{2(i-1)-\gamma}\\c_{2,i}y^{2(i-1)}\end{array}\right.\right).
\eee
The control of higher derivatives follows similarly, and hence $J^i\Psi_i$ is admissible of degree $(i-1,i-1)$. We now prove \fref{decopmohii} by induction on $i\geq 1$.\\
\underline{$i=1$}: From \fref{develpopemtn}, there holds for $y\geq 1$: $$\Phi_0=\left|\begin{array}{ll} \Lambda Q\\Q\end{array}\right.=\left|\begin{array}{ll} \frac{c_{1,0}}{y^{\gamma}}+O\left(\frac{1}{y^{\gamma+g}}\right),\ \ g=\min\{\alpha,\sqrt{\rm Discr}\}>2\\ \frac{c_{2,0}}{y^{\frac2{p-1}}}+O\left(\frac{1}{y^{\gamma}}\right).\end{array}\right.$$ We then invert $$\Lt\Phi_1= \left|\begin{array}{ll} L_-\Im \Phi_1\\-L_+\Re\Phi_1\end{array}\right.=\left|\begin{array}{ll} \Lambda Q\\ Q\end{array}\right.$$ From \fref{invrsionaufminus}:
\bee
A_-\Im\Phi_1& =& \frac{1}{y^{d-1}Q}\int_0^y \Lambda Q x^{d-1}Qdx\\
& = & \frac{1}{y^{d-1}Q}\left[O(1)+\int_1^y \left[\frac{c}{x^{\gamma}}+O\left(\frac{1}{x^{\gamma+g}}\right)\right]\left[\frac{c}{x^{\frac{2}{p-1}}}+O\left(\frac{1}{x^{\gamma}}\right)\right]x^{d-1}dx\right]\\
& = &\frac{c}{y^{d-1-\frac{2}{p-1}}\left[1+O\left(\frac{1}{y^{\gamma}}\right)\right]}\left[O(1)+\int_1^yc x^{d-1-\gamma-\frac{2}{p-1}}\left[1+O\left(\frac{1}{x^{g}}\right)\right]dx\right].
\eee
We now use the lower bounds: 
\bee
&&d-1-\gamma-\frac{2}{p-1}-\alpha=d-1-2\gamma=1+\sqrt{\rm Discr}>-1\\
&&d-1-\gamma-\frac{2}{p-1}-\sqrt{\rm Discr}\geq d-1-2\gamma-\sqrt{\rm Discr}=1>-1
\eee
to conclude:
\bee
A_-\Im\Phi_1& =& \frac{1}{y^{d-1-\frac{2}{p-1}}\left[1+O\left(\frac{1}{y^{\gamma}}\right)\right]}cy^{d-1-\gamma-\frac{2}{p-1}+1}\left[1+O\left(\frac{1}{y^{g}}\right)\right]\\
& = & cy^{1-\gamma}\left[1+O\left(\frac{1}{y^{g}}\right)\right]
\eee

This implies using $\alpha>2$: $$\int_0^{+\infty}\frac{|A_-\Im\Phi_1|}{Q}dx\lesssim 1+\int_1^{+\infty}\frac{dx}{x^{\gamma-1-\frac{2}{p-1}}}<+\infty$$ 
and hence from \fref{formalinversemoins}:
\bee
\Im \Phi_1& = & Q\int_y^{+\infty}\frac{A_-\Im\Phi_1}{Q}dx= \frac{c}{y^{\frac{2}{p-1}}\left[1+O\left(\frac{1}{y^{\gamma}}\right)\right]}\int_y^{+\infty}\frac{1}{x^{\gamma-1-\frac{2}{p-1}}}\left[1+O\left(\frac{1}{x^{g}}\right)\right]dx\\
& = & \frac{c}{x^{\gamma-2}}\left[1+O\left(\frac{1}{x^{g}}\right)\right]= \frac{c}{x^{\gamma-2}}+O\left(\frac{1}{x^{\gamma}}\right)
\eee
from our assumption $g>2$.
Similarily, using \fref{invrsionauf} and since the integral term is the same:
\bee
A_+\Re\Phi_1& =& \frac{-1}{y^{d-1}\Lambda Q}\int_0^y Q x^{d-1}\Lambda Qdx=cy^{1-\frac2{p-1}}\left[1+O\left(\frac{1}{y^{g}}\right)\right]
\eee
and hence from \fref{invrsionauf}:
\bee
\Re\Phi_1& =&- \Lambda Q\int_0^{y}\frac{A_+\Re\Phi_1}{\Lambda Q}dx\\
&=& \frac{c}{y^{\gamma}}\left[1+O\left(\frac{1}{y^{g}}\right)\right]\left[O(1)+\int_1^y x^{1-\frac2{p-1}+\gamma}\left[1+O\left(\frac{1}{y^{g}}\right)\right]dx\right]\\
& = & cy^{2-\frac2{p-1}}\left[1+O\left(\frac{1}{y^{g}}\right)\right]=cy^{2-\frac2{p-1}}+O\left(y^{-\frac{2}{p-1}}\right)
\eee
where we used $$1+\gamma-\frac{2}{p-1}-g=1+\alpha-g\geq 1>-1$$ 
and $g>2$.
The bound \fref{decopmohii} for $i=1$ now easily follows by differentation in space.

\underline{$i\to i+1$} We invert $$\Lt\Phi_{i+1}= \left|\begin{array}{ll} L_-\Im \Phi_{i+1}\\-L_+\Re\Phi_{i+1}\end{array}\right.=\left|\begin{array}{ll} \Re\Phi_i\\ \Im\Phi_i \end{array}\right.$$ 
\underline{case $i=2k-1$, $k\geq 2$}. By induction, $J^{i}\Phi_i=(-1)^kJ\Phi_i$ satisfies \fref{decopmohii}. Hence: $$\left|\begin{array}{ll} L_-\Im \Phi_{i+1}\\-L_+\Re\Phi_{i+1}\end{array}\right.=\left|\begin{array}{ll} c_{2,i}y^{2i-\frac{2}{p-1}}+O\left(y^{2i-2-\frac{2}{p-1}}\right)\\ c_{1,i}y^{2i-\gamma}+O\left(y^{2i-2-\gamma}\right) \end{array}\right.$$
From \fref{invrsionaufminus} and using $d-\frac 4{p-1}> d-2\gamma> 2$:
\bee
A_-\Im\Phi_{i+1}& =& \frac{1}{y^{d-1}Q}\int_0^y \Re\Phi_i x^{d-1}Qdx\\
& = & \frac{1}{y^{d-1}Q}\left[O(1)+\int_1^y c\frac{x^{2i-\frac2{p-1}}}{x^{\frac 2{p-1}}}\left[1+O\left(\frac{1}{x^2}\right)\right]\left[1+O\left(\frac{1}{x^{\alpha}}\right)\right]x^{d-1}dx\right]\\
& = &\frac{c}{y^{d-1-\frac{2}{p-1}}\left[1+O\left(\frac{1}{y^{\gamma}}\right)\right]}\left[O(1)+\int_1^yc x^{2i+d-1-\frac{4}{p-1}}\left[1+O\left(\frac{1}{x^2}\right)\right]dx\right]\\
& = & \frac{1}{y^{d-1-\frac{2}{p-1}}\left[1+O\left(\frac{1}{y^{\gamma}}\right)\right]}cy^{2i+d-\frac{4}{p-1}}\left[1+O\left(\frac{1}{y^2}\right)\right]\\
& = & cy^{2i+1-\frac{2}{p-1}}\left[1+O\left(\frac{1}{y^{2}}\right)\right].
\eee
Since $2i+1>1$, $\int _0^{+\infty}\left|\frac{A_-\Im\Phi_1}{Q}\right|dy=+\infty$ 
and\footnote{one easily checks by induction starting from \fref{develpopemtn} with $a_1\neq 0$ that the leading order terms in \fref{decopmohii} do not vanish i.e. $c_{1,i},c_{2,i}\neq 0$.}  thus: 
\bee
\Im \Phi_{i+1}&=&-Q\int_0^y\frac{A_-\Im\Phi_{i+1}}{Q}dy=\frac{1}{y^{\frac 2{p-1}}\left[1+O\left(\frac{1}{y^{\alpha}}\right)\right]}\left[O(1)+\int_1^{y}cx^{2i+1}\left[1+O\left(\frac{1}{x^{2}}\right)\right]dx\right]\\
&=&  cy^{2i+2-\frac{2}{p-1}}\left[1+O\left(\frac{1}{y^{2}}\right)\right].
\eee
Similarily, from \fref{invrsionauf}:
\bee
A_+\Re\Phi_{i+1}& =& \frac{1}{y^{d-1}\Lambda Q}\int_0^y \Re\Phi_i x^{d-1}\Lambda Qdx\\
& = & \frac{1}{y^{d-1}\Lambda Q}\left[O(1)+\int_1^y c\frac{x^{2i-\gamma}}{x^{\gamma}}\left[1+O\left(\frac{1}{x^2}\right)\right]x^{d-1}dx\right]\\
& = &\frac{c}{y^{d-1-\gamma}\left[1+O\left(\frac{1}{y^{2}}\right)\right]}\left[O(1)+\int_1^yc x^{2i+d-1-2\gamma}\left[1+O\left(\frac{1}{x^2}\right)\right]dx\right]\\
& = & \frac{1}{y^{d-1-\gamma}\left[1+O\left(\frac{1}{y^{2}}\right)\right]}cy^{2i+d-2\gamma}\left[1+O\left(\frac{1}{y^2}\right)\right]\\
& = & cy^{2i+1-\gamma}\left[1+O\left(\frac{1}{y^{2}}\right)\right],
\eee
and thus: 
\bee
\Re \Phi_{i+1}&=&-\Lambda Q\int_0^y\frac{A_+\Re\Phi_{i+1}}{\Lambda Q}dy=\frac{1}{y^{\gamma}\left[1+O\left(\frac{1}{y^{2}}\right)\right]}\left[O(1)+\int_1^{y}cx^{2i+1}\left[1+O\left(\frac{1}{x^{2}}\right)\right]dx\right]\\
&=&  cy^{2i+2-\gamma}\left[1+O\left(\frac{1}{y^{2}}\right)\right].
\eee
The bound \fref{decopmohii} for $i+1$ now easily follows by differentiation in $y$.\\
\underline{case $i=2k$, $k\geq 1$}. By induction, $J^{i}\Phi_i=(-1)^k\Phi_i$ satisfies \fref{decopmohii}. Hence: $$\left|\begin{array}{ll} L_-\Im \Phi_{i+1}\\-L_+\Re\Phi_{i+1}\end{array}\right.=\left|\begin{array}{ll}  c_{1,i}y^{2i-\gamma}+O\left(y^{2i-2-\gamma}\right)\\c_{2,i}y^{2i-\frac{2}{p-1}}+O\left(y^{2i-2-\frac{2}{p-1}}\right) \end{array}\right.$$
From \fref{invrsionaufminus}and using $d-\gamma-\frac{2}{p-1}>d-2\gamma>2$:
\bee
A_-\Im\Phi_{i+1}& =& \frac{1}{y^{d-1}Q}\int_0^y \Re\Phi_i x^{d-1}Qdx\\
& = & \frac{1}{y^{d-1}Q}\left[O(1)+\int_1^y c\frac{x^{2i-\gamma}}{x^{\frac 2{p-1}}}\left[1+O\left(\frac{1}{x^2}\right)\right]x^{d-1}dx\right]\\
& = &\frac{c}{y^{d-1-\frac{2}{p-1}}\left[1+O\left(\frac{1}{y^{\gamma}}\right)\right]}\left[O(1)+\int_1^yc x^{2i+d-1-\gamma-\frac{2}{p-1}}\left[1+O\left(\frac{1}{x^2}\right)\right]dx\right]\\
& = & \frac{1}{y^{d-1-\frac{2}{p-1}}\left[1+O\left(\frac{1}{y^{\gamma}}\right)\right]}cy^{2i+d-\gamma-\frac{2}{p-1}}\left[1+O\left(\frac{1}{y^2}\right)\right]\\
& = & cy^{2i+1-\gamma}\left[1+O\left(\frac{1}{y^{2}}\right)\right].
\eee
If $2i+1-\gamma+\frac{2}{p-1}<-1$, then $\int_0^{+\infty}\frac{|A_-\Im\Phi_{i+1}|}{Q}dy<+\infty$ and thus:
\bee
\Im\Phi_{i+1}&=&Q\int_y^{+\infty}\frac{A_-\Im\Phi_{i+1}}{Q}dx=Q\int_y^{+\infty}cx^{2i+1-\gamma+\frac{2}{p-1}}\left[1+O\left(\frac{1}{x^{2}}\right)\right]dx\\
&=&cy^{2i+2-\gamma+\frac{2}{p-1}-\frac{2}{p-1}}\left[1+O\left(\frac 1{y^2}\right)\right]=y^{2i+2-\gamma}\left[1+O\left(\frac 1{y^2}\right)\right].
\eee
If $2i+1-\gamma+\frac{2}{p-1}\geq -1$, then $2i+1-\gamma+\frac{2}{p-1}=2i+1-\alpha>-1$ from \fref{conditionalpha}.
 Hence $\int_0^{+\infty}\frac{|A_-\Im\Phi_{i+1}|}{Q}dy=+\infty$ and:
\bee
\Im\Phi_{i+1}&=& -Q\int_0^{y}\frac{A_-\Im\Phi_{i+1}}{Q}dx=-Q\int_0^ycx^{2i+1-\gamma+\frac{2}{p-1}}\left[1+O\left(\frac{1}{x^{2}}\right)\right]dx\\
&=&cy^{2i+2-\gamma+\frac{2}{p-1}-\frac{2}{p-1}}\left[1+O\left(\frac 1{y^2}\right)\right]=y^{2i+2-\gamma}\left[1+O\left(\frac 1{y^2}\right)\right].
\eee
Similarily:
\bee
A_+\Re\Phi_{i+1}& =& \frac{1}{y^{d-1}\Lambda Q}\int_0^y \Re\Phi_i x^{d-1}\Lambda Qdx\\
& = & \frac{1}{y^{d-1}\Lambda Q}\left[O(1)+\int_1^y c\frac{x^{2i-\frac{2}{p-1}}}{x^{\gamma}}\left[1+O\left(\frac{1}{x^2}\right)\right]x^{d-1}dx\right]\\
& = &\frac{c}{y^{d-1-\gamma}\left[1+O\left(\frac{1}{y^{\alpha}}\right)\right]}\left[O(1)+\int_1^yc x^{2i+d-1-\gamma-\frac{2}{p-1}}\left[1+O\left(\frac{1}{x^2}\right)\right]dx\right]\\
& = & \frac{1}{y^{d-1-\gamma}\left[1+O\left(\frac{1}{y^{\gamma}}\right)\right]}cy^{2i+d-\frac{2}{p-1}-\gamma}\left[1+O\left(\frac{1}{y^2}\right)\right]\\
& = & cy^{2i+1-\frac{2}{p-1}}\left[1+O\left(\frac{1}{y^{2}}\right)\right],
\eee
and thus: 
\bee
\Re \Phi_{i+1}&=&-\Lambda Q\int_0^y\frac{A_+\Re\Phi_{i+1}}{\Lambda Q}dy\\
&=& \frac{1}{y^{\gamma}\left[1+O\left(\frac{1}{y^{2}}\right)\right]}\left[O(1)+\int_1^{y}cx^{2i+1+\gamma-\frac{2}{p-1}}\left[1+O\left(\frac{1}{x^{2}}\right)\right]dx\right]\\
&=&  cy^{2i+2-\frac{2}{p-1}}\left[1+O\left(\frac{1}{y^{2}}\right)\right].
\eee
The bound \fref{decopmohii} for $i+1$ now easily follows by differentiation in $y$.
\end{proof}

%%%%%%%%%%%%%%%%%%%%%%%%%%%
%%%%%%%%%%%%%%%%%%%%%%%%%%%

\section{Construction of the approximate profile}
\label{sectiontwo}

%%%%%%%%%%%%%%%%%%%%%%%%%%%
%%%%%%%%%%%%%%%%%%%%%%%%%%%

This section is devoted to the construction of the approximate  blow up profile $Q_{b,a}$ and the study of the associated dynamical system for 
the parameters $b=(b_1,\dots,b_{L_+})$ and $a=(a_1,\dots.a_{L_-})$.

%%%%%%%%%%%%%%%%%%%%%%%%%%%%%%%%%%%%%%%%%%%%%%%%%%%%%%%%%%

\subsection{Slowly modulated blow up profiles and growing tails}

%%%%%%%%%%%%%%%%%%%%%%%%%%%%%%%%%%%%%%%%%%%%%%%%%%%%%%%%%%

We introduce a simple notion of homogeneous admissible function.

\begin{definition}[Homogeneous functions]
\label{defadmissiblehomog}
Given parameters $b=(b_m)_{1\leq k\leq L_+}$, $a=(a_n)_{1\leq n\leq L_-}$, we say a function $S(b,a,y)$ is homogeneous of degree $(p_1,p_2,j,\pm)\in \Bbb N\times \Bbb N\times \Bbb N$ if it is a finite linear combination of monomials 
$$\left[\Pi_{k=1}^{L_+}b^{m_k}_{k}\Pi_{\ell=1}^{L_-}a^{n_\ell}_{\ell}\right]f_{\pm}$$ with $$\sum^{L_+}_{m=1}km_k=p_1,\ \ \sum^{L_-}_{k=1}kn_k=p_2, \ \ (m_k,n_k)\in \Bbb N^2
$$
with $f_{\pm}$ homogeneous of degree $(j,\pm)$ in the sense of Definition \ref{defadmissible}. We set $\rm{deg}(S):=(p_1,p_2,j,
\pm)$. 
\end{definition}

We are now in position to construct a slowly modulated blow up profile as a deformation of the solitary wave.

\begin{proposition}[Construction of the approximate profile]
\label{consprofapproch}
Let $L_+$ a large  integer
\be
\label{defL}
L_+\gg\frac{\alpha}{2}=\frac 12(\gamma-\frac 2{p-1}),
\ee
and $L_-$ be given by \fref{deflmoins}. Let $M>0$ be a large enough universal constant, then there exists a small enough universal constant $b^*(M,L_+)>0$ such that the following holds true. Let two $\mathcal C^1$ maps $$b=(b_j)_{1\leq j\leq L_+}:[s_0,s_1]\mapsto (-b^*,b^*)^{L_+} , \ \ a=(a_j)_{1\leq j\leq L_-}:[s_0,s_1]\mapsto (-b^*,b^*)^{L_-} $$ with a priori bounds on $[s_0,s_1]$: 
\be
\label{aprioirbound}
\left\{\begin{array}{ll} 0<b_1<b^*, \ \ |b_j|\lesssim b_1^j, \ \ 1\leq j\leq L_+\\
 |a_j|\leq b_1^{j+\alpha}\ \ \mbox{for}\ \ 1\leq j\leq L_-.
 \end{array}\right.
\ee
Then there exist homogeneous profiles $$\left\{\begin{array}{l}S_{j,\pm}=S_{j,\pm}(b,a,y), \ \ 2\leq j\leq L_\pm+2\\ S_{1,\pm}=0
\end{array}\right .$$ such that
\be
\label{decompqb}
Q_{b(s),a(s)}(y)= Q(y) + \zeta_{b(s),a(s)}(y)
\ee
with
\be
\label{defalphab}
\zeta_{b,a}(y)=\sum_{j=1}^{L_+}b_j\Phi_{j,+}(y)+\sum_{j=1}^{L_-}a_j\Phi_{j,-}(y)+ \sum_{j=2}^{L_\pm+2}S_{j,\pm}(b,a,y),
\ee
with $\Phi_{j,\pm}$ defined in \eqref{defpsiplus}, \eqref{defpsiminus},
generates an approximate solution to the renormalized flow, see \eqref{eqselfsimilar}:
\be
\label{deferreur}
\pa_sQ_{b,a}-J(\Delta Q_{b,a}+f(Q_{b,a}))+b_1 \Lambda Q_{b,a}+Ja_1Q_{b,a}=\Psi +Mod(t)
\ee
with the following properties:\\
\noindent{\em (i) Modulation equations}:
\bea
\label{defmodtun}
&& Mod(t)=\\
\nonumber&&\sum_{j=1}^{L_+}\left[(b_j)_s+(2j-\alpha)b_1b_j-b_{j+1}\right]\left[\Phi_{j,+}+\sum_{m=j+1}^{L_++2}\frac{\partial S_{m,+}}{\partial b_j}+\sum_{m=j+1}^{L_-+2}\frac{\pa S_{m,-}}{\pa b_j}\right]\\
\nonumber& + &  \sum_{j=1}^{L_-}\left[(a_j)_s+2jb_1a_j-a_{j+1}\right]\left[\Phi_{j,-}+\sum_{m=j+1}^{L_++2}\frac{\pa S_{m,+}}{\pa a_j}+\sum_{m=j+1}^{L_-+2}\frac{\partial S_{m,-}}{\partial a_j}\right]
\eea
where we used the convention $$ \left\{\begin{array}{ll}b_j=0\ \ \mbox{for} \ \ j\geq L_++1\\ a_j=0\ \ \mbox{for}\ \ j\geq L_-+1\end{array}\right.\ \ \mbox{and}\  \ \left\{\begin{array}{ll}S_{1,+}=S_{1,-}=0\\ S_{j,-}=0\ \ \mbox{for}\ \ j\geq L_-+3\end{array}\right..$$
\noindent {\em (ii) Estimate on the profile}: $S_{j,\pm}$ is a finite\footnote{the total number of terms is bounded by $C(p,L_+)<+\infty$.}linear combination of terms $S_{j,\pm}^{(1)}$,$S_{j,\pm}^{(2)}$  
with 
\bea
\label{degsiplus}
&&\left\{\begin{array}{ll}\deg S_{j,+}^{(1)}=(k_1,k_2,j-1,+), \ \ k_1+k_2=j,\\
\deg S_{j,+}^{(2)}=(k_1,k_2,j,+), \ \ k_1+k_2=j,\ \ k_2\geq 1.
\end{array}\right.\\
\label{degsiminus}
&&\left\{\begin{array}{ll}\deg S_{j,-}^{(1)}=(k_1,k_2,j-1,-), \ \ k_1+k_2=j,\ \ k_2\geq 1\\
\deg S_{j,-}^{(2)}=(k_1,k_2,j,-), \ \ k_1+k_2=j,\ \ k_2\geq 2.
\end{array}\right.
\eea
and
\be
\label{degnenarxs}
\left\{\begin{array}{ll}
\frac{\partial S^{(k)}_{j,\pm}}{\partial b_m}=0, \ \ 2\leq j\leq m\leq L_\pm, \ \ 1\leq k\leq 2\\
\frac{\partial S^{(k)}_{j,\pm}}{\partial a_m}=0,\ \ 2\leq j\leq m\leq L_\pm,  \ \ 1\leq k\leq 2
 \end{array}\right .,\\
\ee

\noindent{\em (iii) Estimate on the error $\Psi$}: let $B_1$ be given by \fref{defbnotbone}, then $\forall 0\leq j_+\leq L_+$, there holds a global weighted bound: 
\bea
\label{controleh4erreur}
&&\int_{y\leq 2B_1}(1+y^2)|\Lt^* J\Lt^{k_++j_+}\Psi|^2+\int_{y\leq 2B_1} \frac{|\Psi|^2}{1+y^{4(k_++j_++2)}}\\
\nonumber & \lesssim &b_1^{2j_++4+2(1-\delta_{k_+})-C_{L_+}\eta}.
\eea
{\bf CHECKER LA DEUXIEME NORME QUI
N'EST PAS DEMONTREE}\\
and the improved local control:
\be
\label{fluxcomputationone}
\forall B\geq 1, \ \ \int_{y\leq 2B}(1+y^2)|\Lt^* J\Lt^{k_++j_+}\Psi|^2 \lesssim B^Cb_1^{2L_++6}.
\ee
\end{proposition}

{\bf tout reprendre avec le nouvel alpha}\\

\begin{proof}[Proof of Proposition \ref{consprofapproch}] To ease the notation, we denote $\zeta=\zeta_{a,b}$. We compute from \fref{decompqb}, \fref{deferreur}:
\bea
\label{computationerrorone}
&&\pa_sQ_{b,a}-J(\Delta Q_{b,a}+f(Q_{b,a}))+b_1 \Lambda Q_{b,a}+Ja_1Q_{b,a}\\
\nonumber& = &  \pa_s\zeta -\Lt\zeta+b_1\Lambda\zeta+Ja_1\zeta-J\left[f(Q+\zeta)-f(Q)-f'(Q)\zeta\right]+b_1\Lambda Q+Ja_1Q.
\eea

\noindent {\bf step 1} Computation of the linear term. We compute the linear term from \fref{defalphab} using $\Lt\Phi_{i,\pm}=\Phi_{i-1,\pm}$ for $i\geq 1$:
\bee
A_1& = & \pa_s\zeta -\Lt\zeta+b_1\Lambda\zeta+Ja_1\zeta+b_1\Lambda Q+Ja_1Q\\
& = & \sum_{j=1}^{L_+}(b_j)_s\Phi_{j,+}+b_1b_j\Lambda \Phi_{j,+}+Ja_1b_j\Phi_{j,+}-b_j\Lt\Phi_{j,+}\\
& + & \sum_{j=1}^{L_-}(a_j)_s\Phi_{j,-}+b_1a_j\Lambda \Phi_{j,-}+Ja_1a_j\Phi_{j,-}-a_j\Lt\Phi_{j,-}\\
& + & \sum_{j=2}^{L_\pm+2}\pa_s S_{j,\pm}+b_1\Lambda S_{j,\pm}+Ja_1S_{j,\pm}-\Lt S_{j,\pm}\\
& + & b_1\Lambda Q+Ja_1Q\\
& = & b_1(\Lambda Q-\Phi_{0,+})+a_1(JQ-\Phi_{0,-})\\
& + & \sum_{j=1}^{L_+}[(b_j)_s+(2j-\alpha)b_1b_j-b_{j+1}]\Phi_{j,+} +  \sum_{j=1}^{L_-}[(a_j)_s+2jb_1a_j-a_{j+1}]\Phi_{j,-}\\
& + & \sum_{j=1}^{L_+}\left[b_1b_j\Psi_{j,+}+a_1b_jJ\Phi_{j,+}\right]+\sum_{j=1}^{L_-}\left[b_1a_j\Psi_{j,-}+a_1a_jJ\Phi_{j,-}\right]\\
& + & \sum_{j=2}^{L_\pm+2}\pa_s S_{j,\pm}+b_1\Lambda S_{j,\pm}+Ja_1S_{j,\pm}-\Lt S_{j,\pm}
\eee
where we recall the convention $b_{L_++1}=a_{L_-+1}=0$. We now treat the time dependence using the anticipated approximate modulation equation:
\bee
\pa_sS_{j,\pm}& = & \sum_{m=1}^{L_+}(b_m)_s\frac{\partial S_{j,\pm}}{\partial b_m}+\sum_{m=1}^{L_-}(a_m)_s\frac{\partial S_{j,\pm}}{\partial a_m}\\
& = & \sum_{m=1}^{L_+}((b_m)_s+(2m-\alpha)b_1b_m-b_{m+1})\frac{\partial S_{j,\pm}}{\partial b_m}-\sum_{m=1}^{L_+}((2m-\alpha)b_1b_m-b_{m+1})\frac{\partial S_{j,\pm}}{\partial b_m}\\
& + & \sum_{m=1}^{L_-}((a_m)_s+2mb_1a_m-a_{m+1})\frac{\partial S_{j,\pm}}{\partial a_m}-\sum_{m=1}^{L_-}(2mb_1a_m-a_{m+1})\frac{\partial S_{j,\pm}}{\partial a_m}
\eee 
and thus:
\bea
\label{copmutationaone}
\nonumber A_1 & = & \sum_{j=1}^{L_+}[(b_j)_s+(2j-\alpha)b_1b_j-b_{j+1}]\left[\Phi_{j,+}+\sum_{m=j+1}^{L_++2}\frac{\pa S_{m,+}}{\pa b_j}+\sum_{m=j+1}^{L_-+2}\frac{\pa S_{m,-}}{\pa b_j}\right]\\
\nonumber &+& \sum_{j=1}^{L_-}[(a_j)_s+2jb_1a_j-a_{j+1}]\left[\Phi_{j,-}+\sum_{m=j+1}^{L_++2}\frac{\pa S_{m,+}}{\pa a_j}+\sum_{m=j+1}^{L_-+2}\frac{\pa S_{m,-}}{\pa a_j}\right]\\
& + &  \sum_{j=1}^{L_++1}\left[E_{j+1,+}-\Lt S_{j+1,+}\right]+  \sum_{j=1}^{L_-+1}\left[E_{j+1,-}-\Lt S_{j+1,-}\right]\\
\nonumber & + & (b_1\Lambda +a_1J)\Lambda S_{L_++2,+}\\
\nonumber &-& \sum_{m=1}^{L_+}[(2m-\alpha)b_1b_m-b_{m+1}]\frac{\pa S_{L_++2,+}}{\pa b_m}-\sum_{m=1}^{L_-}[2mb_1a_m-a_{m+1}]\frac{\pa S_{L_++2,+}}{\pa a_m}\\
\nonumber & +& (b_1\Lambda +a_1J)S_{L_-+2,-}\\
\nonumber & - & \sum_{m=1}^{L_+}[(2m-\alpha)b_1b_m-b_{m+1}]\frac{\pa S_{L_-+2,-}}{\pa b_m}-\sum_{m=1}^{L_-}(2mb_1a_m-a_{m+1})\frac{\pa S_{L_-+2,-}}{\pa a_m}
\eea
with for $1\leq j\leq L_++1$:
\bea
\label{defEiplusone}
E_{j+1,+}&=&b_1b_j\Psi_{j,+}+b_1\Lambda S_{j,+}+Ja_1S_{j,+}+Ja_1b_j\Phi_{j,+}+\\
\nonumber & - & \sum_{m=1}^{j-1}\left\{[(2m-\alpha)b_1b_m-b_{m+1}]\frac{\pa S_{j,+}}{\pa b_m}+(2mb_1a_m-a_{m+1})\frac{\pa S_{j,+}}{\pa a_m}\right\}
\eea
and for $1\leq j\leq L_-+1$:
\bea
\label{defEiminusone}
E_{j+1,-}&=&b_1a_j\Psi_{j,-}+J\left[a_1a_j\Phi_{j,-}+a_1S_{j,-}\right]+b_1\Lambda S_{j,-}\\
\nonumber & - & \sum_{m=1}^{j-1}\left\{[(2m-\alpha)b_1b_m-b_{m+1}]\frac{\pa S_{j,-}}{\pa b_m}+(2mb_1a_m-a_{m+1})\frac{\pa S_{j,-}}{\pa a_m}\right\}
\eea
This immediately yields by induction on \fref{degsiplus}, \fref{degsiminus} using Lemma \ref{lemmapropinverse} and Lemma \ref{lemmaradiation} that $E_{j+1,\pm}$ is a finite linear combination of terms $E_{j+1,\pm}^{(1)},E_{j+1,\pm}^{(2)}$ with \bea
\label{degsipluse}
&&\left\{\begin{array}{ll}\deg E_{j+1,+}^{(1)}=(k_1,k_2,j-1,+), \ \ k_1+k_2=j+1,\\
\deg E_{j+1,+}^{(2)}=(k_1,k_2,j,+), \ \ k_1+k_2=j+1,\ \ k_2\geq 1.
\end{array}\right.\\
\label{degsiminuse}
&&\left\{\begin{array}{ll}\deg E_{j+1,-}^{(1)}=(k_1,k_2,j-1,-), \ \ k_1+k_2=j+1\ \ k_2\geq 1\\
\deg E_{j+1,-}^{(2)}=(k_1,k_2,j,-), \ \ k_1+k_2=j+1,\ \ k_2\geq 2.
\end{array}\right.
\eea

\noindent{\bf step 2} Expansion of the nonlinear term. We claim a decomposition
\be
\label{defriplus}
f(Q+\zeta)-f(Q)-f'(Q)\zeta=\sum_{j=2}^{L_++2}R_{j,+}+\sum_{j=2}^{L_-+2}\left(R^{(1)}_{j,-}+R_{j,-}^{(2)}\right)+\mathcal R_1
\ee
where $R_{j,+}$ is a linear combination of terms of degree $$\deg R_{j,+}=(k_1,k_2,j-2,+), \ \ k_1+k_2=j$$ $R_{j,-}^{(1)}$ is a linear combination of terms of degree $$\deg R_{j,-}^{(1)} =(k_1,k_2,j-2,-), \ \ k_1+k_2=j, \ \ k_2\geq 1$$ and $R_{j,-}^{(2)}$ is a linear combination of terms of degree $$\deg R_{j,-}^{(1)} =(k_1,k_2,j-1,-), \ \ k_1+k_2=j, \ \ k_2\geq 2.$$ Moreover, the remainder has a decomposition 
\be
\label{decopmrone}
\mathcal R_1=\mathcal R_{1,+}+\matchal R_{1,-}^{(1)}+\mathcal R_{1,-}^{(2)}
\ee 
where $\mathcal R_{1,+}$ is a linear combination of terms of degree $$\deg \mathcal R_{1,+}=(k_1,k_2,j-2,+), \ \ k_1+k_2\geq L_++3$$ $\mathcal R_{1,-}^{(1)}$ is a linear combination of terms of degree $$\deg \mathcal R_{1,-}^{(1)} =(k_1,k_2,j-2,-), \ \ k_1+k_2\geq L_-+3, \ \ k_2\geq 1$$ and $\mathcal R_{1,-}^{(2)}$ is a linear combination of terms of degree $$\deg\mathcal R_{1,-}^{(1)} =(k_1,k_2,j-1,-), \ \ k_1+k_2\geq L_-+3, \ \ k_2\geq 2.$$ 
{\it Proof of \fref{defriplus}, \fref{decopmrone}}: We expand the nonlinear term using that $p=2q+1$. Let the set $$\mathcal J=\{0\leq j_1\leq q+1, \ \ 0\leq j_2\leq q, \ \  j_1+j_2\geq 2\},$$ then
$$
f(Q+\zeta)-f(Q)-f'(Q)\zeta=(Q+\zeta)^{q+1}(Q+\overline{\zeta})^q=\sum_{j\in \mathcal J} c_{j_1,j_2}Q^{p-(j_1+j_2)}\zeta^{j_1}\overline{\zeta}^{j_2}.$$
Let $(j_1,j_2)\in \matchal J$ and $j=j_1+j_2$, then each monomial in the above decomposition is by construction of $\zeta$ a linear combination of monomials
 \bee
 &&M_\Gamma=\\
 &&Q^{p-j}\Pi_{k=1}^{L_+}(b_k\Phi_{k,+})^{\gamma_{1,k}}\Pi_{k=2}^{L_++2}(S_{k,+}^{(1)})^{\gamma_{2,k}}(S_{k,+}^{(2)})^{\gamma_{3,k}}\Pi_{k=1}^{L_-}(a_k\Phi_{k,-})^{\gamma_{4,k}}\Pi_{k=2}^{L_-+2}(S_{k,-}^{(1)})^{\gamma_{5,k}}(S_{k,-}^{(2)})^{\gamma_{6,k}}.
 \eee
 We note 
 \bee
 &&|J|_1=\sum_{k=1}^{L_+}\gamma_{1,k}+\sum_{k=2}^{L_++2}(\gamma_{2,k}+\gamma_{3,k})+\sum_{k=1}^{L_-}\gamma_{4,k}+\sum_{k=2}^{L_++2}(\gamma_{5,k}+\gamma_{6,k})\\
 &&  |J|_2=\sum_{k=1}^{L_+}k\gamma_{1,k}+\sum_{k=2}^{L_++2}k(\gamma_{2,k}+\gamma_{3,k})+\sum_{k=1}^{L_-}k\gamma_{4,k}+\sum_{k=2}^{L_++2}k(\gamma_{5,k}+\gamma_{6,k}),
 \eee
 and observe the constraint $$|J|_1=j\geq 2$$. Each monomial is a polynomial in $(b,a)$ with
 $$\deg M_\Gamma =(k_1,k_2,S,\pm), \ \ k_1+k_2=|J|_2\geq |J|_1$$ for some degree $S$ which we now compute in various regimes of parameters:\\
 \noindent\underline{case $\gamma_{4,k}=\gamma_{5,k}=\gamma_{6,k}=0$}: in this case, using $|J_1|=j$, the rate $S$ of the asymptotic decay in $y$ is given by
 \bee
S&=&-\frac{2(p-j)}{p-1}+\sum_{k=1}^{L_+}(2k-\gamma)\gamma_{1,k}+\sum_{k=1}^{L_++2}(2(k-1)-\gamma)\gamma_{2,k}+(2k-\gamma)\gamma_{3,k}\\
& \leq & -2+2\frac{j-1}{p-1}+2|J|_2-\gamma|J_1|=2(|J|_2-2)+2+(j-1)\frac{2}{p-1}-(j-1)\gamma-\gamma\\
& = & 2(|J|_2-2)-\gamma-\left\{(j-1)\left[\gamma-\frac{2}{p-1}\right]-2\right\}\leq  2(|J|_2-2)-\gamma
 \eee
 from $$j\geq 2, \ \ \gamma-\frac{2}{p-1}>2.$$ We estimate higher order derivatives similarly and hence:
\be
\label{degreemgamma}
\deg M_\Gamma=(k_1,k_2,|J|_2-2,+), \ \ k_1+k_2=|J|_2.
\ee
\noindent\underline{case $(\gamma_{4,k},\gamma_{5,k},\gamma_{6,k})\neq (0,\dots,0)$}: in this case, we use $y^{-\gamma}\leq y^{-\frac 2{p-1}-2}$ for $y\geq 1$ to estimate:
 \bee
 S&\leq &-\frac{2(p-j)}{p-1}\\
 & + & \sum_{k=1}^{L_+}(2k-\frac2{p-1}-2)\gamma_{1,k}+\sum_{k=2}^{L_++2}(2(k-1)-\frac2{p-1}-2)\gamma_{2,k}+(2k-\frac2{p-1}-2)\gamma_{3,k}\\
 & + & \sum_{k=1}^{L_-}(2k-\frac2{p-1})\gamma_{4,k}+\sum_{k=2}^{L_-+2}(2(k-1)-\frac2{p-1})\gamma_{5,k}+(2k-\frac2{p-1})\gamma_{6,k}\\
& \leq & -2+2\frac{j-1}{p-1}+2|J|_2-\frac2{p-1}|J_1|-2\sum_k\left[\gamma_{1,k}+\gamma_{2,k}+\gamma_{3,k}+\gamma_{5,k}\right]\\&\leq & 2\left(|J|_2-1-\sum_k\left[\gamma_{1,k}+\gamma_{2,k}+\gamma_{3,k}+\gamma_{5,k}\right]\right)-\frac2{p-1}.
 \eee
If one of the $\gamma_{1,k},\gamma_{2,k},\gamma_{3,k},\gamma_{5,k}$ is non zero, then 
$$
S\leq 2(|J|_2-2), \ \ \deg M_\Gamma=(k_1,k_2, |J_2|-2,-), \ \ k_1+k_2=|J|_2, \ \ k_2\geq 1.
$$
Otherwise, $\gamma_{1,k}=\gamma_{2,k}=\gamma_{3,k}=\gamma_{5,k}=0$ and hence $$|J|_1=\gamma_{4,k}+\gamma_{6,k}\geq 2$$ implies  $$\deg M_\Gamma=(k_1,k_2,|J|_2-1,-), \ \ k_1+k_2=|J|_2, \ \ k_2\geq 2.$$
We now sort all the above polynomials in terms of $|J|_2\geq 2$ and obtain \fref{defriplus}.\\

\noindent{\bf step 3} Choice of $S_{j,\pm}$. We compute from the definition \fref{deferreur} of $\Psi$ and the modulation equation \fref{defmodtun}, the linear computation \fref{copmutationaone} and the expansion of the nonlinear term \fref{defriplus}:
\bee
\Psi & = &  \sum_{j=1}^{L_++1}\left[E_{j+1,+}+JR_{j+1,+}-\Lt S_{j+1,+}\right]\\
&+&  \sum_{j=1}^{L_-+1}\left[E_{j+1,-}+JR_{j+1,-}^{(1)}+JR_{j+1,-}^{(2)}-\Lt S_{j+1,-}\right]\\
\nonumber & + & (b_1\Lambda +a_1J)\Lambda S_{L_++2,+}\\
\nonumber &-& \sum_{m=1}^{L_+}[(2m-\alpha)b_1b_m-b_{m+1}]\frac{\pa S_{L_++2,+}}{\pa b_m}-\sum_{m=1}^{L_-}[2mb_1a_m-a_{m+1}]\frac{\pa S_{L_++2,+}}{\pa a_j}\\
\nonumber & +& (b_1\Lambda +a_1J)S_{L_-+2,-}\\
\nonumber & - & \sum_{m=1}^{L_+}[(2m-\alpha)b_1b_m-b_{m+1}]\frac{\pa S_{L_-+2,-}}{\pa b_m}-\sum_{m=1}^{L_-}(2mb_1a_m-a_{m+1})\frac{\pa S_{L_-+2,-}}{\pa a_j}\\
& + & J\mathcal R_1. 
\eee
We therefore solve $$\Lt S_{j+1,+}=E_{j+1,+}+JR_{j+1,+}, \ \  \Lt S_{j+1,-}=E_{j+1,-}+JR_{j+1,-}^{(1)}+JR_{j+1,-}^{(2)}$$ and conclude from \fref{degsipluse}, \fref{degsiminuse}, the properties of the decomposition \fref{defriplus} and the inversion Lemma \ref{lemmapropinverse}, that $S_{j+1,\pm}$ satisfies \fref{degsiplus}, \fref{degsiminus}, \fref{degnenarxs} at the order $j+1$.\\

\noindent{\bf step 4} Estimating the error. It remains to estimate the error:
\bea
\label{formulapsib}
\Psi& = &   (b_1\Lambda +a_1J)\Lambda S_{L_++2,+} +(b_1\Lambda +a_1J)S_{L_-+2,-}\\
\nonumber &-& \sum_{m=1}^{L_+}[(2m-\alpha)b_1b_m-b_{m+1}]\frac{\pa S_{L_\pm+2,\pm}}{\pa b_m}-\sum_{m=1}^{L_-}[2mb_1a_m-a_{m+1}]\frac{\pa S_{L_\pm+2,\pm}}{\pa a_j}\\
\nonumber & + & J\mathcal R_1.
\eea
Let $$k_++j_+=k_-+j_-, \ \ 0\leq j_\pm\leq L_\pm.$$ We start by estimating $S_{L_\pm+2}$ terms and split the contribution according to \fref{degsiplus}, \fref{degsiminus}.\\
\noindent\underline{$S_{2L+2,+}^{(1)}$ terms}. A term 
\bee
\sum_{+}^{(1)}&=&(b_1\Lambda +a_1J)\Lambda S_{L_++2,+}^{(1)} \\
\nonumber &-& \sum_{m=1}^{L_+}[(2m-\alpha)b_1b_m-b_{m+1}]\frac{\pa S^{(1)}_{L_++2,+}}{\pa b_m}-\sum_{m=1}^{L_-}[2mb_1a_m-a_{m+1}]\frac{\pa S^{(1)}_{L_++2,+}}{\pa a_j}
\eee
is of degree $$(k_1,k_2,L_++1,+), \ \ k_1+k_2=L_++3.$$
We recall from \fref{dgammkrealtion} the relation $d-2\gamma-4k_+=4\delta_{k_+}-2$ 
and use the definition \fref{defbnotbone} of $B_1$ 
to estimate:
\bee
&&\int_{y\leq B_1}(1+y^2)|\Lt J\Lt^{k_++j_+}\sum_{+}^{(1)}|^2\lesssim b_1^{2L_++6}\int_{y\leq B_1}y^2|y^{2(L_++1)-\gamma-2(k_++j_++1)}|^2y^{d-1}dy\\
& \lesssim & b_1^{2L_++6}\int_{y\leq B_1} y^{4(L_+-j_+)+d-2\gamma-4k_++1}dy=b_1^{2L_++6}\int_{y\leq B_1} y^{4(L_+-j_++\delta_{k_+})-1}dy\\
& \lesssim & b_1^{(2L_++6)-2(L_+-j_++\delta_{k_+})-C_{L_+}\eta}=b_1^{2j_++4+2(1-\delta_{k_+})-C_{L_+}\eta}
\eee
where we recall $$\eta=\eta(L_+), \ \ 0<\eta\ll 1.$$
\noindent\underline{$S_{2L+2,+}^{(2)}$ terms}. A term 
\bee
\sum_{+}^{(2)}&=&(b_1\Lambda +a_1J)\Lambda S_{L_++2,+}^{(2)} \\
\nonumber &-& \sum_{m=1}^{L_+}[(2m-\alpha)b_1b_m-b_{m+1}]\frac{\pa S^{(2)}_{L_++2,+}}{\pa b_m}-\sum_{m=1}^{L_-}[2mb_1a_m-a_{m+1}]\frac{\pa S^{(2)}_{L_++2,+}}{\pa a_j}
\eee
is of degree $$(k_1,k_2,L_++2,+), \ \ k_1+k_2=L_++3, \ \ k_2\geq 1.$$
We then estimate as above using the gain \fref{aprioirbound} from $k_2\geq 1$:
\bee
&&\int_{y\leq B_1}(1+y^2)|\Lt J\Lt^{k_++j_+}\sum_{+}^{(2)}|^2\lesssim b_1^{2L_++6+\alpha}\int_{y\leq B_1} y^{4(L_+-j_++\delta_{k_+})+3}dy\\
& \lesssim &b_1^{2j_++2+\alpha+2(1-\delta_{k_+})-C_{L_+}\eta}\leq b_1^{2j_++4+2(1-\delta_{k_+})-C_{L_+}\eta}
\eee
from $\alpha>2$.\\
\noindent\underline{$S_{2L+2,-}^{(1)}$ terms}. A term 
\bee
\sum_{-}^{(1)}&=&(b_1\Lambda +a_1J)\Lambda S_{L_-+2,-}^{(1)} \\
\nonumber &-& \sum_{m=1}^{L_+}[(2m-\alpha)b_1b_m-b_{m+1}]\frac{\pa S^{(1)}_{L_-+2,+}}{\pa b_m}-\sum_{m=1}^{L_-}[2mb_1a_m-a_{m+1}]\frac{\pa S^{(1)}_{L_-+2,+}}{\pa a_j}
\eee
is of degree $$(k_1,k_2,L_-+1,-), \ \ k_1+k_2=L_-+3, \ \ k_2\geq 1$$
We define $$k_++j_+=k_-+j_-, \ \ -\Delta k\leq j_-\leq L_-.$$ We then use from \fref{dgammkrealtion} the relation $
d-\frac{4}{p-1}-4k_-=4\delta_{k_-}-2$
and the definition \fref{defbnotbone} of $B_1$ 
to estimate:
\bee
&&\int_{y\leq 2B_1}(1+y^2)|\Lt J\Lt^{k_-+j_-}\sum_{-}^{(1)}|^2\lesssim b_1^{2L_-+6+\alpha}\int_{y\leq B_1}y^2|y^{2(L_-+1)-\frac{2}{p-1}-2(k_-+j_-+1)}|^2y^{d-1}dy\\
& \lesssim & b_1^{2L_-+6+2\Delta k}\int_{y\leq B_1} y^{4(L_--j_-)+d-\frac{4}{p-1}-4k_-+1}dy=b_1^{2L_-+6+\alpha}\int_{y\leq B_1} y^{4(L_--j_-+\delta_{k_-})-1}dy\\
& \lesssim & b_1^{2L_-+6-2(L_--j_-+\delta_{k_-})+\alpha-C_{L_+}\eta}=b_1^{2j_++4+2(1-\delta_{k_-})+\alpha-2\Delta k-C_{L_+}\eta}\\
& = & b_1^{2j_++4+2(1-\dk)-C_{L_+}\eta}
\eee
where we used \fref{dgammkrealtion} in the last step.\\

\noindent\underline{$S_{2L+2,-}^{(2)}$ terms}. A term 
\bee
\sum_{-}^{(2)}&=&(b_1\Lambda +a_1J)\Lambda S_{L_-+2,-}^{(2)} \\
\nonumber &-& \sum_{m=1}^{L_+}[(2m-\alpha)b_1b_m-b_{m+1}]\frac{\pa S^{(2)}_{L_-+2,+}}{\pa b_m}-\sum_{m=1}^{L_-}[2mb_1a_m-a_{m+1}]\frac{\pa S^{(2)}_{L_-+2,+}}{\pa a_j}
\eee
is of degree $$(k_1,k_2,L_-+2,-), \ \ k_1+k_2=L_-+3, \ \ k_2\geq 2,$$ and we therefore estimate as above:
\bee
&&\int_{y\leq B_1}(1+y^2)|\Lt J\Lt^{k_-+j_-}\sum_{-}^{(2)}|^2\\
&\lesssim & b_1^{2L_-+6+2\alpha}\int_{y\leq B_1}y^2|y^{2(L_-+2)-\frac{2}{p-1}-2(k_-+j_-+1)}|^2y^{d-1}dy\\
& \lesssim &b_1^{2j_++2+2(1-\delta_{k_-})+\alpha+\alpha-2\Delta k-C_{L_+}\eta}\lesssim  b_1^{2j_++4+2(1-\dk)-C_{L_+}\eta}
\eee
from $\alpha>2$.\\
The $\mathcal R_1$ term is estimated exactly along the same lines, using the properties of the decomposition \fref{decopmrone}. Moreover since the above estimate does not use any cancellation induced by $\Lt$, the control of $\int_{y\leq 2B_1} \frac{|\Psi|^2}{1+y^{4(k_++j_++2)}}$ can be obtained along the exact same lines as above. 
This concludes the proof of \fref{controleh4erreur}.\\ 
The global bound \fref{fluxcomputationone} is a direct consequence of the homogeneity in $(a,b)$ of the terms in \fref{formulapsib}. 
This concludes the proof of Proposition \ref{consprofapproch}.
\end{proof}

We now proceed to a brute force space localization of the profile $Q_{b,a}$. This is done  to avoid the growth of tails, which becomes irrelevant for $y\gtrsim B_1\gg B_0$. However {\it we do not localize $Q$}, as this would produce uncontrollable error terms. These considerations 
force us to work with norms above scaling, which are finite when evaluated on $Q$. 
 
\begin{proposition}[Localization]
\label{consprofapprochloc}
Let the assumptions of Proposition \ref{consprofapproch} hold true. Assume in addition the a priori bound 
\be
\label{aprioribound}
|(b_1)_s|\lesssim b_1^{2}
\ee
Define the localized profile 
\be
\label{decompqbt}
\tilde{Q}_{b(s),a(s)}(y)= Q + \tilde{\zeta}(y), \ \ \zetat=\chi_{B_1}\zeta,
\ee
i.e.,
\bea
\label{defalphabtilde}
&&\zetat =\sum_{j=1}^{L_+}b_j\Phit_{j,+}+\sum_{j=1}^{L_-}a_j\Phit_{j,-}+ \sum_{j=2}^{L_\pm+2}\St_{j,\pm}\\
\nonumber  \ \ &&\mbox{with}\ \ \Phit_{j,\pm}=\chi_{B_1}\Phi_{j,\pm}, \ \ \St_{j,\pm}=\chi_{B_1}S_{j,\pm}.
\eea
Then
 \be
\label{deferreutilder}
\pa_s\qbt-J[\Delta \qbt+f(\qbt)]+b_1 \Lambda \qbt+Ja_1\qbt=\Psit+ \chi_{B_1}\Mod
\ee
where $\Psit$ satisfies the bounds:\\
{\em (i) Large Sobolev bound}: let $j_++k_+=j_-+k_-$, then for $0\leq j_-\leq L_--1$:
\be
\label{controleh4erreurlocone}
\int (1+y^2)|\Lt ^*J\Lt^{k_++j_+}\Psit |^2+\int \frac{|\Psit|^2}{1+y^{4(k_++j_+)+2}} \lesssim b_1^{2j_++2+2(1-\dk)-C_{L_+}\eta}
\ee
and
\be
\label{controleh4erreurloc}
\int (1+y^2)|\Lt^* J\Lt^{k_++L_+}\Psit |^2+\int \frac{|\Psit|^2}{1+y^{4(k_++j_+)+2}} \lesssim b_1^{2L_++2+2(1-\dk)+2\eta(1-\delta_p)}
\ee
with $\dep$ given by \fref{defdeltazero}.

{\em (ii) Very local bound}: $\forall B\leq \frac{B_1}{2}$, $\forall 0\le j_+\leq L_+$,
\be
\label{fluxcomputationonelocbis}
\int_{y\leq 2B}(1+y^2)|\Lt^* J\Lt^{k_++j_+}\Psi|^2 \lesssim B^Cb_1^{2L_++6}
\ee
{\em (iii) Refined local bound near $B_0$}: $\forall 0\le j_+\leq L_+$,
\bea
\label{nkonenoenve}
&&\int_{y\leq 2B_0}(1+y^2)|\Lt^* J\Lt^{k_++j_+}\Psit|^2+\int_{y\leq 2B_0} \frac{|\Psit|^2}{1+y^{4(k_++j_+)+2}}\\
 \nonumber &\lesssim &b_1^{2j_++4+2(1-\dk)-C_{L_+}\eta}.
\eea
{\em (iii) Small Sobolev bound}: let a universal constant $$\sigma>s_c, \ \ |\sigma-s_c|\ll1,$$ then:
\be
\label{smallsobolevbound}
\|\nabla^{\sigma}\Psit\|^2_{L^2}\leq b_1^{\sigma-s_c+2+\nu_1}
\ee
for some universal constant $\nu_1(d,p)>0$.

\end{proposition}

\begin{remark} Observe the loss in \fref{controleh4erreurloc} with respect to \fref{controleh4erreur}. This is a an unavoidable consequence of the localization of the profile, which generates the worst case bound in \fref{controleh4erreurloc}.
\end{remark}

\begin{remark} We can take $$\nu=\frac{\alpha-2}{2}>0$$ in \fref{smallsobolevbound}.
\end{remark}

\begin{proof}[Proof of Proposition \ref{consprofapprochloc}]

{\bf step 1} Algebraic identity. 
We compute from localization:
\bee
&&\pa_s\qbt-J[\Delta \qbt+f(\qbt)]+b_1 \Lambda \qbt+Ja_1\qbt\\
& = & \chi_{B_1}\left[\pa_s\zeta-J(\Delta \zeta+f(Q_{b,a})-f(Q))+b_1 \Lambda \zeta+Ja_1\zeta\right]+b_1\Lambda Q+Ja_1Q\\
& + & \zeta\left[\pa_s\chi_{B_1}+b_1y\chi_{B_1}'-J\Delta\chi_{B_1}\right]-2J\nabla \zeta\cdot\nabla \chi_{B_1}\\
& = & \chi_{B_1}\left[\Psi+\Mod\right]+(1-\chi_{B_1})(b_1\Lambda Q+a_1JQ)-J\left[f(\qbt)-f(Q)-\chi_{B_1}(f(Q_{b,a})-f(Q))\right]\\
&+&  \zeta\left[\pa_s\chi_{B_1}+b_1y\chi_{B_1}'-J\Delta\chi_{B_1}\right]-2J\zeta' \chi'_{B_1}
\eee 
or equivalently according to \fref{deferreutilder}:
$$\Psit=\chi_{B_1}\Psi+\Psih$$ with 
\bea
\label{defpsihat}
\nonumber \Psih&=&(1-\chi_{B_1})(b_1\Lambda Q+a_1JQ)-J\left[f(\qbt)-f(Q)-\chi_{B_1}(f(Q_{b,a})-f(Q))\right]\\
&+&  \zeta\left[\pa_s\chi_{B_1}+b_1y\chi_{B_1}'-J\Delta\chi_{B_1}\right]-2J \zeta'\chi'_{B_1}.
\eea
{\bf step 2} Estimating integer derivatives. The bound \fref{controleh4erreurloc} for $\chi_{B_1}\Psi$ follows verbatim the proof of \fref{controleh4erreur}, \fref{fluxcomputationone} which, in fact, yield a stronger estimate for $0<\eta<\eta(L_+)$ small enough. We therefore left to estimate the $\Psih$ terms. Note that all terms in \fref{defpsihat} are localized in $B_1\leq y\leq 2B_1$ except the first one\footnote{which is in fact the leading order term.} for which $\mbox{Supp}\left\{(1-\chi_{B_1})(b_1\Lambda Q+a_1JQ)\right\}\subset\{y\geq B_1\}.$ Hence \fref{fluxcomputationonelocbis}, \fref{nkonenoenve} follow directly from \fref{fluxcomputationone}, \fref{controleh4erreur}. In order to treat the far away localized remaining error, we split:
\be
\label{defhatpsibis}
\Psit=\Psit_++\Psit_-, \ \ \Psit_-=a_1(1-\chi_{B_1})JQ,
\ee
and we claim the bounds:
\bea
\label{boundpsitildeplus}
\nonumber &&\int (1+y^2)|\Lt J\Lt^{k_++j_+}\Psit_+ |^2+\int \frac{|\Psit_+|^2}{1+y^{4(k_++j_+)+2}}\\
&\lesssim &\left\{\begin{array}{ll}  b_1^{2j_++2+2(1-\dk)-C_{L_+}\eta}\ \ \mbox{for}\ \ 0\leq j_+\leq L_+-1\\ b_1^{2L_++2+2(1-\dk)+2\eta(1-\delta_p)}\ \ \mbox{for}\ \ j_+=L_+
\end{array}\right.,
\eea
and
\bea
\label{boundpsitildeminus}
&&\nonumber \int (1+y^2)|\Lt J\Lt^{k_-+j_-}\Psit_- |^2+\int \frac{|\Psit_-|^2}{1+y^{4(k_++j_+)+2}}\\
&\lesssim& \left\{\begin{array}{ll}  b_1^{2j_++2+2(1-\dk)-C_{L_+}\eta}\ \ \mbox{for}\ \ 0\leq j_-\leq L_--1\\ b_1^{2L_++2+2(1-\dk)+2\eta(1-\delta_p)}\ \ \mbox{for}\ \ j_-=L_-.
\end{array}\right.,
\eea

\noindent{\it Proof of \fref{boundpsitildeminus}}. Let $j_+\geq 0$. We first observe from \fref{aprioribound} the bound: 
\be
\label{estimatecut}
|\pa_s\chi_{B_1}|\lesssim \frac{|(b_1)_s|}{b_1}|y\chi'_{B_1}|\lesssim b_1{\bf 1}_{B_1\leq y\leq 2B_1}.
\ee
Let now $j_+\geq 0$. We estimate 
$$\forall k\geq 0, \ \ \left|\frac{d^k}{dy^k}\left[(1-\chi_{B_1})\Lambda Q\right]\right|\lesssim \frac{1}{y^{\gamma+k}}{\bf 1}_{y\geq B_1}$$ from which, using \fref{dgammkrealtion} and the definition \fref{defbnotbone} of $B_1$:
  \bee
 && \int(1+y^2)\left|\Lt J\Lt^{k_++j_+}\left(b_1 (1-\chi_{B_1})\Lambda Q\right)\right|^2\lesssim b_1^2\int_{y\geq B_1}\frac{y^{d-1}dy}{y^{4(k_++j_++1)+2\gamma-2}}\\
  & \lesssim & \frac{b_1^2}{B_1^{4j_++4(1-\dk)}}\lesssim b_1^{2j_++2+2(1-\dk)(1+\eta)}.
  \eee

We now split: $$\zeta=\zeta_\pm^{(0)}+\zeta_\pm^{(1)}, \ \ \zeta_+^{(0)}=\sum_{j=1}^{L_+}b_j\Phi_{j,+}, \ \ \zeta^{(0)}_-=\sum_{j=1}^{L_-}a_j\Phi_{j,-}, \ \ \zeta^{(1)}_\pm=\sum_{j=2}^{L_{\pm}+2}S_{j,\pm}.$$
From Lemma \ref{lemmaradiation}: for all $B_1\leq y\leq 2B_1$,
 \be
 \label{estaklphaba}
 \left|\frac{\pa^k}{\pa y^k}\zeta_+^{(0)}\right|\lesssim \sum_{j=1}^{L_+}b_1^jy^{2j-\gamma-k}
 \ee
from which, using \fref{dgammkrealtion}: for all $0\leq j_+\leq L_+$: 
 \bee
 \nonumber &&\int(1+y^2)\left|\Lt J\Lt^{k_++j_+}\left((\pa_s\chi_{B_1})\zeta_+^{(0)}-2J\pa_y\chi_{B_1}\pa_y\zeta_+^{(0)}-J\zeta_+^{(0)}\Delta \chi_{B_1}+b_1\zeta_+^{(0)}y\chi'_{B_1}\right)\right|^2\\
  \nonumber&\lesssim & \sum_{j=1}^{L_+}b_1^2b_1^{2j}\int_{B_1\leq y\leq 2B_1}y^2\left|y^{2j-\gamma-2(k_++j_++1)}\right|^2y^{d-1}dy\\
\nonumber  &\lesssim & b_1^2\sum_{j=1}^{L_+} b_1^{2j}B_1^{4(j-j_+)-4(1-\dk)}\lesssim b_1^{2j_++2}\sum_{j=1}^{L_+} (b_1B^2_1)^{2(j-j_+)-2(1-\dk)}\\
& \lesssim &\left\{\begin{array}{ll}  b_1^{2j_++2+2(1-\dk)-C_{L_+}\eta}\ \ \mbox{for}\ \ 0\leq j_+\leq L_+-1\\ b_1^{2L_++2+2(1-\dk)(1+\eta)}\ \ \mbox{for}\ \ j_+=L_+
\end{array}\right.,
 \eee
 Similarily, using \fref{dgammkrealtion} and the a priori bound \fref{aprioirbound}:
  \bee
 \nonumber &&\int(1+y^2)\left|\Lt J\Lt^{k_++j_+}\left((\pa_s\chi_{B_1})\zeta_-^{(0)}-2J\pa_y\chi_{B_1}\pa_y\zeta_-^{(0)}-J\zeta_-^{(0)}\Delta \chi_{B_1}+b_1\zeta_-^{(0)}y\chi'_{B_1}\right)\right|^2\\
  \nonumber&\lesssim & \sum_{j=1}^{L_-}b_1^{2}b_1^{2j+\alpha}\int_{B_1\leq y\leq 2B_1}y^2\left|y^{2j-\frac{2}{p-1}-2(k_-+j_-+1)}\right|^2y^{d-1}dy\\
\nonumber  &\lesssim & \sum_{j=1}^{L_-} b_1^{2j+2+\alpha}B_1^{4(j-j_-)-4(1-\delta_{k_-})}\lesssim b_1^{2+2j_++2(1-\dk)}\sum_{j=1}^{L_-} (b_1B^2_1)^{2(j-j_-)-2(1-\delta_{k_-})}\\
& \lesssim &\left\{\begin{array}{ll}  b_1^{2j_++2+2(1-\dk)-C_{L+}\eta}\ \ \mbox{for}\ \ 0\leq j_+\leq L_+-1\\ b_1^{2L_++2+2(1-\dk)+2\eta(1-\dkm)}\ \ \mbox{for}\ \ j_+=L_+.
\end{array}\right.
 \eee
We now derive from \fref{degsiplus} the bound:
$$\left|\pa_y^k S_{j,+}\right|\lesssim b_1^jy^{2(j-1)-\gamma-k}+b_1^{j+\frac\alpha 2}y^{2j-\gamma-k}$$ from which
 \bee
 \nonumber &&\int(1+y^2)\left|\Lt J\Lt^{k_++j_+}\left((\pa_s\chi_{B_1})\zeta_+^{(1)}-2J\pa_y\chi_{B_1}\pa_y\zeta_+^{(1)}-J\zeta_+^{(1)}\Delta \chi_{B_1}+b_1\zeta_+^{(1)}y\chi'_{B_1}\right)\right|^2\\
  \nonumber&\lesssim & \sum_{j=2}^{L_++2}b_1^2b_1^{2j}\int_{B_1\leq y\leq 2B_1}\left\{y^2\left|y^{2(j-1)-\gamma-2(k_++j_++1)}\right|^2y^{d-1}+b_1^{\alpha}y^2\left|y^{2j-\gamma-2(k_++j_++1)}\right|^2y^{d-1}dy\right\}\\
\nonumber  &\lesssim & b_1^2\sum_{j=2}^{L_++2} \left\{b_1^{2j}B_1^{4(j-j_+-1)-4(1-\dk)}+ b_1^{2j+\alpha}B_1^{4(j-j_+)-4(1-\dk)}\right\}\\
&\lesssim &b_1^{2j_++4}\sum_{j=2}^{L_++2} (b_1B^2_1)^{2(j-j_+-1)-2(1-\dk)}+b_1^{2j_++2+\alpha}\sum_{j=2}^{L_++2} (b_1B^2_1)^{2(j-j_+)-2(1-\dk)}\\
& \lesssim &b_1^{2j_++4-C_{L_+}\eta}\leq b_1^{2L_++2+2(1-\dk)(1+\eta)}
 \eee
 for $0<\eta\ll1 $ small enough, thanks to the conditions $\alpha> 2$ and $0<\dk<1.$ We next estimate from \fref{degsiminus}:
 $$\left|\pa_y^k S_{j,-}\right|\lesssim b_1^{j+\frac{\alpha}{2}}y^{2(j-1)-\frac{2}{p-1}-k}+b_1^{j+\alpha}y^{2j-\frac{2}{p-1}-k}$$
 and obtain the bound:
 \bee
 \nonumber &&\int(1+y^2)\left|\Lt J\Lt^{k_++j_+}\left((\pa_s\chi_{B_1})\zeta_-^{(1)}-2J\pa_y\chi_{B_1}\pa_y\zeta_-^{(1)}-J\zeta_-^{(0)}\Delta \chi_{B_1}+b_1\zeta_-^{(1)}y\chi'_{B_1}\right)\right|^2\\
  \nonumber&\lesssim & \sum_{j=2}^{L_-+2}b_1^{2+\alpha}b_1^{2j}\int_{B_1\leq y\leq 2B_1}y^2\left\{\left|y^{2(j-1)-\frac{2}{p-1}-2(k_-+j_-+1)}\right|^2+b_1^{\alpha}\left|y^{2j-\frac{2}{p-1}-2(k_-+j_-+1)}\right|^2\right\}y^{d-1}dy\\
\nonumber  &\lesssim & b_1^{2+\alpha}\sum_{j=2}^{L_-+2} \left\{b_1^{2j}B_1^{4(j-j_--1)-4(1-\delta_{k_-})}+b_1^{2j+\alpha}B_1^{4(j-j_-)-4(1-\delta_{k_-})}\right\}\\
&\lesssim &b_1^{2j_-+\alpha+4}\sum_{j=2}^{L_-+2} (b_1B^2_1)^{2(j-j_--1)-2(1-\delta_{k_-})}+b_1^{2j_-+2\alpha+2}\sum_{j=2}^{L_-+2} (b_1B^2_1)^{2(j-j_-)-2(1-\delta_{k_-})}\\
& \lesssim &b_1^{2j_++4+\alpha-2\Delta k-C_{L_+}\eta}=b_1^{2j_++2(1-\dk)+2\dkm-C_{L_+}\eta}\lesssim b_1^{2j_++2(1+\eta)(1-\dk)}
 \eee
 for $\eta<\eta(L_+)$ small enough.\\

To estimate the nonlinear term, we first observe:
 $$
 \left|f(\qbt)-f(Q)-\chi_{B_1}(f(Q_b)-f(Q))\right| \lesssim  {\bf 1}_{B_1\leq y\leq 2B_1}[Q^{p-1}+|\zeta|^{p-1}]|\zeta|.
 $$ 
 We then estimate for $y\sim B_1$: $$Q^{p-1}\lesssim \frac{1}{y^2}\lesssim b_1^{1+\eta}$$ and observe the rough bound:
 \bee
 |\zeta|&\lesssim &\sum_{j=1}^{L_++2}b_1^jy^{2j-\gamma}+\sum_{j=1}^{L_-+2}b_1^{j+\frac\alpha 2}y^{2j-\frac{2}{p-1}}\lesssim\frac{b_1^{-C_{L_+}\eta}}{y^{\gamma}}
 \eee
 from which using $\gamma>2$, $p-1>1$: for $y\sim B_1$,
 $$|\zeta|^{p-1}\lesssim b_1^{\frac{\gamma(p-1)}{2}-C_{L_+}\eta}\leq b_1$$
 for $\eta$ small enough. Similar estimates also hold  for derivatives. The bound 
 \bee
&& \int (1+y^2)\left|\Lt J\Lt^{k_++j_+}\left (f(\qbt)-f(Q)-\chi_{B_1}(f(Q_b)-f(Q))\right)\right|^2\\
& \lesssim &\left\{\begin{array}{ll}  b_1^{2j_++2+2(1-\dk)-C_{L+}\eta}\ \ \mbox{for}\ \ 0\leq j_+\leq L_+-1\\ b_1^{2L_++2+2(1-\dk)+2\eta(1-\delta_p)}\ \ \mbox{for}\ \ j_+=L_+.
\end{array}\right.
 \eee
now easily follows. Note that this argument does not not use any cancellation induced by $\Lt$.
This concludes the proof of \fref{boundpsitildeplus}.\\
\noindent{\it Proof of \fref{boundpsitildeminus}}. We now assume the stronger condition $j_-\geq 0$
 to estimate the last non localized term. Using \fref{dgammkrealtion},
 \bee
 && \int(1+y^2)\left|\Lt J\Lt^{k_++j_+}\left(a_1(1-\chi_{B_1})Q\right)\right|^2++\int \frac{|\Psit_-|^2}{1+y^{4(k_++j_+)+2}}\\
 &\lesssim&  b_1^{2+\alpha}\int_{y\geq B_1}\frac{y^{d-1}dy}{y^{4(k_-+j_-+1)+\frac{4}{p-1}-2}}\lesssim  b_1^{2+\alpha}\int_{y\geq B_1}\frac{dy}{y^{1+4(j_-+1-\delta_{k_-})}}\\
 &\lesssim& b_1^{2+\alpha+2j_-+2(1-\dkm)}\left(\frac{B_0}{B_1}\right)^{4j_-+4(1-\delta_{k_-})}\lesssim  b_1^{2+2j_++2(1-\dk)+2\eta(1-\dkm)},
  \eee
and \fref{boundpsitildeminus} is proved.\\

\noindent{\bf step 3} Control of fractional derivatives. Let now $s_c<\sigma<\frac d2$.  Arguing as in the proof of \fref{boundpsitildeplus}, we estimate:
$$\ \ \int\left|\nabla^{2k_++2j_++1}(\chi_{B_1}\Psi+\Psit_+)\right|^2\lesssim \left\{\begin{array}{ll}  b_1^{2j_++2+2(1-\dk)-C_{L_+}\eta}\ \ \mbox{for}\ \ 0\leq j_+\leq L_+-1\\ b_1^{2L_++2+2(1-\dk)+2\eta(1-\delta_p)}\ \ \mbox{for}\ \ j_+=L_+
\end{array}\right.,$$
Now from \fref{dgammkrealtion} and $\alpha>2$:
\be
\label{reltiondeuxkpluspluuns}
2k_++1=s_c-2\left[\frac \alpha 2-(1-\dk)\right]<s_c<\sigma.
\ee 
We interpolate using the notation \fref{calculimportant}:
\be
\label{ejbebvebiev}
\sigma=z(2k_++1)+(1-z)s_+, \ \ 1-z=\frac{\sigma-2k_+-1}{2L_+}
\ee so that:
\bee
\|\nabla^{\sigma}(\chi_{B_1}\Psi+\Psit_+)\|_{L^2}^2& \lesssim &b_1^{(2+2(1-\dk)-C_{L_+}\eta)z+(2L_++2(1-\dk)+2\eta(1-\delta_p))(1-z)}
\eee
We then compute using \fref{reltiondeuxkpluspluuns}, \fref{ejbebvebiev}:
\bee
&&(2+2(1-\dk)-C_{L_+}\eta)z+(2L_++2(1-\dk)+2\eta(1-\delta_p))(1-z)=\\
&&2+2(1-\dk)-C_{L_+}\eta+\sigma-(2k_++1)+O\left(\frac{1}{L_+}\right)\\
&&=\sigma-s_c+\alpha-C_{L_+}\eta+O\left(\frac{1}{L_+}\right)
\eee
and obtain the bound from $\alpha>2$ for $L_+$ large enough and $\eta<\eta(L_+)$ small enough:
$$\|\nabla^{\sigma}(\chi_{B_1}\Psi+\Psit_+)\|_{L^2}^2\lesssim b_1^{2+\sigma-s_c+\nu(d,p)}.$$

 For the $\Psit_-$ term, we use the expansion $$\pa_y^kQ=\frac{c}{y^{\frac{2}{p-1}+k}}+O\left(\frac 1{y^{\gamma+k}}\right), \ \ k\geq 0$$ and standard commutator estimates to bound $$\|\nabla^{\sigma}Q\|_{L^2(y\geq B_1)}^2\lesssim \frac{1}{B_1^{2(\sigma-s_c)}}$$
from which using \fref{aprioirbound}: $$\|\nabla^{\sigma}\Psit_-\|_{L^2}^2\lesssim \frac{b_1^{2+\alpha}}{B_1^{2(\sigma-s_c)}}\lesssim b_1^{2+\sigma-s_c+\alpha}.$$ This concludes the proof of \fref{smallsobolevbound} and of Proposition \ref{consprofapprochloc}.
\end{proof} 
  
   %%%%%%%%%%%%%%%%%%%%%%%%%%%%%%%%%%%%%%%%%%%%%%%%%%%%%%%%%%%%%%%%%%%%%%%%%%%%%%%%%%%%%%%%%%%%%%%%%%%%%%%%%%%

\subsection{Study of the dynamical system for $b=(b_1,\dots,b_{L_+})$ and $a=(a_1,\dots,a_{L_-})$}

%%%%%%%%%%%%%%%%%%%%%%%%%%%%%%%%%%%%%%%%%%%%%%%%%%%%%%%%%%%%%%%%%%%%%%%%%%%%%%%%%%%%%%%%%%%%%%%%%%%%%%%%%%%

The construction of the $Q_{b,a}$ profile together with the yet described orthogonality relations will generate a 
finite dimensional dynamical system for $b=(b_1,\dots,b_{L_+})$ and $a=(a_1,\dots,a_{L_-})$. At a formal level 
this system is obtained by setting to zero the inhomogeneous $Mod(t)$ term \fref{defmodtun} of the renormalized 
flow. 
\be
\label{systdynfund}
\left\{\begin{array}{ll} (b_j)_s+\left(2j-\alpha\right)b_1b_j-b_{j+1}=0,  \ \ 1\leq j\leq L_+, \ \ b_{L_++1}\equiv 0,\\
 (a_j)_s+2jb_1a_j-a_{j+1}=0, \ \ 1\leq j\leq L_-, \ \ a_{L_-+1}\equiv 0.
 \end{array}\right.
 \ee 
In this section we show that \fref{systdynfund} admits a family of explicit solutions indexed by $\ell\in \Bbb N^*$, $\ell>\frac \alpha 2$. 
This family has a special property that its linearized flow is explicit as well and provides a direct description of its stable and unstable 
manifolds.

\begin{lemma}[Solution to the $a,b$  system]
\label{lemmaexplicitsol}
Let $$\frac{\alpha}{2}<\ell\ll L_+ , \ \ \ell\in \Bbb N^*$$ and the sequence
\be
\label{defalphai}
\left\{\begin{array}{lll}c_1=\frac{\ell}{2\ell-\alpha},\\
c_{j+1}=-\frac{\alpha(\ell-j)}{2l-\alpha}c_j, \ \ 1\leq j\leq \ell-1,\\
c_j=0, \ \ j\geq \ell+1
\end{array}\right.
\ee
then with the explicit choice 
\be
\label{approzimatesolution}
\left\{\begin{array}{ll}
b_j^{e}(s)=\frac{c_j}{s^j} \ \ 1\leq j\leq L_+\\
a_j^e(s)=0
\end{array}\right., \ \  \ \ s>0
\ee
is a solution to \fref{systdynfund}.
\end{lemma}

The proof of Lemma \ref{lemmaexplicitsol} is an explicit computation which is left to the reader. We now claim that this solution has  a codimension $(\ell+k_\ell-1)$ stable 
manifold with $k_\ell$ given by \fref{defkell}. We  note that the stability and instability of the $(b,a)$ system is considered in the class of solutions
\begin{align*}
&\sup_s s^j |b_j(s)|\le C_j,\qquad j=1,...,L_+\\
&\sup_s s^{j+\frac \alpha 2} |a_j(s)|\le C_j,\qquad j=1,...,L_-
\end{align*}

We start with the $b$ instabilities:

\begin{lemma}[Linearization of the unstable b-subsystem]
\label{lemmalinear}
{\em 1. Computation of the linearized system}: Let 
\be
\label{defuk}
b_k(s)=b_k^e(s)+\frac{U_k(s)}{s^k},\ \ 1\leq k\leq \ell, 
\ee
and note $U=(U_1,\dots,U_\ell)$.  Then: for $1\leq k\leq \ell-1$,
\be
\label{reformulaion}
(b_k)_s+\left(2k-\alpha\right)b_1b_k-b_{k+1}=  \frac{1}{s^{k+1}}\left[s(U_k)_s-(M_\ell U)_k+O\left(|U|^2\right)\right]
\ee
and
\bea
\label{reformulaionl}
(b_\ell)_s+\left(2\ell-\alpha\right)b_1b_\ell =  \frac{1}{s^{\ell+1}}\left[s(U_\ell)_s-(M_\ell U)_\ell+O\left(|U|^2\right)\right]
\eea

where 
\be
\label{defa}
M_\ell=(a_{i,j})_{1\leq i,j,\leq \ell} \ \ \mbox{with}\ \ \left\{\begin{array}{lllll}
a_{11}=\frac{\alpha(\ell-1)}{2\ell-\alpha}-(2-\alpha)c_1\\
a_{i,i+1}=1, \ \ 1\leq i\leq \ell-1\\
a_{1,i}=-(2i-\alpha)c_{i}, \ \ 2\leq i\leq  \ell\\
a_{i,i}=\frac{\alpha(\ell-i)}{2\ell-\alpha}, \ \ 2\leq i\leq \ell\\
a_{i,j}=0\ \ \mbox{otherwise}
\end{array}\right. .
\ee
{\em 2. Diagonalization of the linearized matrix}: $M_\ell$ is diagonalizable:
\be
\label{specta}
M_\ell=P_\ell^{-1}D_\ell P_\ell, \ \ D_\ell=\mbox{diag}\left\{-1,\frac{2\alpha}{2\ell-\alpha},\frac{3\alpha}{2\ell-\alpha},\dots, \frac{\ell\alpha}{2\ell-\alpha}\right\}.
\ee
\end{lemma}
\begin{remark}
Positive eigenvalues of the matrix $M_\ell$ correspond to $(\ell-1)$ unstable directions of both the truncated and the full system for $b$. On the other
hand, the negative eigenvalue direction together with the submanifold of solutions of the form $(0,...,0,b_{\ell+1},...,b_{L_+})$ generate the stable manifold.
Solutions of the form $(0,...,0,b_{\ell+1},...,b_{L_+})$ automatically obey the linear system
$$
(b_j)_s+(2j-\alpha) b_1^e b_j - b_{j+1}=0,\qquad j=\ell+1,...,L_+,\qquad b_{L_++1}=0.
$$
Its stability in the class of solutions with uniform bounds on $s^j |b_j(s)|$ is ensured by the positivity of 
$$
(2j-\alpha) c_1-j=\frac {2j-\alpha}{2\ell-\alpha} \ell - j=\alpha\frac  {j-\ell}{2\ell-\alpha}>0 
$$
for $j>\ell$.
\end{remark}
\begin{proof}[Proof of Lemma \ref{lemmalinear}]
{\bf step 1} Linearization. A simple computation from \fref{approzimatesolution} gives for $1\leq k\leq \ell-1$:
\bee
&&(b_k)_s+\left(2k-\alpha\right)b_1b_k-b_{k+1}\\
& = & \frac{1}{s^{k+1}}\left[s(U_k)_s-kU_k+(2k-\alpha)c_1U_k+(2k-\alpha) c_k U_1-U_{k+1}+O(U_1U_k)\right],
\eee
and the relation $$(2k-\alpha)c_1-k=-\frac{\alpha(\ell-k)}{2\ell-\alpha}$$ implies
\bee
&&(b_k)_s+\left(2k-1+\frac{2}{\log s}\right)b_1b_k-b_{k+1}\\
& = &   \frac{1}{s^{k+1}}\left[s(U_k)_s+(2k-\alpha)c_kU_1-\frac{\alpha(\ell-k)}{2l-\alpha}U_k-U_{k+1}+O\left(|U|^2\right)\right].
\eee
For $k=\ell$,
\bee
&&(b_\ell)_s+\left(2\ell-\alpha\right)b_1b_\ell-b_{\ell+1}\\
& = & \frac{1}{s^{\ell+1}}\left[s(U_\ell)_s-\ell U_\ell+(2\ell-\alpha)c_1U_\ell+(2\ell-\alpha) c_\ell U_1+O(U_1U_\ell)\right]\\
& = & \frac{1}{s^{\ell+1}}\left[s(U_\ell)_s+(2\ell-\alpha) c_\ell U_1+O(|U|^2)\right]
\eee
thanks to $$-\ell+(2\ell-\alpha)c_1=0.$$
These two relations are equivalent to \fref{reformulaion}, \fref{reformulaionl}, \fref{defa}.\\

\noindent{\bf step 2} Diagonalization. We compute the characteristic polynomial. The cases $\ell=2,3$ are done by direct inspection. Let us assume $\ell\geq 4$ and 
compute $$P_\ell(X)=\det(M_\ell-X\mbox{Id})$$ by expanding in the last row. This yields:
\bee
& & P_\ell(X) =  (-1)^{\ell+1}(-1)(2\ell-\alpha)c_\ell+(-X)\bigg\{(-1)^{\ell}(-1)(2(\ell-1)-\alpha)c_{\ell-1}\\
& + & \left(\frac{\alpha}{2\ell-\alpha}-X\right)\left[(-1)^{\ell-1}(-1)(2(\ell-2)-\alpha)c_{\ell-2}+\left(\frac{2\alpha}{2\ell-\alpha}-X\right)\left[\dots...\right]\right]\bigg\}.
\eee
We use the recurrence relation \fref{defalphai} to compute explicitly:
\bee
&&(-1)^{\ell+1}(-1)(2\ell-\alpha)c_l\\
& + & (-X)\bigg\{(-1)^{\ell}(-1)(2(\ell-1)-\alpha)c_{\ell-1}+ \left(\frac{\alpha}{2\ell-\alpha}-X\right)\left[(-1)^{\ell-1}(-1)(2(\ell-2)-\alpha)c_{\ell-2}\right]\bigg\}\\
& = & (-1)^\ell\bigg\{(2(\ell-1)-\alpha)c_{\ell-1}\left(X-\frac{\alpha}{2(\ell-1)-\alpha}\right)+(2(\ell-2)-\alpha)c_{\ell-2}\left(X-\frac{\alpha}{2\ell-\alpha}\right)X\bigg\}.
\eee
We now compute from \fref{defalphai} for $1\leq k\leq l-2$:
\bea
\label{recurelation}
\nonumber &&(2(\ell-k)-\alpha))c_{\ell-k}\left(X-\frac{\alpha}{2(\ell-k)-\alpha)}\right)+(2(\ell-k-1)-\alpha))c_{\ell-(k+1)}X\left(X-\frac\alpha{2\ell-\alpha}\right)\\
\nonumber & = & (2(\ell-k-1)-\alpha))c_{\ell-(k+1)}\left[X\left(X-\frac{\alpha}{2\ell-\alpha}\right)\right.\\
\nonumber & - & \left. \frac{2(\ell-k)-\alpha}{2(\ell-k-1)-\alpha}\frac{\alpha(k+1)}{2\ell-\alpha}\left(X-\frac{\alpha}{2(\ell-k)-\alpha)}\right)\right]\\
& = &  (2(\ell-k-1)-\alpha))c_{\ell-(k+1)}\left(X-\frac{\alpha(k+1)}{2\ell-\alpha}\right)\left(X-\frac{\alpha}{2(\ell-k-1)-\alpha)}\right).
\eea
We therefore obtain inductively:
\bee
&&P_\ell(X)\\
&&= (-1)^\ell\bigg\{(2\ell-1-\alpha)c_{\ell-1}\left(X-\frac{\alpha}{2(\ell-1)-\alpha}\right)+(2(\ell-2)-\alpha)c_{\ell-2}\left(X-\frac{\alpha}{2\ell-\alpha}\right)X\bigg\}\\
& + & (-X)\left(\frac{\alpha}{2\ell-\alpha}-X\right)\left(\frac{2\alpha}{2\ell-\alpha}-X\right)\left[(-1)^{\ell-2}(-1)(2(\ell-3)-\alpha)c_{\ell-3}+\left(\frac{3\alpha}{2\ell-\alpha}-X\right)[\dots]\right]\\
&= & (-1)^\ell\left(X-\frac{2\alpha}{2\ell-\alpha}\right)\bigg\{(2(\ell-2)-\alpha)c_{\ell-2}\left(X-\frac{\alpha}{2\ell-2)-\alpha}\right)\\
& + & (2(\ell-3)-\alpha)c_{\ell-3}X\left(X-\frac{\alpha}{2\ell-\alpha}\right)\bigg\}\\
& + & (-X)\left(\frac{\alpha}{2\ell-\alpha}-X\right)\left(\frac{2\alpha}{2\ell-\alpha}-X\right)\left(\frac{3\alpha}{2\ell-\alpha}-X\right)[(-1)^{\ell-3}(-1)(2(\ell-4)-\alpha)c_{\ell-4}\dots]\\
& = & (-1)^\ell\left(X-\frac{2\alpha}{2\ell-\alpha}\right)\dots \left(X-\frac{(\ell-2)\alpha}{2\ell-\alpha}\right)\\
&& \times\bigg\{(4-\alpha)c_2\left(X-\frac\alpha{4-\alpha}\right)+X\left(X-\frac{\alpha}{2\ell-\alpha}\right)\left((2-\alpha)c_1+X-\frac{\alpha(\ell-1)}{2\ell-\alpha}\right)\bigg\}.
\eee
We use \fref{recurelation} with $k=l-2$ to compute the last polynomial:
\bee
&& (4-\alpha)c_2\left(X-\frac\alpha{4-\alpha}\right)+X\left(X-\frac{\alpha}{2\ell-\alpha}\right)\left((2-\alpha)c_1+X-\frac{\alpha(\ell-1)}{2\ell-\alpha}\right)\\
& =& \bigg\{(4-\alpha)c_2\left(X-\frac \alpha{4-\alpha}\right)+(2-\alpha)c_1X\left(X-\frac{\alpha}{2\ell-\alpha}\right)\bigg\}\\
&+&X\left(X-\frac{\alpha}{2\ell-\alpha}\right)\left(X-\frac{\alpha(\ell-1)}{2\ell-\alpha}\right)\\
& = & (2-\alpha)c_1\left(X-\frac{\alpha(\ell-1)}{2\ell-\alpha}\right)\left(X-\frac{\alpha}{2-\alpha}\right)+X\left(X-\frac{\alpha}{2\ell-\alpha}\right)\left(X-\frac{\alpha(\ell-1)}{2\ell-\alpha}\right)\\
& = &\left(X-\frac{\alpha(\ell-1)}{2\ell-\alpha}\right)\left[\frac{(2-\alpha)\ell}{2\ell-\alpha}\left(X-\frac{\alpha}{2-\alpha}\right)+X\left(X-\frac{\alpha }{2\ell-\alpha}\right)\right]\\
& = & \left(X-\frac{\alpha(\ell-1)}{2\ell-\alpha}\right)\left(X-\frac{\alpha \ell}{2\ell-\alpha}\right)\left(X+1\right).
\eee
We have therefore computed:
$$P_\ell(x)=(-1)^l\left(X-\frac{2\alpha}{2\ell-\alpha}\right)\dots\left(X-\frac{3\alpha}{2\ell-\alpha}\right)\left(X-\frac{(\ell-1)\alpha}{2l-\alpha}\right)\left(X-\frac{\ell\alpha}{2\ell-\alpha}\right)\left(X+1\right)$$ and \fref{specta} is proved.
\end{proof}

We now compute the $a$ instabilities:

\begin{lemma}[Linearization of the unstable $a$-subsystem] 
\label{lemmaak}
Assume $k_\ell\geq 1$. Let $$A_k=s^{k+\frac \alpha 2}a_k,\ \ \mathcal A=(A_k)_{1\leq k\leq k_\ell},$$ then for $1\leq k\leq k_{\ell}-1$ (if $k_\ell\geq 2$):
\bee
(a_k)_s+\frac{2kc_1}{s}a_k-a_{k+1}=\frac{1}{s^{k+\frac\alpha 2+1}}\left[s(A_k)_s-(\mathcal M_{k_\ell}\mathcal A)_k\right]
\eee
and for $k=k_{\ell}$:
\bee
(a_k)_s+\frac{2kc_1}{s}a_k=\frac{1}{s^{k+\frac\alpha 2+1}}\left[s(A_k)_s-(\mathcal M_{k_\ell}\mathcal A)_k\right]
\eee
with  $$\left\{\begin{array}{lll} (\mathcal M_{k_{\ell}})_{i,i}=-\frac{\alpha}{(2\ell-\alpha)}\left[k-(k_\ell+\delta_{\ell})\right], \ \ 1\leq i\leq k_{\ell}\\
(\mathcal M_{k_{\ell}})_{i,i+1}=1, \ \ 1\leq i\leq k_{\ell}-1\\ ( \mathcal M_{k_{\ell}})_{i,j}=0\ \ \mbox{otherwise}.
\end{array}\right.$$
We can diagonalize the matrix $\matchal M_{k_\ell}$:
\be
\label{Pdiegmle}
\mathcal M_{k_{\ell}}=Q_{\ell}D_{k_{\ell}}Q_{\ell}^{-1}, \ \ D_{k_{\ell}}={\rm Diag}\left( -\frac{\alpha}{(2\ell-\alpha)}\left[k-(k_\ell+\delta_{\ell})\right]\right)_{1\leq k\leq k_{\ell}}.
\ee
\end{lemma}
\begin{remark}
All $k_\ell$ eigenvalues of the matrix $M_{k_\ell}$ are positive and thus generate unstable directions of the truncated 
(and full) $a$-system. Similar to the analysis of the $b$-system the solutions of the form $(0,...,0,a_{k_\ell+1},....,a_{L_-})$
give rise to the stable directions of the $a$-system. We omit the computation.
\end{remark}
\begin{proof}[Proof of Lemma \ref{lemmaak}] This is an elementary computation based on the value of $c_1$ from \fref{defalphai}. 
Here, the explicit diaganolization of $M_{k_\ell}$  is obvious.
\end{proof}

%%%%%%%%%%%%%%%%%%%%%%%%%%%%%%%%%%%%%%%%%%%%%%%%%%%%%%%%%%%%%%%%%%%%%%%%%%%%%%%%%%%%%%%%%%%%%%%%%%%%%%%%%%%

\section{The trapped regime}
\label{sectionboot}
%%%%%%%%%%%%%%%%%%%%%%%%%%%%%%%%%%%%%%%%%%%%%%%%%%%%%%%%%%%%%%%%%%%%%%%%%%%%%%%%%%%%%%%%%%%%%%%%%%%%%%%%%%%

In this section, we introduce the main dynamical tools at the heart of the proof of Theorem \ref{thmmain}. We start with the description of  the bootstrap regime 
in which the blow up solutions of Theorem \ref{thmmain} will be trapped, based on the splitting of the motion into the finite dimensional part driven by the modulation parameters and the remaining infinite dimensional dispersive dynamics. We then establish the control of the finite dimensional dynamics by the infinite dimensional part. The infinite dimensional part will in turn be controlled through the derivation of a mixed Energy/Morawetz Lyapunov functional in section \ref{sectionmonoton}.
%%%%%%%%%%%%%%%%%%%%%%%%%%%%%%%%%%%%%%%%%%%%%%%%%%%%%%%%%%%%%%%%%%%%%%%%%%%%%%%%%%%%%%%%%%%%%%%%%%%%%%%%%%%

\subsection{Localized generators of the kernel of the iterates of $\Lt$}
\label{sectionsetup}
%%%%%%%%%%%%%%%%%%%%%%%%%%%%%%%%%%%%%%%%%%%%%%%%%%%%%%%%%%%%%%%%%%%%%%%%%%%%%%%%%%%%%%%%%%%%%%%%%%%%%%%%%%%

We start by constructing two directions $\Xi_{M,\pm}$ with the property that their iterates $(\Lt^k\Xi_{M,\pm})_{1\leq k\leq L_\pm}$ are a well localized approximation of the explicit kernel of $\Lt^{k_+L_+}$.\\

\noindent\underline{Construction of $\Xi_{M,+}$.}  First observe from \fref{boundgamma} that since $$d-\gamma-\frac{2}{p-1}>d-2\gamma>0,$$ for any $M\gg1$: 
\be
\label{nondgeenracy}
M^{d-\gamma-\frac 2{p-1}}\lesssim \left|(J\chi_M\Phi_{0,+},\Phi_{0,-})\right|=\int \chi_M\Lambda Q Q\lesssim M^{d-\gamma-\frac 2{p-1}}.
\ee
We then consider the fixed vector:
\be
  \label{defdirectionplus}
  \Xi_{M,+} = \sum_{m=0}^{L_-}c^+_{m,-} (\Lt^*)^m(J\chi_M\Phi_{0,+})+\sum_{m=0}^{L_+}c^+_{m,+} (\Lt^*)^m(J\chi_M\Phi_{0,-})
    \ee
 with the explicit choice:   
  $$c^+_{0,+}=1, \ \ c^+_{0,-}=0$$ and the inductive relation: for $1\leq k\leq L_{+}$,
  $$c^+_{k,+}=-\frac{\sum_{m=0}^{\min\{L_-,k-1\}}c^+_{m,-} (J\chi_M\Phi_{0,+},\Lt^m\Phi_{k,+})+\sum_{m=0}^{k-1}c^+_{m,+} (J\chi_M\Phi_{0,-},\Lt^m\Phi_{k,+})}{(\chi_MJ\Phi_{0,+},\Phi_{0,-})},
 $$
 and for $1\leq k\leq L_-$,
 $$c^+_{k,-}=-\frac{\sum_{m=0}^{k-1}c^+_{m,-} (J\chi_M\Phi_{0,+},\Lt^m\Phi_{k,-})+\sum_{m=0}^{k-1}c^+_{m,+} (J\chi_M\Phi_{0,-},\Lt^m\Phi_{k,-})}{(\chi_MJ\Phi_{0,+},\Phi_{0,-})}.$$ We compute:
 \bea
 \label{nondegeenr}
 \nonumber |(\Xi_{M,+},\Phi_{0,+})|&=&\left|c^+_{0,-}(J\chi_M\Phi_{0,+},\Phi_{0,+})+c^+_{0,+}(J\chi_M\Phi_{0,-},\Phi_{0,+})\right|\\
 &=& \left|(J\chi_M\Phi_{0,+},\Phi_{0,-})\right|\gtrsim M^{d-\gamma-\frac 2{p-1}},
 \eea
 and
 \be
 \label{nondegeenrbis}
  (\Xi_{M,+},\Phi_{0,-})=c^+_{0,-}(J\chi_M\Phi_{0,+},\Phi_{0,-})+c^+_{0,+}(J\chi_M\Phi_{0,-},\Phi_{0,-})=0,
  \ee
 and for $1\leq k\leq L_+$:
 \bee
 &&(\Xi_{M,+},\Phi_{k,+})=\sum_{m=0}^{L_-}c^+_{m,-} (J\chi_M\Phi_{0,+},\Lt^m\Phi_{k,+})+\sum_{m=0}^{L_+}c^+_{m,+} (J\chi_M\Phi_{0,-},\Lt^m\Phi_{k,+})\\
  & = & c_{k,+}^+(J\chi_M\Phi_{0,-},\Phi_{0,+})+\sum_{m=0}^{\min\{L_-,k-1\}}c^+_{m,-} (J\chi_M\Phi_{0,+},\Lt^m\Phi_{k,+})+\sum_{m=0}^{k-1}c^+_{m,+} (J\chi_M\Phi_{0,-},\Lt^m\Phi_{k,+})\\
  & = & 0
  \eee
 and for $1\leq k\leq L_-$:
 \bee
 && (\Xi_{M,+},\Phi_{k,-})=\sum_{m=0}^{L_-}c^+_{m,-} (J\chi_M\Phi_{0,+},\Lt^m\Phi_{k,-})+\sum_{m=0}^{L_+}c^+_{m,+} (J\chi_M\Phi_{0,-},\Lt^m\Phi_{k,-})\\
 & = & c_{k,-}^+(J\chi_M\Phi_{0,+},\Phi_{0,-})+\sum_{m=0}^{k-1}c^+_{m,-} (J\chi_M\Phi_{0,+},\Lt^m\Phi_{k,-})+\sum_{m=0}^{k-1}c^+_{m,+} (J\chi_M\Phi_{0,-},\Lt^m\Phi_{k,-})\\
 & = & 0.
 \eee
 In particular:
 \be
 \label{keyrealtions}
 \left\{\begin{array}{ll}(\Lt^i\Phi_{j,+},\Xi_{M,+})=(J\chi_M\Phi_{0,+},\Phi_{0,-})\delta_{i,j}, \ \ 0\leq i,j\leq L_{+}\\
  (\Lt^i\Phi_{j,-},\Xi_{M,+})=0, \ \ 0\leq j\leq L_{-}, \ \ 0\leq i \leq L_+.
  \end{array}\right.
\ee
We now claim by induction on $k$ the bound 
\be
\label{boundcm}
|c^+_{k,+}|\lesssim M^{2k}, \ \ |c^+_{k,-}|\lesssim M^{2k+\alpha}.
\ee
and indeed\footnote{using $d-2\gamma>0$ so that all integrals diverge.}
\bee
|c^+_{k+1,+}|&\lesssim& \frac{1}{M^{d-\gamma-\frac 2{k-1}}}\left[\sum_{m=0}^{k}M^{2m+\alpha}M^{d-2\gamma+2(k+1-m)}+M^{2m}M^{d-\gamma-\frac 2{m-1}+2(k+1-m)}\right]\\
&\lesssim &M^{2(k+1)},
\eee
\bee
 |c^+_{k+1,-}|&\lesssim &\frac{1}{M^{d-\gamma-\frac 2{p-1}}}\left[\sum_{p=0}^{k}M^{2p+\alpha}M^{d-\gamma-\frac 2{p-1}+2(k+1-p)}+M^{2p}M^{d-\frac 4{p-1}+2(k+1-p)}\right]\\
&\lesssim &M^{2(k+1)+\alpha}.
\eee
Using the cancellation $\Lt^*(J\Phi_{0,\pm})=0$ this yields
 the bound:
\be
\label{ltwoboundxim}
 \int |\Xi_{M,+}|^2\lesssim  \sum_{k=0}^{L_-}M^{4k+2\alpha}M^{d-2\gamma-4k}+\sum_{k=0}^{L_+}M^{4k}M^{d-\frac{4}{p-1}-4k}\lesssim M^{d-\frac 4{p-1}}
\ee
and similarly\footnote{using $d-\frac 4{p-1}-2=d-2\gamma+2\alpha-2>0$.}
\be
\label{ltwoboundximbis}
 \int (1+y^2)|\Lt^*\Xi_{M,+}|^2\lesssim M^{d-\frac 4{p-1}-2}.
 \ee

\noindent\underline{Construction of $\Xi_{M,-}$}. We now consider along the same lines the direction:
\be
  \label{defdirectionminus}
  \Xi_{M,-} = \sum_{m=0}^{L_-}c^-_{m,-} (\Lt^*)^m(J\chi_M\Phi_{0,+})+\sum_{m=0}^{L_+}c^-_{m,+} (\Lt^*)^m(J\chi_M\Phi_{0,-})
    \ee
  with the explicit choice:   
  $$c^-_{0,+}=0, \ \ c^-_{0,-}=1$$ and the induction relations: for $1\leq k\leq L_{+}$,
  $$c^-_{k,+}=-\frac{\sum_{m=0}^{\min\{L_-,k-1\}}c^-_{m,-} (J\chi_M\Phi_{0,+},\Lt^m\Phi_{k,+})+\sum_{m=0}^{k-1}c^-_{m,+} (J\chi_M\Phi_{0,-},\Lt^m\Phi_{k,+})}{(\chi_MJ\Phi_{0,+},\Phi_{0,-})},
 $$
 and for $1\leq k\leq L_-$,
 $$c^-_{k,-}=-\frac{\sum_{m=0}^{k-1}c^-_{m,-} (J\chi_M\Phi_{0,+},\Lt^m\Phi_{k,-})+\sum_{m=0}^{k-1}c^-_{m,+} (J\chi_M\Phi_{0,-},\Lt^m\Phi_{k,-})}{(\chi_MJ\Phi_{0,+},\Phi_{0,-})}$$
so that 
  \be
  \label{nondegeenrtrois}
 (\Xi_{M,-},\Phi_{0,+})=c^-_{0,-}(J\chi_M\Phi_{0,+},\Phi_{0,+})+c^-_{0,+}(J\chi_M\Phi_{0,-},\Phi_{0,+})=0
 \ee
 \bea
 \label{nondegeenrquatre}
 \nonumber \left|(\Xi_{M,-},\Phi_{0,-})\right|&=&\left|c^-_{0,-}(J\chi_M\Phi_{0,+},\Phi_{0,-})+c^-_{0,+}(J\chi_M\Phi_{0,-},\Phi_{0,-})\right|\\
 &=& \left|(J\chi_M\Phi_{0,+},\Phi_{0,-})\right|\gtrsim M^{d-\gamma-\frac 2{p-1}}
 \eea
 and
  \bee
 &&(\Xi_{M,-},\Phi_{k,+})=0\ \ \mbox{for}\ \ 1\leq k\leq L_+\\
 && (\Xi_{M,-},\Phi_{k,-})=0\ \ \mbox{for}\ \ 1\leq k\leq L_-
  \eee
  In particular:
 \be
 \label{keyrealtionsbis}
 \left\{\begin{array}{ll} (\Lt^i\Phi_{j,+},\Xi_{M,-})=0, \ \ 0\leq i,j\leq L_{+}\\
   (\Lt^i\Phi_{j,-},\Xi_{M,-})=(J\chi_M\Phi_{0,+},\Phi_{0,-})\delta_{i,j}, \ \ 0\leq j\leq L_{-}, \ \ 0\leq i \leq L_+.
   \end{array}\right.
\ee
The bounds
\be
\label{ltwoboundximter}
 \int |\Xi_{M,-}|^2\lesssim M^{d-\frac 4{p-1}}, \ \ \int (1+y^2)|\Lt^*\Xi_{M,-}|^2\lesssim M^{d-\frac 4{p-1}-2}.
\ee
now follow verbatim as in the proof of \fref{ltwoboundxim}, \fref{ltwoboundximbis}.

%%%%%%%%%%%%%%%%%%%%%%%%%%%%%%%%%%%%%%%%%%%%%%%%%%%%%%%%%%%%%%%%%%%%%%%%%%%%%%%%%%%%%%%%%%%%%%%%%%%%%%%%%%%

\subsection{Setting up the bootstrap}
%%%%%%%%%%%%%%%%%%%%%%%%%%%%%%%%%%%%%%%%%%%%%%%%%%%%%%%%%%%%%%%%%%%%%%%%%%%%%%%%%%%%%%%%%%%%%%%%%%%%%%%%%%%

We are now in position to describe the set of initial data leading to the blow up scenario of Theorem  \ref{thmmain}.\\
We assume that the initial data $u_0\in H^{\infty}(\Bbb R^d)$. Since the nonlinearity is smooth, there exist a unique solution $u\in \matchal C^0([0,T0,H^s)$ for all $s>0$ with the blow up criterion $$T<+\infty\ \ \mbox{implies}\ \ \lim_{t\uparrow T}\|u(t)\|_{H^s}=\infty\ \ \mbox{for}\ \ s>s_c.$$ We now restrict our class of initial data. We pick $$L_+\gg 1$$ and a Sobolev exponent $\sigma$ with
\be
\label{choicesigma}
\frac1{L_+}\ll \sigma-s_c\ll 1, \ \ s_c<\sigma<\frac d2.
\ee 
and require that initially 
\be
\label{inoinnoenone}
\|u_0-Q\|_{\dot{H^{s}}}\ll1\ \ \mbox{for}\ \ s\in[\sigma,L_+].
\ee

\noindent\underline{Modulation}.  By continuity of the flow, the smallness \fref{inoinnoenone} is propagated on a small time interval $[0,t_1)$. 
On $[0,t_1)$ we then define the unique decomposition:
\bea
\label{decompu}
&&u(t,r)=\frac{1}{\l(t)^{\frac 2{p-1}}}(\tilde{Q}_{b(t),a(t)}+\e)\left(t,\frac{r}{\l(t)}\right)e^{i\gamma(t)},\\
\nonumber && \lambda(t)>0, \ \ \nonumber b=(b_1,\dots,b_{L_+}), \ \ a=(a_1,\dots,a_{L_-})
\eea where the modulation parameters $(a,b,\lambda,\gamma)$ are determined from the requirement that 
$\e(t)$ satisfies the $L_++L_-+2$ orthogonality conditions: 
\be
\label{ortho}
 (\e,(\Lt^*)^{k}\Xi_{M,\pm})=0, \ \ 0\leq k\leq L_\pm.
\ee
The existence of the decomposition \fref{decompu} is a standard consequence of the implicit function theorem and the explicit relations from \fref{defalphab}, \fref{degnenarxs}:
\bee 
&&\left(\frac{\pa}{\pa\l}(\qbt)_\lambda e^{i\gamma},\frac{\pa}{\pa b_1}(\qbt)_\lambda e^{i\gamma}, \dots, \frac{\pa}{\pa b_{L_+}}(\qbt)_\lambda e^{i\gamma},\right.\\
&&\left. \frac{\pa}{\pa\gamma}(\qbt)_\lambda e^{i\gamma}, \frac{\pa}{\pa a_1}(\qbt)_\lambda e^{i\gamma}, \dots, \frac{\pa}{\pa a_{L_-}}(\qbt)_\lambda e^{i\gamma}\right)|_{\l=1,b=0, \gamma=0, a=0}\\
& = & (\Phi_{0,+},\Phi_{1,+}, \dots,\Phi_{L_+,+}\Phi_{0,-},\Phi_{1,-},\dots,\Phi_{L_-,-})
\eee
which, using \fref{nondgeenracy}, \fref{keyrealtions}, \fref{keyrealtionsbis}, imply the non degeneracy of the Jacobian:
 \bee
&&\left|\left(\frac{\pa}{\pa (\l,b_j,a_k)}({\qbt})_\l,(\Lt^*)^i\Phi_M\right)_{1\leq j\leq L_+,1\leq k\leq L_-,0\leq i\leq L_\pm} \right|_{\l=1, b=0,\gamma=0,a=0}\\
&=& (\chi_MJ\Phi_{0,+},\Phi_{0,-})^{L_++L_-+2}\neq 0
\eee 
for $M\geq M^*$ large enough. The decomposition \fref{decompu}, in fact, exists as long as $t<T$ and $\e(t,r)$ remains small in $\dot{H}^s\cap\dot H^{L_+}$.

\noindent\underline{Setting up the bootstrap}. We now set up the bootstrap for the control of the geometrical parameters $(\l,b,\gamma,a)$ and the radiation $\e$. We will measure the regularity of the map through the following coercive norms of $\e$:\\
 \begin{itemize}
 \item High Sobolev norms adapted to the linearized operator: let $s_+$ be given by \fref{calculimportant}, we consider the high order Sobolev norm adapted to $\Lt$:
\be
\label{defnorme}
\mathcal E_{s_+}=(J\Lt\Lt^{k_++L_+}\e,\Lt^{k_++L_+}\e)\geq C(M)\left[ \int |\nabla^{s_+}\e|^2+\int \frac{|\e|^2}{1+y^{2s_+}}\right],
\ee
where the coercivity property follows from Lemma \ref{propcorc} and the choice of orthogonality conditions \fref{ortho}.
\item Low Sobolev norm: let $\sigma$ be chosen in the range \fref{choicesigma}, we will also control $\e$ in the norm:
\be
\label{defenergyspcar}
\int|\nabla^{\sigma}\e|^2.
\ee
\end{itemize}
We now choose our set of initial data in a more restricted way. More precisely, pick a large enough time $s_0\gg1 $ and rewrite the decomposition \fref{decompu}:
\be
\label{neneonneo}
u(t,r)=(\tilde{Q}_{b(s),a(s)}+\e)e^{i\gamma(t)}(s,y)
\ee
where we introduced the renormalized variables:
\be
\label{resacledtime}
y=\frac{r}{\l(t)}, \ \ s(t)=s_0+\int_0^{t}\frac{d\tau}{\l^2(\tau)}.
\ee
The renormalized  time variable $s$ will be shown to range in the interval $[s_0,+\infty)$ with $s=\infty$ corresponding to the blow up time $T$.  
We introduce a decomposition, see \fref{defuk}:
\be
\label{defukbis}
b_k=b_k^e+\frac{U_k}{s^{k}},\ \ 1\leq k\leq \ell
\ee 
and consider the variable 
\be
\label{devjj}
V=P_\ell U
\ee
where $P_\ell$ refers to the diagonalization \fref{specta} of $M_\ell$. Similarily, if $k_{\ell}\geq 1$, we let from \fref{Pdiegmle}: 
\be
\label{defatildegib}
A_k=s^{k+\frac \alpha 2}a_k,\ \ \mathcal A=(A_k)_{1\leq k\leq k_\ell}, \ \ \tilde{\matchal A}=Q_\ell \mathcal A.
\ee 
We recall that $0<\eta\ll1 $ is given by \fref{defetal} and assume initially:
\begin{itemize}
\item Smallness of the initial perturbation for the $b_k$ unstable modes:
\be
\label{estintial}
\left(s_0^{\frac{\eta}{2}(1-\dep)}V_k(s_0)\right)_{2\leq k\leq \ell}\in \mathcal B_{\ell}\left(1\right).
\ee
\item Smallness of the initial perturbation for the $a_k$ unstable modes: if $k_{\ell}\geq 1$,
\be
\label{estintiala}
\left(s_0^{\frac{\eta}{2}(1-\dep)}\tilde{\mathcal A}_k(s_0)\right)_{1\leq k\leq k_\ell}\in \mathcal B_{k_\ell}\left(1\right).
\ee
\item Smallness of the initial perturbation for the stable $b$ modes:
\be
\label{estintialbis}
|V_1(s_0)|<\frac{1}{s^{\frac{\eta}2(1-\dep)}},\quad \forall \ell+1\leq k\leq L_+, \ \  |b_k(s_0)|<b_1(s_0)^{k+\frac{5(2k-\alpha)\ell}{2\ell-\alpha}}.
\ee
\item Smallness of the initial perturbation for the stable $a$ modes:
\be
\label{estintialbisa}
\forall k_\ell+1\leq k\leq L_-, \ \ |a_k(s_0)|<b_1(s_0)^{k+\frac{\alpha}{2}+\frac{5(2k)\ell}{2\ell-\alpha}}.
\ee
\item Smallness of the data in high and low Sobolev norms:
\be
\label{init2energy}
\int | \nabla^{\sigma} \varepsilon (s_0)|^2 + \mathcal E_{s_+}(s_0) < b_1(s_0)^{\frac{10\ell}{2\ell-\alpha}L_+}.
\ee
\item Normalization: up to a fixed rescaling, we may always assume
\be
\label{intilzero}
\l(s_0)=1.
\ee
\end{itemize}

The heart of our analysis is the following:

\begin{proposition}
\label{bootstrap'}
Let $$K=K(d,p,M,L_+,\sigma)\gg 1$$
denote some large enough universal constant, then for any $s_0$ large enough, there exists initial data for the unstable modes $$\left(V_k(s_0)s_0^{\frac{\eta}{2}(1-\dep)}\right)_{2\leq k\leq \ell}\times \left(\tilde{\mathcal A}_k(s_0)s_0^{\frac{\eta}{2}(1-\dep)}\right)_{1\leq k\leq k_\ell}\in \mathcal B_{\ell+k_{\ell}-1}\left(1\right) $$ such that
the corresponding solution satisfies the bounds: $\forall s\geq s_0$,
\begin{itemize}
\item Control of the unstable modes:
\be
\label{controlunstable}
 \left(s^{\frac\eta 2(1-\dep)}V_k(s)\right)_{2\leq k\leq \ell}\times \left(s^{\frac\eta 2(1-\dep)}\tilde{A}_k(s)\right)_{1\leq k\leq k_\ell}\in \mathcal B_{\ell+k_\ell-1}\left(1\right).
\ee
\item  Control of the stable $b_k$ modes:
\be
\label{controlssstable}
|V_1(s)|\leq \frac{10}{s^{\frac{\eta}2(1-\dep)}},\ \ |b_k(s)|\leq \frac{10}{s^k}, \ \ \ell+1\leq k\leq L_+.
\ee
\item Control of the stable $a_k$ modes:
\be
\label{controlsaksstable}
|a_k(s)|\leq \frac{1}{s^{k+\frac \alpha 2}}, \ \ k_\ell+1\leq k\leq L_-.
\ee
\item Control of the radiation in high Sobolev norm:  \be
\label{bootnorm}
\mathcal E_{s_+}(s)\leq K b_1(s)^{2L_++2(1-\dk)+2\eta(1-\delta_p)}.
\ee
\item Control of the radiation in low Sobolev norm: 
\be
\label{bootsmallsigma}
\|\nabla^{\sigma}\e\|_{L^2}^2\leq K b_1(s)^{\frac{2\ell}{2\ell-\alpha}(\sigma-s_c)}.
\ee
\end{itemize} 
\end{proposition}

\begin{remark} Note in particular from \fref{defukbis} that the above bounds imply that for $\eta$ small enough
$$b_1(s)\sim\frac{c_1}{s}, \ \ |b_k(s)|\lesssim (b_1(s))^k, \ \ |a_k(s)|\leq(b_1(s))^{k+\alpha}$$
which are consistent with \fref{aprioirbound}.
\end{remark}

The proof of Proposition \ref{bootstrap'} proceeds via bootstrap combined with a standard topological argument \`a la Brouwer. 
Given $(\e(0),V(0))$ as above, we introduce the exit time
 \bea
\label{defexitsstar}
s^*& = & s^*(\e(s_0),V(s_0),\tilde{\mathcal A}(s_0))\\
\nonumber & = & \sup\{s\geq s_0\ \ \mbox{such that} \ \ \fref{controlunstable}, \fref{controlssstable}, \fref{controlsaksstable}, \fref{bootnorm}, \fref{bootsmallsigma}\ \ \mbox{hold on }\ \ [s_0,s]\},
\eea
assume that for any choice of
\be
\label{hypcontr}
\left(V_k(s_0)s_0^{\frac\eta 2(1-\dk)}\right)_{2\leq k\leq \ell}\times \left(\tilde{A}_k(s_0)s_0^{\frac\eta 2(1-\dk)}\right)_{1\leq k\leq k_\ell}\in \mathcal B_{\ell+k_\ell-1}\left(1\right)
\ee
the exit time $s^*<+\infty$.
 and look for a contradiction for $s_0$ large enough. Our main claim is that the a priori control of the unstable modes \fref{controlunstable} is enough to improve the bounds  \fref{controlssstable}, \fref{controlsaksstable}, \fref{bootnorm}, \fref{bootsmallsigma}. The contradiction claim, i.e. existence of the data for $\ell+k_{\ell}-1$ unstable modes
 resulting in the exit time $s_*=\infty$, is then established through a Brouwer type argument.
 
 We formalize the first part of this argument in the following proposition.
 \begin{proposition}[Bootstrap under the a priori control of the unstable modes]
\label{bootstrap}
Under the assumptions of Proposition \ref{bootstrap'} let the solution $(\e(s),a(s),b(s),\lambda(s),\gamma(s))$ obey the bounds \fref{controlunstable}, \fref{controlssstable}, \fref{controlsaksstable}, \fref{bootnorm}, \fref{bootsmallsigma} on a finite interval $[s_0,s^*]$. Then the bounds \fref{controlssstable}, \fref{controlsaksstable}, \fref{bootnorm}, \fref{bootsmallsigma} in fact hold with an improved factor, e.g. $1/2$, on the same interval $[s_0,s^*]$.
 \end{proposition}
 
 The end of this section is devoted to the derivation of the modulation equations. They follow from  the construction of the directions $\Xi_{M,\pm}$ and the choice of the 
 orthogonality conditions \fref{ortho}. The key monotonicity Lemmas for the control of $\e$ in the $\dot{H}^{\sigma}\times \dot{H}^{s_+}$ topology are then proved in section \ref{sectionmonoton}. The proof of Proposition \ref{bootstrap} is then completed in section \ref{sectionbootstrappouet}. We will make a systematic implicit use of the interpolation bounds of Lemma \ref{lemmainterpolation} following from the coercivity of the $\mathcal E_{s_+}$ energy established in Lemma \ref{propcorc}.

%%%%%%%%%%%%%%%%%%%%%%%%%%%%%%%%%%%%%

\subsection{Equation for the radiation}

%%%%%%%%%%%%%%%%%%%%%%%%%%%%%%%%%%%%%

Recall the decomposition of the flow:
  \be
  \label{decmopvemeope}
  u(t,r) = \frac{1}{\l^{\frac 2{p-1}}}(\tilde{Q}_{b(t),a(t)} + \varepsilon) (s,y)e^{i\gamma} =\left[\frac{1}{\l^{\frac 2{p-1}}}(Q + \zeta)(s,y) + w(t,r)\right]e^{i\gamma}.
  \ee
 We use the rescaling formulas
  \bea
  \nonumber u(t,r) =\frac{1}{\l^{\frac{2}{p-1}}}v(s,y)e^{i\gamma}\mbox{,} \ \ \ y = \frac{r}{\lambda(t)} \mbox{,}  \ \ \ \ \ \partial_t u = \frac{1}{\lambda^{2+\frac 2{p-1}}(t)}(\partial_s v - \frac{\lambda_s}{\lambda}\Lambda v+i\gamma_sv)(s,y)e^{i\gamma}
  \eea  
 and \fref{deferreutilder} to derive the equation for $\e$ in renormalized variables:
  \be
  \label{eqepsilon}
  \partial_s \varepsilon  - \frac{\lambda_s}{\lambda} \Lambda \varepsilon - \Lt \varepsilon = F - \Modt-\gamma_sJ\e = \mathcal F
  \ee
with 
\be
\label{defmodtuntilde}
\Modt =  -\left(\lsl+b_1\right)\Lambda \qbt+(\gamma_s-a_1)J\qbt-\chi_{B_1}\Mod
\ee
and  
\be
  \label{defF}
  F = -\Psit_b +L(\varepsilon) - N(\varepsilon)
  \ee
where $L(\varepsilon)$ is the linear part arising from replacing $\qbt$ with $Q$ in the nonlinear term:
  \be
  \label{defLe}
 L(\varepsilon) =  J(f'(Q)-f'(\qbt))\e, \ \ f(u)=u|u|^{p-1},
  \ee
 while the remainder higher order term:
  \be
  \label{defNe}
  N(\varepsilon) = J\left[f(\qbt+ \varepsilon) - f(\qbt)-\e f'(\qbt)\right]. 
  \ee
We also need to write the flow \fref{eqepsilon} in original variables. For this, let the rescaled linearized operator
$$(L_+)_{\l}=-\Delta-\frac{p}{\l^2}Q^{p-1}\left(\frac r{\lambda}\right), \ \ (L_-)_{\l}=-\Delta-\frac{1}{\l^2}Q^{p-1}\left(\frac r{\lambda}\right)$$ and the renormalized matrix  operator  
$$
\Lt_{\l}=\left(\begin{array}{ll} 0&(L_-)_\l\\-(L_+)_{\l}&0\end{array}
\right),
$$
then the renormalized function $$w(t,r)=\frac{1}{\l^{\frac2{p-1}}}\e\left(s,y\right)$$  satisfies
  \be
  \label{eqenwini}
  \partial_t w -\Lt_\lambda w = \frac 1{\lambda^2} \mathcal F_{\lambda},\ \ \mathcal F_{\lambda}(t,r)=\frac{1}{\l^{\frac 2{p-1}}}\mathcal F(s,y).
  \ee
Observe from \fref{approzimatesolution}, \fref{controlssstable} that for $s<s^*$, 
 $$ |b_k|\lesssim b_1^k, \ \ 0<b_1\ll 1, \ \ 1\leq k\leq L_+ 
 $$
 for some universal constant independent of the constant $\eta$ in \fref{controlunstable} in the range $0<\eta\leq 1$,
 and similarly from \fref{controlunstable}, \fref{controlsaksstable}:
  $$ |a_k|\leq b_1^{k+\alpha}, \ \ 1\leq k\leq L_-
 $$
 for $\eta$ in \fref{controlunstable} small enough.
As a consequence the a priori bound \fref{aprioirbound} as well as the conclusions of Proposition \ref{consprofapprochloc} hold 
with constants independent of $\eta$, chosen to be sufficiently small.

  %%%%%%%%%%%%%%%%%%%%%%%%%%%%%%%%%%%%%

    \subsection{Modulation equations}
  
  %%%%%%%%%%%%%%%%%%%%%%%%%%%%%%%%%%%%%

We now derive the modulation equations for $(\l,b,\gamma,a)$ from the orthogonality conditions \fref{ortho}.

  \begin{lemma}[Modulation equations]
\label{modulationequations}
We have the following bounds on the modulation parameters :
\bea
\label{parameters}
\nonumber && \sum_{k=1}^{L_+-1}|(b_k)_s+(2k-\alpha)b_1b_k-b_{k+1}|+\sum_{k=1}^{L_--1} |(a_k)_s+2kb_1a_k-a_{k+1}|\\
&+ &  \left|\frac{\lambda_s}{\lambda} + b_1\right| +\left|\gamma_s-a_1\right|\lesssim b_1^{L_++1+(1-\dk)+\eta(1-\delta_p)},
\eea
the sharp bound for $b_{L_+}$ term:
\be
\label{parameterspresicely}
\left| (b_{L_+})_s + (2L_+-\alpha)b_1b_{L_+} \right| \lesssim \frac{\sqrt{\mathcal E_{s_+}}}{M^{2\dk}}+  b_1^{L_++1+(1-\dk)+\eta(1-\delta_p)}
\ee
 and the lossy bound for $a_{L_-}$ term:
 \be
 \label{parameterspresicelya}
\left| (a_{L_-})_s + 2L_-b_1a_{L_-} \right| \lesssim M^C\sqrt{\mathcal E_{s_+}}+  b_1^{L_++1+(1-\dk)+\eta(1-\delta_p)}.
\ee
for some universal constant $c=c_{d,p,L_+}>0$.
\end{lemma}

\begin{remark} Note that under the bootstrap assumptions the above bounds  imply:
\be
\label{rougboundpope}
|(b_1)_s|\lesssim b_1^2
\ee
and in particular \fref{aprioribound}. 
\end{remark}

\begin{proof}[Proof of Lemma \ref{modulationequations}] This Lemma is a consequence of our choice of orthogonality conditions and the construction of the compactly supported directions $\Xi_{M,\pm}$.\\

\noindent{\bf step 1} Law for $b_{L_+}$. Let 
\bea
 \label{defut}
 D(t) &=& \left|\frac{\lambda_s}{\lambda} + b_1\right|+\left|\gamma_s-a_1\right|\\
 \nonumber & + & \sum_{k=1}^{L_+}|(b_k)_s + (2k-\alpha)b_1b_k-b_{k+1}|+\sum_{k=1}^{L_-}|(a_k)_s + (2k-\alpha)b_1a_k-a_{k+1}| 
 \eea
 We take the inner product of \fref{eqepsilon} with $(\Lt^*)^{L_+}\Xi_{M,+}$ and obtain using the orthogonality \fref{ortho}:
\bea
\label{modequone}
(\widetilde{\mbox{Mod}}(t),(\Lt^*)^{L_+}\Xi_{M,+})& = & -(\Psit_b,(\Lt^*)^{L_+}\Xi_{M,+})+(\Lt\e,(\Lt^*)^{L_+}\Xi_{M,+})\\
\nonumber & + & \left(L(\e)-N(\e)+\lsl\Lambda \e-\gamma_sJ\e,(\Lt^*)^{L_+}\Xi_{M,+}\right).
\eea
We now evaluate all terms in \fref{modequone}.  The lhs is computed using \fref{defmodtun}, \fref{defmodtuntilde}, ${\rm Supp}(\Xi_{m,+})\subset \{y\leq 2M\}$ and the scalar products \fref{keyrealtions}:
\bee
(\widetilde{\mbox{Mod}}(t),(\Lt^*)^{L_+}\Xi_{M,+})& = & \left(-\left(\lsl+b_1\right)\Lambda \qbt+(\gamma_s-a_1)J\qbt-\chi_{B_1}\Mod,(\Lt^*)^{L_+}\Xi_{M,+}\right)\\
& = & ((b_{L_+})_s+(2L_+-\alpha)b_1b_{L_+})(J\chi_M\Phi_{0,+},\Phi_{0,-})+O(M^Cb_1|D(t)|).
\eee
We now turn to the rhs of \fref{modequone}. The error term is estimated from \fref{fluxcomputationonelocbis}:
$$\left|(\Psit_b,(\Lt^*)^{L_+}\Xi_{M,+})\right|\lesssim M^Cb_1^{L_++3}\leq b_1^{L_++1+(1-\dk)+\eta(1-\dep)}.$$
To estimate the linear term, we apply \fref{coerciviteuk} to $\Lt^{L_++1}\e$ and estimate:
$$\mathcal E_{s_+}=(J\Lt\Lt^{k_++L_+}\e,\Lt^{k_++L_+}\e)=(J\Lt\Lt^{k_+-1}\Lt^{L_++1}\e,\Lt^{k_+-1}\Lt^{L_++1}\e) \geq c_0\int \frac{|\Lt^{L_++1}\e|^2}{1+y^{4k_+-2}}$$ for some universal constant $c_0>0$ {\it independent of $M$}, and hence using \fref{ltwoboundxim}:
\bee
\left|(\Lt\e,(\Lt^*)^{L_+}\Xi_{M,+})\right|&\lesssim &\|\Lt^{L_++1}\e\|_{L^2(y\leq 2M)}\|\|\Xi_{M,+}\|_{L^2}\lesssim \sqrt{\mathcal E_{s_+}}M^{2k_+-1}\|\Xi_{M,+}\|_{L^2}\\
& \lesssim & M^{2k_+-1+\frac d2-\frac{2}{p-1}}\sqrt{\mathcal E_{s_+}}.
\eee
We conclude using \fref{nondgeenracy}, \fref{dgammkrealtion}:
\be
\label{kenvenovnevo}
\left|\frac{(\Lt\e,(\Lt^*)^{L_+}\Xi_{M,+})}{(\chi_MJ\Phi_{0,+},\Phi_{0,-})}\right|\lesssim \frac{M^{2k_+-1+\frac d2-\frac{2}{p-1}}}{M^{d-\gamma-\frac 2{p-1}}}\sqrt{\mathcal E_{s_+}}\lesssim \frac{\sqrt{\mathcal E_{s_+}}}{M^{2\delta_+}}.
\ee
The remaining terms are estimated using the Hardy bounds of Appendix \ref{appendixhardy} and the size of the support of $\Xi_{M,+}$:
$$\left|\left(L(\e)-N(\e)+\lsl\Lambda \e-\gamma_sJ\e,(\Lt^*)^{L_+}\Xi_{M,+}\right)\right|\lesssim M^Cb_1(\sqrt{\mathcal E_{s_+}}+|D(t)|).$$ The collection of above bounds yields:
\be
\label{estblplusone}
|(b_{L_+})_s+(2L_+-\alpha)b_1b_{L_+}|\lesssim \frac{\sqrt{\mathcal E_{s_+}}}{M^{2\delta_+}}+ b_1^{L_++1+(1+\eta)(1-\dk)}+M^Cb_1D(t).
\ee

\noindent{\bf step 2} Law for $a_{L_-}$. We follow a similar chain of estimates to compute the modulation equation for $a_{L_-}$.  We take the inner product of \fref{eqepsilon} with $(\Lt^*)^{L_-}\Xi_{M,-}$ and obtain using the orthogonality \fref{ortho}:
\bea
\label{modequonebis}
(\widetilde{\mbox{Mod}}(t),(\Lt^*)^{L_-}\Xi_{M,-})& = & -(\Psit_b,(\Lt^*)^{L_-}\Xi_{M,-})+(\Lt\e,(\Lt^*)^{L_-}\Xi_{M,-})\\
\nonumber & + & \left(L(\e)-N(\e)+\lsl\Lambda \e-\gamma_sJ\e,(\Lt^*)^{L_-}\Xi_{M,-}\right).
\eea
We now evaluate all the terms in \fref{modequonebis}. The lhs is computed using \fref{defmodtun}, \fref{defmodtuntilde}, ${\rm Supp}(\Xi_{M,-})\subset \{y\leq 2M\}$ and the scalar products \fref{keyrealtions}:
\bee
(\widetilde{\mbox{Mod}}(t),(\Lt^*)^{L_-}\Xi_{M,-})& = & \left(-\left(\lsl+b_1\right)\Lambda \qbt+(\gamma_s-a_1)J\qbt-\chi_{B_1}\Mod,(\Lt^*)^{L_+}\Xi_{M,-}\right)\\
& = & ((a_{L_-})_s+2L_-b_1a_{L_-})(J\chi_M\Phi_{0,+},\Phi_{0,-})+O(M^Cb_1D(t)).
\eee
The error term is estimated from \fref{fluxcomputationonelocbis} which implies:
$$\left|(\Psit_b,(\Lt^*)^{L_-}\Xi_{M,-})\right|\lesssim M^Cb_1^{L_++3}\leq b_1^{L_++1+(1-\dk)+\eta(1-\dep)}.$$
The remaining terms are estimated using the Hardy bounds of Appendix \ref{appendixhardy} and the size of the support of $\Xi_{M,-}$:
$$\left|\left(\Lt\e+L(\e)-N(\e)+\lsl\Lambda \e-\gamma_sJ\e,(\Lt^*)^{L_-}\Xi_{M,-}\right)\right|\lesssim M^C\sqrt{\mathcal E_{s_+}}+b_1M^CD(t).$$
 The collection of above bounds yields
\be
\label{estblplustwo}
|(a_{L_-})_s+2L_-b_1a_{L_-})|\lesssim M^C\sqrt{\mathcal E_{s_+}}+ b_1^{L_++1+(1-\dk)+\eta(1-\dep)}+b_1M^CD(t).
\ee

\noindent{\bf step 3} Law for $-\lsl$ and $b_{k}$, $1\leq k\leq L_+-1$. We take the inner product of \fref{eqepsilon} with $(\Lt^*)^{k}\Xi_{M,+}$ and obtain using the orthogonality \fref{ortho}:
\bee
(\widetilde{\mbox{Mod}}(t),(\Lt^*)^{k}\Xi_{M,+})& = & -(\Psit_b,(\Lt^*)^{k}\Xi_{M,+})\\
\nonumber & + & \left(L(\e)-N(\e)+\lsl\Lambda \e-\gamma_sJ\e,(\Lt^*)^{k}\Xi_{M,+}\right)
\eee
where in particular the linear term dropped thanks to \fref{ortho} and $k\leq L_+-1$. We compute from \fref{defmodtun}, \fref{defmodtuntilde}, ${\rm Supp}(\Xi_{M,-})\subset \{y\leq 2M\}$ and the scalar products \fref{keyrealtions}:
$$(\widetilde{\mbox{Mod}}(t),(\Lt^*)^{k}\Xi_{M,+})= ((b_{k})_s+(2k-\alpha)b_1b_k-b_{k+1})(J\chi_M\Phi_{0,+},\Phi_{0,-})+O(M^Cb_1D(t)).$$ The remaining terms are estimated using \fref{fluxcomputationonelocbis}, the Hardy bounds of Appendix \ref{appendixhardy} and the compact support of $\Xi_{M,+}$ giving the bound:
\be
\label{estblplusthree}
|(b_{k})_s+(2k-\alpha)b_1b_{k}-b_{k+1}|\lesssim  b_1^{L_++1+(1-\dk)+\eta(1-\dep)}+M^Cb_1(D(t)+\sqrt{\matchal E_{s_+}}).
\ee
Taking the inner product of \fref{eqepsilon} with $\Xi_{M,+}$ yields similarly: 
\be
\label{estblplusthreebis}
\left|\lsl+b_1\right|\lesssim  b_1^{L_++1+(1-\dk)+\eta(1-\dep)}+M^Cb_1(D(t)+\sqrt{\matchal E_{s_+}}).
\ee

\noindent{\bf step 4} Law for $\gamma_s$, $a_{k}$, $1\leq k\leq L_--1$. We take the inner product of \fref{eqepsilon} with $(\Lt^*)^{k}\Xi_{M,-}$ and obtain using the orthogonality \fref{ortho}:
\bee
(\widetilde{\mbox{Mod}}(t),(\Lt^*)^{k}\Xi_{M,-})& = & -(\Psit_b,(\Lt^*)^{k}\Xi_{M,-})\\
\nonumber & + & \left(L(\e)-N(\e)+\lsl\Lambda \e-\gamma_sJ\e,(\Lt^*)^{k}\Xi_{M,-}\right)
\eee
where again the linear term dropped thanks to \fref{ortho} and $k\leq L_--1$. We compute from \fref{defmodtun}, \fref{defmodtuntilde}, ${\rm Supp}(\Xi_{M,-})\subset \{y\leq 2M\}$ and the scalar products \fref{keyrealtions}:
$$(\widetilde{\mbox{Mod}}(t),(\Lt^*)^{k}\Xi_{M,-})= ((a_{k})_s+2kb_1a_k-a_{k+1})(J\chi_M\Phi_{0,+},\Phi_{0,-})+O(M^Cb_1|D(t)|).$$ The remaining terms are estimated using \fref{fluxcomputationonelocbis}, the Hardy bounds of Appendix \ref{appendixhardy} and the compact support of $\Xi_{M,+}$ resulting in the bound:
\be
\label{estblplusfour}
|(a_{k})_s+2kb_1a_{k}-a_{k+1}|\lesssim  b_1^{L_++1+(1-\dk)+\eta(1-\dep)}+M^Cb_1(D(t)+\sqrt{\matchal E_{s_+}}).
\ee
Taking the inner product of \fref{eqepsilon} with $\Xi_{M,-}$ yields similarly: 
\be
\label{estblplusfourbis}
\left|\gamma_s-a_1\right|\lesssim  b_1^{L_++1+(1-\dk)+\eta(1-\dep)}+M^Cb_1(D(t)+\sqrt{\matchal E_{s_+}}).
\ee
\noindent{\bf step 5} Conclusion. Summing \fref{estblplusone}, \fref{estblplustwo}, \fref{estblplusthree}, \fref{estblplusthreebis} \fref{estblplusfour}, \fref{estblplusfourbis} gives the rough bound: $$|D(t)|\lesssim M^C\sqrt{\mathcal E_{s_+}}+ b_1^{L_++1+(1-\dk)+\eta(1-\dep)}$$ which reinserted into \fref{estblplusone}, \fref{estblplustwo}, \fref{estblplusthree}, \fref{estblplusthreebis} \fref{estblplusfour}, \fref{estblplusfourbis} yields \fref{parameters}, \fref{parameterspresicely}, \fref{parameterspresicelya} for $|b_1|<b_1^*(M)$ small enough.
\end{proof}

 %%%%%%%%%%%%%%%%%%%%%%%%%%%%%%%%%%%%%

    \subsection{Improved modulation equation for $b_{L_+},a_{L_-}$}
  
  %%%%%%%%%%%%%%%%%%%%%%%%%%%%%%%%%%%%%
 
 The modulation equations for $b_{L_+},a_{L_-}$ correspond to the unstable directions {\it linear} in $\e$ due to our choice of orthogonality conditions \fref{ortho}, and the fact that $\Xi_{M,\pm}$ is merely  an {\it approximation} of the kernel of $\Lt^{k_++L_+}$. Indeed  \fref{parameterspresicely}, \fref{bootnorm} would only yield the pointwise bound $$\left| (b_{L_+})_s + (2L_+-\alpha)b_1b_{L_+} \right| \lesssim b_1^{L_++(1-\dk)+\eta(1-\dep)}$$ which is not good enough to close the expected modulation equation $$\left| (b_{L_+})_s + (2L_+-\alpha)b_1b_{L_+} \right| \ll b_1^{L_++1},$$  and similarly for the $a_{L_-}$ modulation equation \fref{parameterspresicelya}.  
 We however claim that the main linear term can be removed modulo a term with a time oscillation:

 \begin{lemma}[Improved modulation equation]
 \label{improvedbl}
Then there holds the improved bounds:
\bea
\label{improvedboundbl}
&&\left|(b_{L_+})_s+(2L_+-\alpha)b_1b_{L_+}+\frac{d}{ds}\left\{\frac{(\Lt^{L_+}\e,\chi_{B_{\mu}}J\Phi_{0,-})}{(\Phi_{0,+},\chi_{B_{\mu}}J\Phi_{0,-})}\right\}\right|\\
\nonumber & \lesssim & \frac{1}{B_{\mu}^{2\dk}}\left[C(M)\sqrt{\mathcal E_{s_+}}+b_1^{L_++(1-\dk)+\eta(1-\dep)}\right],
\eea
\bea
\label{improvedboundal}
&&\left|(a_{L_-})_s+2L_-b_1a_{L_-}+\frac{d}{ds}\left\{\frac{(\Lt^{L_-}\e,\chi_{B_{\mu}}J\Phi_{0,+})}{(\Phi_{0,-},\chi_{B_{\mu}}J\Phi_{0,+})}\right\}\right|\\
\nonumber & \lesssim & \frac{1}{B_{\mu}^{2\delta_{k_-}}}\left[C(M)\sqrt{\mathcal E_{s_+}}+b_1^{L_++(1-\dk)+\eta(1-\dep)}\right].
\eea
\end{lemma}

\begin{proof}[Proof of Lemma \ref{improvedbl}]
\noindent{\bf step 1} Proof of \fref{improvedboundbl}. We commute \fref{eqepsilon} with $\Lt^{L_+}$ and take the scalar product with $\chi_{B_0}J\Phi_{0.-}$. This yields:
\bee
&&\frac{d}{ds}\left\{(\Lt^{L_+}\e,\chi_{B_0}J\Phi_{0,-})\right\}-(\Lt^{L_+}\e,J\Phi_{0,-}\pa_s(\chi_{B_0}))\\
& = & (\Lt^{L_++1}\e,J\chi_{B_0}\Phi_{0,-})+\lsl(\Lt^{L_+}\Lambda \e,\chi_{B_0}J\Phi_{0,-})-\gamma_s(\Lt^{L_+}J\e,J\chi_{B_0}\Phi_{0,-})\\
& +& (\Lt^{L_+}(F-\widetilde{\mbox{Mod}}),J\chi_{B_0}\Phi_{0,-}).
\eee
The linear term is estimated by Cauchy-Schwarz using Lemma \ref{propcorc}, \fref{dgammkrealtion} and $\Lt^*(J\Phi_{0,-})=0$:
\bee
\nonumber |(\Lt^{L_++1}\e,J\chi_{B_0}\Phi_{0,-})|&\lesssim &B_0^{1+2k_+}\|\Lt^*(J\chi_{B_0}\Phi_{0,-})\|_{L^2}\left(\int\frac{|\Lt^{L_+}\e|^2}{1+y^{2+4k_+}}\right)^{\frac 12}\\
&\lesssim &C(M)B_0^{1+2k_+}B_0^{\frac d2-\frac{2}{p-1}-2}\sqrt{\mathcal E_{s_+}}\\
& = & C(M)B_0^{d-\gamma-\frac{2}{p-1}-2\dk}\sqrt{\mathcal E_{s_+}}.
\eee
Similarily:
\bee
\nonumber &&\left|\lsl(\Lt^{L_+}\Lambda \e,\chi_{B_0}J\Phi_{0,-})\right|+\left|\gamma_s(\Lt^{L_+}J\e,J\chi_{B_0}\Phi_{0,-})\right|\\
&\lesssim& b_1\left(\int \frac{|\Lambda \e|^2+|\e|^2}{1+y^{4(L_++k_+)+2}}\right)^{\frac 12}\left(\int (1+y^{4(L_++k_+)+2})|(\Lt^*)^{L_+}\chi_{B_0}J\Phi_{0,-}|^2\right)^{\frac 12}\\
& \lesssim & b_1C(M)\sqrt{\mathcal E_{s_+}}B_0^{\frac d2-\frac 2{p-1}+2k_++1}\lesssim C(M)B_0^{d-\gamma-\frac{2}{p-1}-2\dk}\sqrt{\mathcal E_{s_+}}\\
& \lesssim & C(M)B_0^{d-\gamma-\frac{2}{p-1}-2\dk}\sqrt{\mathcal E_{s_+}},
\eee
and
\bee
&&\left|(\Lt^{L_+}\e,J\Phi_{0,-}\pa_s(\chi_{B_0}))\right)|\\
&\lesssim& \left|\frac{(b_1)_s}{b_1}\right|\left(\int \frac{|\Lt^{L_+}\e|^2}{1+y^{4k_++2}}\right)^{\frac 12}\left(\int_{B_0\leq y\leq 2B_0} (1+y^{4k_++2})|y^{-\frac{2}{p-1}}|^2\right)^{\frac 12}\\
& \lesssim & C(M)b_1B_0^{2k_++1+\frac d2-\frac{2}{p-1}}\sqrt{\mathcal E_{s_+}}\lesssim  C(M)B_0^{d-\gamma-\frac{2}{p-1}-2\dk}\sqrt{\mathcal E_{s_+}}\\
&\lesssim &  C(M)B_0^{d-\gamma-\frac{2}{p-1}-2\dk}\sqrt{\mathcal E_{s_+}}.
\eee
We now estimate the $F$ terms. We anticipate the bound \fref{controlnonlineaire} to estimate:
\bee
&&\left|(\Lt^{L_+}N(\e),J\chi_{B_0}\Phi_{0,-})\right|\\
&\lesssim&\left(\int \frac{|N(\e)|^2}{1+y^{2s_+}}\right)^{\frac 12}\left(\int_{y\leq 2B_0}(1+y^{2(k_++L_+)+1-\frac{2}{p-1}-2L_+})^2y^{d-1}dy\right)^{\frac 12}\\
& \lesssim & b_1^{1+\frac{\nu(d,p)}{2}}\sqrt{\mathcal E_{s_+}}B_0^{d-\gamma-\frac{2}{p-1}-2\dk+2}\lesssim B_0^{d-\gamma-\frac{2}{p-1}-2\dk}\sqrt{\mathcal E_{s_+}}
\eee
and similarly using \fref{boundle}:
\bee
&&\left|(\Lt^{L_+}L(\e),J\chi_{B_0}\Phi_{0,-})\right|\\
&\lesssim &\left(\int \frac{|L(\e)|^2}{1+y^{2s_+-4}}\right)^{\frac 12}\left(\int_{y\leq 2B_0}(1+y^{2(k_++L_+)+1-2-\frac{2}{p-1}-2L_+})^2y^{d-1}dy\right)^{\frac 12}\\
&\lesssim & B_0^{d-\gamma-\frac{2}{p-1}-2\dk}\sqrt{\mathcal E_{s_+}}.
\eee
We estimate the $\Psit_b$ term from \fref{nkonenoenve}:
\bee
&&|(\Lt^{L_+}\Psit,\chi_{B_0}J\Phi_{0,-})|=|(\Psit,(\Lt^*)^{L_+}\chi_{B_0}J\Phi_{0,-})|\\
&\lesssim & \left(\int \frac{|\Psit|^2}{1+y^{4(k_++L_+)+2}}\right)^{\frac 12}\left(\int_{B_0\leq y\leq 2B_0}y^{4(k_++L_+)+2}|y^{-\frac{2}{p-1}-2L_+}|^2\right)^{\frac 12}\\
&\lesssim & b_1^{L_++2+(1-\dk)-C_{L_+}\eta}B_0^{2k_++1-\frac{2}{p-1}+\frac d2}=B_0^{d-\gamma-\frac 2{p-1}-2\dk+2} b_1^{L_++2+(1-\dk)-C_{L_+}\eta}\\
& \lesssim & B_0^{d-\gamma-\frac 2{p-1}-2\dk}b_1^{L_++1+(1-\dk)-C_{L_+}\eta}\lesssim B_0^{d-\gamma-\frac 2{p-1}-2\dk}b_1^{L_++(1-\dk)+\eta(1-\dep)}.
\eee
We now compute the leading order term from \fref{parameters}, \fref{defmodtuntilde}. We derive from \fref{degsiplus}, \fref{degsiminus} the rough bound: for $y\leq 2B_0$
\be
\label{reougboundweta}
|\zeta_{b,a}|+|y\cdot\nabla \zeta_{a,b}|\lesssim \frac{b_1(1+y^2)}{1+y^{\gamma}}+\frac{b_1^{1+\frac \alpha 2}(1+y^2)}{1+y^{\frac{2}{p-1}}}\lesssim \frac{b_1(1+y^2)}{1+y^{\gamma}}
\ee 
which together with the cancellation $\Lt^*J\Phi_{0,-}=0$ and \fref{parameters} gives:
\bee
&&\left|\lsl+b_1\right||(\Lt^{L_+}\Lambda \qbt,\chi_{B_0}J\Phi_{0,-})|+|\gamma_s-a_1|(\Lt^{L_+}J\qbt,\chi_{B_0}J\Phi_{0,-})|\\
&\lesssim&  b_1^{L_++1+(1-\dk)+\eta(1-\delta_p)},\int_{B_0\leq y\leq 2B_0}\frac{b_1(1+y^2)}{1+y^{\gamma}}\frac{y^{d-1}dy}{1+y^{2L_++\frac{2}{p-1}}}\\
& \lesssim &   b_1^{L_++1+(1-\dk)+\eta(1-\delta_p)}.
\eee
To estimate the lower order terms, we first observe the rough bound for $y\leq 2B_0$, $1\leq j\leq L_+$:
\bea
\label{estroughtermmodun}
\nonumber &&\left|\sum_{m=j+1}^{L_++2}\frac{\partial S_{m,+}}{\partial b_j}+\sum_{m=j+1}^{L_-+2}\frac{\pa S_{m,-}}{\pa b_j}\right|\\
\nonumber & \lesssim & \sum_{m=j+1}^{L_++2}b_1^{m-j}\left[y^{2(m-1)-\gamma}+b_1^{\frac\alpha 2}y^{2m-\gamma}+b_1^{\frac\alpha 2}y^{2(m-1)-\frac{2}{p-1}}+b_1^{\alpha }y^{2m-\frac{2}{p-1}}\right]\\
& \lesssim & b_1y^{2j-\gamma}
\eea
and hence for $1\leq j\leq L_+$:
\bee
&&\left|\left(\sum_{m=j+1}^{L_++2}\frac{\partial S_{m,+}}{\partial b_j}+\sum_{m=j+1}^{L_-+2}\frac{\pa S_{m,-}}{\pa b_j},(\Lt^*)^{L_+}(\chi_{B_0}J\Phi_{0,-})\right)\right|\\
& \lesssim & \int_{B_0\leq y \leq 2B_0}\frac{b_1y^{2j-\gamma}}{y^{2L_++\frac{2}{p-1}}}y^{d-1}dy\lesssim  b_1B_0^{d-\gamma-\frac2{p-1}}\lesssim b_1^{1-\dk}B_0^{d-\gamma-\frac2{p-1}-2\dk}.
\eee
We obtain, using  $$B_0^{2-\gamma-\frac2{p-1}}\lesssim (\chi_{B_0}\Phi_{0,+},J\Phi_{0,-})\lesssim B_0^{2-\gamma-\frac2{p-1}},$$ the cancellation $\Lt^{L_+}\Phi_{j,+}=0$ for $j\leq L_+-1$ and \fref{parameters}:
\bee
&&\left|\sum_{j=1}^{L_+-1}\left[(b_j)_s+(2j-\alpha)b_1b_j-b_{j+1}\right]\right.\\
&\times& \left.\left(\Phi_{j,+}+\sum_{m=j+1}^{L_++2}\frac{\partial S_{m,+}}{\partial b_j}+\sum_{m=j+1}^{L_-+2}\frac{\pa S_{m,-}}{\pa b_j},(\Lt^*)^{L_+}(\chi_{B_0}J\Phi_{0,-})\right)\right|\\
& \lesssim &  b_1^{L_++1+(1-\dk)+\eta(1-\delta_p)}b_1^{1-\dk}B_0^{d-\gamma-\frac2{p-1}-2\dk}\\
&\lesssim & b_1^{L_++1+(1-\dk)+\eta(1-\delta_p)}B_0^{d-\gamma-\frac2{p-1}-2\dk}
\eee
and using \fref{parameterspresicely} for the leading order term:
\bee
&&\left[(b_{L_+})_s+(2L_+-\alpha)b_1b_{L_+}]\right.\\
&\times& \left.\left(\Phi_{L_+,+}+\sum_{m=L_++1}^{L_++2}\frac{\partial S_{L_+,+}}{\partial b_{L_+}}+\sum_{m=L_++1}^{L_-+2}\frac{\pa S_{m,-}}{\pa b_{L_+}},(\Lt^*)^{L_+}(\chi_{B_0}J\Phi_{0,-})\right)\right]\\
& = & [(b_{L_+})_s+(2L_+-\alpha)b_1b_{L_+}]\left[(\Phi_{0,+},\chi_{B_0}J\Phi_{0,-})+O\left(b_1^{1-\dk}B_0^{d-\gamma-\frac2{p-1}-2\dk}\right)\right]\\
& = & [b_{L_+})_s+(2L_+-\alpha)b_1b_{L_+}](\Phi_{0,+},\chi_{B_0}J\Phi_{0,-})\\
&+& O\left(\left[\frac{\sqrt{\mathcal E_{s_+}}}{M^{2\dk}}+  b_1^{L_++1+(1-\dk)+\eta(1-\delta_p)}\right]b_1^{1-\dk}B_0^{d-\gamma-\frac{2}{p-1}-2\dk}\right).
\eee
We now observe the rough bound for $y\leq 2B_0$, $1\leq j\leq L_-$:
\bea
\label{estintermdeuxe}
\nonumber &&\left|\sum_{m=j+1}^{L_++2}\frac{\partial S_{m,+}}{\partial a_j}+\sum_{m=j+1}^{L_-+2}\frac{\pa S_{m,-}}{\pa a_j}\right|\\
\nonumber& \lesssim & \sum_{m=j+1}^{L_++2}b_1^{m-j}\left[y^{2(m-1)-\gamma}+y^{2m-\gamma}\right]+\sum_{m=j+1}^{L_-+2}b_1^{m-j}\left[y^{2(m-1)-\frac{2}{p-1}}+b_1^{\frac\alpha 2}y^{2m-\frac{2}{p-1}}\right]\\
& \lesssim & y^{2j-\gamma}+b_1y^{2j-\frac{2}{p-1}}
\eea
and hence:
\bee
&&\left|\left(\Phi_{j,-}+\sum_{m=j+1}^{L_++2}\frac{\partial S_{m,+}}{\partial a_j}+\sum_{m=j+1}^{L_-+2}\frac{\pa S_{m,-}}{\pa a_j},(\Lt^*)^{L_+}(\chi_{B_0}J\Phi_{0,-})\right)\right|\\
& \lesssim & \int_{y\leq B_0\leq 2B_0}\frac{y^{2j-\gamma}+b_1y^{2j-\frac{2}{p-1}}}{y^{2L_++\frac{2}{p-1}}}y^{d-1}dy\lesssim  B_0^{d-\gamma-\frac2{p-1}+2(j-L_+)}+b_1B_0^{d-\frac 4{p-1}+2(j-L_+)}\\
& \lesssim &  B_0^{d-\gamma-\frac2{p-1}-2\Delta k}++b_1B_0^{d-\frac 4{p-1}-2\Delta k}\lesssim B_1^{d-\gamma-\frac2{p-1}-2}+b_1B_1^{d-\gamma-\frac 2{p-1}-2\dk+2\dkm}\\
& \lesssim &  B_1^{d-\gamma-\frac2{p-1}-2\dk}
\eee
where we used \fref{dgammkrealtion}. Hence using \fref{parameters}, \fref{parameterspresicelya}:
\bee
&&\left|\sum_{j=1}^{L_-}\left[(a_j)_s+2jb_1a_j-a_{j+1}\right]\right.\\
&\times& \left.\left(\Phi_{j,-}+\sum_{m=j+1}^{L_++2}\frac{\partial S_{m,+}}{\partial a_j}+\sum_{m=j+1}^{L_-+2}\frac{\pa S_{m,-}}{\pa a_j},(\Lt^*)^{L_+}(\chi_{B_0}J\Phi_{0,-})\right)\right|\\
& \lesssim & \left[M^C\sqrt{\mathcal E_{s_+}}+  b_1^{L_++1+(1-\dk)+\eta(1-\delta_p)}\right]B_0^{d-\gamma-\frac2{p-1}-2\dk}.
\eee
The collection of above bounds together with the lower bound 
$$ (\Phi_{0,+},\chi_{B_0}J\Phi_{0,-})\gtrsim B_0^{d-\gamma-\frac2{p-1}}$$ yield the preliminary estimate:
\bea
\label{cnoneonoenevknoe}
\nonumber &&\left|\left[(b_{L_+})_s+(2L_+-\alpha))b_1b_{L_+}\right]+\frac{1}{(\Phi_{0,+},\chi_{B_0}J\Phi_{0,-})}\frac{d}{ds}\left\{(\Lt^{L_+}\e,\chi_{B_0}J\Phi_{0,-})\right\}\right|\\
\nonumber & \lesssim &  \frac{B_0^{d-\gamma-\frac 2{p-1}-2\dk}}{B_0^{d-\gamma-\frac 2{p-1}}}\left[C(M)\sqrt{\mathcal E_{s_+}}+b_1^{L_++(1-\dk)+\eta(1-\delta_p)}\right]\\
 & \lesssim &  \frac{1}{B_0^{2\dk}}\left[C(M)\sqrt{\mathcal E_{s_+}}+b_1^{L_++(1-\dk)+\eta(1-\delta_p)}\right].
\eea
We now observe the bound
\bea
\label{neinoeneo}
\nonumber \frac{|(\Lt^{L_+}\e,\chi_{B_0}J\Phi_{0,-})|}{(\Phi_{0,+},\chi_{B_0}J\Phi_{0,-})}& \lesssim &  \left(\int \frac{|\Lt^{L_+}\e|^2}{1+y^{2+4k_+}}\right)^{\frac 12}\frac{B_0^{1+2k_++\frac d2-\frac{2}{p-1}}}{B_0^{d-\gamma-\frac{2}{p-1}}}\\
& \lesssim & C(M)B_0^{2(1-\dk)} \sqrt{\mathcal E_{s_+}}
\eea
which implies:
\bee
&&\left|(\Lt^{L_+}\e,\chi_{B_0}J\Phi_{0,-})|\frac{d}{ds}\frac{1}{(\Phi_{0,+},\chi_{B_0}J\Phi_{0,-})}\right|\\
&\lesssim& \frac{|(\Lt^{L_+}\e,\chi_{B_0}J\Phi_{0,-})|}{(\Phi_{0,+},\chi_{B_0}J\Phi_{0,-})^2}b_1\int_{B_0\leq y\leq 2B_0} |\Lambda Q|Q\\
& \lesssim & C(M)b_1\frac{B_0^{2(1-\dk)} \sqrt{\mathcal E_{s_+}}}{B_0^{d-\gamma-\frac 2{p-1}}}B_0^{d-\gamma-\frac 2{p-1}}\lesssim \frac{C(M)\sqrt{\mathcal E_{s_+}}}{B_0^{2\dk}}.
\eee
Injecting this into \fref{cnoneonoenevknoe} yields the expected bound \fref{improvedboundbl}.\\

\noindent{\bf step 2} Proof of \fref{improvedboundal}. We commute \fref{eqepsilon} with $\Lt^{L_-}$ and take the scalar product with $\chi_{B_0}J\Phi_{0,+}$. This yields:
\bee
&&\frac{d}{ds}\left\{(\Lt^{L_-}\e,J\chi_{B_0}J\Phi_{0,+})\right\}-(\Lt^{L_-}\e,J\Phi_{0,+}\pa_s(\chi_{B_0}))\\
& = & (\Lt^{L_-+1}\e,J\chi_{B_0}\Phi_{0,+})+\lsl(\Lt^{L_-}\Lambda \e,\chi_{B_0}J\Phi_{0,+})-\gamma_s(\Lt^{L_-}J\e,J\chi_{B_0}\Phi_{0,+})\\
& +& (\Lt^{L_+}(F-\widetilde{\mbox{Mod}}),J\chi_{B_0}\Phi_{0,+}).
\eee
We recall the notation $L_++k_+=L_-+k_-$. The linear term is estimated by Cauchy-Schwarz using the estimate \fref{coerciviteuk} and $ \Lt^*(\Phi_{0,+})=0$:
\bee
&&\nonumber |(\Lt^{L_-+1}\e,J\chi_{B_0}\Phi_{0,-})|\lesssim B_0^{1+2k_-}\|\Lt^*(J\chi_{B_0}\Phi_{0,+})\|_{L^2}\left(\int\frac{|\Lt^{L_-}\e|^2}{1+y^{2+4k_-}}\right)^{\frac 12}\\
&\lesssim &C(M)B_0^{1+2k_-}B_0^{\frac d2-\gamma-2}\sqrt{\mathcal E_{s_+}}=  C(M)B_0^{d-\gamma-\frac{2}{p-1}-2\dkm}\sqrt{\mathcal E_{s_+}}.
\eee
Similarily:
\bee
\nonumber &&\left|\lsl(\Lt^{L_-}\Lambda \e,\chi_{B_0}J\Phi_{0,+})\right|+\left|\gamma_s(\Lt^{L_-}J\e,J\chi_{B_0}\Phi_{0,+})\right|\\
&\lesssim& b_1\left(\int \frac{|\Lambda \e|^2+|\e|^2}{1+y^{4(L_-+k_-)+2}}\right)^{\frac 12}\left(\int (1+y^{4(L_-+k_-)+2})|(\Lt^*)^{L_-}\chi_{B_0}J\Phi_{0,+}|^2\right)^{\frac 12}\\
& \lesssim & b_1C(M)\sqrt{\mathcal E_{s_+}}B_0^{\frac d2-\gamma+2k_-+1}\leq C(M)B_0^{d-\gamma-\frac{2}{p-1}-2\dkm}\sqrt{\mathcal E_{s_+}}\\
&\lesssim & C(M)B_0^{d-\gamma-\frac{2}{p-1}-2\dkm}\sqrt{\mathcal E_{s_+}}
\eee
\bee
&&\left|(\Lt^{L_-}\e,J\Phi_{0,-}\pa_s(\chi_{B_0}))\right)|\lesssim \left|\frac{(b_1)_s}{b_1}\right|\left(\int \frac{|\Lt^{L_-}\e|^2}{1+y^{4k_-+2}}\right)^{\frac 12}\left(\int_{B_0\leq y\leq 2B_0} (1+y^{4k_-+2})|y^{-\frac{2}{p-1}}|^2\right)^{\frac 12}\\
& \lesssim & b_1C(M)B_0^{2k_-+1+\frac d2-\gamma}\sqrt{\mathcal E_{s_+}}\leq C(M)B_0^{d-\gamma-\frac{2}{p-1}-2\dkm}\sqrt{\mathcal E_{s_+}}.
\eee
We now estimate the $F$ terms. We anticipate the bound \fref{tobeprovedfnien} to estimate:
\bee
&&\left|(\Lt^{L_-}N(\e),J\chi_{B_0}\Phi_{0,+})\right|\\
&\lesssim&\left(\int \frac{|N(\e)|^2}{1+y^{2s_+}}\right)^{\frac 12}\left(\int_{y\leq 2B_0}(1+y^{2(k_-+L_-)+1-\gamma-2L_-})^2y^{d-1}dy\right)^{\frac 12}\\
& \lesssim & b_1^{1+\frac{\nu(d,p)}{2}}\sqrt{\mathcal E_{s_+}}B_0^{d-\gamma-\frac{2}{p-1}-2\dkm+2}\lesssim B_0^{d-\gamma-\frac{2}{p-1}-2\dkm}\sqrt{\mathcal E_{s_+}}
\eee
and similarly using \fref{boundle}:
\bee
&&\left|(\Lt^{L_-}L(\e),J\chi_{B_0}\Phi_{0,+})\right|\\
&\lesssim &\left(\int \frac{|L(\e)|^2}{1+y^{2s_+-4}}\right)^{\frac 12}\left(\int_{y\leq 2B_0}(1+y^{2(k_-+L_-)+1-2-\gamma-2L_-})^2y^{d-1}dy\right)^{\frac 12}\\
&\lesssim & B_0^{d-\gamma-\frac{2}{p-1}-2\dkm}\sqrt{\mathcal E_{s_+}}.
\eee

We estimate the $\Psit_b$ term from \fref{nkonenoenve}:
\bee
&&|(\Lt^{L_-}\Psit,\chi_{B_0}J\Phi_{0,+})|=|(\Psit,(\Lt^*)^{L_-}\chi_{B_0}J\Phi_{0,+})|\\
&\lesssim & \left(\int \frac{|\Psit|^2}{1+y^{4(k_-+L_-)+2}}\right)^{\frac 12}\left(\int_{B_0\leq y\leq 2B_0}y^{4(k_-+L_-)+2}|y^{-\gamma-2L_-}|^2\right)^{\frac 12}\\
&\lesssim & b_1^{L_++2+(1-\dk)-C_{L_+}\eta}B_0^{2k_-+1-\gamma+\frac d2}=B_0^{d-\gamma-\frac 2{p-1}-2\dkm+2}b_1^{L_++2+(1-\dk)-C_{L_+}\eta}\\
& \lesssim & B_0^{d-\gamma-\frac 2{p-1}-2\dkm}b_1^{L_++(1-\dk)+\eta(1-\delta_p)}.
\eee
We now estimate using \fref{reougboundweta} , the cancellation $\Lt^*J\Phi_{0,+}=0$ and \fref{parameters}:
\bee
&&\left|\lsl+b_1\right||(\Lt^{L_-}\Lambda \qbt,\chi_{B_0}J\Phi_{0,+})|+|\gamma_s-a_1|(\Lt^{L_-}J\qbt,\chi_{B_0}J\Phi_{0,+})|\\
&\lesssim& b_1^{L_++1+(1+\eta)(1-\dk)}\left\{\int_{B_0\leq y\leq 2B_0}\frac{b_1}{1+y^{\gamma}}\frac{y^{d-1}dy}{1+y^{2L_-+\gamma}}\right\}\\
& \lesssim & b_1^{L_++1+(1-\dk)+\eta(1-\delta_p)}
\eee
Next, from \fref{estroughtermmodun} for $1\leq j\leq L_+$:
\bee
&&\left|\left(\Phi_{j,+}+\sum_{m=j+1}^{L_++2}\frac{\partial S_{m,+}}{\partial b_j}+\sum_{m=j+1}^{L_-+2}\frac{\pa S_{m,-}}{\pa b_j},(\Lt^*)^{L_-}(\chi_{B_0}J\Phi_{0,+})\right)\right|\\
& \lesssim & \int_{B_0\leq y\leq 2B_0}\frac{b_1y^{2j-\gamma}}{y^{2L_-+\gamma}}y^{d-1}dy\lesssim  b_1B_0^{2\Delta k -\alpha}B_0^{d-\gamma-\frac2{p-1}}\\
& \lesssim & B_0^{d-\gamma-\frac2{p-1}-2\dkm}b_1B_0^{2\dk} \lesssim B_0^{d-\gamma-\frac2{p-1}-2\dkm}
\eee
and hence:
\bee
&&\left|\sum_{j=1}^{L_+}\left[(b_j)_s+(2j-\alpha)b_1b_j-b_{j+1}\right]\right.\\
&\times& \left.\left(\Phi_{j,+}+\sum_{m=j+1}^{L_++2}\frac{\partial S_{m,+}}{\partial b_j}+\sum_{m=j+1}^{L_-+2}\frac{\pa S_{m,-}}{\pa b_j},(\Lt^*)^{L_-}(\chi_{B_0}J\Phi_{0,+})\right)\right|\\
& \lesssim & \left[\frac{\sqrt{\mathcal E_{s_+}}}{M^{2\dk}}+ b_1^{L_++1+(1-\dk)+\eta(1-\delta_p)}\right] B_0^{d-\gamma-\frac2{p-1}-2\dkm}.
\eee
From \fref{estintermdeuxe} and $\alpha>2$  for $1\leq j\leq L_-$:
\bee
&&\left|\left(\sum_{m=j+1}^{L_++2}\frac{\partial S_{m,+}}{\partial a_j}+\sum_{m=j+1}^{L_-+2}\frac{\pa S_{m,-}}{\pa a_j},(\Lt^*)^{L_-}(\chi_{B_0}J\Phi_{0,+})\right)\right|\\
& \lesssim & \int_{B_0\leq y\leq 2B_0}\frac{y^{2j-\gamma}+b_1y^{2j-\frac{2}{p-1}}}{y^{2L_-+\gamma}}y^{d-1}dy\lesssim  B_0^{d-\gamma-\frac2{p-1}-\alpha}+b_1B_0^{d-\frac2{p-1}-\gamma}\\
& \lesssim &  B_0^{d-\gamma-\frac2{p-1}-2\dkm},
\eee
which together with \fref{parameters} gives:
\bee
&&\left|\sum_{j=1}^{L_--1}\left[(a_j)_s+2jb_1a_j-a_{j+1}\right]\right.\\
&\times& \left.\left(\Phi_{j,-}+\sum_{m=j+1}^{L_++2}\frac{\partial S_{m,+}}{\partial a_j}+\sum_{m=j+1}^{L_-+2}\frac{\pa S_{m,-}}{\pa a_j},(\Lt^*)^{L_+}(\chi_{B_0}J\Phi_{0,-})\right)\right|\\
& \lesssim &B_0^{d-\gamma-\frac2{p-1}-2\dkm} b_1^{L_++1+(1-\dk)+\eta(1-\delta_p)}
\eee
Finally, from \fref{parameterspresicelya}:
\bee
&&\left[(a_{L_-})_s+2L_-b_1a_{L_-}]\right.\\
&\times& \left.\left(\Phi_{L_-,-}+\sum_{m=L_++1}^{L_++2}\frac{\partial S_{L_+,+}}{\partial b_{L_+}}+\sum_{m=L_++1}^{L_-+2}\frac{\pa S_{m,-}}{\pa b_{L_+}},(\Lt^*)^{L_-}(\chi_{B_0}J\Phi_{0,+})\right)\right]\\
& = & [(a_{L_-})_s+2L_-b_1a_{L_-}]\left[(\Phi_{0,-},\chi_{B_0}J\Phi_{0,+})+O(B_0^{d-\gamma-\frac2{p-1}-2\dkm})\right]\\
& = & [(a_{L_-})_s+2L_-b_1a_{L_-}](\Phi_{0,+},\chi_{B_0}J\Phi_{0,-})\\
&+& O\left(\left[M^C\sqrt{\mathcal E_{s_+}}+  b_1^{L_++1+(1-\dk)+\eta(1-\delta_p)}\right]B_0^{d-\gamma-\frac{2}{p-1}-2\dkm}\right).
\eee
The collection of above bounds yields the preliminary estimate:
\bea
\label{cnoneonoenevknoebis}
\nonumber &&\left|[(a_{L_-})_s+2L_-b_1a_{L_-}]+\frac{1}{(\Phi_{0,-},\chi_{B_0}J\Phi_{0,+})}\frac{d}{ds}\left\{(\Lt^{L_-}\e,\chi_{B_0}J\Phi_{0,+})\right\}\right|\\
\nonumber & \lesssim &  \frac{B_0^{d-\gamma-\frac 2{p-1}-2\dkm}}{B_0^{d-\gamma-\frac 2{p-1}}}\left[C(M)\sqrt{\mathcal E_{s_+}}+b_1^{L_++(1-\dk)+\eta(1-\delta_p)}\right]\\
 & \lesssim &  \frac{1}{B_0^{2\dkm}}\left[C(M)\sqrt{\mathcal E_{s_+}}+b_1^{L_++(1-\dk)+\eta(1-\delta_p)}\right].
\eea
We now observe the bound
\bea
\label{neinoeneobis}
\nonumber \frac{|(\Lt^{L_-}\e,\chi_{B_0}J\Phi_{0,+})|}{(\Phi_{0,-},\chi_{B_0}J\Phi_{0,+})}& \lesssim &  \left(\int \frac{|\Lt^{L_-}\e|^2}{1+y^{2+4k_-}}\right)^{\frac 12}\frac{B_0^{1+2k_-+\frac d2-\gamma}}{B_0^{d-\gamma-\frac{2}{p-1}}}\\
& \lesssim & C(M)B_0^{2(1-\dkm)} \sqrt{\mathcal E_{s_+}}
\eea
which implies:
\bee
&&\left|(\Lt^{L_-}\e,\chi_{B_0}J\Phi_{0,-})|\frac{d}{ds}\frac{1}{(\Phi_{0,-},\chi_{B_0}J\Phi_{0,+})}\right|\\
&\lesssim& \frac{|(\Lt^{L_-}\e,\chi_{B_0}J\Phi_{0,+})|}{(\Phi_{0,-},\chi_{B_0}J\Phi_{0,+})^2}b_1\int_{B_0\leq y\leq 2B_0} |\Lambda Q|Q\\
& \lesssim & C(M)b_1\frac{B_0^{2(1-\dkm)} \sqrt{\mathcal E_{s_+}}}{B_0^{d-\gamma-\frac 2{p-1}}}B_0^{d-\gamma-\frac 2{p-1}}\lesssim \frac{C(M)\sqrt{\mathcal E_{s_+}}}{B_0^{2\dkm}}.
\eee
Inserting this into \fref{cnoneonoenevknoebis} yields the expected bound \fref{improvedboundal}.

\end{proof}

%%%%%%%%%%%%%%%%%%%%%%%%%%%%%%%%%%%%%
%%%%%%%%%%%%%%%%%%%%%%%%%%%%%%%%%%%%%%%

\section{Monotonicity}
\label{sectionmonoton}
%%%%%%%%%%%%%%%%%%%%%%%%%%%%%%%%%%%%%
%%%%%%%%%%%%%%%%%%%%%%%%%%%%%%%%%%%%%%%

We are now in position to derive the main monotonicity tools at the heart of the control of the infinite dimensional part of the solution. We rely on two classical sets of estimates: energy estimates, at both high and low level of regularity, yet above scaling, and a Morawetz bound to control local errors on the soliton core. 
Note that neither of these two estimates is sufficient to provide decay on its own, only the combination of the two is successful. Roughly speaking, the energy bound provides the outer control in the self-similar region, while the Morawetz estimate controls radiation on the soliton core.

%%%%%%%%%%%%%%%%%%%%%%%%%%%%%%%%%%%%%
%%%%%%%%%%%%%%%%%%%%%%%%%%%%%%%%%%%%%%%

\subsection{Monotonicity for the high Sobolev norm}

%%%%%%%%%%%%%%%%%%%%%%%%%%%%%%%%%%%%%%
%%%%%%%%%%%%%%%%%%%%%%%%%%%%%%%%%%%%%%

We now turn to  the derivation of a suitable Lyapunov functional for the $\mathcal E_{s_+}$ energy.\\
Recall the decomposition of the flow \fref{decmopvemeope}. We define the derivatives of $w,\e$ adapted to the corresponding linearized Hamiltonians $\Lt_\l,\Lt$:  
$$w_k = \Lt^k_{\lambda}w,\ \ \e_k = \Lt^k\e, \ \ k\geq 0$$ and claim:

\begin{proposition}[Lyapunov monotonicity for the high Sobolev norm]
\label{AEI2}
Let 
\be
\label{defg}
g=\frac{1-\dep}{4},
\ee
then there holds:
\bea
\label{monoenoiencle}
&&\frac{d}{dt} \left\{\frac{\mathcal E_{s_+}}{\lambda^{2(s_+-s_c)}}\left[1+O(b_1^{\eta(1-\dep)})\right]\right\}\leq   \frac{ b_1}{\lambda^{2(s_+-s_c)+2}}\left\{ \frac{\mathcal E_{s_+}}{M^{c\dk}}\right.\\
\nonumber &+& \left. C(M)b_1^{2L_++2(1-\dk)+\eta(1-\dep)}+C(M)\int \frac{1}{1+y^{4g}}\left[|\nabla\e_{k_++L_+}|^2+\frac{|\e_{k_++L_+}|^2}{1+y^2}\right]\right\}
\eea
for some universal constant $c>0$ independent of $M,\eta$ and of the bootstrap constant $K$ in \fref{bootnorm}, \fref{bootsmallsigma}.
\end{proposition}

\begin{proof}[Proof of Proposition \ref{AEI2}] {\bf step 1} Suitable derivatives and energy identity. Using the notation \fref{eqenwini} we compute from \fref{eqenwini}:
\be
\label{eqtwol}
\partial_{t} w_{k_++L_+} - \Lt_{\lambda} w_{k_++L_+} =  [\pa_t,\Lt_\l^{k_++L_+}]w + \Lt^{k_++L_+}_{\lambda} \left(\frac 1{\lambda^2} \mathcal F_{\lambda}\right)
\ee
We now derive the energy identity for \fref{eqtwol} using the self-adjointness \fref{adjointltilde}:
\par
\bea
\label{firstestimate}
&&\nonumber  \frac{d}{dt}\frac{\mathcal E_{s_+}}{2} =  \frac{1}{2} \frac{d}{d t}\left\{ (J\Lt_\lambda w_{k_++L_+},w_{k_++L_+})\right\}\\
\nonumber&=& \frac 12(J[\pa_t,\Lt_\lambda] w_{k_++L_+},w_{k_++L_+})+(\pa_tw_{k_++L_+},J\Lt_\lambda w_{k_++L_+})\\
\nonumber& = &  \frac 12(J[\pa_t,\Lt_\lambda] w_{k_++L_+},w_{k_++L_+})+( [\pa_t,\Lt_\l^{k_++L_+}]w,J\Lt_\lambda w_{k_++L_+})\\
& + & \left(\Lt^{k_++L_+}\left[\frac 1{\lambda^2} \mathcal F_{\lambda}\right],J\Lt_\lambda w_{k_++L_+}\right)
\eea
Our next goal is to  estimate all the terms in \fref{firstestimate}.\\

\noindent{\bf step 2} Well localized quadratic terms. By definition:
$$J\Lt_\lambda=\left(\begin{array}{ll}-\Delta+1-p\frac{1}{\l^2}Q^{p-1}\left(\frac{r}{\lambda}\right)&0\\ 0& -\Delta+1-\frac{1}{\l^2}Q^{p-1}\left(\frac{r}{\lambda}\right)\end{array}\right)$$ from which
\be
\label{defvero}
J[\pa_t,\Lt_\lambda]=\frac{1}{\l^4}\lsl\left(\begin{array}{ll}pV_0\left(\frac{r}{\lambda}\right)&0\\ 0& V_0\left(\frac{r}{\lambda}\right)\end{array}\right), \ \ V_0=(p-1)Q^{p-2}\Lambda Q.
\ee
We observe the improved decay
\be
\label{improvedecayvnot}
|\nabla^kV_0|\lesssim \frac{1}{y^{\gamma+\frac{2(p-2)}{p-1}+k}}=\frac{1}{y^{2+\alpha+k}}, \ \ k\geq 0
\ee
which yields the bound:
$$
\left|(J[\pa_t,\Lt_\lambda] w_{k_++L_+},w_{k_++L_+})\right|\lesssim \frac{b_1}{\l^{2(s_+-s_c)+2}}\int \frac{|\e_{k_++L_+}|^2}{1+y^{2+\alpha}}.
$$
We now claim the estimate
\bea
\label{commutator}
\nonumber&&\int (1+y^{2\alpha})\left|\nabla [\pa_t,\Lt_\l^{k_++L_+}]w\right|^2+\int (1+y^{2\alpha+2})\frac{\left|[\pa_t,\Lt_\l^{k_++L_+}]w\right|^2}{1+y^2}\\
& \lesssim & C(M)\frac{b_1^2}{\l^{2(s_+-s_c)+2}}\mathcal E_{s_+},
\eea
which is proved below. This implies:
\bee
&&\left|( [\pa_t,\Lt_\l^{k_++L_+}]w,J\Lt_\lambda w_{k_++L_+})\right|\\
& \leq& \frac{b_1}{\l^{2(s_+-s_c)+2}}\left\{\frac{\mathcal E_{s_+}}{M^{c\dk}}+C(M)\int\frac{1}{1+y^{2\alpha}}\left[|\nabla \e_{k_++L_+}|^2+\frac{|\e_{k_++L_+}|^2}{1+y^2}\right]\right\}
\eee
\noindent{\it Proof of \fref{commutator}}. A simple induction argument gives 
the formula: $$[\pa_t,\Lt_\l^{k_++L_+}]w=\sum_{k=0}^{k_++L_+-1}\Lt_{\l}^k[\pa_t,\Lt_\l] \Lt_{\l}^{k_++L_+-(k+1)}w.$$ We renormalize and compute explicitely from \fref{defvero}:
\bea
\label{formuacommut}
&&[\pa_t,\Lt_\l^{k_++L_+}]w\\
\nonumber &=&\frac{1}{\l^{2(k_++L_+)+2+\frac 2{p-1}}}\sum_{k=0}^{k_++L_+-1}\Lt^k\left(\begin{array}{ll} 0&\lsl V_0 \\-\lsl V_0&0\end{array}\right)\Lt^{(k_++L_+)-(k+1)}\e.
\eea
The regularity of $V_0$ at the origin and a simple application of the Leibniz rule with the improved decay \fref{improvedecayvnot} give the pointwise bound:
\bee
\left|[\pa_t,\Lt_\l^{k_++L_+}]w\right|&\lesssim &\frac{b_1}{\l^{2(k_++L_+)+2+\frac 2{p-1}}}\sum_{p=0}^{2(k_++L_+-1)}\frac{|\nabla^{2(k_++L_+-1)-p}\e|}{1+y^{2+\alpha+p}}\\
&=&\frac{b_1}{\l^{2(k_++L_+)+2+\frac 2{p-1}}}\sum_{m=0}^{2(k_++L_+-1)}\frac{|\nabla^m\e|}{1+y^{2(k_++L_+)+\alpha-m}}
\eee
$$\left|\nabla [\pa_t,\Lt_\l^{k_++L_+}]w\right|\lesssim \frac{b_1}{\l^{2(k_++L_+)+3+\frac 2{p-1}}}\sum_{m=0}^{2(k_++L_+-1)+1}\frac{|\nabla^m\e|}{1+y^{2(k_++L_+)+1+\alpha-m}}.$$
We conclude from \fref{coerciviityleybis}:
\bee
&&\int (1+y^{2\alpha})\left|\nabla [\pa_t,\Lt_\l^{k_++L_+}]w\right|^2+\int (1+y^{2\alpha})\frac{\left|[\pa_t,\Lt_\l^{k_++L_+}]w\right|^2}{1+y^2}\\
& \lesssim & \frac{b_1^2}{\l^{2(s_+-s_c)+2}}\sum_{m=0}^{2(k_++L_+-1)+1}\int |\nabla^m\e|^2\frac{1+y^{2\alpha}}{1+y^{4(k_++L_+)+2+2\alpha-2m}}\\
& \lesssim &  \frac{b_1^2}{\l^{2(s_+-s_c)+2}}\sum_{m=0}^{s_+}\int \frac{|\nabla^m\e|^2}{1+y^{2(s_+-m)}}\lesssim C(M) \frac{b_1^2}{\l^{2(s_+-s_c)+2}}\mathcal E_{s_+},
\eee
and \fref{commutator} is proved.\\

\noindent{\bf step 3} $\Psit$ terms. From \fref{controleh4erreurloc} and by the coercivity of $L_+,L_-$:
\bee
&&\left|\left(\Lt^{k_++L_+}\left[\frac 1{\lambda^2} \Psit_{\lambda}\right],J\Lt_\lambda w_{k_++L_+}\right)\right|\\
& \lesssim &\frac{1}{\l^{2(s_+-s_c)+2}} \left(\int \frac{|\e_{k_++L_+}|^2}{1+y^2}\right)^{\frac 12}\left(\int (1+y^2)|\Lt^* J\Lt^{k_++L_+}\Psit|^2\right)^{\frac 12}\\
&\lesssim& \frac{1}{\l^{2(s_+-s_c)+2}} \left(C\mathcal E_{s_+}\right)^{\frac 12}\left(b_1^{2L_++2+2(1+\eta)(1-\dk)}\right)^{\frac 12}\\
&\leq &\frac{b_1}{\l^{2(s_+-s_c)+2}}\left[\frac{\mathcal E_{s_+}}{M^{c\dk}}+C(M)b_1^{2L_++2(1-\dk)+\eta(1-\dep)}\right].
\eee

\noindent{\bf step 4} $\Modt$ terms. Recall \fref{defmodtuntilde}:
\bee
\widetilde{Mod}(t)& = & -\left(\lsl+b_1\right)\Lambda \qbt+(\gamma_s-a_1)J\qbt\\
\nonumber & + & \sum_{j=1}^{L_+}\left[(b_j)_s+(2j-\alpha)b_1b_j-b_{j+1}\right]\chi_{B_1}\left[\Phi_{j,+}+\sum_{m=j+1}^{L_++2}\frac{\partial S_{m,+}}{\partial b_j}+\sum_{m=j+1}^{L_-+2}\frac{\pa S_{m,-}}{\pa b_j}\right]\\
\nonumber& + &  \sum_{j=1}^{L_-}\left[(a_j)_s+2jb_1a_j-a_{j+1}\right]\chi_{B_1}\left[\Phi_{j,-}+\sum_{m=j+1}^{L_++2}\frac{\pa S_{m,+}}{\pa a_j}+\sum_{m=j+1}^{L_-+2}\frac{\partial S_{m,-}}{\partial a_j}\right].
\eee
We need to remove the last modulation equations for $(b_{L_+},a_{L_-})$ in order to take advantage of the improved bounds of Lemma \ref{improvedbl} since 
the pointwise bounds \fref{parameterspresicely}, \fref{parameterspresicelya} are not good enough to close. Let the directions
\be
\label{deftlplus} T_{L_+}=\chi_{B_1}\Phi_{L_+,+},\ \  T_{L_-}=\chi_{B_1}\Phi_{L_-,-}
\ee
and the vectors:
\be
\label{defradidaition}
\xi_+=\frac{(\Lt^{L_+}\e,\chi_{B_0}J\Phi_{0,-})}{(\Phi_{0,+},\chi_{B_0}J\Phi_{0,-})}T_{L_+}, \ \ \xi_-=\frac{(\Lt^{L_-}\e,\chi_{B_0}J\Phi_{0,+})}{(\Phi_{0,-},\chi_{B_0}J\Phi_{0,+})}T_{L_-}
\ee
We decompose
\be
\label{defwdoheat}
\widetilde{\Mod}=\widehat{\Mod}-\pa_s\xi_+-\pa_s\xi_-, \ \ \widehat{\Mod}=\widehat{\Mod}_{\rm rad}+\widehat{\Mod}_1+\widehat{\Mod}_2
\ee where
\bee
\nonumber \widehat{\Mod}_{\rm rad}&= & \frac{(\Lt^{L_+}\e,\chi_{B_0}J\Phi_{0,-})}{(\Phi_{0,+},\chi_{B_0}J\Phi_{0,-})}\pa_sT_{L_+}+\frac{(\Lt^{L_-}\e,\chi_{B_0}J\Phi_{0,+})}{(\Phi_{0,-},\chi_{B_0}J\Phi_{0,+})}\pa_sT_{L_-}\\
& + & T_{L_+}\left[O\left( \frac{1}{B_0^{2\dk}}\left[C(M)\sqrt{\mathcal E_{s_+}}+b_1^{L_++(1-\dk)+\eta(1-\dep)}\right]\right)\right]\\
&+&T_{L_-}\left[O\left( \frac{1}{B_0^{2\dkm}}\left[C(M)\sqrt{\mathcal E_{s_+}}+b_1^{L_++(1-\dk)+\eta(1-\dep)}\right]\right)\right],
\eee
 according to \fref{improvedboundbl}, \fref{improvedboundal} applied with $\mu=1$,
\bee
&&\widehat{\Mod}_1 =   -\left(\lsl+b_1\right)\Lambda \qbt+(\gamma_s-a_1)J\qbt+ \widehat{\Mod}_++\widehat{\Mod}_-,\\
&&\widehat{\Mod}_+=\sum_{j=1}^{L_+-1}\left[(b_j)_s+(2j-\alpha)b_1b_j-b_{j+1}\right]\chi_{B_1}\left[\Phi_{j,+}+\sum_{m=j+1}^{L_++2}\frac{\partial S_{m,+}}{\partial b_j}+\sum_{m=j+1}^{L_-+2}\frac{\pa S_{m,-}}{\pa b_j}\right],\\
&&\widehat{\Mod}_-= \sum_{j=1}^{L_--1}\left[(a_j)_s+2jb_1a_j-a_{j+1}\right]\left[\Phi_{j,-}+\sum_{m=j+1}^{L_++2}\frac{\pa S_{m,+}}{\pa a_j}+\sum_{m=j+1}^{L_-+2}\frac{\partial S_{m,-}}{\partial a_j}\right],
\eee
and the remaining term:
\bea
\label{defmodhat2}
\nonumber \widehat{\Mod}_2&=&\left[(b_{L_+})_s+(2L_+-\alpha)b_1b_{L_+}\right]\left[\sum_{m=L_++1}^{L_++2}\frac{\partial S_{m,+}}{\partial b_{L_+}}+\sum_{m=L_++1}^{L_-+2}\frac{\pa S_{m,-}}{\pa b_{L_+}}\right]\\
& + &  \left[(a_{L_-})_s+2jb_1a_{L_-}\right]\left[\sum_{m=L_-+1}^{L_++2}\frac{\pa S_{m,+}}{\pa a_{L_-}}+\sum_{m=L_-+1}^{L_-+2}\frac{\partial S_{m,-}}{\partial a_{L_-}}\right]
\eea
The bounds:
\bea
\label{tobeprovedmoeperad}
&&\int (1+y^2)|\Lt^*J\Lt^{k_++L_+}\widehat{\Mod}_{\rm rad}|^2\\
\nonumber &\lesssim &b_1^2\left[b_1^{(1-\dep)\eta}\mathcal E_{s_+}+b_1^{2L_++2(1-\dk)+2\eta(1-\dep)}\right],
\eea
\bea
\label{tobeprovedmoepe}
&&\int (1+y^2)|\Lt^*J\Lt^{k_++L_+}\widehat{\Mod}_1|^2\\
\nonumber &\lesssim &b_1^2\left[b_1^{(1-\dep)\eta}\mathcal E_{s_+}+b_1^{2L_++2(1-\dk)+2\eta(1-\dep)}\right]
\eea
\be
\label{tobeprovedmoepebis}
\int (1+y^{2+4g})|\Lt^*J\Lt^{k_++L_+}\widehat{\Mod}_2|^2\lesssim b_1^2\left[C(M)\mathcal E_{s_+}+b_1^{2L_++2(1-\dk)+2\eta(1-\dep)}\right]
\ee
with $g$ given by \fref{defg} follow by direct inspection. We then estimate the corresponding term in \fref{firstestimate}:
\bee
&&\left|\left(\Lt^{k_++L_+}\left[\frac 1{\lambda^2} (\widehat{\Mod}_1+\widehat{\Mod}_{\rm rad})_{\lambda}\right],J\Lt_\lambda w_{k_++L_+}\right)\right|\\
&\lesssim & \frac{1}{\l^{2(s_+-s_c)+2}}\left(\int (1+y^2)|\Lt^*J\Lt^{k_++L_+}(\widehat{\Mod}_1+\widehat{\Mod}_{\rm rad})|^2\right)^{\frac 12}\left(\int \frac{|\e_{k_++L_+}|^2}{1+y^2}\right)^{\frac 12}\\
&\lesssim & \frac{b_1}{\l^{2(s_+-s_c)+2}}\left[b_1^{(1-\dep)\eta}\mathcal E_{s_+} +b_1^{2L_++2(1-\dk)+2\eta(1-\dep)}\right]
\eee
and 
\bee
&&\left|\left(\Lt^{k_++L_+}\left[\frac 1{\lambda^2} (\widehat{\Mod}_2)_{\lambda}\right],J\Lt_\lambda w_{k_++L_+}\right)\right|\\
&\lesssim & \frac{1}{\l^{2(s_+-s_c)+2}}\left(\int (1+y^{2+4g})|\Lt^*J\Lt^{k_++L_+}\widehat{\Mod}_1|^2\right)^{\frac 12}\left(\int \frac{|\e_{k_++L_+}|^2}{1+y^{2+4g}}\right)^{\frac 12}\\
&\leq & b_1\left[C(M)\int \frac{|\e_{k_++L_+}|^2}{1+y^{2+4g}}+b_1^{2L_++2(1-\dk)+2\eta(1-\dep)}+\frac{\matchal E_{s_+}}{M^{c\dk}}\right].
\eee
\noindent{\it Proof of \fref{tobeprovedmoeperad}}: Using the cancellation $\Lt^* J\Lt^{L_++k_+}\Phi_{L_+,+}=0$ we first estimate:
\bee
&&\int (1+y^2)\left|\Lt^* J\Lt^{L_++k_+}(\chi_{B_1}\Phi_{L_+,+})\right|^2\lesssim  B_1^{d-2\gamma-4-4k_++2}=\frac{1}{B_1^{4(1-\dk)}}\\
& \lesssim & B_0^{4\dk}\frac{b^2_1B_0^4}{B_0^{4\dk}}\frac{1}{B_1^{4(1-\dk)}}\lesssim b^2_1B_0^{4\dk} \left(\frac{B_0}{B_1}\right)^{4(1-\dk)}.
\eee
This implies:
\be
\label{esttlplus}
\int(1+y^2)\left|\Lt^*J\Lt^{k_++L_+}T_{L_+}\right|^2
\lesssim b^2_1B_0^{4\dk} \left(\frac{B_0}{B_1}\right)^{4(1-\dk)}
\ee
and hence:
\bee
&&\frac{1}{B_0^{4\dk}}\left[C(M)\mathcal E_{s_+}+b_1^{2L_++2(1-\dk)+2\eta(1-\dep)}\right]\int(1+y^2)\left|\Lt^*J\Lt^{k_++L_+}T_{L_+}\right|^2\\
& \lesssim & \frac{1}{B_0^{4\dk}}\left[C(M)\mathcal E_{s_+}+b_1^{2L_++2(1-\dk)+2\eta(1-\dep)}\right]b^2_1B_0^{4\dk} \left(\frac{B_0}{B_1}\right)^{4(1-\dk)}\\
& \lesssim & b_1^2\left[b_1^{(1-\dk)\eta}\mathcal E_{s_+}+b_1^{2L_++2(1-\dk)+2\eta(1-\dep)}\right]
\eee
Similarly,
\bea
\label{eiovniennevo}
\nonumber &&\int (1+y^2)|\Lt^* J\Lt^{L_-+k_-}(\chi_{B_1}\Phi_{L_-,-})|^2\lesssim B_1^{d-\frac{4}{p-1}-4-4k_-+2}=\frac{1}{B_1^{4(1-\dkm)}}\\
& \lesssim &  B_0^{4\dkm} b^2_1\left(\frac{B_0}{B_1}\right)^{4(1-\dkm)}.
\eea
This implies:
\be
\label{esttlominus}
\int(1+y^2)\left|\Lt^*J\Lt^{k_++L_+}T_{L_-}\right|^2\lesssim b^2_1B_0^{4\dkm} \left(\frac{B_0}{B_1}\right)^{4(1-\dk)}
\ee
and hence:
\bee
&&\frac{1}{B_0^{4\dkm}}\left[C(M)\mathcal E_{s_+}+b_1^{2L_++2(1-\dk)+2\eta(1-\dep)}\right]\int(1+y^2)\left|\Lt^*J\Lt^{k_++L_+}T_{L_-}\right|^2\\
& \lesssim & \frac{1}{B_0^{4\dkm}}\left[C(M)\mathcal E_{s_+}+b_1^{2L_++2(1-\dk)+2\eta(1-\dep)}\right]b^2_1B_0^{4\dkm} \left(\frac{B_0}{B_1}\right)^{4(1-\dk)}\\
& \lesssim & b_1^2\left[b_1^{(1-\dkm)\eta}\mathcal E_{s_+}+b_1^{2L_++2(1-\dk)+2\eta(1-\dep)}\right].
\eee
Similarily,
\bea
\label{timedrivaotve}
&&\int (1+y^2)\left|\Lt^*J\Lt^{L_++k_+}\Phi_{L,+}\pa_s\chi_{B_1}\right|^2\lesssim b_1^2B_1^{2-4k_++d-2\gamma-4}=\frac{b_1^2}{B_1^{4(1-\dk)}}\\
&&\int (1+y^2)\left|\Lt^*J\Lt^{L_-+k_-}\Phi_{L,-}\pa_s\chi_{B_1}\right|^2\lesssim b_1^2B_1^{2-4k_-+d-\frac{4}{p-1}-4}=\frac{b_1^2}{B_1^{4(1-\dkm)}}
\eea
 Therefore, using \fref{neinoeneo}:
\bee
\int (1+y^2)\left|\frac{(\Lt^{L_+}\e,\chi_{B_0}J\Phi_{0,-})}{(\Phi_{0,+},\chi_{B_0}J\Phi_{0,-})}\Lt^*J\Lt^{L_-+k_-}\pa_sT_{L_+}\right|^2&\lesssim  &C(M)B_0^{4(1-\dk)} \mathcal E_{s_+}\left[\frac{b_1^2}{B_1^{4(1-\dk)}}\right]\\
&\lesssim &b_1^2b_1^{\eta(1-\dep)}\mathcal E_{s_+},
\eee
and with \fref{neinoeneobis}:
\bee
\int (1+y^2)\left|\frac{(\Lt^{L_-}\e,\chi_{B_0}J\Phi_{0,+})}{(\Phi_{0,-},\chi_{B_0}J\Phi_{0,+})}\Lt^*J\Lt^{L_-+k_-}\pa_sT_{L_-}\right|^2&\lesssim  &C(M)B_0^{4(1-\dkm)} \mathcal E_{s_+}\left[\frac{b_1^2}{B_1^{4(1-\dkm)}}\right]\\
&\lesssim &b_1^2b_1^{\eta(1-\dep)}\mathcal E_{s_+}.
\eee
This concludes the proof of \fref{tobeprovedmoeperad}.\\
\noindent{\it Proof of \fref{tobeprovedmoepe}, \fref{tobeprovedmoepebis}}:\\
\noindent\underline{ $\widehat{\Mod}_+$ terms}. From \fref{degsiplus}  by a brute force estimate for $1\leq j\leq L_+$:
\bee
&&\int (1+y^{2+4g})\left|\sum_{m=j+1}^{L_++2}\Lt^*J\Lt^{k_++L_+}\chi_{B_1}\frac{\partial S_{m,+}}{\partial b_j}\right|^2\\
&\lesssim &\int_{y\leq 2B_1} (1+y^{2+4g})\sum_{m=j+1}^{L_++2}b_1^{2(m-j)}\left[1+|y^{2(m-1)-\gamma-2(k_++L_++1)}|^2\right]\\
& + & \int_{y\leq 2B_1} (1+y^2)\sum_{m=j+1}^{L_++2}b_1^{2(m-j)+\alpha}\left[1+|y^{2m-\gamma-2(k_++L_++1)}|^2\right]\\
&\lesssim & b_1^2\sum_{m=j}^{L_++1} b_1^{2(m-j)}\int_{y\leq 2B_1}\frac{dy}{1+y^{1+4(L_+-m)+4(1-\dk)-4g}}\\
& + & \sum_{m=j+1}^{L_++2}b_1^{2(m-j)+\alpha}\int_{y\leq 2B_1}\frac{dy}{1+y^{1+4(L_+-m)+4(1-\dk)-4g}}\\
& \lesssim & b_1^2\left[1+b_1^{2(L_++1-j)}B_1^{4\dk+4g}\right]+b_1^{\alpha}\left[b_1^{2(L_++1-j)}B_1^{4\dk+4g}+b_1^{2(L_++2-j)}B_1^{4+4\dk+4g}\right]\\
& \lesssim & b_1^2\left[1+b_1^{2-2(\dk+g)-C\eta}\right]+b_1^{\alpha}b_1^{2-2(\dk+g)-C\eta}\lesssim  b_1^2
\eee
for $0<\eta\ll 1$ small enough, thanks to $\alpha>2$ and \fref{defg}. Similarily, from \fref{degsiminus}, \fref{dgammkrealtion}:
\bea
\label{estoneoneone}
\nonumber &&\int (1+y^{2+4g})\left|\sum_{m=j+1}^{L_++2}\Lt^*J\Lt^{k_++L_+}\chi_{B_1}\frac{\partial S_{m,-}}{\partial b_j}\right|^2\\
\nonumber &\lesssim &\int_{y\leq 2B_1} (1+y^{2+2g})\sum_{m=j+1}^{L_++2}b_1^{2(m-j)+\alpha}\left[1+|y^{2(m-1)-\frac{2}{p-1}-2(k_-+L_-+1)}|^2\right]\\
\nonumber & + & \int_{y\leq 2B_1} (1+y^{2+4g})\sum_{m=j+1}^{L_++2}b_1^{2(m-j)+2\alpha}\left[1+|y^{2m-\frac{2}{p-1}-2(k_-+L_-+1)}|^2\right]\\
\nonumber &\lesssim & b_1^2\sum_{m=j}^{L_++1} b_1^{2(m-j)+\alpha}\int_{y\leq 2B_1}\frac{dy}{1+y^{1+4(L_--m)+4(1-\dkm)-4g}}\\
\nonumber & + & \sum_{m=j+1}^{L_++2}b_1^{2(m-j)+2\alpha}\int_{y\leq 2B_1}\frac{dy}{1+y^{1+4(L_--m)+4(1-\dkm)-4g}}\\
\nonumber & \lesssim &b_1^{2}\left[1+\sum_{m=L_-+1}^{L_++2} b_1^{2(m-j)+\alpha}B_1^{4(m-L_-)-4(1-\dkm)+4g}\right]\\
\nonumber &\lesssim &  \lesssim  b_1^{2}\left[1+\sum_{m=L_-+1}^{L_++2} (b_1B_1^2)^{2(m-L_-)}b_1^{\alpha+2(L_--j)+2(1-\dkm)-2g-C_{L_+}\eta}\right]\\
\nonumber & \lesssim & b_1^{2}\left[1+b_1^{\alpha-2\Delta k+2(1-\dkm)-C_{L_+}\eta}\right]=b_1^{2}\left[1+b_1^{2(1-\dk)-C_{L_+}\eta}\right]\\
& \lesssim & b_1^2
\eea
Hence, using \fref{parameters}:
\bee
&&\int(1+y^2)\left|\sum_{j=1}^{L_+-1}\left[(b_j)_s+(2j-\alpha)b_1b_j-b_{j+1}\right]\Lt^*J\Lt^{k_++L_+}\right.\\
&&\left.\left[\chi_{B_1}\left(\Phi_{j,+}+\sum_{m=j+1}^{L_++2}\frac{\partial S_{m,+}}{\partial b_j}+\sum_{m=j+1}^{L_-+2}\frac{\pa S_{m,-}}{\pa b_j}\right)\right]\right|^2\\
& \lesssim & b_1^2b_1^{2L_++2+2(1-\dk)+2\eta(1-\dep)}
\eee
and \fref{tobeprovedmoepe} follows for $\widehat{\Mod}_+$. Moreover from \fref{parameterspresicely}:
\bee
&&\int(1+y^{2+4g})\left|\left[(b_{L_+})_s+2L_+b_1b_{L_+}\right]\Lt^*J\Lt^{k_++L_+}\right.\\
&&\left.\left[\chi_{B_1}\sum_{m=L_++1}^{L_++2}\frac{\partial S_{m,+}}{\partial b_{L_+}}+\sum_{m={L_+}+1}^{L_-+2}\frac{\pa S_{m,-}}{\pa b_{L_+}}\right]\right|^2\\
&\lesssim & b_1^2\left[C(M)\mathcal E_{s_+}+b_1^{2L_++2(1-\dk)+2\eta(1-\dep)}\right].
\eee

\noindent\underline{\it $\widehat{\Mod}_-$ terms}: From \fref{degsiplus} for $1\leq j\leq L_-$:
\bee
&&\int (1+y^{2+4g})\left|\sum_{m=j+1}^{L_++2}\Lt^*J\Lt^{k_++L_+}\chi_{B_1}\frac{\partial S_{m,+}}{\partial a_j}\right|^2\\
&\lesssim &\int_{y\leq 2B_1} (1+y^{2+4g})\sum_{m=j+1}^{L_++2}b_1^{2(m-j)}\left[1+|y^{2(m-1)-\gamma-2(k_++L_++1)}|^2\right]\\
& + & \int_{y\leq 2B_1} (1+y^2)\sum_{m=j+1}^{L_++2}b_1^{2(m-j)}\left[1+|y^{2m-\gamma-2(k_++L_++1)}|^2\right]\\
&\lesssim & b_1^2+\sum_{m=j+1}^{L_++2}b_1^{2(m-j)}\int_{y\leq 2B_1}\frac{dy}{1+y^{1+4(L_+-m)+4(1-\dk)-4g}}\\
&\lesssim & b_1^2+\sum_{m=L_++1}^{L_++2}b_1^{2(m-j)}B_1^{4(m-L_+)-4(1-\dk)+4g}\\
& \lesssim & b_1^2+b_1^{2(L_+-j)+2(1-\dk)-4g-C_{L_+}\eta}\lesssim b_1^2+b_1^{2\Delta k}\lesssim b_1^2,
\eee
since $\Delta k\geq 1$, and from \fref{degsiminus}:
\bee
&&\int (1+y^{2+4g})\left|\sum_{m=j+1}^{L_++2}\Lt^*J\Lt^{k_++L_+}\chi_{B_1}\frac{\partial S_{m,-}}{\partial a_j}\right|^2\\
&\lesssim &\int_{y\leq 2B_1} (1+y^2)\sum_{m=j+1}^{L_++2}b_1^{2(m-j)}\left[1+|y^{2(m-1)-\frac{2}{p-1}-2(k_-+L_-+1)}|^2\right]\\
& + & \int_{y\leq 2B_1} (1+y^{2+4g})\sum_{m=j+1}^{L_++2}b_1^{2(m-j)+\alpha}\left[1+|y^{2m-\frac{2}{p-1}-2(k_-+L_-+1)}|^2\right]\\
&\lesssim & b_1^2\sum_{m=j}^{L_++1} b_1^{2(m-j)}\int_{y\leq 2B_1}\frac{dy}{1+y^{1+4(L_--m)+4(1-\dkm)-4g}}\\
& + & \sum_{m=j+1}^{L_++2}b_1^{2(m-j)+\alpha}\int_{y\leq 2B_1}\frac{dy}{1+y^{1+4(L_--m)+4(1-\dkm)-4g}}\\
& \lesssim &b_1^2+\sum_{m=L_-}^{L_++1} b_1^{2(m-j)}B_1^{4(m-L_-)-4(1-\dkm)+4g}+\sum_{m=L_-+1}^{L_++2}b_1^{2(m-j)+\alpha}B_1^{4(m-L_-)-4(1-\dkm)+4g}\\
& \lesssim & b_1^2+b_1^{\alpha-C_{L_+}\eta+2(1-\dkm)-2g}\lesssim b_1^2.
\eee
since $\alpha>2$. Hence, using \fref{parameters}:
\bee
&&\int(1+y^{2})\left|\sum_{j=1}^{L_+-1}\left[(a_j)_s+2jb_1a_j-a_{j+1}\right]\Lt^*J\Lt^{k_++L_+}\right.\\
&&\left.\left[\chi_{B_1}\left(\Phi_{j,-}+\sum_{m=j+1}^{L_++2}\frac{\partial S_{m,+}}{\partial a_j}+\sum_{m=j+1}^{L_-+2}\frac{\pa S_{m,-}}{\pa a_j}\right)\right]\right|^2\\
& \lesssim & b_1^2b_1^{2L_++2+2(1-\dk)+2\eta(1-\dep)}
\eee
and \fref{tobeprovedmoepe} is proved for $\widehat{\Mod}_-$. Moreover from \fref{parameterspresicelya}:
\bee
&&\int(1+y^{2+4g})\left|\left[(a_{L_-})_s+2L_-b_1a_{L_-}\right]\Lt^*J\Lt^{k_++L_+}\right.\\
&&\left.\left[\chi_{B_1}\sum_{m=L_-+1}^{L_++2}\frac{\partial S_{m,+}}{\partial a_{L_-}}+\sum_{m={L_-}+1}^{L_-+2}\frac{\pa S_{m,-}}{\pa a_{L_-}}\right]\right|^2\\
&\lesssim & b_1^2\left[C(M)\mathcal E_{s_+}+b_1^{2L_++2(1-\dk)+2\eta(1-\dep)}\right].
\eee
This concludes the proof of \fref{tobeprovedmoepebis}.\\
\noindent\underline{\it Lower order modulation parameters}. We use $\Lt \Lambda Q=0$, $\Lt JQ=0$ and \fref{dgammkrealtion} to estimate:
\bee
&&\left|\int (1+y^2)|\Lt^*J\Lt^{k_++L_+}\Lambda \qbt\right|^2+\left|\int (1+y^2)|\Lt^*J\Lt^{k_++L_+}J \qbt\right|^2\\
&\lesssim & \int_{y\leq 2B_1}\sum_{j=1}^{L_+}b_1^{2j}(1+y^2|y^{2j-\gamma-2(k_++L_++1)}|^2)\\
& + & \int_{y\leq 2B_1}\sum_{j=2}^{L_++2}b_1^{2j}(1+y^2|y^{2(j-1)-\gamma-2(k_++L_++1)}|^2)+b_1^{2j+\alpha}(1+y^2|y^{2j-\gamma-2(k_++L_++1)}|^2)\\
&+& \int_{y\leq 2B_1}\sum_{j=1}^{L_-}b_1^{2j+\alpha}(1+y^2|y^{2j-\frac 2{p-1}-2(k_-+L_-+1)}|^2)\\
& + & \int_{y\leq 2B_1}\sum_{j=2}^{L_-+2}b_1^{2j+\alpha}(1+y^2|y^{2(j-1)-\frac{2}{p-1}-2(k_-+L_-+1)}|^2)+b_1^{2j+2\alpha}(1+y^2|y^{2j-\frac{2}{p-1}-2(k_-+L_-+1)}|^2)\\
& \lesssim & \sum_{j=1}^{L_+}b_1^{2j}\int_{y\leq 2B_1}\frac{dy}{1+y^{1+4(1-\dk)+4(L_+-j)}}+  \sum_{j=2}^{L_++2}b_1^{2j}\int_{y\leq 2B_1}\frac{dy}{1+y^{1+4(1-\dk)+4(L_++1-j)}}\\
&+& \sum_{j=2}^{L_++2}b_1^{2j+\alpha}\int_{y\leq 2B_1}\frac{dy}{1+y^{1+4(1-\dk)+4(L_+-j)}}+  \sum_{j=1}^{L_-}b_1^{2j+\alpha}\int_{y\leq 2B_1}\frac{dy}{1+y^{1+4(1-\delta_{k_-})+4(L_--j)}}\\
&+&  \sum_{j=2}^{L_-+2}b_1^{2j+\alpha}\int_{y\leq 2B_1}\frac{dy}{1+y^{1+4(1-\delta_{k_-})+4(L_-+1-j)}}\\
&+& \sum_{j=2}^{L_-+2}b_1^{2j+2\alpha}\int_{y\leq 2B_1}\frac{dy}{1+y^{1+4(1-\delta_{k_-})+4(L_--j)}}\\
& \lesssim & b_1^2
\eee
and hence from \fref{modulationequations}:
\bee
&&\int (1+y^2)\left|-\left(\lsl+b_1\right)\Lt^*J\Lt^{k_++L_+}\Lambda \qbt+(\gamma_s-a_1)\Lt^*J\Lt^{k_++L_+}J\qbt\right|^2\\
&\lesssim & b_1^2 b_1^{2L_++2+2(1-\dk)+2\eta(1-\dep)},
\eee
which concludes the proof of \fref{tobeprovedmoepe}.\\

\noindent{\bf step 7} Nonlinear term $N(\e)$. We now turn to the control of the nonlinear term. We expand using $p=2q+1$:
$$N(\e)=\sum N_{k_1,k_2}(\e), \ \ N_{k_1,k_2}(\e)=\e^{k_1}\overline{\e^{k_2}}\qbt^{q+1-k_1}\overline{\qbt^{q-k_2}}, \ \ \left\{\begin{array}{lll}0\leq k_1\leq q+1\\ 0\leq k_2\leq q\\ k_1+k_2\geq 2.\end{array}\right..$$ We claim the bound:
\be
\label{tobeprovedfnien}
\int |\nabla J\Lt^{k_++L_+}N_{k_1,k_2}(\e)|^2+\int \frac{|N_{k_1,k_2}(\e)|^2}{1+y^{2s_+}}\lesssim b_1^{2+O\left(\frac 1{L_+}\right)}\left(\frac{\|\e\|_{\dot{H^\sigma}}^2}{b_1^{\sigma-s_c}}\right)^{d_{k_1,k_2}}\mathcal E_{s_+}
\ee
for some $d_{k_1,k_2}=d(k_1,k_2,d,p)>0$ which, together with \fref{bootsmallsigma} and Hardy, yields
\bea
\label{controlnonlineaire}
\nonumber && \int |\nabla \Lt^{k_++L_+}N_{k_1,k_2}(\e)|^2+\int \frac{|J\Lt^{k_++L_+}N_{k_1,k_2}(\e)|^2}{1+y^2}+\int \frac{|N_{k_1,k_2}(\e)|^2}{1+y^{2s_+}}\\
&\lesssim &b_1^{2+\frac{(\sigma-s_c)\nu(d,p)}{2}}\mathcal E_{s_+}
 \eea
 thanks to  $$(\sigma-s_c)\nu(d,p)\gg \frac{1}{L_+}$$  from \fref{choicesigma}. This gives the control of the corresponding term in \fref{firstestimate}:
\bee
&&\left|\left(\Lt^{k_++L_+}\left[\frac 1{\lambda^2}(N(\e))_{\lambda}\right],J\Lt_\lambda w_{k_++L_+}\right)\right|\\
& \lesssim & \frac{1}{\l^{2(s_+-s_c)+2}}\left( b_1^{2+\frac{(\sigma-s_c)\nu(d,p)}{2}}\mathcal E_{s_+}\right)^{\frac 12}\left(\int |\nabla \e_{k_++L_+}|^2+\int \frac{|\e_{k_++L_+}|^2}{1+y^2}\right)^{\frac 12}\\
& \leq & \frac{b_1}{\l^{2(s_+-s_c)+2}}\frac{\mathcal E_{s_+}}{M}.
\eee
\noindent{\it Proof of \fref{tobeprovedfnien}}. We first derive from the $\qbt$ construction the bound:
\be
\label{ebibvebibebve}
|\pa_y^k\qbt|\lesssim \frac{1}{1+y^{\frac{2}{p-1}+k}}, \ \ k\geq 0.
\ee
Using \fref{hardywehghtgeenral} we estimate:
\bee
&&\int |\nabla \Lt^{k_++L_+}N_{k_1,k_2}(\e)|^2+\int \frac{|N_{k_1,k_2}(\e)|^2}{1+y^{2s_+}}\\
& \lesssim & \sum_{j=0}^{s_+}\frac{|D^jN_{k_1,k_2}(\e)|^2}{1+y^{2(s_+-j)}}\lesssim \sum_{j=0}^{s_+}\sum_{l=0}^j\frac{|D^l(\e^{k_1}\overline{\e^{k_2}})|^2}{1+y^{2(s_+-j)+\frac{4(p-k_1-k_2)}{p-1}+2(j-l)}}\\
& \lesssim & \int \frac{|D^{s_+}(\e^{k_1}\overline{\e^{k_2}})|^2}{1+y^{\frac{4(p-k_1-k_2)}{p-1}}}+\int \frac{|\e^{k_1}\overline{\e^{k_2}}|^2}{1+y^{2s_++\frac{4(p-k_1-k_2)}{p-1}}}.
\eee
Near the origin, $H^{s_+}(y\leq 1)$ is an algebra and therefore:
\bee
\int_{y\leq 1} \frac{|D^{s_+}(\e^{k_1}\overline{\e^{k_2}})|^2}{1+y^{\frac{4(p-k_1-k_2)}{p-1}}}+\int_{y\leq 1} \frac{|\e^{k_1}\overline{\e^{k_2}}|^2}{1+y^{2s_++\frac{4(p-k_1-k_2)}{p-1}}}\lesssim \|\e\|_{H^{s_+}(y\leq 1)}^{2(k_1+k_2)}\lesssim \mathcal E_{s_+}^2\lesssim b_1^3 \mathcal E_{s_+}.
\eee
We now claim the bounds:
\be
\label{cneobneoveooeivh}
 \int_{y\geq 1} \frac{|D^{s_+}(\e^{k_1}\overline{\e^{k_2}})|^2}{1+y^{\frac{4(p-k_1-k_2)}{p-1}}}\lesssim K^Cb_1^{2+O\left(\frac1{L_+}\right)}b_1^{2L_++2(1-\dk)+2\eta(1-\dep)}\left(\frac{\|\e\|_{\dot{H^\sigma}}^2}{b_1^{\sigma-s_c}}\right)^{d_k}
 \ee
 \be
 \label{nceonekoneiioehe}
 \int_{y\ge 1}\frac{|\e|^{2(k_1+k_2)}}{1+y^{2s_++\frac{4}{p-1}(p-k_1-k_2)}}\lesssim  K^Cb_1^{2+O\left(\frac 1{L_+}\right)}b_1^{2L_++2(1-\dk)+2\eta(1-\dep)}\left(\frac{\|\e\|_{\dot{H^\sigma}}^2}{b_1^{\sigma-s_c}}\right)^{d_k}
 \ee
  which yield \fref{tobeprovedfnien}.\\
 \noindent{\em Proof of \fref{cneobneoveooeivh}}: We let $$k=k_1+k_2, \ \ 2\leq k\leq p.$$ We split the integral in two.\\
 \noindent\underline{Term $y\geq B_0$}: We estimate:
$$\int_{y\geq B_0} \frac{\left|\nabla^{s_+}(\e^{k_1}\overline{\e^{k_2}})\right|^2}{1+y^{\frac{4}{p-1}(p-k_1-k_2)}}\lesssim b_1^{\frac{2(p-k)}{p-1}}\|\nabla^{s_+}(\e^{k_1}\overline{\e^{k_2}})\|_{L^2}^2.$$
We claim the nonlinear estimate: 
\be
\label{nonlinearest}
\forall m\in \Bbb N, \ \ m>\frac d2, \ \ \|\nabla^m(\e^{k_1}\overline{\e^{k_2}})\|_{L^2}\lesssim (\|\e\|_{L^{\infty}}^{k-1}+\|\nabla^{\frac d2}\e\|_{L^{2}}^{k-1})\|\nabla^m\e\|_{L^2}
\ee
which is proved below. Using \fref{bebebebeo}:
\bee
&&\int_{y\geq B_0} \frac{\left|\nabla^{s_+}(\e^{k_1}\overline{\e^{k_2}})\right|^2}{1+y^{\frac{4}{p-1}(p-k_1-k_2)}}\lesssim b_1^{\frac{2(p-k)}{p-1}}\left[\|\nabla^{\sigma}\e\|^{1+O\left(\frac1{L_+}\right)}_{L^2}b_1^{\frac 12\left(\frac d2-\sigma\right)+O\left(\frac 1{L_+}\right)}\right]^{2(k-1)}\|\nabla^{s_+}\e\|_{L^2}^2\\
&\lesssim & C(M)\mathcal E_{s_+} b_1^{\frac{2(p-k)}{p-1}+(k-1)(\sigma-s_c)+(k-1)(\frac d2-\sigma)+O\left(\frac1{L_+}\right)}\left(\frac{\|\nabla^{\sigma}\e\|_{L^2}^2}{b_1^{\sigma-s_c}}\right)^{k-1}\\
& \lesssim & b_1^{2+O\left(\frac 1L\right)}\mathcal E_{s_+}\left(\frac{\|\nabla^{\sigma}\e\|_{L^2}^2}{b_1^{\sigma-s_c}}\right)^{(k-1)\left[1+O\left(\frac1{L_+}\right)\right]}
\eee
\noindent{\it Proof of \fref{nonlinearest}}: By Leibniz:
$$|\nabla^m(\e^{k_1}\overline{\e^{k_2}})|\lesssim \Pi_{l_1+\dots l_k=m}|\nabla^{l_i}\e|.$$ Let $p_i=\frac{2m}{l_i}\in [2,+\infty],$ then from H\"older: $$\|\nabla^m\e^k\|_{L^2}\lesssim \| \Pi_{l_1+\dots l_k=m}\nabla^{l_i}\e\|_{L^2}\lesssim \Pi\|\nabla^{l_i}\e\|_{L^{p_i}}.$$ Let $$-l_i+\frac d{p_i}=-m_i+\frac d2$$ then from Sobolev: $$\|\nabla^{l_i}\e\|_{L^{p_i}}\lesssim \|\nabla^{m_i}\e\|_{L^2}\ \ \mbox{for}\ \ p_i<+\infty\ \ \mbox{i.e.}\ \ l_i\neq 0.$$ Observe that 
\be
\label{weheremi}
m_i=\frac d2+l_i\left(1-\frac{d}{2m}\right)>\frac d2
\ee
and we can interpolate: $$\|\nabla^{m_i}\e\|_{L^2}\lesssim \|\nabla^{\frac d2}\e\|_{L^2}^{1-z_i}\|\nabla^m\e\|_{L^2}^{z_i}$$ with $$ m_i=\frac d2(1-z_i)+mz_i\ \ \mbox{i.e.}\ \ z_i=\frac{l_i}{m}\in [0,1].$$
This yields
\bee
\|\nabla^m\e^k\|_{L^2}&\lesssim &\Pi_{l_1+\dots l_k=m}\|\nabla^{l_i}\e\|_{L^{p_i}}\lesssim \sum_{j=1}^k\|\e\|_{L^{\infty}}^{k-j}\Pi_{l_1+\dots l_j=m, l_i>0}\|\nabla^{m_i}\e\|_{L^{2}}\\
& \lesssim & \sum_{j=1}^k\|\e\|_{L^{\infty}}^{k-j}\Pi_{l_1+\dots l_j=m, l_i>0}\|\nabla^{\frac d2}\e\|_{L^2}^{1-\frac{l_i}{m}}\|\nabla^{m}\e\|^{\frac{l_i}{m}}_{L^{2}}\\
& \lesssim & \sum_{j=1}^k \|\e\|_{L^{\infty}}^{k-j}\|\|\nabla^{\frac d2}\e\|^{j-1}_{L^2}\|\nabla^m\e\|_{L^2}\\
&\lesssim & (\|\e\|_{L^{\infty}}^{k-1}+\|\nabla^{\frac d2}\e\|_{L^{2}}^{k-1})\|\nabla^m\e\|_{L^2}
\eee
by H\"older. This is \fref{nonlinearest}.\\
\noindent\underline{Term $y\leq B_0$}: We now control the inner integral. Note that for $p=k$, the nonlinear estimate \fref{nonlinearest} treats the inner integral as well 
and we may therefore assume $k\leq p-1$. We expand using the Leibniz rule: $$|\nabla^{s_+}(\e^{k_1}\overline{\e^{k_2}})|\lesssim \sum_{l_1+\dots l_k=s_+}|\Pi_{i=1}^k\nabla^{l_i}\e|$$ and distinguish three cases.\\
\underline{case $l_i=s_+$}: In this case, using the $L^{\infty}$ bound \fref{Linfttybound} with $\delta=\frac{2(p-k)}{(p-1)(k-1)}$, we have:
\bee
\int_{y\geq 1} \frac{\left|\e^{k-1}\nabla^{s_+}\e\right|^2}{1+y^{\frac{4}{p-1}(p-k)}}&\lesssim & \left\|\frac{\e}{1+y^{\frac{2(p-k)}{(p-1)(k-1)}}}\right\|_{L^{\infty}}^{2(k-1)}C(M)\mathcal E_{s_+}\\
& \lesssim & C(M)\mathcal E_{s_+}b_1^{\frac{2(p-k)}{p-1}+(k-1)(\frac d2-\sigma)+(k-1)(\sigma-s_c)+O(\frac 1{L_+})}\left(\frac{\|\nabla^{\sigma}\e\|_{L^2}^2}{b_1^{\sigma-s_c}}\right)^{k-1}\\
& \lesssim & b_1^{2+O(\frac 1{L_+})}\left(\frac{\|\nabla^{\sigma}\e\|_{L^2}^2}{b_1^{\sigma-s_c}}\right)^{(k-1)\left[1+O(\frac{1}{L_+})\right]}\mathcal E_{s_+}.
\eee
\underline{case $l_i=s_+-1$}: In this case,
$$\int_{y\geq 1} \frac{\left|\Pi \nabla^{l_i}\e\right|^2}{1+y^{\frac{4}{p-1}(p-k)}}\lesssim \left\|\frac{\e}{1+y^{\alpha_k}}\right\|_{L^{\infty}}^{2(k-2)}\|\nabla\e\|_{L^{\infty}}^2\int |\nabla^{s_+-1}\e|^2, \ \ \alpha_k=\frac{2(p-k)}{(p-1)(k-2)}.$$ We interpolate:
$$\|\nabla^{s_+-1}\e\|_{L^2}\lesssim \|\nabla^{s_+}\e\|_{L^2}^{\alpha_+}\|\nabla^{\sigma}\e\|_{L^2}^{1-\alpha_+}$$ with 
\be
\label{delalphal}
\alpha_+=\frac{s_+-1-\sigma}{s_+-\sigma}=1-\frac{1}{s_+-\sigma}=1-\frac{1}{s_+}+O\left(\frac{1}{L_+^2}\right)
\ee
We now invoke the $L^{\infty}$ bound \fref{Linfttybound} with $\delta=\alpha_k\ll1$, \fref{bebebebeogradient} and the bootstrap bound \fref{bootnorm} to estimate:
\bee
 &&\left\|\frac{\e}{1+y^{\alpha_k}}\right\|_{L^{\infty}}^{2(k-2)}\|\nabla \e\|_{L^{\infty}}^2\int |\nabla^{s_+-1}\e|^2\\
\nonumber & \lesssim & \|\nabla^{\sigma}\e\|_{L^2}^{2(k-1)\left[1+O(\frac{1}{L_+})\right]}b_1^{(k-2)\alpha_k+(k-2)\left(\frac d2-\sigma\right)+\left(\frac d2+1-\sigma\right)+O\left(\frac 1{L_+}\right)}\|\nabla^{s_+}\e\|^{2-\frac{2}{s_+}+O\left(\frac 1{L_+^2}\right)}_{L^2}\\
\nonumber& \lesssim & K^C\left(\frac{\|\nabla^{\sigma}\e\|^2_{L^2}}{b_1^{\sigma-s_c}}\right)^{(k-1)\left[1+O(\frac{1}{L_+})\right]}b_1^{2\frac{(p-k)}{p-1}+(k-1)\left(\frac d2-s_c\right)+1+O\left(\frac 1{L_+}\right)}\\
\nonumber &\times & b_1^{\left(1-\frac{1}{2L_+}+O(\frac{1}{L_+})\right)\left(2L_++2(1-\dk)+2\eta(1-\dep)+O\left(\frac 1{L_+}\right)\right)}\\
& \lesssim & K^Cb_1^{2+O(\frac{1}{L_+})}\left(\frac{\|\nabla^{\sigma}\e\|^2_{L^2}}{b_1^{\sigma-s_c}}\right)^{(k-1)\left[1+O(\frac{1}{L_+})\right]}b_1^{2L_++2(1-\dk)+2\eta(1-\dep)}.
\eee
\underline{case $l_i\leq s_+-2$}: Up to reordering, we have $$l_1+\dots l_j=s_+, \  l_{j+1}=\dots=l_k=0, \ \ l_i>0\ \ \mbox{for} \ \ 1\leq i\leq j.$$By H\"older:
$$\int_{y\leq B_0}\frac{|\Pi\nabla^{l_i}\e|^2}{1+y^{\frac{4(p-k)}{p-1}}}\lesssim \|\e\|_{L^{\infty}}^{2(k-j)}|\log b_1|^C\|\Pi_{1\leq i\leq j} \nabla^{l_i}\e\|_{L^q}^2, \ \ \mbox{with}\ \  1-\frac2{q}=\frac{4(p-k)}{d(p-1)}.$$ Using H\"older again:
$$\|\Pi_{1\leq i\leq j} \nabla^{l_i}\e\|_{L^q}\lesssim \Pi_{1\leq i\leq j}\|\nabla^{l_i}\e\|_{L^{q_i}}, \ \ q_i=\frac{qs_+}{l_i}\in (2,+\infty].$$ From Sobolev and $l_i>0$: $$\|\nabla^{l_i}\e\|_{L^{q_i}}\lesssim \|\nabla^{m_i}\e\|_{L^2},  \ \ m_i=\frac d2-\frac d{q_i}+l_i.$$ We interpolate: $$m_i=\frac d2(1-z_i)+z_is_+\ \ \mbox{ie}\ \  z_i=\frac{l_i}{s_+}\frac{1-\frac{d}{qs_+}}{1-\frac{d}{2s_+}}.$$ Observe that $z_i\geq 0$ for $L_+$ large enough, and from $l_i\leq s_+-2$:
\bee z_i&\leq &\frac{s_+-2}{s_+}\frac{1-\frac{d}{qs_+}}{1-\frac{d}{2s_+}}\\
 & = & \left[1-\frac 2{2L_+}+O\left(\frac 1{L_+^2}\right)\right]\left[1-\frac{d}{2qL_+}+O\left(\frac 1{L_+^2}\right)\right]\left[1+\frac{d}{4L_+}++O\left(\frac 1{L_+^2}\right)\right]\\
 & = & 1+\frac 1{2L_+}\left[\frac{d}{2}-\frac dq-2\right]+O\left(\frac 1{L_+^2}\right).
\eee
Now $$\frac d2-\frac dq=\frac d2\left[1-\frac 2q\right]=\frac d2\frac{4(p-k)}{d(p-1)}=\frac{2(p-k)}{p-1}$$
 and thus:
 $$0\leq z_i\leq 1+\frac 1{2L_+}\left[\frac{2(p-k)}{p-1}-2\right]+O\left(\frac 1{L_+^2}\right)=1-\frac{k-1}{(p-1)L_+}+O\left(\frac 1{L_+^2}\right)<1$$ for $L_+$ large enough since 
 $k\geq 2$. Moreover,
 $$\sum_{i=1}^jz_i=\frac{1-\frac{d}{qs_+}}{1-\frac{d}{2s_+}}=1+\frac 1{2L_+}\left[\frac{d}{2}-\frac dq\right]+O\left(\frac 1{L_+^2}\right)=1+\frac{p-k}{(p-1)L_+}+O\left(\frac 1{L_+^2}\right).$$
 We therefore obtain the bound:
 \bee
 \|\Pi_{1\leq i\leq j} \nabla^{l_i}\e\|_{L^q}&\lesssim& \Pi_{1\leq i\leq j}\|\nabla^{m_i}\e\|_{L^2}\lesssim \Pi_{1\leq i\leq j}\|\nabla^{s_+}\e\|_{L^2}^{z_i}\nabla^{\frac d2}\e\|_{L^2}^{1-z_i}\\
 & \lesssim & \|\nabla^{s_+}\e\|^{1+\frac{p-k}{(p-1)L_+}+O\left(\frac 1{L_+^2}\right)}\|\nabla^{\frac d2}\e\|_{L^2}^{j-1+O\left(\frac 1{L_+}\right)}
\eee
and therefore using \fref{bebebebeo}:
\bee
&&\int_{y\leq B_0}\frac{|\Pi\nabla^{l_i}\e|^2}{1+y^{\frac{4(p-k)}{p-1}}}\lesssim   \|\nabla^{s_+}\e\|^{2+\frac {2(p-k)}{(p-1)L_+}+O\left(\frac 1{L_+^2}\right)}\|\nabla^{\frac d2}\e\|_{L^2}^{2j-2+O\left(\frac 1{L_+}\right)} \|\e\|_{L^{\infty}}^{2(k-j)}|\log b_1|^C\\
& \lesssim & b_1^{\left[2+\frac {2(p-k)}{(p-1)L_+}+O\left(\frac 1{L_+^2}\right)\right]\left[L_++(1-\dk)+\eta(1-\dep)\right]}\|\nabla^{\sigma}\e\|_{L^2}^{2(k-1)\left[1+O(\frac{1}{L_+})\right]}b_1^{(k-1)\left[\frac d2-\sigma\right]+O(\frac1{L_+})}\\
& \lesssim & K^3b_1^{O\left(\frac 1{L_+}\right)} b_1^{2L_++2(1-\dk)+2\eta(1-\dep)+\frac{2(p-k)}{p-1}+(k-1)\left[\frac d2-s_c\right]}\left(\frac{\|\nabla^{\sigma}\e\|_{L^2}^2}{b_1^{\sigma-s_c}}\right)^{k-1\left[1+O(\frac{1}{L_+})\right]}\\
& \lesssim & K^Cb_1^{2+O\left(\frac 1L_+\right)}  b_1^{2L+2(1-\dk)+2\eta(1-\dep)}\left(\frac{\|\nabla^{\sigma}\e\|_{L^2}^2}{b_1^{\sigma-s_c}}\right)^{(k-1)\left[1+O(\frac{1}{L_+})\right]},
\eee
which concludes the proof of \fref{cneobneoveooeivh}.\\
\noindent{\it Proof of \fref{nceonekoneiioehe}}. We estimate from \fref{Linfttybound} with $\delta=\alpha_k=\frac{2(p-k)}{(p-1)(k-1)}$:
\bee
&& \int_{y\ge 1}\frac{|\e|^{2(k_1+k_2)}}{1+y^{2s_++\frac{4}{p-1}(p-k_1-k_2)}}\lesssim \left\|\frac{\e}{1+y^{\alpha_k}}\right\|_{L^{\infty}}^{2(k-1)}\int \frac{|\e|^2}{1+y^{2s_+}}\\
& \lesssim & C(M)\mathcal E_{s_+}\|\nabla^\sigma\e\|_{L^2}^{2(k-1)\left[1+O(\frac 1{L_+})\right]}b_1^{(k-1)\alpha_k+(k-1)(\frac d2-\sigma)+O(\frac 1{L_+})}\\
& \lesssim & b_1^{2+O(\frac{1}{L_+})}\mathcal E_{s_+}\left(\frac{\|\nabla^{\sigma}\e\|_{L^2}^2}{b_1^{\sigma-s_c}}\right)^{(k-1)\left[1+O(\frac{1}{L_+})\right]}\\
& \lesssim & K^Cb_1^{2+O\left(\frac 1L_+\right)}  b_1^{2L+2(1-\dk)+2\eta(1-\dep)}\left(\frac{\|\nabla^{\sigma}\e\|_{L^2}^2}{b_1^{\sigma-s_c}}\right)^{(k-1)\left[1+O(\frac{1}{L_+})\right]},
 \eee
 this is \fref{nceonekoneiioehe}.\\
 
\noindent{\bf step 8} Small linear term $L(\e)$. We claim the bound:
\bea
\label{boundle}
\nonumber &&\int_{y\geq 1}(1+y^4)\left[|\nabla \Lt^{k_++L_+}L(\e)|^2+\frac{|\Lt^{k_++L_+}L(\e)|^2}{1+y^2}+\frac{|L(\e)|^2}{1+y^{2s_+}}\right]\\
\
&\lesssim& b_1^2C(M)\mathcal E_{s_+}.
\eea
Assume \fref{boundle}, we then estimate the corresponding term in \fref{firstestimate}:
\bee
&&\left|\left(\Lt^{k_++L_+}\left[\frac 1{\lambda^2}(L(\e))_{\lambda}\right],J\Lt_\lambda w_{k_++L_+}\right)\right|\\
& \lesssim & \frac{1}{\l^{2(s_+-s_c)+2}}\left( b_1^{2}C(M)\mathcal E_{s_+}\right)^{\frac 12}\left(\int \frac{|\nabla \e_{k_++L_+}|^2}{1+y^4}+\int \frac{|\e_{k_++L_+}|^2}{1+y^6}\right)^{\frac 12}\\
& \leq & \frac{b_1}{\l^{2(s_+-s_c)+2}}\frac{\mathcal E_{s_+}}{M^{c\dk}}+ C(M)\left[b_1\int \frac{|\nabla \e_{k_++L_+}|^2}{1+y^4}+\int \frac{|\e_{k_++L_+}|^2}{1+y^6}\right]
\eee
\noindent{\it Proof of \fref{boundle}}: We compute explicitly from \fref{decompqbt}:
$$f'(u)\e=\frac{p+1}{2}(u\overline{u})^{2q}\e+\frac{p-1}{2}(u\overline{u})^{2(q-1)}\overline{\e}, \ \ p=2q+1.$$ 
We estimate the first contribution $$L_1(\e)=\frac{p+1}{2}\left[(Q\overline{Q})^{2q}-(\qbt\overline{\qbt})^{2q}\right]\e,$$ the second contribution is estimated similarily. We expand:
$$L_1(\e)=\left[\sum_{1\leq k_1+k_2\leq 2q}c_{k_1,k_2}\zetat^{k_1}\overline{\zetat^{k_2}}Q^{2q-k_1-k_2}\right]\e.$$ We first observe from the $Q_b$ construction: for $y\leq 2B_1$, $$|\zetat|\lesssim \sum_{j=1}^{L_+}b_1^jy^{2j-\gamma}+\sum_{j=1}^{L_-}b_1^{j+\frac\alpha 2}y^{2j-\frac{2}{p-1}}.$$ For the second term $$b_1^{j+\frac\alpha 2}y^{2j-\frac{2}{p-1}}\lesssim \frac{b_1}{1+y^{\frac{2}{p-1}}}(b_1y^2)^{j}b_1^{\frac\alpha 2-1}\lesssim \frac{b_1}{1+y^{\frac{2}{p-1}}}$$ from $\alpha>2$ and for $\eta<\eta(L_+)$ small enough. For the first term, if $\alpha-2j>0$, then $$b_1^jy^{2j-\gamma}=b_1b_1^{j-1}y^{2j-\alpha}y^{-\frac 2{p-1}}\lesssim b_1y^{-\frac 2{p-1}}$$ and if $\alpha-2j<0$: $$b_1^jy^{2j-\gamma}<b_1y^{-\frac{2}{p-1}}\ \ \mbox{iff}\ \ y\leq \frac{1}{b_1^{\frac{j-1}{2j-\alpha}}}=B_0^{1+\frac{\frac\alpha2-1}{j-\frac\alpha2}}$$ which holds for $\eta$ small enough. Therefore,
$$|\zetat|\lesssim \frac{b_1}{1+y^{\frac 2{p-1}}}$$ and similarily for higher derivatives: 
\be
\label{estradiiation}
|\pa_y^j\zetat|\lesssim \frac{b_1}{1+y^{\frac{2}{p-1}+j}}
\ee from which:
\be
\label{cneorgoeoho}
\left|\pa_y^j\left[\sum_{1\leq k_1+k_2\leq 2q}c_{k_1,k_2}\zetat^{k_1}\overline{\zetat^{k_2}}Q^{2q-k_1-k_2}\right]\right|\lesssim \frac{b_1}{1+y^{2+j}.}
\ee
The function $f'(Q)-f'(\qbt)$ is radially symmetric. Therefore, a simple application of the Leibniz rule and Sobolev gives near the origin:
$$\int_{y\leq 1}|\nabla \Lt^{k_++L_+}L_1(\e)|^2+\int |L_1(\e)|^2\lesssim b_1^2 C(M)\matchal E_{s_+}\lesssim b_1 b_1^{2L+2(1-\dk)+2\eta(1-\dep)}.$$ For $y\geq 1$, we estimate from \fref{cneorgoeoho}:
\bee
&&\int_{y\geq 1}(1+y^4)\left[|\nabla \Lt^{k_++L_+}L_1(\e)|^2+\frac{|L_1(\e)|^2}{1+y^{2s_+}}\right]\\
&\lesssim& b_1^2\sum_{j=0}^{s_+}\int\left|\frac{\nabla^j\e}{1+y^{2+(s_+-j)}}\right|^2(1+y^4)= b_1^2\sum_{j=0}^{s_+}\int\frac{|\nabla^j\e|^2}{1+y^{2(s_+-j)}}\\
& \lesssim & b_1^2C(M)\mathcal E_{s_+}.
\eee
The second term, $L_2(\e)$ is estimated similarly and \fref{boundle} follows.\\

\noindent{\bf step 9} Time oscillations. Injecting the collections of above bounds into \fref{firstestimate} and recalling the definition \fref{defwdoheat} yields the first estimate:
\bea
\label{enebviebibvbve}
&&  \frac{d}{ds} \frac{\mathcal E_{s_+}}2\leq \frac{ b_1}{\lambda^{2(s_+-s_c)}}\left\{ \frac{\mathcal E_{s_+}}{M^{c\dk}}+C(M)b_1^{2L_++2(1-\dk)+2\eta(1-\dep)}\right.\\
& +& \left. C(M)\int \frac{1}{1+y^{4g}}\left[|\nabla\e_{k_++L_+}|^2+\frac{|\e_{k_++L_+}|^2}{1+y^2}\right]\right\}\\
\nonumber& + & \frac{1}{\l^{2(s_+-s_c)}}\left(\Lt^{k_++L_+}(\pa_s\xi_++\pa_s\xi_-),J\Lt \e_{k_++L_+}\right).
\eea
We now extract the full time derivative from the last term above:
\bee
&&\frac{1}{\l^{2(s_+-s_c)}}\left(\Lt^{k_++L_+}(\pa_s\xi_++\pa_s\xi_-),J\Lt \e_{k_++L_+}\right) = \frac{d}{ds}\left\{\frac{\left(\Lt^{k_++L_+}(\xi_++\xi_-),J\Lt \e_{k_++L_+}\right)}{\l^{2(s_+-s_c)}}\right\}\\
& + & \frac{1}{\l^{2(s_+-s_c)}}\left[2(s_+-s_c)\lsl \left(\Lt^{k_++L_+}(\xi_++\xi_-),J\Lt \e_{k_++L_+}\right)-\left(\Lt^{k_++L_+}(\xi_++\xi_-),J\Lt \pa_s\e_{k_++L_+}\right)\right].
\eee
We estimate from \fref{neinoeneo}, \fref{esttlplus}, 
\bea
\label{vvehuijbjbv}
\nonumber \int (1+y^2)|\Lt^*J\Lt^{k_++L_+}\xi_+|^2&\lesssim& C(M)B_0^{4(1-\dk)}\matchal E_{s_+}\left[b^2_1B_0^{4\dk} \left(\frac{B_0}{B_1}\right)^{4(1-\dk)}\right]\\
& \leq & C(M)b_1^{2\eta(1-\dk)}\mathcal E_{s_+},
\eea
and from \fref{neinoeneobis}, \fref{esttlominus}:
\bea
\label{vvehuijbjbvbis}
\nonumber \int (1+y^2)|\Lt^*J\Lt^{k_++L_+}\xi_-|^2&\lesssim& C(M)B_0^{4(1-\dkm)}\matchal E_{s_-}\left[b^2_1B_0^{4\dkm} \left(\frac{B_0}{B_1}\right)^{4(1-\dkm)}\right]\\
& \leq & C(M)b_1^{2\eta(1-\dkm)}\mathcal E_{s_+},
\eea
which gives the bound
\bee
\left|\left(\Lt^{k_++L_+}(\xi_++\xi_-),J\Lt \e_{k_++L_+}\right)\right|\lesssim C(M)b_1^{\eta(1-\dep)}\mathcal E_{s_+}
\eee
and the control of the first error term:  
$$
\left|\lsl \left(\Lt^{k_++L_+}(\xi_++\xi_-),J\Lt \e_{k_++L_+}\right)\right|\lesssim C(M)b_1^{\eta(1-\dep)}\mathcal E_{s_+}.
$$
We now rewrite \fref{eqepsilon} with \fref{defwdoheat}:
 $$\partial_s \varepsilon  - \frac{\lambda_s}{\lambda} \Lambda \varepsilon - \Lt \varepsilon = F - \widehat{\Mod}-\gamma_sJ\e +\pa_s\xi_++\pa_s\xi_-
$$
from which:
\be
\label{eqioneonenevo}
\pa_s\e_{2(k_++L_+)}=\Lt^{k_++L_++1}\e+\Lt^{k_++L_+}\left[\lsl\Lambda \e-\gamma_sJ\e+ F - \widehat{\Mod}-\pa_s\xi_+-\pa_s\xi_-\right]
\ee
and hence:
\bee
&&\left(\Lt^{k_++L_+}(\xi_++\xi_-),J\Lt \pa_s\e_{k_++L_+}\right)=(\Lt^{k_++L_+}(\xi_++\xi_-),J\Lt\Lt^{k_++L_++1}\e)\\
& + & \left(\Lt^{k_++L_+}(\xi_++\xi_-),J\Lt\Lt^{k_++L_+}(\lsl\Lambda \e-\gamma_sJ\e+ F - \widehat{\Mod})\right)\\
& + & \frac 12\frac{d}{ds}\left\{(\Lt^{k_++L_+}(\xi_++\xi_-),J\Lt\Lt^{k_++L_+}(\xi_++\xi_-))\right\}.
\eee
We estimate from \fref{vvehuijbjbv}, \fref{vvehuijbjbvbis}:
\bee
&&\left| \left(\Lt^{k_++L_+}(\xi_++\xi_-),J\Lt\Lt^{k_++L_+}\left[\lsl\Lambda \e-\gamma_sJ\e\right]\right)\right|\lesssim C(M)b_1b_1^{\eta(1-\dep)}\mathcal E_{s_+}.
\eee
As in the proof of \fref{vvehuijbjbv}, \fref{vvehuijbjbvbis}:
\bee
&&\int (1+y^2)|(\Lt^*)^2J\Lt^{k_++L_+}\xi_+|^2\lesssim C(M)b_1^2b_1^{2\eta(1-\dk)}\mathcal E_{s_+}\\
&& \int (1+y^2)|\Lt^*J\Lt^{k_++L_+}\xi_-|^2\lesssim C(M)b_1^2b_1^{2\eta(1-\dkm)}\mathcal E_{s_+}
\eee
from which
$$
\left|(\Lt^{k_++L_+}(\xi_++\xi_-),J\Lt\Lt^{k_++L_++1}\e)\right|\lesssim b_1C(M)b_1^{\eta(1-\dep)}\mathcal E_{s_+}.$$
By the coercivity of $L_+,L_-$ we have that for any $v\in \dot{H}^1$:
$$ \int |\nabla v|^2+\int \frac{|v|^2}{1+y^2}\lesssim (J\Lt v,v)\gtrsim\left(\int (1+y^2)|\Lt v|^2\right)^{\frac 12}\left(\int \frac{|v|^2}{1+y^2}\right)^{\frac 12}$$ and hence
\be
\label{boundleltj}
\int \frac{|v|^2}{1+y^2}\lesssim \int (1+y^2)|\Lt v|^2
\ee
from which using the relation $J\Lt=-\Lt^*J$ from \fref{adjointltilde} and \fref{controleh4erreurloc}:
\be
\label{estpistitit}
\int \frac{|\Lt^{k_++L_+}\Psit|^2}{1+y^2}\lesssim \int (1+y^2)|\Lt^*J\Lt^{k_++L_+}\Psit|^2\lesssim b_1^{2L_++2+2(1-\dk)+2\eta(1-\dep)}
\ee
and hence using \fref{vvehuijbjbv}, \fref{vvehuijbjbvbis}:
\bee
&&\left|\left(\Lt^{k_++L_+}(\xi_++\xi_-),J\Lt\Lt^{k_++L_+}\Psit\right)\right|\\
&\lesssim& \left(\int \frac{|\Lt^{k_++L_+}\Psit |^2}{1+y^2}\right)^{\frac 12}\left(\int (1+y^2)\left[|\Lt^*J\Lt^{k_++L_+}\xi_+|^2+|\Lt^*J\Lt^{k_++L_+}\xi_-|^2\right]\right)^{\frac 12}\\
& \lesssim & b_1b_1^{L_++(1-\dk)+\eta(1-\dep)}b_1^{\eta(1-\dep)}\sqrt{\matchal E_{s_+}}.
\eee
We now estimate from \fref{controlnonlineaire} using again $J\Lt=-\Lt^*J$ and \fref{vvehuijbjbv}, \fref{vvehuijbjbvbis}:
\bee
&& \left|\left(\Lt^{k_++L_+}(\xi_++\xi_-),J\Lt\Lt^{k_++L_+}N(\e)\right)\right|\\
 & \lesssim & \left(\int(1+y^2)\left[ |\Lt^{k_++L_++1}\xi_+|^2+ |\Lt^{k_++L_++1}\xi_-|^2\right]\right)^{\frac 12}\left(\int \frac{|J\Lt^{k_++L_+}N(\e)|^2}{1+y^2}\right)^{\frac 12}\\
 & \lesssim & \left(b_1^{2+\nu(d,p)}\mathcal E_{s_+}\right)^{\frac 12}\left(b_1^{2\eta(1-\dep)}\matchal E_{s_+}\right)^{\frac 12}\lesssim b_1b_1^{\eta(1-\dep)}\matchal E_{s_+}.
 \eee
Finally, using \fref{boundle}:
\bee
&& \left|\left(\Lt^{k_++L_+}(\xi_++\xi_-),J\Lt\Lt^{k_++L_+}L(\e)\right)\right|\\
 & \lesssim & \left(\int\frac{|\Lt^{k_++L_++1}\xi_+|^2+ |\Lt^{k_++L_++1}\xi_-|^2}{1+y^2}\right)^{\frac 12}\left(\int (1+y^2)|J\Lt^{k_++L_+}L(\e)|^2\right)^{\frac 12}\\
 & \lesssim & \left(b_1^{2}C(M)\mathcal E_{s_+}\right)^{\frac 12}\left(b_1^{2\eta(1-\dep)}\matchal E_{s_+}\right)^{\frac 12} \lesssim b_1b_1^{\eta(1-\dep)}\matchal E_{s_+}.
 \eee
 Injecting the collection of above bounds into \fref{enebviebibvbve} we obtain
 \bea
 \label{bebvebeibevebie}
\nonumber &&\frac12 \frac{d}{ds}\left\{\mathcal E_{s_+}+\frac{-2(\Lt^{k_++L_+}(\xi_++\xi_-),J\Lt \e_{k_++L_+})+(\Lt^{k_++L_+}(\xi_++\xi_-),J\Lt\Lt^{k_++L_+}(\xi_++\xi_-)}{\l^{2(s_+-s_c)}}\right\}\\
\nonumber & \lesssim & \frac{ b_1}{\lambda^{2(s_+-s_c)+2}}\left\{ \frac{\mathcal E_{s_+}}{M^{c\dk}}+Cb_1^{2L_++2(1+\eta)(1-\dk)}+C(M)\int \frac{1}{1+y^{4g}}\left[|\nabla\e_{k_++L_+}|^2+\frac{|\e_{k_++L_+}|^2}{1+y^2}\right]\right.\\
& + & \left.\left|(\Lt^{k_++L_+}(\xi_++\xi_-),J\Lt\Lt^{k_++L_+}(\xi_++\xi_-)\right|\right\}.
\eea
To control the corrections to the energy ${\mathcal E}_{s_+}$ we argue as follows.
First, the linear in $\e$ term is estimated using \fref{vvehuijbjbv}, \fref{esttlominus}:
\bee
&&\left|(\Lt^{k_++L_+}(\xi_++\xi_-),J\Lt \e_{k_++L_+})\right|\\
&\lesssim & \left(\int(1+y^2)\left[ |\Lt^{k_++L_++1}\xi_+|^2+ |\Lt^{k_++L_++1}\xi_-|^2\right]\right)^{\frac 12}\left(\int \frac{|\e_{k_++L_+}|^2}{1+y^2}\right)^{\frac 12}\\
& \lesssim & b_1^{\eta(1-\dep)}\matchal E_{s_+}
\eee
We then estimate by brute force, using $\Lt^{k_++L_++1}\Phi_{L_+,\pm}=0$:
\bee
&&\left|(\Lt^{k_++L_+}(\chi_{B_1}\Phi_{L_+,+}),J\Lt\Lt^{k_++L_+})(\chi_{B_1}\Phi_{L_+,+})\right|\lesssim B_1^{d-2\gamma-4k_+-2}\lesssim \frac{1}{B_1^{4(1-\dk)}}\\
&&\left|(\Lt^{k_-+L_-}(\chi_{B_1}\Phi_{L_-,-}),J\Lt\Lt^{k_-+L_-})(\chi_{B_1}\Phi_{L_-,-})\right|\lesssim B_1^{d-\frac{4}{p-1}-4k_--2}\lesssim \frac{1}{B_1^{4(1-\dkm)}}\\
&&\left|(\Lt^{k_-+L_-}(\chi_{B_1}\Phi_{L_-,-}),J\Lt\Lt^{k_++L_+})(\chi_{B_1}\Phi_{L_+,+})\right|\lesssim \frac{1}{B_1^{4(1-\frac{\dk+\dkm}{2})}},
\eee
which with the help of  \fref{neinoeneo}, \fref{neinoeneobis} produces the bounds:
\bee
&&\left|(\Lt^{k_++L_+}\xi_+,J\Lt\Lt^{k_++L_+}\xi_+)\right|\lesssim \frac{1}{B_1^{4(1-\dk)}}C(M)B_0^{4(1-\dk)} \sqrt{\mathcal E_{s_+}}\lesssim b_1^{2\eta(1-\dk)}\mathcal E_{s_+}\\
&&\left|(\Lt^{k_++L_+}\xi_-,J\Lt\Lt^{k_++L_+}\xi_-)\right|\lesssim \frac{1}{B_1^{4(1-\dkm)}}C(M)B_0^{4(1-\dkm)} \sqrt{\mathcal E_{s_+}}\lesssim b_1^{2\eta(1-\dkm)}\mathcal E_{s_+}\\
&&\left|(\Lt^{k_++L_+}\xi_-,J\Lt\Lt^{k_++L_+}\xi_+)\right|\lesssim \frac{C(M)B_0^{4(1-\frac{\dk+\dkm}{2})}}{B_1^{4(1-\frac{\dk+\dkm}{2})}} \lesssim b_1^{2\eta(1-\dep)}\mathcal E_{s_+}.
\eee
Inserting these final bounds into \fref{bebvebeibevebie} concludes the proof of \fref{monoenoiencle} and of Proposition \ref{AEI2}.
\end{proof}

%%%%%%%%%%%%%%%%%%%%%%%%%%%%%%%%%%%%%%%%%%%%%
%%%%%%%%%%%%%%%%%%%%%%%%%%%%%%%%%%%%%%%%%%%%%

\subsection{Local Morawetz control}
\label{appendenergy}

%%%%%%%%%%%%%%%%%%%%%%%%%%%%%%%%%%%%%%%%%%%%%
%%%%%%%%%%%%%%%%%%%%%%%%%%%%%%%%%%%%%%%%%%%%%

We now establish a Morawetz type identity. This identity will be used in particular to control the remaining quadratic term on the rhs of \fref{monoenoiencle} which is better localized on the soliton core. This estimate is a replacement for the dissipative bounds available in the parabolic setting\footnote{see \cite{RSc2}.} and relies on the coercivity of the virial quadratic form. This in turn is a direct consequence of the fact that the linearized operator is pointwise strictly lower bounded by the sharp Hardy potential\footnote{a fundamental structural property of the super critical problem $p>p_{JL}$.}. Moreover, we may afford to use a lossy Morawetz multiplier at infinity since in the setting of the energy estimate \fref{monoenoiencle}, the far away zone $y\gg 1$ is already under control with a stronger norm than the one provided by the Morawetz bound. 
This feature reenforces the analogy with the inner/outer control in a parabolic flow.

\begin{lemma}[Local Morawetz control]
\label{boundlemmamorawetz}
Let $0<\delta\ll1 $ denote a small enough universal constant and let 
\be
\label{choicecut}
\psi'_A(y)=\chi_A(y) y^{1-\delta}, \ \ \chi_A(y)=\chi\left(\frac yA\right),\ \ A\gg1,
\ee
then there holds the bound:
\bea
\label{morawetzbound}
\nonumber \frac{d}{ds}\left\{\frac{\mathcal M}{\l^{2(s_+-s_c)}}\right\}&\geq &\frac{b_1}{\l^{2(s_+-s_c)}}\left[ \delta\int \frac{1}{1+y^{\delta}}\left(|\nabla \e_{2(k_++L_+)}|^2+\frac{|\e_{2(k_++L_+)}|^2}{y^2}\right)\right.\\
 &-& \left. b_1^{2L_++2(1-\dk)+2\eta(1-\dep)}-\frac{C}{A^{\delta}}\matchal E_{s_+}\right]
\eea
with 
\bea
\label{defmahtcam}
\nonumber \mathcal M&=&b_1\Im\left(\int\nabla \psi_A\cdot\nabla \e_{2(k_++L_+)}\overline{\e_{2(k_++L_+)}}\right)+O\left(\sqrt{b_1}\matchal E_{s_+}\right)\\
& = & O\left(\sqrt{b_1}\matchal E_{s_+}\right).
\eea
\end{lemma}

\begin{proof}[Proof of Lemma \ref{boundlemmamorawetz}]

\noindent{\bf step 1} The Morawetz identity. Let $v$ be a solution of
\be
\label{ceceee}
\pa_sv=\Lt v+G
\ee
For a compactly supported smooth function $\psi$ the Morawetz type identity takes the form 
\bee
&&\frac 12\frac{d}{ds}\left\{\Im\left(\int \nabla \psi\cdot\nabla v\overline{v}\right)\right\}=-\Im\left(\int\pa_sv\left[\overline{\frac{\Delta \psi}{2}v+\nabla \psi\cdot\nabla v}\right]\right)\\
& = & -\Im\left(\int\left[\Lt v+G\right]\left[\overline{\frac{\Delta \psi}{2}v+\nabla \psi\cdot\nabla v}\right]\right)\\
& = & \int L_+\Re v\left[\frac{\Delta \psi}{2}\Re v+\nabla \psi\cdot\nabla \Re v\right]+\int L_-\Im v\left[\frac{\Delta \psi}{2}\Im v+\nabla \psi\cdot\nabla \Im v\right]\\
&-& \Im\left(\int G\left[\overline{\frac{\Delta \psi}{2}v+\nabla \psi\cdot\nabla v}\right]\right)\\
\eee
For any potential $V$ and real valued radially symmetric function $u$:
$$
\int (-\Delta -V)u\left[\frac{\Delta \psi}{2} u+\nabla \psi\cdot\nabla u\right]=\int \psi''|\nabla u|^2-\frac 14\int \Delta^2\psi u^2+\frac12\int\nabla V\cdot\nabla \psi u^2.
$$
Using \fref{envnoevne} we observe that for $V=V_+=pQ^{p-1}$ we have the {\it lower bound} :
\bee
\nonumber \frac 12 y\pa_yV&=&\frac{p(p-1)}{2}y\pa_y QQ^{p-2}=\frac{p(p-1)}{2}Q^{p-2}\left[\frac{2}{p-1}Q+y\pa_y Q\right]-pQ^{p-1}\\
& = & \frac{p(p-1)}{2}Q^{p-2}\Lambda Q-pQ^{p-1}\geq -pQ^{p-1}\geq -\frac{\left[\frac{(d-2)^2}{4}-c_p\right]}{y^2},
\eee
for some universal constant $c_p>0$, where the last inequality follows from the positivity of the operator $L_+$, \eqref{positivityh}. The same argument also applies to $V=V_-=Q^{p-1}$.
This gives the lower bound on the virial quadratic form:
\bea
\label{lowerboundpoetnei}
\nonumber && \int L_+\Re v\left[\frac{\Delta \psi}{2}\Re v+\nabla \psi\cdot\nabla \Re v\right]+\int L_-\Im v\left[\frac{\Delta \psi}{2}\Im v+\nabla \psi\cdot\nabla \Im v\right]\\
 & \geq & \int \psi''|\nabla v|^2 -\left[\frac{(d-2)^2}{4}-c_p\right]\int \frac{|\pa_y\psi|}{y}\frac{|v|^2}{y^2}-\frac 14\int \Delta^2\psi|v|^2
 \eea
 Let now $u$ be spherically symmetric, real valued. We have the following  weighted Hardy bound for $0<\delta\ll1$:
 \bee
 &&\int \frac{\chi}{y^{\delta}}\left(\pa_yu+\frac{\beta}{y}u\right)^2y^{d-1}dy= \int \frac{\chi}{y^{\delta}}\left[(\pa_yu)^2+\frac{\beta^2}{y^2}u^2+2\frac{\beta}{y}u\pa_y u\right]y^{d-1}dy\\
 & = & \int \frac{\chi}{y^{\delta}}(\pa_y u)^2+\int \frac{u^2}{y^{2+\delta}}\left[(\beta^2-\beta(d-\delta-2))\chi-\beta y\chi'\right]
\eee 
For the optimal choice $\beta=\frac{d-2-\delta}2$,
\bee
 \int \frac{\chi}{y^{\delta}}(\pa_yu)^2\geq \left(\frac{d-2-\delta}{2}\right)^2\int \chi\frac{u^2}{y^{2+\delta}}-C\int \frac{|y\chi'|}{y^{2+\delta}}u^2
 \eee
 with $C$ independent of $\chi,\delta$ in the range $0<\delta\ll1$. With the choices of $\psi$ in  \fref{choicecut} and $\chi$ in \fref{sahepchi}:
 \bee
 \int \psi_A''|\nabla v|^2&=&\int \left[\chi'_Ay^{1-\delta}+\frac{\chi_A(1-\delta)}{y^{\delta}}\right]|\nabla v|^2\\
 &\geq & \delta \int \frac{\chi_A}{y^{\delta}}|\nabla v|^2+(1-\delta)^2\left(\frac{d-2-\delta}{2}\right)^2\int \chi_A\frac{u^2}{y^{2+\delta}}-\frac{C}{A^{\delta}}\int_{y\geq A}\left[ |\nabla u|^2+\frac{u^2}{1+y^2}\right].
 \eee
 Moreover, by a  direct computation:
 $$-\Delta^2\psi_A=\frac{\delta(d-\delta)(d-\delta-2)}{4}\frac{\chi_A}{y^{2+\delta}}+O\left(\frac{1}{A^{\delta}y^2}{\bf 1}_{A\leq y\leq 2A}\right)$$ and hence using \fref{lowerboundpoetnei}:
\bee
&& \int L_+\Re v\left[\frac{\Delta \psi_A}{2}\Re v+\nabla \psi_A\cdot\nabla \Re v\right]+\int L_-\Im v\left[\frac{\Delta \psi_A}{2}\Im v+\nabla \psi_A\cdot\nabla \Im v\right]\\
& \geq & \delta \int \frac{\chi_A}{y^{\delta}}|\nabla v|^2+\left[c_p-\frac{(d-2)^2}{4}+(1-\delta)^2\left(\frac{d-2-\delta}{2}\right)^2\right]\int  \chi_A\frac{u^2}{y^{2+\delta}}\\
& - & \frac 14\int \Delta^2\psi_A |v|^2-\frac{C}{A^{\delta}}\int_{y\geq A}\left[ |\nabla v|^2+\frac{|v|^2}{1+y^2}\right]\\
& \geq & \delta\int \frac{\chi_A}{y^{\delta}}\left[|\nabla v|^2+\frac{|v|^2}{y^2}\right]-\frac{C}{A^{\delta}}\int_{y\geq A}\left[ |\nabla v|^2+\frac{|v|^2}{1+y^2}\right]
\eee
for $0<\delta<\delta(p)$ small enough. We have therefore obtained the monotonicity formula for solutions to \fref{ceceee}:
\bee
&&\frac 12\frac{d}{ds}\left\{\Im\left(\int \nabla \psi_A\cdot\nabla v\overline{v}\right)\right\}\geq \delta\int \frac{1}{1+y^{\delta}}\left[|\nabla v|^2+\frac{|v|^2}{y^2}\right]\\
\nonumber &-& \Im\left(\int G\left[\overline{\frac{\Delta \psi_A}{2}v+\nabla \psi\cdot\nabla v_A}\right]\right) -\frac{C}{A^{\delta}}\int_{y\geq A}\left[ |\nabla u|^2+\frac{u^2}{1+y^2}\right]
\eee
with $C>0$ independent of $A,\delta$. We now fix, once and for all, a small $\delta$ with $$0<\delta \ll g$$ where $g$ is given by \fref{defg}, and apply this identity to \fref{eqioneonenevo} to obtain:
\bee
&&\frac 12\frac{d}{ds}\left\{\Im\left(\int \nabla \psi_A\cdot\nabla \e_{2(k_++L_+)}\overline{\e_{2(k_++L_+)}}\right)\right\}\\
\nonumber &\geq&  \delta\int \frac{1}{1+y^{\delta}}\left[|\nabla \e_{2(k_++L_+)}|^2+\frac{|\e_{2(k_++L_+)}|^2}{y^2}\right]- \frac{C}{A^{\delta}}\matchal E_{s_+}\\
\nonumber &-& \Im\left(\int \Lt^{k_++L_+}\left[\lsl\Lambda \e-\gamma_sJ\e+ F - \widetilde{\Mod}\right]\right.\\
\nonumber & \times& \left.\left[\overline{\frac{\Delta \psi_A}{2}\e_{2(k_++L_+)}+\nabla \psi_A\cdot\nabla \e_{2(k_++L_+)}}\right]\right).
\eee
The space localization of $\chi_A$ gives the rough bound: 
$$\left|\Im\left(\int \nabla \psi_A\cdot\nabla \e_{2(k_++L_+)}\overline{\e_{2(k_++L_+)}}\right)\right|\lesssim A^CC(M)\matchal E_{s_+}.$$ 
Combining it with $$\left|\lsl\right|\lesssim b_1, \ \ |(b_1)_s|\lesssim b_1^2,$$ we obtain:
\bea
\label{morawetzfinal}
&&\frac{\l^{2(s_+-s_c)}}2\frac{d}{ds}\left\{\frac{b_1}{\l^{2(s_+-s_c)}}\Im\left(\int \nabla \psi_A\cdot\nabla \e_{2(k_++L_+)}\overline{\e_{2(k_++L_+)}}\right)\right\}\\
\nonumber & \geq & \delta b_1 \int \frac{1}{1+y^{\delta}}\left[|\nabla \e_{2(k_++L_+)}|^2+\frac{|\e_{2(k_++L_+)}|^2}{y^2}\right]- \left[\frac{C}{A^{\delta}}+A^Cb_1\right]b_1\matchal E_{s_+}\\
\nonumber &-& b_1\Im\left(\int \Lt^{k_++L_+}\left[\lsl\Lambda \e-\gamma_sJ\e+ F - \widetilde{\Mod}\right]\right.\\\nonumber &\times& \left.\left[\overline{\frac{\Delta \psi_A}{2}\e_{2(k_++L_+)}+\nabla \psi_A\cdot\nabla \e_{2(k_++L_+)}}\right]\right).
\eea
We now estimate the last term on the rhs of \fref{morawetzfinal}.\\

\noindent{\bf step 2} Quadratic terms. Using the space localization of $\chi_A$,
\bee
\left|b_1\Im\left(\int \Lt^{k_++L_+}\left[\lsl\Lambda \e-\gamma_sJ\e\right]\left[\overline{\frac{\Delta \psi_A}{2}\e_{2(k_++L_+)}+\nabla \psi_A\cdot\nabla \e_{2(k_++L_+)}}\right]\right)\right|\lesssim b_1^2C(M)A^C\matchal E_{s_+}.
\eee

\noindent{\bf step 3} Nonlinear terms. We estimate from \fref{estpistitit}:
\bee
&&\left|b_1\Im\left(\int \Lt^{k_++L_+}\Psit\left[\overline{\frac{\Delta \psi_A}{2}\e_{2(k_++L_+)}+\nabla \psi_A\cdot\nabla \e_{2(k_++L_+)}}\right]\right)\right|\\
&\lesssim& b_1C(M)A^Cb_1^{L_++1+(1-\dk)+\eta(1-\dep)}\sqrt{\mathcal E_{s_+}}\lesssim b_1\left[b_1^{2L_++2(1-\dk)+2\eta(1-\dep)}+C(M)A^Cb_1\mathcal E_{s_+}\right].
\eee
Similarily, from \fref{controlnonlineaire}:
\bee
&&\left|b_1\Im\left(\int \Lt^{k_++L_+}N(\e)\left[\overline{\frac{\Delta \psi_A}{2}\e_{2(k_++L_+)}+\nabla \psi_A\cdot\nabla \e_{2(k_++L_+)}}\right]\right)\right|\\
& \lesssim & b_1b_1\sqrt{\mathcal E_{s_+}}C(M)A^C\sqrt{\matchal E_{s_+}}\leq b_1\sqrt{b_1}\matchal E_{s_+}.
\eee
Next from \fref{boundle}:
\bee
&&\left|b_1\Im\left(\int \Lt^{k_++L_+}L(\e)\left[\overline{\frac{\Delta \psi_A}{2}\e_{2(k_++L_+)}+\nabla \psi_A\cdot\nabla \e_{2(k_++L_+)}}\right]\right)\right|\\
& \lesssim & b_1b_1\sqrt{\mathcal E_{s_+}}C(M)A^C\sqrt{\matchal E_{s_+}}\leq b_1\sqrt{b_1}\matchal E_{s_+}.
\eee

\noindent{\bf step 4} Modulation equation terms. We recall the explicit expression \fref{defmodtuntilde}:
\bee
\widetilde{Mod}(t)& = & -\left(\lsl+b_1\right)\Lambda \qbt+(\gamma_s-a_1)J\qbt\\
\nonumber & + & \sum_{j=1}^{L_+}\left[(b_j)_s+(2j-\alpha)b_1b_j-b_{j+1}\right]\chi_{B_1}\left[\Phi_{j,+}+\sum_{m=j+1}^{L_++2}\frac{\partial S_{m,+}}{\partial b_j}+\sum_{m=j+1}^{L_-+2}\frac{\pa S_{m,-}}{\pa b_j}\right]\\
\nonumber& + &  \sum_{j=1}^{L_-}\left[(a_j)_s+2jb_1a_j-a_{j+1}\right]\chi_{B_1}\left[\Phi_{j,-}+\sum_{m=j+1}^{L_++2}\frac{\pa S_{m,+}}{\pa a_j}+\sum_{m=j+1}^{L_-+2}\frac{\partial S_{m,-}}{\partial a_j}\right].
\eee
Observe that since $k_+\geq 1$ 
$$\Lt^{k_++L_+}(\chi_{B_1}\Phi_{L_+,+})=\Lt^{k_+}\Lambda Q=0, \ \ \Lt^{k_++L_+}(\chi_{B_1}\Phi_{L_-,-})=\Lt^{k_++\Delta k}JQ=0\ \ \mbox{on}\ \ {\rm Supp}\,\psi_A'$$ and thus with the decomposition \fref{defwdoheat}: 
\be
\label{cnjeijbbeibje}
\Lt^{k_++L_+}\widetilde{\Mod}=\Lt^{k_++L_+}\left[\widehat{\Mod}_1+\widehat{\Mod}_2\right]\ \ \mbox{on}\ \  {\rm Supp}\,\psi_A'.
\ee
We estimate from \fref{tobeprovedmoepe}, \fref{tobeprovedmoepebis}, \fref{boundleltj}: for $j=1,2$,
$$\int\frac{|\Lt^{k_++L_+}\widehat{\Mod}_j|^2}{1+y^2}\lesssim b_1^2\left[C(M)\mathcal E_{s_+}+b_1^{2L_++2(1-\dk)+2\eta(1-\dep)}\right]$$
and therefore 
\bee
&&\left|b_1\Im\left(\int \Lt^{k_++L_+}\widehat{\Mod}_j\left[\overline{\frac{\Delta \psi_A}{2}\e_{2(k_++L_+)}+\nabla \psi_A\cdot\nabla \e_{2(k_++L_+)}}\right]\right)\right|\\
& \lesssim & b_1b_1\left[\sqrt{\mathcal E_{s_+}}+b_1^{L_++(1-\dk)+\eta(1-\dep)}\right]C(M)A^C\sqrt{\matchal E_{s_+}}\\
&\leq & b_1\left[b^{\frac 12}_1\matchal E_{s_+}+b_1^{2L_++2(1-\dk)+2\eta(1-\dep)}\right].
\eee
This concludes the proof of \fref{morawetzbound}, \fref{defmahtcam}.
\end{proof}

%%%%%%%%%%%%%%%%%%%%%%%%%%%%%%%%%%%%%%%%%%%%%
%%%%%%%%%%%%%%%%%%%%%%%%%%%%%%%%%%%%%%%%%%%%%

\subsection{Monotonicity for the low Sobolev norm}
\label{monotlowsob}

%%%%%%%%%%%%%%%%%%%%%%%%%%%%%%%%%%%%%%%%%%%%%
%%%%%%%%%%%%%%%%%%%%%%%%%%%%%%%%%%%%%%%%%%%%%

We claim a similar monotonicity formula for the low Sobolev energy.

\begin{lemma}[Monotonicity for the low Sobolev energy]
\label{lemmapouetue}
For $0<b_1<b_1^*(L_+,d,p,M)$ small enough:
\be
\label{monotnoesigma}
\frac{d}{dt}\left\{\frac{\|\nabla^{\sigma}\e\|_{L^2}^2}{\l^{2(\sigma-s_c)}}\right\}\leq \frac{b_1}{\l^{2(\sigma-s_c)+2}}\left[b_1^{\frac{c}{L_+}}\|\nabla^{\sigma}\e\|_{L^2}^2+b_1^{\sigma-s_c+\nu_0}\right]
\ee
with some universal constants $c(d,p,\ell),\nu_0(d,p)>0$
independent of $\sigma$ in the range \fref{choicesigma}.

\end{lemma}

\begin{proof}[Proof of Lemma \ref{lemmapouetue}]

\noindent{\bf step 1} Energy identity. Recall \fref{defF}, \fref{eqenwini}, we compute using \fref{wlpusmoinus}:
\bea
\label{nceoneobehe}
 &&\frac 12\frac d{dt}\int|\nabla^{\sigma}w|^2=\Re\left(\int\pa_tw\overline{\nabla^{2\sigma}w}\right)=\Re\left(\int\left[\Lt_\l w+\frac 1{\l^2}\mathcal F_\lambda\right]\overline{\nabla^{2\sigma}w}\right)\\
\nonumber & = & \frac{1}{\l^{2+2(\sigma-s_c)}}\Re\left(\int \left[\left(\begin{array}{ll} 0&-W_-\\ W_+&0\end{array}\right)\e-\Psit_b+\widetilde{Mod}+L(\e)-N(\e)\right]\overline{\nabla^{2\sigma}\e}\right).
\eea
We now estimate all the terms on the rhs of \fref{nceoneobehe}.

\noindent {\bf step 2} Potential term. The potentials $W_\pm$ satisfy \fref{estf} with $\mu=2$. Using Lemma \ref{fracthardy} with $\nu=\sigma-2$ so that $\nu+\mu=\sigma<\frac d2$:
\bee
\left|\int\left(\begin{array}{ll} 0&-W_-\\ W_+&0\end{array}\right)\e\overline{\nabla^{2\sigma}\e}\right|&\lesssim& \left[\|\nabla^{\sigma-2}(W_+\e)\|_{L^2}+\|\nabla^{\sigma-2}(W_-\e)\|_{L^2}\right]\|\nabla ^{\sigma+2}\e\|_{L^2}\\
&\lesssim & \|\nabla^{\sigma}\e\|_{L^2} \|\nabla ^{\sigma+2}\e\|_{L^2}\lesssim \|\nabla^{\sigma}\e\|_{L^2}^{1+z_{L_+}}\|\nabla^{s_+}\e\|_{L^2}^{1-z_{L_+}}\\
& \lesssim & (b^{\frac{1+\nu}2}_1\|\nabla^{\sigma}\e\|_{L^2})^{1+z_{L_+}}(b_1^{-\frac{(1+\nu)(1+z_{L_+})}{2(1-z_{L_+})}}\|\nabla^{s_+}\e\|_{L^2})^{1-z_{L_+}}\\
& \lesssim & b^{1+\nu}_1\|\nabla^{\sigma}\e\|_{L^2}^2+M^{C_{L_+}}b_1^{-\frac{(1+\nu)(1+z_{L_+})}{1-z_{L_+}}}\|\nabla^{s_+}\e\|^2_{L^2}
\eee
with $$\nu=\frac{1}{2L_+}, \ \ \sigma+2=z_{L_+}\sigma+(1-z_{L_+})s_+.$$
We now compute $$-\frac{1+z_{L_+}}{1-z_{L_+}}=1-\frac{2}{1-z_{L_+}}=1-(s_+-\sigma)$$ and hence using \fref{calculimportant}, \fref{bootnorm}, \fref{dgammkrealtion}:
\bee
b_1^{-\frac{(1+\nu)(1+z_{L_+})}{1-z_{L_+}}}\|\nabla^{s_+}\e\|_{L^2}^2&\lesssim& Kb_1^{2L_++2(1-\dk)+2\eta(1-\dep)-(1+\nu)(s_+-\sigma-1)}\\
&\lesssim&  Kb_1^{\sigma-s_c+\alpha+1+2\eta(1-\dep)-\nu(2L_++O(1))}\leq b_1b_1^{\sigma-s_c+\frac\alpha 2}
\eee
for $b_1<b_1^*(M)$ small enough. We have therefore obtained the expected bound:
$$\left|\int\left(\begin{array}{ll} 0&-W_-\\ W_+&0\end{array}\right)\e\overline{\nabla^{2\sigma}\e}\right|\leq b_1\left[b_1^{\frac{c}{L_+}}\|\nabla^{\sigma}\e\|_{L^2}^2+b_1^{\sigma-s_c+\frac\alpha 2}\right].$$

\noindent{\bf step 3} $\Psit_b$ term. We recall the Sobolev bound \fref{smallsobolevbound}:
$$\|\nabla^{\sigma}\Psit\|^2_{L^2}\le b_1^{2+\sigma-s_c+\nu_1}, \ \ \nu_1=\nu(d,p)>0$$
which implies
\bee|(\Psit_b,\nabla^{2\sigma}\e)|&\lesssim &\|\nabla^\sigma \e\|_{L^2}\|\nabla^{\sigma}\Psit_b\|_{L^2}\lesssim b_1\|\nabla^{\sigma}\e\|_{L^2}\left( b_1^{\sigma-s_c+\nu_1}\right)^{\frac 12}\\
& \lesssim & b_1\left[b_1^{\frac{\nu_1} 2}\|\nabla^\sigma\e\|_{L^2}^2+b_1^{\sigma-s_c+\frac{\nu_1} 2}\right].
\eee

\noindent{\bf step 4} $\widetilde{Mod}$ term. Let $$\widetilde{\widetilde{\Mod}}=\widetilde{\Mod}-(\gamma_s-a_1)JQ,$$ we claim the bound:
\be
\label{estmodsobolevinterm}
\|\nabla^{2k_++1}\widetilde{\widetilde{\Mod}}\|_{L^2}^2\lesssim b_1^{2(1-\dk)}.
\ee
Assume \fref{estmodsobolevinterm}, we then observe $$2\sigma-(2k_++1)=\sigma+\sigma-s_c+\alpha-2+2\dk>\sigma$$ and interpolate:
$$|(\widetilde{\widetilde{\Mod}},\nabla^{2\sigma}\e)|\lesssim \|\nabla^{2\sigma-(2k_++1)}\e\|_{L^2}\|\nabla^{2k_++1}\widetilde{Mod}\|_{L^2}\lesssim \|\nabla^{\sigma}\e\|_{L^2}^{z_+}\|\nabla^{s_+}\e\|_{L^2}^{1-z_+}b_1^{2(1-\dk)}$$
with $$2\sigma-(2k_++1)=z_+\sigma+(1-z_+)s_+, \ \ 1-z_+=\frac{\sigma-(2k_++1)}{s_+-\sigma}=\frac{\sigma-s_c+\alpha-2(1-\dk)}{2L_+}+O\left(\frac 1{L_+^2}\right).$$
Therefore, 
\bee
|(\widetilde{\widetilde{\Mod}},\nabla^{2\sigma}\e)|&\lesssim& \|\nabla^{\sigma}\e\|_{L^2}b_1^{1-\dk}b_1^{\frac{\sigma-s_c+\alpha-2(1-\dk)}{2}+O\left(\frac 1{L_+}\right)}\\
& \lesssim & b_1^{\frac{\sigma-s_c+\alpha}{2}}\|\nabla^{\sigma}\e\|_{L^2}\lesssim b_1\left[b_1^{\frac {\nu_0}{2}}\|\nabla^{\sigma}\e\|^2_{L^2}+b_1^{\sigma-s_c+\frac{\nu_0}{2}}\right]
\eee
for some $\nu_0(d,p)>0$, thanks to $\alpha>2$.
The second term is estimated from \fref{parameters}:
\bee
|((\gamma_s-a_1)JQ,\nabla^{2\sigma}\e)|\lesssim  b_1^{L_++1+(1-\dk)+\eta(1-\dep)}\|\nabla^{\sigma}Q\|_{L^2}\|\nabla^\sigma\e\|_{L^2}\lesssim b_1\left[b_1\|\nabla^{\sigma}\e\|_{L^2}^2+b_1^{2L_+-1}\right].
\eee
This gives the desired control of the $\widetilde{\Mod}$ terms.\\
\noindent{\it Proof of \fref{estmodsobolevinterm}}: By the $\qbt$ construction:
\bee
&&\int |\nabla^{2k_++1}\Lambda \qbt|^2+\left|\int |\nabla^{2k_++1}(\qbt-Q)\right|^2 \\
& \lesssim & \int_{y\leq 2B_1}\sum_{j=0}^{L_+}b_1^{2j}(1+|y^{2j-\gamma-(2k_++1)}|^2)+ \int_{y\leq 2B_1}\sum_{j=1}^{L_-}b_1^{2j+\alpha}(1+|y^{2j-\frac 2{p-1}-(2k_++1)}|^2)\\
& + & \int_{y\leq 2B_1}\sum_{j=2}^{L_++2}b_1^{2j}(1+|y^{2(j-1)-\gamma-(2k_++1)}|^2)+b_1^{2j+\alpha}(1+|y^{2j-\gamma-(2k_++1)}|^2)\\
& + & \int_{y\leq 2B_1}\sum_{j=2}^{L_-+2}b_1^{2j+\alpha}(1+|y^{2(j-1)-\frac{2}{p-1}-(2k_++1)}|^2)+b_1^{2j+2\alpha}(1+|y^{2j-\frac{2}{p-1}-(2k_++1)}|^2)\\
& \lesssim & \sum_{j=0}^{L_++2}b_1^{2j}\int_{y\leq 2B_1}\frac{1+y^{4j}}{1+y^{1+4(1-\dk)}}dy+\sum_{j=1}^{L_-+2}\int_{y\leq 2B_1}b_1^{2j+\alpha}\frac{1+y^{4(j+\Delta k)}}{1+y^{1+4(1-\dkm)}}dy\\
&\lesssim & 1.
\eee
We now estimate for $1\leq j\leq L_+$:
\bee
&&\int _{y\leq 2B_1}\left|\nabla^{2k_++1}\left[\Phi_{j,+}+\sum_{m=j+1}^{L_++2}\frac{\pa S_{m,+}}{\pa b_j}+\sum_{m=j+1}^{L_-2+2}\frac{\pa S_{m,-}}{\pa b_j}\right]\right|^2\\
& \lesssim & \int_{y\leq 2B_1} |y^{2j-\gamma-(2k_++1)}|^2+b_1^{2(m-j)}|y^{2(m-1)-\gamma-(2k_++1)}|^2+b_1^{2(m-j)+\alpha}|y^{2m-\frac{2}{p-1}-(2k_++1)}|^2\\
& \lesssim & B_1^{4j-4(1-\dk)}+\sum_{m=j+1}^{L_++2}b_1^{2(m-j)}B_1^{4(m-1)-4(1-\dk)}+\sum_{m=j+1}^{L_-2+2}b_1^{2(m-j)+\alpha}B_1^{4m-4(1-\dkm)+4\Delta k}\\
& \lesssim & B_1^{4j-4(1-\dk)}\left[1+\sum_{m=j+1}^{L_++2}\frac{(b_1B_1^2)^{2(m-j)}}{B_1^{4}}+\sum_{m=j+1}^{L_-2+2}(b_1B_1^2)^{2(m-j)+\alpha}\right]\\
& \lesssim & B_1^{4j-4(1-\dk)-C_{L_+}\eta}.
\eee
Similarily for $1\leq j\leq L_-$:
\bee
&&\int _{y\leq 2B_1}\left|\nabla^{2k_++1}\left[\Phi_{j,-}+\sum_{m=j+1}^{L_++2}\frac{\pa S_{m,+}}{\pa a_j}+\sum_{m=j+1}^{L_-+2}\frac{\pa S_{m,-}}{\pa a_j}\right]\right|^2\\
& \lesssim & \int_{y\leq 2B_1} |y^{2j-\frac{2}{p-1}-(2k_++1)}|^2+\sum_{m=j+1}^{L_++2}b_1^{2(m-j)}|y^{2m-\frac{2}{p-1}-(2k_++1)}|^2\\
& \lesssim & B_1^{4j-4(1-\dkm)+4\Delta k}
\eee
The collection of above bounds together with \fref{parameters}, \fref{parameterspresicely}, \fref{parameterspresicelya} imply:
\bee
&&\|\nabla^{2k_++1}\widetilde{\widetilde{\Mod}}\|_{L^2}^2\lesssim K^2b_1^{2L_++2(1-\dk)+2\eta(1-\dep)}\left[B_1^{4L_+-4(1-\dk)-C_{L_+}\eta}+B_1^{4L_--4(1-\dkm)+4\Delta k}\right]\\
& \leq & b_1^{2(1-\dk)+2(1-\dep)-C_{L_+}\eta}\lesssim b_1^{2(1-\dk)}.
\eee

\noindent {\bf step 5} Nonlinear term $N(\e)$. We claim:
$$\|\nabla ^{\sigma}N(\e)\|_{L^2}^2\lesssim b_1^{2+O\left(\frac 1{L_+}\right)}\|\nabla^{\sigma}\e\|_{L^2}^2\left(\frac{\|\nabla^{\sigma}\e\|_{L^2}^2}{b_1^{\sigma-s_c}}\right)^{(k-1)\left[1+O(\frac1{L_+})\right]}$$ which, upon expanding the nonlinearity 
$$N(\e)=\sum N_{k_1,k_2}(\e), \ \ N_{k_1,k_2}(\e)=\e^{k_1}\overline{\e^{k_2}}\qbt^{q+1-k_1}\overline{\qbt^{q-k_2}}, \ \ \left\{\begin{array}{lll}0\leq k_1\leq q+1\\ 0\leq k_2\leq q\\ k_1+k_2\geq 2.\end{array}\right..$$ 
follows from: $\forall 2\leq k=k_1+k_2\leq p$,
\be
\label{tiovboebeo}
 \|\nabla^{\sigma}N_{k_1,k_2}(\e)\|_{L^2}\lesssim  b_1^{2+O\left(\frac 1{L_+}\right)}\|\nabla^{\sigma}\e\|_{L^2}^2\left(\frac{\|\nabla^{\sigma}\e\|_{L^2}^2}{b_1^{\sigma-s_c}}\right)^{(k-1)\left[1+O(\frac1{L_+})\right]}.
\ee
This implies from the bootstrap bound \fref{bootsmallsigma} and \fref{choicesigma}:
\bee
|(N(\e),\nabla^{2\sigma}\e)|&\lesssim &b_1^{1+O\left(\frac 1{L_+}\right)}\sum_{k=2}^p\|\nabla^{\sigma}\e\|_{L^2}^2\left(\frac{\|\nabla^{\sigma}\e\|_{L^2}^2}{b_1^{\sigma-s_c}}\right)^{\frac{k-1}{2}}\leq b_1b_1^{c(\sigma-s_c)}\|\nabla^\sigma\e\|_{L^2}^2\\
& \lesssim & b_1b_1^{\frac{c}{L_+}}\|\nabla^\sigma\e\|_{L^2}^2
\eee
for some universal constant $c>0$, where we used \fref{choicesigma} in the last step.\\
\noindent{\it Proof of \fref{tiovboebeo}}: We observe from \fref{ebibvebibebve} the bound:
\be\label{nekneoenone}
\left|\pa_y^j\left(\qbt^{q+1-k_1}\overline{\qbt^{q-k_2}}\right)\right|\lesssim \frac{1}{1+y^{\frac{2(p-k)}{p-1}+j}}, \ \ j\geq0.
\ee 
We decompose $$\sigma=s+\delta_\sigma, \ \ s\in \Bbb N^*, \ \ 0\leq \delta_{\sigma}<1$$ and develop with the help of the Leibniz rule:
$$\|\nabla^{\sigma}\e^{k_1}\overline{\e^{k_2}}\qbt^{q+1-k_1}\overline{\qbt^{q-k_2}}\|_{L^2}\lesssim \sum_{i=0}^s\|\nabla^{\delta_{\sigma}}\left[\nabla^i(\e^{k_1}\overline{\e}^{k_2})\nabla^{s-i}(\qbt^{q+1-k_1}\overline{\qbt^{q-k_2}})\right]\|_{L^2}.$$ We now consider separate cases, depending on the value of $i$:\\
\noindent\underline{Case $2\leq i\leq s$}: In this case, from $2\leq k\leq p$,
$$0< \delta_\sigma+\frac{2(p-k)}{p-1}+s-i=\sigma+\frac{2(p-k)}{p-1}-i\leq \sigma+\frac{2(p-2)}{p-1}-2<\sigma<\frac d2$$ and we therefore estimate using \fref{nekneoenone} and the fractional Hardy estimate \fref{fracthardy}:
\bee
\|\nabla^{\delta_{\sigma}}\left[\nabla^i(\e^{k_1}\overline{\e}^{k_2})\nabla^{s-i}(\qbt^{q+1-k_1}\overline{\qbt^{q-k_2}})\right]\|_{L^2}&\lesssim& \|\nabla^{\delta_\sigma+\frac{2(p-k)}{p-1}+s-i}\nabla^i(\e^{k_1}\overline{\e}^{k_2})\|_{L^2}\\
&\lesssim& \|\nabla^{\sigma+\frac{2(p-k)}{p-1}}(\e^{k_1}\overline{\e}^{k_2})\|_{L^2}.
\eee
We now claim the nonlinear estimate: $\forall 2\leq k=k_1+k_2\leq p$, $\forall \sigma\leq \beta \ll s_+$,
\be
\label{estnonlineseigma}
 \|\nabla ^{\beta}(\e^{k_1}\overline{\e}^{k_2})\|_{L^2}^2\lesssim b_1^{2-\frac{2(p-k)}{p-1}+\beta-\sigma+O\left(\frac 1{L_+}\right)}\|\nabla^{\sigma}\e\|_{L^2}^2\left(\frac{\|\nabla^{\sigma}\e\|_{L^2}^2}{b_1^{\sigma-s_c}}\right)^{(k-1)\left[1+O(\frac1{L_+})\right]}
\ee
which is proved below. This yields the expected bound for $i\geq2$:
\bee
&&\|\nabla^{\delta_{\sigma}}\left[\nabla^i(\e^{k_1}\overline{\e}^{k_2})\nabla^{s-i}(\qbt^{q+1-k_1}\overline{\qbt^{q-k_2}})\right]\|_{L^2}\lesssim \|\nabla^{\sigma+\frac{2(p-k)}{p-1}}(\e^{k_1}\overline{\e}^{k_2})\|_{L^2}^2\\
&\lesssim & b_1^{2-\frac{2(p-k)}{p-1}+\sigma+\frac{2(p-k)}{p-1}-\sigma+O\left(\frac 1{L_+}\right)}\|\nabla^{\sigma}\e\|_{L^2}^2\left(\frac{\|\nabla^{\sigma}\e\|_{L^2}^2}{b_1^{\sigma-s_c}}\right)^{(k-1)\left[1+O(\frac1{L_+})\right]}\\
&\lesssim & b_1^{2+O\left(\frac 1{L_+}\right)}\|\nabla^{\sigma}\e\|_{L^2}^2\left(\frac{\|\nabla^{\sigma}\e\|_{L^2}^2}{b_1^{\sigma-s_c}}\right)^{(k-1)\left[1+O(\frac1{L_+})\right]}.
\eee 
\noindent{\it Proof of \fref{estnonlineseigma}}. If $\beta\in \Bbb N$, since $\beta\geq \sigma>\frac d2$ we estimate from \fref{nonlinearest}:
$$\|\nabla^\beta(\e^{k_1}\overline{\e}^{k_2})\|_{L^2}\lesssim (\|\e\|_{L^{\infty}}^{k-1}+\|\nabla^{\frac d2}\e\|_{L^{2}}^{k-1})\|\nabla^\beta\e\|_{L^2}$$ and thus from \fref{bebebebeo}, \fref{interpolationboundbeta}:
\bee
&&\|\nabla^\beta(\e^{k_1}\overline{\e}^{k_2})\|_{L^2}^2\lesssim  \|\nabla^{\sigma}\e\|^{2(k-1)+O\left(\frac 1{L_+}\right)}_{L^2}b_1^{(k-1)(\frac d2-\sigma)+O\left(\frac 1{L_+}\right)}\|\nabla^\sigma\e\|_{L^2}^{2+O(\frac1{L_+})}b_1^{\beta-\sigma+O(\frac1{L_+})}\\
& \lesssim & \|\nabla^{\sigma}\e\|_{L^2}^2\left(\frac{\|\nabla^{\sigma}\e\|_{L^2}^2}{b_1^{\sigma-s_c}}\right)^{(k-1)\left[1+O(\frac1{L_+})\right]}b_1^{(k-1)(\frac d2-s_c)+\beta-\sigma+O(\frac1{L_+})}\\
&= & b_1^{2-\frac{2(p-k)}{p-1}+\beta-\sigma+O\left(\frac 1{L_+}\right)}\|\nabla^{\sigma}\e\|_{L^2}^2\left(\frac{\|\nabla^{\sigma}\e\|_{L^2}^2}{b_1^{\sigma-s_c}}\right)^{(k-1)\left[1+O(\frac1{L_+})\right]}.
\eee
If $\beta\notin \Bbb N$, we split $$\beta=s_\beta+\delta_\beta, \ \ s_\beta\in \Bbb N^*, \ \ \delta_\beta\in (0,1).$$ We recall the standard commutator estimate: let $$0<\nu<1,\ \ 1<p,p_1,p_3<+\infty, \ \ 1\leq p_2,p_4\leq +\infty \ \ \mbox{with}\ \ \frac 1p=\frac 1{p_1}+\frac 1{p_2}=\frac1{p_3}+\frac 1{p_4},$$
then
\be
\label{comutatoreestimate}
\|\nabla^{\nu}(uv)\|_{L^p}\lesssim \|\nabla^{\nu}u\|_{L^{p_1}}\|v\|_{L^{p_2}}+\|\nabla^{\nu}v\|_{L^{p_3}}\|u\|_{L^{p_4}}.
\ee
We therefore expand: 
\bee
\|\nabla^{\delta_{\beta}}\nabla^{s_\beta}(\e^{k_1}\overline{\e}^{k_2})\|_{L^2}&\lesssim& \sum_{l_1+\dots +l_k=s_\beta}\left\|\nabla^{\delta_{\alpha}}\left(\Pi_{1\leq i\leq k}\nabla^{l_i}\e\right)\right\|_{L^2}\\
& \lesssim & \sum_{l_1+\dots +l_k=s_\beta}\Pi_{1\leq i\leq k}\|\nabla^{\tilde{l}_i}\e\|_{L^{p_i}}
\eee
where $$\tilde{l_i}=l_i+\delta_{i=j}\delta_\beta, \  \ 1\leq j\leq k,\ \ p_i=\frac{2\beta}{\tilde{l}_i}, \ \ \sum_{i=1}^k\tilde{l}_i=\beta.$$
We then estimate by Sobolev for $\tilde{l_i}>0$, i.e., $2\leq p_i<+\infty$: $$\|\nabla^{\tilde{l}_i}\e\|_{L^{p_i}}\lesssim \|\nabla^{m_i}\e\|_{L^2}\ \ \mbox{with}\ \ -m_i+\frac d2=-\tilde{l}_i+\frac d{p_i}.$$ We compute 
$$
m_i=\left\{\begin{array}{ll}\left(\frac d2-\beta\right)\left(1-\frac{\tilde{l_i}}{\beta}\right)+\beta\geq \alpha\geq \sigma\ \ \mbox{for}\ \ \beta\leq \frac d2\\
\tilde{l_i}\left(1-\frac d{2\beta}\right)+\frac d2\geq \frac d2\geq \sigma  \ \ \mbox{for}\ \ \beta\geq \frac d2
\end{array}\right.
$$
and thus $\sigma\leq m_i\leq s_+$. We interpolate: $$m_i=z_i\sigma+(1-z_i)s_+\ \ \mbox{with}\ \ z_i=\frac{s_+-m_i}{s_+-\sigma}=1-\frac{m_i-\sigma}{2L_+}+O\left(\frac 1{L_+^2}\right).$$
and count the $j\in [1,k]$ terms $\tilde{l}_j\neq 0$. We compute:
\bee
&&\sum_{i=1}^j m_i=j\frac d2-\frac d2+\beta=(j-1)\frac d2+\beta\\
&& \sum_{i=1}^j z_i=j-\frac1{2L_+}\left[(j-1)\frac d2+\beta-j\sigma\right]+O\left(\frac 1{L_+^2}\right)=j-\frac{j-1}{2L_+}\left(\frac d2-\sigma\right)-\frac{\beta-\sigma}{2L_+}+O\left(\frac 1{L_+^2}\right)
\eee
and estimate using \fref{bebebebeo}: 
\bee
&&\sum_{l_1+\dots +l_k=s}\Pi_{1\leq i\leq k}\|\nabla^{\tilde{l}_i}\e\|_{L^{p_i}}\lesssim \sum_{1\leq j\leq k}\|\e\|_{L^{\infty}}^{k-j}\Pi \left(\|\nabla^{\sigma}\e\|_{L^2}^{z_i}\|\nabla^{s_+}\e\|_{L^2}^{1-z_i}\right)\\
& \lesssim &\sum_{1\leq j\leq k} \left(\|\nabla^{\sigma}\e\|^{1+O(\frac 1{L_+})}_{L^2}b_1^{\frac 12\left(\frac d2-\sigma\right)+O\left(\frac 1{L_+}\right)}\right)^{k-j}\|\nabla^{\sigma}\e\|_{L^2}^{j+O\left(\frac 1{L_+}\right)}\|\nabla^{s_+}\e\|_{L^2}^{\frac{j-1}{2L_+}\left(\frac d2-\sigma\right)+\frac{\beta-\sigma}{2L_+}+O\left(\frac 1{L_+^2}\right)}\\
& \lesssim & \|\nabla^{\sigma}\e\|_{L^2}^{k+O(\frac 1{L_+})}b_1^{\frac{k-1}{2}\left(\frac d2-\sigma\right)+\frac{\beta-\sigma}{2}+O\left(\frac 1L_+\right)}
\eee 
and thus:
\bee
\|\nabla ^{\beta}(\e^{k_1}\overline{\e}^{k_2})\|_{L^2}^2&\lesssim &\|\nabla^{\sigma}\e\|_{L^2}^{2k+O(\frac 1{L_+})}b_1^{(k-1)\left(\frac d2-\sigma\right)+\beta-\sigma+O\left(\frac 1L_+\right)}\\
 &\lesssim&  \|\nabla^{\sigma}\e\|_{L^2}^{2}\|\nabla^{\sigma}\e\|_{L^2}^{2(k-1)+O(\frac 1{L_+})}b_1^{\beta-\sigma+(k-1)\left(\frac d2-\sigma\right)+O\left(\frac 1L_+\right)}\\
& \lesssim & b_1^{2-\frac{2(p-k)}{p-1}+\beta-\sigma+O\left(\frac 1{L_+}\right)}\|\nabla^{\sigma}\e\|_{L^2}^2\left(\frac{\|\nabla^{\sigma}\e\|_{L^2}^2}{b_1^{\sigma-s_c}}\right)^{(k-1)\left[1+O(\frac 1{L_+})\right]}
\eee
and \fref{estnonlineseigma} is proved.\\
\noindent\underline{Case $i=0,1$}: For $i=0$, we estimate from \fref{stnadradharyd}, \fref{nekneoenone}, \fref{Linfttybound}:
\bee
&&\|\nabla^{\delta_{\sigma}}\left[\e^{k_1}\overline{\e}^{k_2}\nabla^{s}(\qbt^{q+1-k_1}\overline{\qbt^{q-k_2}})\right]\|_{L^2}\lesssim \|\nabla\left(\e^{k_2}\overline{\e}^{k_2}(1+y^{1-\delta_\sigma})\nabla^{s}(\qbt^{q+1-k_1}\overline{\qbt^{q-k_2}})\right)\|^2_{L^2}\\
& \lesssim & \left\|\frac{\e^k}{1+y^{\sigma+\frac{2(p-k)}{p-1}}}\right\|^2_{L^2}+\left\|\frac{\nabla (\e^{k_1}\overline{\e}^{k_2})}{1+y^{\sigma-1+\frac{2(p-k)}{p-1}}}\right\|^2_{L^2}\\
&\lesssim &  \left\|\frac{\e}{1+y^{\frac{2(p-k)}{(k-1)(p-1)}}}\right\|^{2(k-1)}_{L^{\infty}}\left[\left\|\frac{\e}{1+y^{\sigma}}\right\|^2_{L^2}+\left\|\frac{\nabla\e}{1+y^{\sigma-1}}\right\|_{L^2}^2\right]\\
& \lesssim & \left(\|\nabla^{\sigma}\e\|_{L^2}b_1^{\frac{2(p-k)}{(k-1)(p-1)} +\left(\frac d2-\sigma\right)+O\left(\frac 1{L_+}\right)}\right)^{k-1}\|\nabla^{\sigma}\e\|_{L^2}^{2+O(\frac1{L_+})}\\
& = &b_1^{2+O(\frac 1{L_+})}\|\nabla^{\sigma}\e\|_{L^2}^2\left(\frac{\|\nabla^{\sigma}\e\|_{L^2}^2}{b_1^{\sigma-s_c}}\right)^{(k-1)\left[1+O(\frac1{L_+})\right]}.
\eee
The case $i=1$ follows similarily and is left to the reader. 
This concludes the proof of \fref{tiovboebeo}.\\

\noindent{\bf step 6} Small linear term $L(\e)$. We use \fref{cneorgoeoho}, \fref{hardyfract} to estimate:
\bee
|(L(\e),\nabla^{2\sigma}\e)|\lesssim\|\nabla^{\sigma-2}L(\e)\|_{L^2}\|\nabla^{\sigma+2}\e\|_{L^2}\lesssim b_1\|\nabla^{\sigma}\e\|_{L^2}\|\nabla^{\sigma+2}\e\|_{L^2}
\eee
and follow the chain of estimates of step 2.\\

\noindent{\bf step 7} Conclusion. The collection of above bounds yields \fref{monotnoesigma}.
\end{proof}

%%%%%%%%%%%%%%%%%%%%%%%%%%%%%%%
%%%%%%%%%%%%%%%%%%%%%%%%%%%%%%%
  
  \section{Closing the bootstrap and proof of Theorem \ref{thmmain}}
  \label{sectionfour}
 %%%%%%%%%%%%%%%%%%%%%%%%%%%%%%%
%%%%%%%%%%%%%%%%%%%%%%%%%%%%%%%

We are now in position to close the bootstrap bounds of Proposition \ref{bootstrap}, and then conclude the proof of Theorem \ref{thmmain}.

%%%%%%%%%%%%%%%%%%%%%%%%%%%%%%%

\subsection{Proof of Proposition \ref{bootstrap}}
\label{sectionbootstrappouet}
 %%%%%%%%%%%%%%%%%%%%%%%%%%%%%%%

Our aim is to show first that for $s<s^*$, the a priori bounds \fref{controlssstable}, \fref{controlsaksstable}, \fref{bootnorm}, \fref{bootsmallsigma} can be improved under the sole a priori control \fref{controlunstable}, and then control the unstable modes $(V_k)_{1\le k\le \ell}$, $(\tilde{A}_k))_{1\leq k\leq k_\ell}$. \\

\noindent{\bf step 1} Integration of the law for the scaling parameter. 
First observe from \fref{defukbis} and the a priori bound \fref{controlunstable} on $U_k$ on $s_0\leq s<s^*$ that 
\be
\label{cneocneneo}
|b_k(s)|\lesssim  |b_k(s_0)|.
\ee
We compute explicitly the scaling parameter for $s<s^*$. From \fref{defuk}, \fref{approzimatesolution}, \fref{controlunstable}, \fref{parameters}, we have  the bound: 
$$-\lsl=\frac{\ell}{(2\ell-\alpha)s}+O\left(\frac{1}{s^{1+c\eta}}\right)$$ which we rewrite 
\be
\label{eneoeoe}
\left|\frac{d}{ds}\left\{\log\left(s^{\frac{\ell}{(2\ell-\alpha)}}\l(s)\right)\right\}\right|\lesssim \frac{1}{s^{1+c\eta}}.
\ee
We integrate this using the initial value $\l(s_0)=1$ and conclude using $s_0\gg 1$ from \fref{defukbis}:
\be
\label{valuel}
\l(s)=\left(\frac{s_0}{s}\right)^{\frac{\ell}{2\ell-\alpha}}\left[1+O\left(\frac{1}{s_0^{c\eta}}\right)\right].
\ee
Together with the law for $b_1$ given by \fref{controlunstable}, \fref{defuk}, \fref{approzimatesolution}, this implies:
\be
\label{keyestimated}
b_1(s_0)^{\frac{\ell}{2\ell-\alpha}}\lesssim \frac{b_1^{\frac{\ell}{2\ell-\alpha}}(s)}{\l(s)}\lesssim b_1(s_0)^{\frac{\ell}{2\ell-\alpha}}.
\ee

\noindent{\bf step 2} Improved control of $\mathcal E_{s_+}$. We now improve the control \fref{bootnorm} of the high order energy $\mathcal E_{s_+}$ by reintegrating the Lyapunov monotonicity \fref{monoenoiencle} coupled with the local Morawetz \fref{morawetzbound} formulas in the regime governed by \fref{keyestimated}, \fref{defuk}: for a large enough universal constant $D=D(M)\gg 1$, $A=A(M)\gg D$, $0<b_1<b_1^*(A)$,
\bee
 &&\frac{d}{ds}\left\{\frac{\mathcal E_{s_+}}{\lambda^{2(s_+-s_c)}}\left[1+O(b_1^{\eta(1-\dep)})\right]-D\mathcal M\right\}\\
\nonumber &\leq &  \frac{ b_1}{\lambda^{2(s_+-s_c)+2}}\left[ C\frac{\mathcal E_{s_+}}{M^{c\dk}}+Cb_1^{2L_++2(1-\dk)+2\eta(1-\dep)}\right.\\
&+& \left. C(M)\int \frac{1}{1+y^{4g}}\left(|\nabla\e_{k_++L_+}|^2+\frac{|\e_{k_++L_+}|^2}{1+y^2}\right)\right]\\
& - & \frac{b_1}{\l^{2(s_+-s_c)}}\left[ D\delta\int \frac{1}{1+y^{\delta}}\left(|\nabla \e_{2(k_++L_+)}|^2+\frac{|\e_{2(k_++L_+)}|^2}{y^2}\right)\right.\\
 &+& \left. Db_1^{2L_++2(1-\dk)+2\eta(1-\dep)}+\frac{CD}{A^{\delta}}\matchal E_{s_+}\right]\leq   \frac{ b_1}{\lambda^{2(s_+-s_c)+2}}\left[\frac{\mathcal E_{s_+}}{M^{c\dk}}+C(M)b_1^{2L_++2(1-\dk)+2\eta(1-\dep)}\right]\\
 &\leq & \left[\frac{K}{M^{c\dk}}+C(M)\right]\frac{ b_1b_1^{2L_++2(1-\dk)+2\eta(1-\dep)}}{\lambda^{2(s_+-s_c)}}
\eee
where we injected the bootstrap bound \fref{bootnorm} in the last step, and where we stress that $C(M)$ is independent of $K(M)$. We integrate in time using $\l(s_0)=1$ and the bound \fref{defmahtcam}:
\bea
\label{keyestimatedbis}
 \mathcal E_{s_+}&\leq& C\l(s)^{2(s_+-s_c)}\mathcal E_{s_+}(s_0)\\
\nonumber & + & C\left[\frac{K}{M^{c\dk}}+C(M)\right]\l(s)^{2(s_+-s_c)}\int_{s_0}^s\frac{ b_1(\tau)^{1+2L_++2(1-\dk)+2\eta(-\dep)}}{\lambda(\tau)^{2(s_+-s_c)}}d\tau.
\eea
We now estimate from \fref{keyestimated}:
\bee
 &&\l(s)^{2(s_+-s_c)}\int_{s_0}^s\frac{ b_1(\tau)^{1+2L_++2(1-\dk)+2\eta(1-\dep)}}{\lambda(\tau)^{2(s_+-s_c)}}d\tau\\
 & \lesssim &\frac{1}{s^{\frac{2\ell(s_+-s_c)}{2\ell-\alpha}}}\int_{s_0}^s \tau^{\frac{2\ell(s_+-s_c)}{2\ell-\alpha}-[1+2L_++2(1-\dk)+2\eta(1-\dep)]}d\tau.
\eee
The above integral is divergent since
\bea
\label{sononoeoe}
&&\frac{2\ell(s_+-s_c)}{2\ell-\alpha}-[1+2L_++2(1-\dk)+2\eta(1-\dep)]\\
\nonumber & = & \frac{2\alpha}{2\ell-\alpha}L_++O_{L_+\to _\infty}(1)\gg -1
\eea
and thus leads to the upper bound:
\bee
 &&\l(s)^{2(s_+-s_c)}\int_{s_0}^s\frac{ b_1(\tau)^{1+2L_++2(1-\dk)+2\eta(1-\dep)}}{\lambda(\tau)^{2(s_+-s_c)}}d\tau\\
 &\lesssim & \frac{1}{s^{\frac{2\ell(s_+-s_c)}{2\ell-\alpha}}}s^{\frac{2\ell(s_+-s_c)}{2\ell-\alpha}-[2L_++2(1-\dk)+2\eta(1-\dep)]}\lesssim  b_1^{2L_++2(1-\dk)+2\eta(1-\dep)}.
\eee
We now estimate the contribution of the initial data in \fref{keyestimatedbis} using \fref{keyestimated}, the initial bounds \fref{intilzero}, \fref{init2energy} and the comparison \fref{sononoeoe}:
\bee
&&\lambda(s)^{2(s_+-s_c)}\mathcal E_{s_+}(0)\lesssim \left(\frac{b_1(s)}{b_1(0)}\right)^{\frac{2\ell(s_+-s_c)}{2\ell-\alpha}}b_1(0)^{\frac{10\ell}{2\ell-\alpha}L_+}\leq b_1(s)^{2L_++2(1-\dk)+2\eta(1-\dep)}
\eee
for $L_+$ large enough. We have therefore obtained
\bea
\label{init3hbbis}
\nonumber \mathcal E_{s_+}(s)&\leq&  \left[C(M)+\frac{K(M)}{M^{c}}\right]b_1(s)^{2L_++2(1-\dk)+2\eta(1-\dep)}\\
&\leq & \frac{K(M)}{2}b_1(s)^{2L_++2(1-\dk)+2\eta(1-\dep)}
\eea
 for $K=K(M)$ large enough.\\

\noindent{\bf step 4} Improved control of $\|\nabla^{\sigma}\e\|^2_{L^2}$. We now turn to the improved control of the low Sobolev norms. We inject the bootstrap bound \fref{bootsmallsigma} into the monotonicity formula \fref{monotnoesigma} and obtain:
$$ \frac{d}{ds}\left\{\frac{\|\nabla^{\sigma}\e\|_{L^2}^2}{\l^{2(\sigma-s_c)}}\right\}\leq \frac{b_1}{\l^{2(\sigma-s_c)}}\left[Kb_1^{\frac{2}{L_+}+\frac{2\ell}{2\ell-\alpha}(\sigma-s_c)}+b_1^{\sigma-s_c+\nu_0}\right]\leq \frac{b_1}{\l^{2(\sigma-s_c)}}b_1^{\frac{1}{L_+}+\frac{2\ell}{2\ell-\alpha}(\sigma-s_c)}$$ for $\sigma-s_c$ small enough and $b_1<b_1^*(L_+)$ small enough. We now integrate in time $s$ and obtain from \fref{init2energy}:
\bee
\|\nabla^{\sigma}\e\|_{L^2}^2\leq \l(s)^{2(\sigma-s_c)}b_1(s_0)^{\frac{10\ell}{2\ell-\alpha}L_+}+\l(s)^{2(\sigma-s_c)}\int_{s_0}^{s_c}\frac{b_1(\tau)^{1+\frac{1}{L_+}+\frac{2\ell}{2\ell-\alpha}(\sigma-s_c)}}{\l(\tau)^{2(\sigma-s_c)}}d\tau.
\eee
The time integral is estimated using \fref{keyestimated}:
\bee
&&\l(s)^{2(\sigma-s_c)}\int_{s_0}^{s}\frac{b_1(\tau)^{1+(\sigma-s_c)(1+\nu)}}{\l(\tau)^{2(\sigma-s_c)}}d\tau\lesssim  \frac{1}{s^{\frac{2\ell(\sigma-s_c)}{2\ell-\alpha}}}\int_{s_0}^{s}\frac{d\tau}{\tau^{1+\frac{1}{L_+}}}\lesssim \frac{1}{s^{\frac{2\ell(\sigma-s_c)}{2\ell-\alpha}}}\\
& \lesssim & b_1(s)^{\frac{2\ell}{2\ell-\alpha}}
\eee
and similarly for the boundary term from \fref{keyestimated}:
\bee
\l(s)^{2(\sigma-s_c)}b_1(s_0)^{\frac{10\ell}{2\ell-\alpha}L_+}\lesssim b_1(s)^{\frac{2\ell(\sigma-s_c)}{2\ell-\alpha}}b_1(s_0)^{\frac{\ell}{2\ell-\alpha}(10L_+-2(\sigma-s_c))}\leq b_1(s)^{\frac{2\ell(\sigma-s_c)}{2\ell-\alpha}}
\eee
and we have therefore obtained the improved bound:
\be
\label{improvedlsifma}
\|\nabla^{\sigma}\e\|_{L^2}^2\lesssim b_1(s)^{\frac{2\ell(\sigma-s_c)}{2\ell-\alpha}}\leq \frac K2b_1(s)^{\frac{2\ell(\sigma-s_c)}{2\ell-\alpha}}
\ee 
for $K$ large enough as expected.\\
\noindent{\bf step 5} Control of the stable $b_k$ modes. We now close the control of the stable modes $(b_{\ell+1},\dots,b_{L_+})$ and claim the bound:\
\be
\label{improvebounderrors}
|b_k|\lesssim \frac{1}{s^{k+\eta(1-\dep)}}, \ \ \ell+1\leq k\leq L_+.
\ee

\noindent\underline{case $k=L_+$}: Let $$\bt_{L_+}=b_{L_+}+\frac{(\Lt^{L_+}\e,\chi_{B_0}J\Phi_{0,-})}{(\Phi_{0,+},\chi_{B_0}J\Phi_{0,-})}$$ then from \fref{neinoeneo}, \fref{init3hbbis}:
\be
\label{poitwidediff}
|\bt_{L_+}-b_L|\lesssim B_0^{2(1-\dk)}\sqrt{\matchal E_{s_+}}\lesssim b_1^{-(1-\dk)+L_++(1-\dk)+\eta(1-\dep)}\lesssim b_1^{L_++\eta(1-\dep)}
\ee
and hence from the improved modulation equation \fref{improvedboundbl}:
\bee
&&\left|(\bt_{L_+})_s+(2L_+-\alpha)b_1\bt_{L_+}\right|\lesssim b_1|b_{L_+}-\bt_{L_+}|+ \frac{1}{B_0^{2\dk}}\left[C(M)\sqrt{\mathcal E_{s_+}}+b_1^{L_++(1-\dk)+\eta(1-\dep)}\right]\\
& \lesssim & b_1^{L_++1+\eta(1-\dep)}+b_1^{\dk}\left[b_1^{L_++(1-\dk)+\eta(1-\dep)}\right]\lesssim b_1^{L_++1+\eta(1-\dep)}.
\eee
Equivalently:
$$\left|\frac{d}{ds}\left\{\frac{\bt_{L_+}}{\l^{2L_+-\alpha}}\right\}\right|\lesssim \frac{b_1^{L_++1+(1-\dep)\eta}}{\l^{2L_+-\alpha}}.$$ We integrate this identity in time from $s_0$. The time integral is estimated from \fref{keyestimated}:
\bee
&&\l(s)^{2L_+-\alpha}\int_{s_0}^s\frac{b_1(\tau)^{L_++1+(1-\dep)\eta}}{\l(\tau)^{2L_+-\alpha}}d\tau\lesssim \frac{1}{s^{\frac{(2L_+-\alpha)\ell}{2\ell-\alpha}}}\int_{s_0}^s\tau^{\frac{(2L_+-\alpha)\ell}{2\ell-\alpha}-L_+-1-(1-\dep)\eta}d\tau\\
& \lesssim & \frac{1}{s^{\frac{(2L_+-\alpha)\ell}{2\ell-\alpha}}}\int_{s_0}^s\tau^{-1+\frac{\alpha(L_+-\ell)}{2\ell-\alpha}-(1-\dep)\eta}d\tau\leq\frac{s^{\frac{\alpha(L_+-\ell)}{2\ell-\alpha}}}{s^{\frac{(2L_+-\alpha)\ell}{2\ell-\alpha}}}\frac{1}{s^{(1-\dep)\eta}}= \frac{1}{s^{L_++(1-\dep)\eta}}\\
& \lesssim & b_1^{L_++(1-\dep)\eta}.
\eee
The boundary term is estimated using \fref{estintialbis}, \fref{init2energy}, \fref{poitwidediff}: 
$$
|\tilde{b}_{L_+}(s_0)|\lesssim b_1(s_0)^{5\frac{2L_+-\alpha}{2\ell-\alpha}}+B_0^{2(1-\dk)}\sqrt{\mathcal E_{s_+}(s_0)}\lesssim b_1(s_0)^{5\frac{2L_+-\alpha}{2\ell-\alpha}}$$
and hence using \fref{keyestimated}:
$$\left(\frac{\l(s)}{\l(s_0)}\right)^{2L_+-\alpha}|\tilde{b}_{L_+}(s_0)|\lesssim \frac{b_1(s_0)^{\frac{5(2L_+-\alpha)\ell}{2\ell-\alpha}}}{b_1(s_0)^{\frac{2L_+-\alpha}{2\ell-\alpha}}}\frac{1}{s^{\frac{(2L_+-\alpha)\ell}{2\ell-\alpha}}}\leq \frac{1}{s^{L_++(1-\dep)\eta}}.$$
 The collection of above bounds yields the bound $$|\bt_{L_+}|\lesssim \frac{1}{s^{L_++(1-\dep)\eta}}$$ which together with \fref{poitwidediff} yields: $$|b_{L_+}|\lesssim \frac{1}{s^{L_++(1-\dep)\eta}}$$ and \fref{improvebounderrors} is proved for $k=L_+$.\\
\noindent\underline{case $\ell+1\leq k\leq L_+-1$}: We now prove \fref{improvebounderrors} by a descending induction: we assume the claim for $k+1$ and proved it for  $k$, $\ell+1\leq k\leq L_+-1$.
 From Lemma \ref{modulationequations} and the induction claim:
$$
\left|(b_k)_s-(2k-\alpha)\lsl b_k\right|\lesssim b_1^{L_++1}+|b_{k+1}|\lesssim b_1^{k+1+\eta(1-\dep)}
$$ from which
$$\left|\frac{d}{ds}\left\{\frac{b_k}{\l^{2k-\alpha}}\right\}\right|\lesssim \frac{b_1^{k+1+\eta(1-\dep)}}{\l^{2k-\alpha}}.$$
We integrate this identity in time from $s_0$. The time integral is estimated from \fref{keyestimated} using $\ell+1\leq k\leq L_+-1$:
\bee
&&\l(s)^{2k-\alpha}\int_{s_0}^s\frac{b_1(\tau)^{k+1+\eta(1-\dep)}}{\l(\tau)^{2k-\alpha}}d\tau\lesssim \frac{1}{s^{\frac{(2k-\alpha)\ell}{2\ell-\alpha}}}\int_{s_0}^s\tau^{\frac{(2k-\alpha)\ell}{2\ell-\alpha}-k-1-\eta(1-\dep)}d\tau\\
& \lesssim& \frac{1}{s^{\frac{(2k-\alpha)\ell}{2\ell-\alpha}}}\int_{s_0}^s\tau^{-1-\eta(1-\dep)+\frac{\alpha(k-\ell)}{2\ell-\alpha}}d\tau\leq \frac{s^{\frac{\alpha(k-\ell)-\eta(1-\dep)}{2\ell-\alpha}}}{s^{\frac{(2k-\alpha)\ell}{2\ell-\alpha}}}=\frac{1}{s^{k+\eta(1-\dep)}}\\
&\lesssim& b_1^{k+\eta(1-\dep)}.
\eee
The boundary term in time is controlled using \fref{keyestimated}, \fref{estintialbis}: 
$$\left(\frac{\l(s)}{\l(s_0)}\right)^{2k-\alpha}|b_k(s_0)|\lesssim \frac{b_1(s_0)^{k+\frac{5(2k-\alpha)}{2\ell-\alpha}}}{b_1(0)^{\frac{(2k-\alpha)\ell}{2\ell-\alpha}}}\frac{1}{s^{\frac{(2k-\alpha)\ell}{2\ell-\alpha}}}\leq \frac{1}{s^{k+\frac{\alpha(k-\ell)}{2\ell-\alpha}}}\lesssim b_1^{k+\eta(1-\dep)}$$ thanks to $k\geq \ell+1$. 
 The collection of above bounds yields the expected bound \fref{improvebounderrors}.\\

\noindent{\bf step 6} Control of the stable $a_k$ modes. Recall \fref{defkell}. We claim a bound:
\be
\label{improvebounderrorsak}
|a_k|\lesssim \frac{1}{s^{k+\frac\alpha 2+\eta(1-\dep)}},  \ \ k_{\ell}+1\leq k\leq L_-.
\ee
\noindent\underline{case $k=L_-$}:  let $$\at_{L_-}=a_{L_-}+\frac{(\Lt^{L_-}\e,\chi_{B_0}J\Phi_{0,+})}{(\Phi_{0,-},\chi_{B_0}J\Phi_{0,+})}$$ then from \fref{neinoeneobis}, \fref{init3hbbis}, \fref{dgammkrealtion}:
\bea
\label{diffattal}
\nonumber |\at_{L_-}-a_{L_-}|&\lesssim &B_0^{2(1-\dkm)}\sqrt{\mathcal E_{s_+}}\lesssim b_1^{-(1-\dkm)+L_++(1-\dk)+\eta(1-\dep)}=b_1^{L_-+\Delta k+\dkm-\dk+\eta(1-\dep)}\\
& = & b_1^{L_-+\frac \alpha 2+\eta(1-\dep)}.
\eea
From the improved modulation equation \fref{improvedboundal},
\bee
&&\left|(\at_{L_-})_s+2L_-b_1\at_{L_-}\right|\lesssim b_1|a_{L_-}-\at_{L_-}|+ \frac{1}{B_0^{2\dkm}}\left[C(M)\sqrt{\mathcal E_{s_+}}+b_1^{L_++(1-\dk)(1+\eta)}\right]\\
& \lesssim & b_1^{L_-+1+\frac \alpha 2+\eta(1-\dep)}.
\eee
Equivalently:
$$\left|\frac{d}{ds}\left\{\frac{\at_{L_-}}{\l^{2L_-}}\right\}\right|\lesssim \frac{b_1^{L_-+1+\frac \alpha 2+\eta(1-\dep)}}{\l^{2L_-}}.$$ We integrate this identity in time from $s_0$. The time integral is estimated from \fref{keyestimated} for $L_-$ large enough:
\bee
&&\l(s)^{2L_-}\int_{s_0}^s\frac{b_1(\tau)^{L_-+1+\frac \alpha 2+(1-\dep)\eta}}{\l(\tau)^{2L_-}}d\tau\lesssim \frac{1}{s^{\frac{(2L_-)\ell}{2\ell-\alpha}}}\int_{s_0}^s\tau^{\frac{(2L_-)\ell}{2\ell-\alpha}-L_--1-\frac \alpha 2-(1-\dep)\eta}d\tau\\
& \lesssim & \frac{1}{s^{\frac{(2L_-)\ell}{2\ell-\alpha}}}\int_{s_0}^s\tau^{-1+\frac{\alpha L_-}{2\ell-\alpha}-\frac\alpha 2-(1-\dep)\eta}d\tau\leq\frac{s^{\frac{\alpha L_-}{2\ell-\alpha}}}{s^{\frac{(2L_-)\ell}{2\ell-\alpha}}}\frac{1}{s^{\frac\alpha2+(1-\dep)\eta}}=\frac{1}{s^{L_-+\frac \alpha 2+(1-\dep)\eta}}.
\eee
The boundary term is estimated using \fref{estintialbisa}, \fref{init2energy}, \fref{diffattal}: $$|\tilde{a}_{L_-}(s_0)|\lesssim b_1(s_0)^{\frac\alpha 2+5\frac{2L_-}{2\ell-\alpha}}+B_0^{2(1-\dk)}\sqrt{\mathcal E_{s_+}(s_0)}\lesssim b_1(s_0)^{\frac\alpha 2+5\frac{2L_-}{2\ell-\alpha}}$$ and hence:
$$\left(\frac{\l(s)}{\l(s_0)}\right)^{2L_-}|\tilde{a}_{L_-}(s_0)|\lesssim \frac{b_1(s_0)^{\frac\alpha 2+5\frac{2L_-}{2\ell-\alpha}}}{b_1(s_0)^{\frac{2L_-}{2\ell-\alpha}}}\frac{1}{s^{\frac{(2L_-)\ell}{2\ell-\alpha}}}\leq \frac{1}{s^{L_-+\frac \alpha 2+\eta(1-\dep)}}.$$
The collection of above bounds yields the bound $$|\at_{L_-}|\lesssim \frac{1}{s^{L_-+\frac \alpha 2+\eta(1-\dep)}}$$ which together with \fref{diffattal} yields: $$|a_{L_-}|\lesssim \frac{1}{s^{L_-+\frac{\alpha}{2}+\eta(1-\dep)}}$$ and \fref{improvebounderrors} is proved for $k=L_-$.\\
\noindent\underline{case $k_\ell+1\leq k\leq L_--1$}: We now prove \fref{improvebounderrors} by a descending induction: we assume the claim for $k+1$ and prove it for  $k$, $k_\ell+1\leq k\leq L_--1$. From Lemma \ref{modulationequations} and the induction claim:
$$
\left|(a_k)_s-2k\lsl a_k\right|\lesssim b_1^{L_++1+(1-\dk)+\eta(1-\dep)}+|a_{k+1}|\lesssim b_1^{k+1+\frac\alpha 2+\eta(1-\dep)}
$$ from which
$$\left|\frac{d}{ds}\left\{\frac{a_k}{\l^{2k}}\right\}\right|\lesssim \frac{b_1^{k+1+\frac\alpha 2+\eta(1-\dep)}}{\l^{2k}}.$$
We integrate this identity in time from $s_0$. The time integral is estimated from \fref{keyestimated} using $k_\ell+1\leq k\leq L_+-1$ and \fref{defkell}:
\bee
&&\l(s)^{2k}\int_{s_0}^s\frac{b_1(\tau)^{k+1+\frac\alpha 2+\eta(1-\dep)}}{\l(\tau)^{2k}}d\tau\lesssim \frac{1}{s^{\frac{(2k)\ell}{2\ell-\alpha}}}\int_{s_0}^s\tau^{\frac{(2k)\ell}{2\ell-\alpha}-k-1-\frac{\alpha}2-\eta(1-\dep)}d\tau\\
& \lesssim& \frac{1}{s^{\frac{(2k)\ell}{2\ell-\alpha}}}\int_{s_0}^s\tau^{-1+\frac{\alpha}{2\ell-\alpha}\left[k-(k_\ell+\delta_{\ell})\right]-\eta(1-\dep)}d\tau\leq \frac{s^{\frac{\alpha}{2\ell-\alpha}\left[k-(k_\ell+\delta_{\ell})\right]-\eta(1-\dep)}}{s^{\frac{(2k-\alpha)\ell}{2\ell-\alpha}}}\\
&=& \frac{1}{s^{k+\frac{\alpha}{2}+\eta(1-\dep)}}.
\eee
The boundary term in time is controlled using \fref{keyestimated}, \fref{estintialbisa}: 
$$\left(\frac{\l(s)}{\l(s_0)}\right)^{2k}|a_k(s_0)|\lesssim \frac{b_1(s_0)^{k+\frac\alpha 2+\frac{5(2k)\ell}{2\ell-\alpha}}}{b_1(s_0)^{\frac{(2k)\ell}{2\ell-\alpha}}}\frac{1}{s^{\frac{(2k)\ell}{2\ell-\alpha}}}\leq \frac{1}{s^{k+\frac{\alpha}{2}+\frac{\alpha}{2\ell-\alpha}\left[k-(k_{\ell}+\delta_{\ell})\right]}}\lesssim \frac{1}{s^{k+\frac \alpha 2+\eta(1-\dep)}}$$ thanks to $k\geq k_\ell+1$.  The collection of above bounds yields the expected bound \fref{improvebounderrorsak}.\\
This concludes the proof of Proposition \ref{bootstrap}, modulo the bound for the stable $b$-mode $V_1$. 
We now turn to the remaining step in the proof of Proposition \ref{bootstrap'} and the proof of the improved bound \eqref{controlunstable} for $V_1$.\\

\noindent{\bf step 6} Contradiction through a topological argument. Recall the decompositions \fref{defukbis}, \fref{defatildegib}
\bee
&&b_k=b_k^e+\frac{U_k}{s^{k}},\ \ 1\leq k\leq \ell, \ \ V=P_{\ell}U\\
&&A_k=s^{k+\frac \alpha 2}a_k,\ \ A=(A_k)_{1\leq k\leq k_\ell}, \ \ \tilde{A}=Q_\ell A,
\eee
where $P_\ell,Q_\ell$  diagonalize the matrices $M_\ell, \matchal M_{k_\ell}$ with spectra \fref{specta}, \fref{Pdiegmle} respectively . We argue by contradiction and assume 
\fref{hypcontr}:
$$
\forall \left(V_k(s_0)s_0^{\frac{\eta}2(1-\dep)}\right)_{2\leq k\leq \ell}\times \left(\tilde{A}_k(s_0)s_0^{\frac{\eta}2(1-\dep)}\right)_{1\leq k\leq k_\ell}\in \mathcal B_{\ell+k_\ell-1}\left(1\right) 
$$
the exit time  \fref{defexitsstar} $s^*<\infty$. We claim that if  $s_0$ is large enough  this contradicts the Brouwer fixed point theorem.\\
Indeed, we first estimate from \fref{reformulaion}, \fref{parameters}: for $1\leq k\leq \ell-1$,
$$
\left|s(U_k)_s-(M_\ell U)_k\right|\lesssim s^{k+1}\left|(b_k)_s+\left(2k-\alpha\right)b_1b_k-b_{k+1}\right|+|U|^2\lesssim  \frac{1}{s^{L_+-k}}+|U|^2,
$$
and for $k=\ell$ using \fref{reformulaionl}, \fref{parameters} and the improved bound \fref{improvebounderrors}:
\bee
\left|s(U_\ell)_s-(M_\ell U)_\ell\right|\lesssim s^{\ell+1}\left[\left|(b_\ell)_s+\left(2\ell-\alpha\right)b_1b_\ell-b_{\ell+1}\right|+|b_{\ell+1}|\right]+|U|^2\lesssim  \frac{1}{s^{\eta(1-\dep)}}+|U|^2.
\eee
This yields using the diagonalization \fref{specta}:
\be
\label{eqvtilde}
sV_s=D_LV_s+O\left(\frac{1}{s^{\eta(1-\dep)}}\right).
\ee
This first implies the control of the stable mode $V_1$ from \fref{specta}:
$$|(sV_1)_s|\lesssim \frac{1}{s^{\eta(1-\dep)}}$$ and thus from \fref{estintial}:
$$
|s^{\eta(1-\dep)}V_1(s)|\leq \left(\frac{s_0}{s}\right)^{1-\eta(1-\dep)}s_0^{\eta(1-\dep)}V_1(0)+1\lesssim s_0^{\eta(1-\dep)}
$$
and thus 
\be
\label{controlstable}
|s^{\frac\eta 2(1-\dep)}V_1(s)|\leq 1
\ee
for $s_0\geq s_0^*(\eta)$ large enough. 

From \fref{init3hbbis}, \fref{improvedlsifma}, \fref{improvebounderrors}, \fref{improvebounderrorsak}, \fref{controlstable}, \fref{defexitsstar} and a standard continuity argument 
we conclude that \fref{hypcontr} implies: 
\be
\label{cneoneno}
\sum_{k=2}^{\ell}\left|(s^*)^{\frac \eta 2(1-\dep)}V_k(s^*)\right|^2+\sum_{k=1}^{k_{\ell}}\left|(s^*)^{\frac \eta 2(1-\dep)}\tilde{A}_k(s^*)\right|^2=1.
\ee 
We then compute from \fref{eqvtilde}, \fref{specta} at the exit time:
\bee
&&\frac 12\frac{d}{ds}\left\{\sum_{k=2}^{\ell}\left|(s^*)^{\frac \eta 2(1-\dep)}V_k(s^*)\right|^2\right\}=(s^*)^{\eta(1-\dep)-1}\sum_{k=2}^{\ell}\left[\frac \eta 2(1-\dep)V_k^2+sV_k(V_k)_s\right](s^*)\\
&=&(s^*)^{\eta(1-\dep)-1}\left[\sum_{k=2}^{\ell}\left(\frac{k\alpha}{2k-\alpha}+\frac \eta 2(1-\dep)\right)V_k^2+O\left(\frac 1{(s^*)^{\frac{3}2\eta(1-\dep)}}\right)\right](s^*)\\
& \geq & \frac{1}{s^*}\left[c(d,p,\ell)\sum_{k=2}^{\ell}\left|(s^*)^{\frac \eta 2(1-\dep)}V_k(s^*)\right|^2+O\left(\frac 1{(s^*)^{\frac{\eta}{2}(1-\dep)}}\right)\right]
\eee
for some universal constant $c(d,p,\ell)>0$. We now estimate from \fref{parameters}, \fref{improvebounderrorsak}: 
\bee
&&|(a_k)_s+2kb_1a_k-a_{k+1}|\lesssim \frac{1}{s^{k+1+\frac\alpha 2+\eta(1-\dep)}}, \ \ 1\leq k\leq k_{\ell}-1,\\
&& |(a_k)_s+2kb_1a_k|\lesssim |a_{k+1}|+\frac{1}{s^{k+1+\frac\alpha 2+\eta(1-\dep)}}\lesssim \frac{1}{s^{k+1+\frac\alpha 2+\eta(1-\dep)}}\ \ \mbox{for}\ \ k=k_\ell.
\eee
Using $$\left|b_1-\frac{c_1}{s}\right|\lesssim \frac1{s^{1+\frac \eta 2(1-\dep)}},$$ Lemma \ref{lemmaak} and \fref{controlunstable} this implies:
$$
\left|s(A_k)_s-(\mathcal M_{k\ell}\matchal A)_k\right|\lesssim \frac{1}{s^{\eta (1-\dep)}}$$ 
or, equivalently, in the diagonal basis:
$$\left|s(\tilde{A}_k)_s+\frac{\alpha}{2\ell-\alpha}\left[k-(k_\ell+\delta_{\ell})\right]\tilde{A}_k\right|\lesssim \frac{1}{s^{\eta (1-\dep)}}.$$
We compute for $k\leq k_\ell$ that at the exit time
\bee
&&\frac 12\frac{d}{ds}\left\{\sum_{i=1}^{k_{\ell}}\left|(s^*)^{\frac \eta 2(1-\dep)}\tilde{A}_k(s^*)\right|^2\right\}=(s^*)^{\eta(1-\dep)-1}\sum_{k=2}^{\ell}\left[\frac \eta 2(1-\dep)\tilde{A}_k^2+s\tilde{A}_k(\tilde{A}_k)_s\right](s^*)\\
&=&(s^*)^{\eta(1-\dep)-1}\left[\sum_{k=2}^{\ell}\left(\frac{(k_\ell+\delta_\ell-k)\alpha}{2k-\alpha}+\frac \eta 2(1-\dep)\right)\tilde{A}_k^2+O\left(\frac 1{(s^*)^{\frac{3}2\eta(1-\dep)}}\right)\right](s^*)\\
& \geq & \frac{1}{s^*}\left[c(d,p,\ell)\sum_{k=2}^{\ell}\left|(s^*)^{\frac \eta 2(1-\dep)}\tilde{A}_k(s^*)\right|^2+O\left(\frac 1{(s^*)^{\frac{\eta}{2}(1-\dep)}}\right)\right]\eee
for some universal constant $c(d,p,\ell)>0$. We therefore obtain the fundamental outgoing behavior at exit time:
\bee
&&\frac12\frac{d}{ds}\left\{\sum_{i=2}^{\ell}\left|(s^*)^{\frac \eta 2(1-\dep)}V_i(s^*)\right|^2+\sum_{i=1}^{k_{\ell}}\left|(s^*)^{\frac \eta 2(1-\dep)}\tilde{A}_k(s^*)\right|^2\right\}\\
&\geq &\frac{c(d,p,\ell)}{s^*}\left [\sum_{i=2}^{\ell}\left|(s^*)^{\frac \eta 2(1-\dep)}V_i(s^*)\right|^2+\sum_{i=1}^{k_{\ell}}\left|(s^*)^{\frac \eta 2(1-\dep)}\tilde{A}_k(s^*)\right|^2+O\left( \frac{1}{(s^*)^{\frac \eta 2(1-\dep)}}\right)\right]\\
& \geq & \frac{c(d,p,\ell)}{s^*}\left [1+O\left( \frac{1}{(s^*)^{\frac \eta 2(1-\dep)}}\right)\right]>0
\eee
for $s_0\geq s_0^*$ large enough.
This strict outgoing behavior at exit time implies the continuity of the map $\Phi:\mathcal B_{\ell+k_\ell-1}(1)\to \mathcal S_{\ell+k_\ell-1}(1)$:
\bee
&&\left(s_0^{\frac\eta2(1-\dk)}{V}_k(s_0)\right)_{2\leq k\leq \ell}\times\left(s_0^{\frac\eta2(1-\dep)}\tilde{A}_k(s_0)\right)_{1\leq k\leq k_\ell}
\mapsto\\ \quad && \left((s^*)^{\frac\eta 2(1-\dk)}{V}_k(s^*)\right)_{2\leq k \leq k_\ell}\times \left((s^*)^{\frac\eta 2(1-\dk)}\tilde{A}_k(s^*)\right)_{1\leq k\leq k_{\ell}}.
\eee
Since $\Phi$ is the identity map on the boundary sphere $\Bbb S_{\ell+k_\ell-1}(1)$, this is a contradiction of the standard fact that a boundary sphere can not be a
continuous retract of the ball.
This concludes the proof of Proposition \ref{bootstrap'}.

%%%%%%%%%%%%%%%%%%%%%%%%%%%%%%%%%%%%%%%%%%

\subsection{Proof of Theorem \ref{thmmain}}

%%%%%%%%%%%%%%%%%%%%%%%%%%%%%%%%%%%%%%%%%%%%%%%%%%%%%%%%%%%%%%%%%%%%%%%%%%%%%%%%%%%%%%%%%%%%%%%%%%%%%%%%%%%

We are now in position to conclude the proof of Theorem \ref{thmmain}.\\

\noindent{\bf step 1} Finite time blow up. We pick initial data satisfying the conclusions of Proposition \ref{bootstrap'}.  In particular, integrating \fref{eneoeoe} from $s$ to $+\infty$ implies the existence of $c(u_0)>0$ such that $$\lambda(s)=\frac{c(u_0)}{s^\frac{\ell}{2\ell-\alpha}}\left[1+O\left(\frac{1}{s^{c\eta}}\right)\right].$$ Then from \fref{parameters}, \fref{defuk}:
$$-\l\l_t=-\lsl=b_1+O\left(\frac{1}{s^L}\right)=\frac{c_1}{s}\left[1+O\left(\frac{1}{s^{\tilde{\eta}}}\right)\right]=c(u_0)\l^{\frac{2\ell-\alpha}{\ell}}\left[1+O\left(\frac{1}{s^{\tilde{\eta}}}\right)\right]$$ and hence the ODE: $$-\l^{1-\frac{2\ell-\alpha}{\ell}}\l_t=c(u_0)(1+o(1)).$$ We easily conclude that $\l$ vanishes at some finite time $T=T(u_0)<+\infty$, 
with near blow up time behavior: 
\be
\label{vbibvbeieb}
\l(t)=c(u_0)(1+o(1))(T-t)^{\frac{\ell}{\alpha}}.
\ee 
The phase parameter is estimated from \fref{parameters}
$$|\gamma_s|\lesssim \frac{1}{s^{1+\frac\alpha2}} \ \ \mbox{and hence} \ \ \int_{s_0}^{+\infty}|\gamma_s|ds<+\infty$$ which implies \fref{convergncephase}.
\begin{remark} Note that this closes the construction of a type II blow up solution in $\dot{H}^{\sigma}\cap \dot{H}^{s_+}$ and no additional regularity is needed on the data. In particular, whether the data has finite energy or mass is irrelevant.
\end{remark}

\noindent{\bf step 2} Control of Sobolev norms. First observe by interpolation between \fref{bootnorm} and \fref{bootsmallsigma} that $$\forall \sigma\leq s\leq s_+, \ \ \lim_{t\uparrow T}\|\nabla^s\e(t)\|_{L^2}=0.$$ We now further assume that $u_0\in L^2$ and aim at controlling low Sobolev norms. By mass conservation:
\be
\label{Ltwoconservation}\|u(t)\|_{L^2}=\|u_0\|_{L^2}.
\ee
We split $$Q+\e=\chi_{\frac 1\lambda}Q+\et \ \ \mbox{i.e.}\ \ \tilde{\e}=(1-\chi_{\frac 1\lambda})Q+\e,$$ then from \fref{Ltwoconservation}:
\be
\label{nvbibviebveib}
\|\et\|_{L^2}\lesssim \|(1-\chi_{\frac 1\lambda})Q\|_{L^2}+\|Q+\e\|_{L^2}\lesssim \frac{C(u_0)}{\l^{s_c}}.
\ee 
Moreover from \fref{bootsmallsigma}, \fref{keyestimated}: $$\|\nabla^\sigma\e\|_{L^2}\lesssim C(u_0)\l^{\sigma-s_c}$$ and hence by direct computation: 
\be
\label{vnovnoneneo}
\|\nabla^\sigma \et\|_{L^2}\lesssim \|\nabla^{\sigma}(1-\chi_{\frac 1\lambda})Q\|_{L^2}+\|\nabla^{\sigma}\e\|_{L^2}\lesssim \l^{\sigma-s_c}.
\ee
We interpolate \fref{nvbibviebveib} and \fref{vnovnoneneo} and conclude: $$\forall 0\leq s\leq \sigma, \ \ \|\nabla^s\et\|_{L^2}\lesssim C(u_0)\l^{s-s_c}.$$
Therefore for $2\leq s<s_c$: $$\|u(t)\|_{\dot{H}^s}\lesssim \l^{s_c-s}\left[\|\nabla^{s}(1-\chi_{\frac 1\lambda})Q\|_{L^2}+\|\nabla^{s}\et\|_{L^2}\right]\lesssim C(u_0)$$ and for the critical norm, using \fref{extensionrq}, \fref{vbibvbeieb}:
\bee
\|u(t)\|_{\dot{H}^{s_c}}&=&\|\nabla^{s_c}(1-\chi_{\frac 1\lambda})Q+\nabla^{s_c}\et\|_{L^2}=\|\nabla^{s_c}(1-\chi_{\frac 1\lambda})Q\|_{L^2}+O(1)\\
&=& \left(c^2_{\infty}|\log \l(t)|\right)^{\frac 12}+O(1)= \left[c_\infty\sqrt{\frac{\ell}{\alpha}}+o(1)\right]\sqrt{|\log(T-t)|}
\eee
as $t\to T$. This concludes the proof of Theorem \ref{thmmain}.

%%%%%%%%%%%%%%%%%%%%%%%%%%%%%%%%%%%%%%%%%%%%%%%%%%%%%%%%%%%%%%%%%%%%%%%%%%%%%%%%%%%%%%%%%%%%%%%%%%%%%%%%%%%

\begin{appendix}

%%%%%%%%%%%%%%%%%%%%%%%%%%%%%%%%%%%%%%%%%%%%%%%%%%%%%%%%%%%%%%%%%%%%%%%%%%%%%%%%%%%%%%%%%%%%%%%%%%%%%%%%%%%

%%%%%%%%%%%%%%%%%%%%%%%%%%%%%%%%%%%%%%%%%%%%%%%%%%%%%%%%%%%%%%%%%%%%%%%%%%%%%%%%%%%%%%%%%%%%%%%%%%%%%%%%%%%

\section{Super critical numerology}
\label{numero}
%%%%%%%%%%%%%%%%%%%%%%%%%%%%%%%%%%%%%%%%%%%%%%%%%%%%%%%%%%%%%%%%%%%%%%%%%%%%%%%%%%%%%%%%%%%%%%%%%%%%%%%%%%%
We collect in this Appendix some algebraic facts induced by the condition $p>p_{JL}$. Recall that the exponent $p_{JL}$ is defined in \eqref{exponentpjl},
the critical Sobolev exponent $s_c=\frac d2-\frac 2{p-1}$ and the parameter $\gamma$ is in \eqref{defgamma}.

\begin{lemma}[Range of parameters]
\label{conditionsc}
Let $d\geq 11$. The condition $$p_{JL}<p<+\infty$$ is equivalent to 
\be
\label{condsc}
2+\sqrt{d-1}<s_c<\frac d2.
\ee
Moreover, there holds the bound:
\be
\label{boundgamma}
2<\alpha=\gamma-\frac{2}{p-1}<\frac d2-1
\ee
\be
\label{estkwplus}
k_+={\rm E}\left[\frac 12+\frac 12\left(\frac d2-\gamma\right)\right]\geq 1.
\ee
\end{lemma}

\begin{proof}[Proof of Lemma \ref{conditionsc}]
Recall the definitions \fref{selfsimilarsolutions}, \fref{defgamma}. We compute the discriminant in terms of $s_c$:
\bee
{\rm Discr}&=&(d-2)^2-4pc_{\infty}^{p-1}=(d-2)^2-4(p-1+1)c_{\infty}^{p-1}\\
& = & (d-2)^2-4\left(\frac{4}{d-2s_c}+1\right)\left(\frac d2-s_c\right)\left(\frac d2-2+s_c\right)\\
& = & (d-2)^2-(4+d-2s_c)(d-4+2s_c)=(d-2)^2-(d+2(2-s_c))(d-2(2-s_c))\\
& = & (d-2)^2-d^2+4(2-s_c)^2=4\left[(s_c-2)^2-(d-1)\right].
\eee
In particular $$s_c(p_{JL})=2+\sqrt{d-1}$$ and hence \footnote{Observe that $ 2+\sqrt{d-1}<\frac d2$ iff $ d^2-120d+20=(d-10)(d-2)>0$ ie $d\geq 11$.}
$$p_{JL}<p<+\infty\ \ \mbox{iff} \ \ 2+\sqrt{d-1}<s_c<\frac d2.
$$
Define $f(s)=s-\sqrt{(s-2)^2-(d-1)}$. From \fref{defgamma}: $$\gamma-\frac2{p-1}=s_c-1-\frac{\sqrt{{\rm Discr}}}{2}=s_c-1-\sqrt{(s_c-2)^2-(d-1)}=f(s_c)-1.$$ We compute
$$f'(s_c)=1-\frac{s_c-2}{\sqrt{(s_c-2)^2-(d-1)}}<0$$ and thus $$f(s_c)>f(\frac{d}{2}).$$ We now compute:
\bee
f(\frac d2)&=&\frac d2-\sqrt{(\frac d2-2)^2-(d-1)}=\frac 12\left[d-\sqrt{d^2-12d+20}\right]=\frac{6d-10}{d+\sqrt{(d-10)(d-2)}}\\
& >3 &
\eee
by a direct check, and \fref{boundgamma} is proved.\\
Finally, from \fref{defgamma}: $$\frac 12+\frac 12\left(\frac d2-\gamma\right)=\frac 12+\frac12\left[1+\frac{\sqrt{\rm Discr}}{2}\right]\geq 1$$ and \fref{estkwplus} follows.
\end{proof}

%%%%%%%%%%%%%%%%%%%%%%%%%%%%%%%%%%%%%%%%%%%%%%%%%%%%%%%%%%%%%%%%%%%%%%%%%%%%%%%%%%%%%%%%%%%%%%%%%%%%%%%%%%%

\section{Hardy inequalities}
\label{appendixhardy}
%%%%%%%%%%%%%%%%%%%%%%%%%%%%%%%%%%%%%%%%%%%%%%%%%%%%%%%%%%%%%%%%%%%%%%%%%%%%%%%%%%%%%%%%%%%%%%%%%%%%%%%%%%%

In this section we recall the standard Hardy type inequalities in dimension $d\geq 3$.  We define the space of radially symmetric test functions $$\mathcal D_{\rm rad}=\{u\in \mathcal C^{\infty}_c(\Bbb R^d)\ \ \mbox{with radial symmetry}\}$$. Note that the notation $\int f$ stands for the integral over ${\Bbb R}^d$ with respect to the standard volume form:
$$
\int f:=\int_0^\infty f(y) y^{d-1} dy
$$ 
We also recall the notation
 $$
D^k= \left\{\begin{array}{llll}\Delta^m,\ \ k=2m,\\
 \partial_y \Delta^m,\ \
k=2m+1\end{array}\right.
$$ 

\begin{lemma}[Hardy with the best constant]
\label{lemmahardy}

(i) {\em Hardy near the origin}: let $u\in \matchal \mathcal D_{\rm rad}$, then:
\be
\label{esthardythree}
\int_{y\leq 1} |\pa_yu|^2y^{d-1}dy\geq \frac{(d-2)^2}{4}\int_{y\leq 1}\frac{u^2}{y^{2}}y^{d-1}dy-C_{d}u^2(1).
\ee
(ii) {\em Hardy away from the origin, non-critical exponent}: let $q>0$, $q\neq \frac{d-2}{2}$ and $u\in \mathcal D_{\rm rad}$, then:
\bea
\label{esthardythreebis}
&&\int_{y\geq 1}\frac{|\pa_yu|^2}{y^{2q}}y^{d-1}dy\geq \frac {|d-(2q+2)|}2\left\|\frac{u}{y^{q+1-\frac d2}}\right\|_{L^{\infty}(y\geq 1)}^2-C_{q,d}u^2(1)\\
\nonumber &&\int_{y\geq 1}\frac{|\pa_yu|^2}{y^{2q}}y^{d-1}\geq  \left(\frac{d-(2q+2)}{2}\right)^2\int_{y\geq 1}\frac{u^2}{y^{2+2q}}y^{d-1}dy-C_{q,d}u^2(1)
\eea
(iii) {\em Hardy away from the origin, critical exponent}: let $q=\frac{d-2}{2}$ and $u\in\mathcal D_{\rm rad}$, then:
\be
\label{esthardytwo}
 \int_{y\geq 1}\frac{|\pa_yu|^2}{y^{2q}}y^{d-1}dy\geq  \frac 14\int_{y\geq 1}\frac{u^2}{y^{2q+2}(1+\log y)^2}y^{d-1}dy-C_{d}u^2(1).
\ee
(iv) {\em General weighted Hardy}: let  $u\in \mathcal D_{\rm rad}$ then for any $\delta>0$, $k\in \Bbb N^*$ with $k\geq 2$ and $1\leq j\leq k-1$,
\be
\label{hardywehghtgeenral}
\int \frac{|D^ju|^2}{1+y^{\delta+2(k-j)}}\lesssim_{j,\delta}\int \frac{|D^ku|^2}{1+y^\delta}+\int \frac{u^2}{1+y^{\delta+2k}}.
\ee
\end{lemma}

\begin{proof}
{\em Proof of (i)}: We integrate by parts:
\bee
\int_\e^1\frac{u^2}{y^2}y^{d-1}dy&=&\frac{1}{d-2}\int_\e^1u^2\pa_y(y^{d-2})dy=\frac{1}{d-2}\left[u^2y^{d-2}\right]_\e^1-\frac{2}{d-2}\int_\e^1\frac{\pa_y uu}{y}y^{d-1}dy\\
& \leq & C_{d}u^2(1)+\frac{2}{d-2}\left(\int_\e^1\frac{u^2}{y^2}y^{d-1}dy\right)^{\frac 12}\left(\int_\e^1|\pa_yu|^2y^{d-1}dy\right)^{\frac 12}\\
& \leq & C_{d}u^2(1)+\frac{1}{d-2}\left[\sigma\int_{\e}^1\frac{u^2}{y^{2}}y^{d-1}dy+\frac{1}{\sigma}\int_\e^1 |\pa_yu|^2y^{d-1}dy\right]
\eee
and hence letting $\e\to 0$:
$$
\left[1-\frac{\sigma}{d-2}\right]\int_{y\leq 1}\frac{u^2}{y^2}y^{d-1}dy\leq C_{d}u^2(1)+\frac{1}{\sigma(d-2)}\int_{y\leq 1} |\pa_yu|^2y^{d-1}dy$$
and the optimal choice $\sigma=\frac{d-2}{2}$ yields \fref{esthardythree}.\\
\noindent{\em Proof of (ii)}: For $0<q<\frac{d-2}{2}$ i.e. $2q+2<d$,  
\be
\label{calculdivergence}
\frac{1}{y^{d-1}}\pa_y\left(\frac{y^{d-1}}{y^{2q+1}}\right)=\frac{d-(2p+2)}{y^{2q+2}}
\ee 
and hence integrating by parts:
\bee
&&\int_{1}^R\frac{u^2}{y^{2q+2}}y^{d-1}dy=\frac{1}{d-(2q+2)}\int_{1}^Ru^2\pa_y\left(\frac{y^{d-1}}{y^{2q+1}}\right)dy\\
& = & \frac{1}{d-(2q+2)}\left[y^{d-(2q+2)}u^2\right]_{1}^R-\frac{2}{d-(2q+2)}\int_{1}^R\frac{u\pa_yu}{y^{2q+1}}y^{d-1}dy\\
& \leq& \frac{R^{d-(2q+2)}}{d-(2q+2)}u^2(R)+ \frac{2}{d-(2q+2)}\left(\int_{1}^{R}\frac{u^2}{y^{2q+2}}y^{d-1}dy\right)^{\frac 12}\left(\int_1^R\frac{|\pa_yu|^2}{y^{2q}}y^{d-1}dy\right)^{\frac 12}.
\eee
We let $R\to +\infty$ and hence $$\int_{1}^{+\infty}\frac{u^2}{y^{2q+2}}y^{d-1}dy\leq  \frac{2}{d-(2q+2)}\left(\int_{1}^{+\infty}\frac{u^2}{y^{2q+2}}y^{d-1}dy\right)^{\frac 12}\left(\int_1^{+\infty}\frac{|\pa_yu|^2}{y^{2q}}y^{d-1}dy\right)^{\frac 12}$$ which implies:
\be
\label{bebeibebevi}
\int_{y\geq 1}\frac{|\pa_yu|^2}{y^{2q}}y^{d-1}dy\geq  \left(\frac{d-(2q+2)}{2}\right)^2\int_{y\geq 1}\frac{u^2}{y^{2+2q}}y^{d-1}dy.
\ee We now estimate from $2q+2<d$ and \fref{bebeibebevi}:
\bee
u^2(R)&=&-2\int_R^{+\infty}\pa_yu udy\lesssim \int_{y\geq R}\frac{(y^{\frac{d-1-2q}{2}}\pa_yu)(y^{\frac{d-3-2q}{2}u)}}{y^{d-2-2q}}dy\\
& \leq &\frac{1}{R^{d-2-2q}}\left(\int_{y\geq 1}\frac{|\pa_yu|^2}{y^{2q}}y^{d-1}dy\right)^{\frac 12}\left(\int_{y\geq 1}\frac{|u|^2}{y^{2q+2}}y^{d-1}dy\right)^{\frac 12}\\
& \leq &\frac{1}{R^{d-2-2q}} \frac 2{d-(2q+2)}\int_{y\geq 1}\frac{|\pa_yu|^2}{y^{2q}}y^{d-1}dy
\eee
and \fref{esthardythreebis} is proved.\\
For $q> \frac{d-2}{2}$, we compute similarly from \fref{calculdivergence}:
\bee
&&\int_{1}^R\frac{u^2}{y^{2q+2}}y^{d-1}dy=-\frac{1}{2q+2-d}\int_{1}^Ru^2\pa_y\left(\frac{y^{d-1}}{y^{2q+1}}\right)dy\\
& = & \frac{-1}{2q+2-d}\left[y^{d-(2q+2)}u^2\right]_{1}^R+\frac{2}{2q+2-d}\int_{1}^R\frac{u\pa_yu}{y^{2q+1}}y^{d-1}dy\\
& \leq& C_{q,d}u^2(1)-\frac{1}{2q+2-d}\frac{u^2(R)}{R^{2q+2-d}}+\frac{2}{|2q+2-d|}\left(\int_{1}^R\frac{u^2}{y^{2q+2}}y^{d-1}dy\right)^{\frac 12}\left(\int_1^R \frac{|\pa_yu|^2}{y^{2q}}y^{d-1}dy\right)^{\frac 12}\\
& \leq & C_{q,d}u^2(1)-\frac{1}{2p+2-d}\frac{u^2(R)}{R^{2q+2-d}}+\frac{1}{|2q+2-d|}\left[\sigma\int_{1}^R\frac{u^2}{y^{2q+2}}y^{d-1}dy+\frac{1}{\sigma}\int_1^R\frac{|\pa_yu|^2}{y^{2q}}y^{d-1}dy\right]
\eee
and hence:
\bee
&&\left[1-\frac{\sigma}{|2q+2-d|}\right]\int_1^R\frac{u^2}{y^{2q+2}}y^{d-1}dy+\frac{1}{2q+2-d}\frac{u^2(R)}{R^{2q+2-d}}\\
& \leq & C_{q,d}u^2(1)+\frac{1}{\sigma(2q+2-d)}\int_1^R\frac{|\pa_yu|^2}{y^{2q}}y^{d-1}dy.
\eee
Passing to the limit $R\to +\infty$ and picking the optimal  $\sigma=\frac{2q+2-d}{2}$ yields \fref{esthardythreebis}.\\
\noindent{\em Proof of (iii)}: In the critical case $p=\frac{d-2}{2}$, we compute:
\bee
&&\int_{1}^R\frac{u^2}{y^{2q+2}(1+\log y)^2}y^{d-1}dy=\int_{1}^R\frac{u^2}{y(1+\log y)^2}dy=-\int_1^Ru^2\pa_y\left(\frac{1}{1+\log y}\right)\\
& = & -\left[\frac{u^2}{1+\log y}\right]_1^R+2\int_1^R\frac{u\pa_y u}{1+\log y}dy\leq  u^2(1)+2\left(\int_{1}^R\frac{u^2}{y(1+\log y)^2}dy\right)^{\frac 12}\left(\int_{1}^R|\pa_yu|^2ydy\right)^{\frac 12}\\
& \leq & u^2(1) +\frac 12\int_{1}^R\frac{u^2}{y^{2q+2}(1+\log y)^2}y^{d-1}dy+2\int_{1}^R\frac{|\pa_yu|^2}{y^{2q}}y^{d-1}dy
\eee
and letting $R\to +\infty$ yields \fref{esthardytwo}.\\
\noindent{\it Proof of (iv)}. We argue by induction on $k$ with the induction claim: \fref{hardywehghtgeenral} holds for all $\delta>0$. For $k=2$, we
integrate by parts and use Cauchy Schwarz to estimate:
\bee
\int \frac{|\nabla u|^2}{1+y^{2+\delta}}&=&-\int u\nabla \cdot\left[\frac{\nabla u}{1+y^{2+\delta}}\right]=-\int \frac{u\Delta u}{1+y^{2+\delta}}+\int u\nabla u\cdot\nabla \left(\frac{1}{1+y^{2+\delta}}\right)\\
& \leq & C\left[\int \frac{|\Delta u|^2}{1+y^{\delta}}+\int \frac{|u|^2}{1+y^{4+\delta}}\right]+\frac12\int \frac{|\nabla u|^2}{1+y^{2+\delta}}
\eee
and \fref{hardywehghtgeenral} is proved. Assume the claim for $k$ and prove it for $k+1$. For $1\leq j\leq k-1$ we have from the induction claim at the level $k$ applied to 
$\tilde\delta=\delta+2$:
\be
\label{cebcbbeiebve}
\int \frac{|D^ju|^2}{1+y^{\delta+2(k+1-j)}}=\int \frac{|D^ju|^2}{1+y^{\delta+2+2(k-j)}}\lesssim \int \frac{|D^ku|^2}{1+y^{2+\delta}}+\int \frac{u^2}{1+y^{\delta+2+2k}}
\ee
For $j=k$,
\bee
&&\int \frac{|D^ku|^2}{1+y^{2+\delta}}=\int D^{k-1}uD\left(\frac{D^ku}{1+y^{2+\delta}}\right)\\
& = & \int D^{k-1}uD^kuD\left(\frac{1}{1+y^{2+\delta}}\right)+O\left[\left(\int \frac{|D^{k-1}u|^2}{1+y^{4+\delta}}\right)^{\frac 12}\left(\int \frac{|D^{k+1}u|^2}{1+y^{\delta}}\right)^{\frac 12}\right]\\
& = & C\int \frac{|D^{k+1}u|^2}{1+y^{\delta}}+C \int \frac{|D^{k-1}u|^2}{1+y^{4+\delta}} +\frac 12 \int \frac{|D^ku|^2}{1+y^{2+\delta}}
\eee
and \fref{hardywehghtgeenral} is proved.
\end{proof}

We now state a refined fractional global Hardy bound:

\begin{lemma}[Fractional Hardy]
\label{fracthardy}
Let $u\in{\mathcal D}_{rad}$ and 
$$ 0<\nu<1 \ \ \mbox{and}\ \  \mu>0\ \ \mbox{with} \ \ \mu+\nu<\frac d2,$$ and a smooth radially symmetric function $f$ with 
\be
\label{estf}
|\pa_y^kf(y)|\lesssim \frac{1}{1+|y|^{\mu+k}}, \ \ k=0,1,
\ee
then
\be
\label{hardyfract}
\left\|\nabla^{\nu}\left(u f\right)\right\|_{L^2}\lesssim \|\nabla^{\mu+\nu}u\|_{L^2}.
\ee
\end{lemma}

\begin{proof} We recall the standard fractional Hardy inequality:
\be
\label{stnadradharyd}
\forall 0<s<\frac d2, \ \ \int\frac{|u|^2}{|x|^{2s}}\lesssim \|\nabla^{s}u\|^2_{L^2}.
\ee
From $0<\nu<1$, $$\left\|\nabla^{\nu}\left(u f\right)\right\|_{L^2}^2=\int\frac{|f(x)u(x)-f(y)u(y)|^2}{|x-y|^{d+2\nu}}dxdy.$$ We split the integral in various zones. First:
\bee
&&\int_{|x-y|\leq \frac{|x|}{2}}\frac{|f(x)u(x)-f(y)u(y)|^2}{|x-y|^{d+2\nu}}dxdy\\
&\lesssim &\int_{|x-y|\leq \frac{|x|}{2}}\frac{|f(x)|^2|u(x)-u(y)|^2}{|x-y|^{d+2\nu}}dxdy+\int_{|x-y|\leq\frac{|x|}{2}}\frac{|f(x)-f(y)|^2|u(y)|^2}{|x-y|^{d+2\nu}}dxdy.
\eee
The first term is the most delicate one:
\bee
&&\int_{|x-y|\leq \frac{|x|}{2}}\frac{|f(x)|^2|u(x)-u(y)|^2}{|x-y|^{d+2\nu}}dxdy\lesssim \int_{|x-y|\leq \frac{|x|}{2}}\frac{|u(x)-u (y)|^2}{|x|^{2\mu}|x-y|^{d+2\nu}}dxdy\\
& \lesssim & \int_x\int_{z}\frac{|u(x+z)-u(x)|^2}{|x|^{2\mu}|z|^{d+2\nu}}dxdz.
\eee
Let $v_z(x)=u(x+z)-u(x)$, then from \fref{stnadradharyd}, Fubini and Plancherel:
\bee
&&\int_x\int_{z}\frac{|u(x+z)-u(x)|^2}{|x|^{2\mu}|z|^{d+2\nu}}dxdz=\int_{z}\frac{dz}{|z|^{d+2\nu}}\int\frac{|v_z(x)|^2}{|x|^{2\mu}}dx\lesssim \int \frac{dz}{|z|^{d+2\nu}}\int|\nabla^\mu v_z(x)|^2dx\\
& \lesssim & \int\frac{dz}{|z|^{d+2\nu}}\int |\xi|^{2\mu}|\hat{v}_z(\xi)|^2d\xi=\int |\xi|^{2\mu}|\hat{u}(\xi)|^2d\xi\int\frac{|1-e^{-i\xi\cdot z}|^2}{|z|^{d+2\nu}}dz\\
& \lesssim & \int|\xi|^{2\mu+2\nu}|\hat{u}(\xi)|^2d\xi=\|\nabla^{\mu+\nu}u\|_{L^2}^2
\eee
where we used from $0<\nu<1$ and a simple homogeneity argument: $$\int\frac{|1-e^{-i\xi\cdot z}|^2}{|z|^{d+2\nu}}dz=c_d|\xi|^{2\nu}.$$
For the second term, we estimate using $|x-y|\leq \frac{|x|}{2}$: $$|f(x)-f(y)|\lesssim \int_{0}^1|x-y||f'(x+t(x-y))|dt\lesssim \int_0^1\frac{|x-y|dt}{1+|x+t(x-y)|^{\mu+1}}\lesssim \frac{|x-y|}{|x|^{\mu+1}}$$ and hence using $|x|\sim |y|$:
\bee
&&\int_{|x-y|\leq \frac{|x|}{2}}\frac{|f(x)-f(y)|^2|u(y)|^2}{|x-y|^{d+2\nu}}dxdy\lesssim \int_{|x-y|\leq\frac{|x|}{2}}\frac{|u(y)|^2}{|x|^{2\mu+2}|x-y|^{d+2\nu-2}}\\
&\lesssim &  \int \frac{|u(y)|^2}{|y|^{2\mu+2}}dy\int_{|x-y|\lesssim |y|}\frac{dx}{|x-y|^{d+2\nu-2}}\lesssim \int \frac{|u(y)|^2}{|y|^{2\mu+2\nu}}dy\lesssim  \|\nabla^{\nu+\mu}u\|_{L^2}^2
\eee
from \fref{stnadradharyd}. By symmetry, we estimate similarly $|x-y|\leq \frac{|y|}{2}.$  For $|x-y|\gtrsim \max\{|x|,|y|\}$, we estimate:
\bee
&&\int_{|x-y|\gtrsim |x|,|y|}\frac{|f(x)u(x)-f(y)u(y)|^2}{|x-y|^{d+2\nu}}dxdy\\
& \lesssim & \int_{|x-y|\gtrsim |x|,|y|}\frac{|u(x)|^2}{|x|^{2\mu}|x-y|^{d+2\nu}}dxdy+\int_{|x-y|\gtrsim |x|,|y|}\frac{|u(y)|^2}{|y|^{2\mu}|x-y|^{d+2\nu}}dxdy\\
& \lesssim & \int\frac{|u(x)|^2dx}{|x|^{2\mu}}\int_{|x-y|\gtrsim |x|}\frac{dy}{|x-y|^{d+2\nu}}+\int\frac{|u(y)|^2dy}{|y|^{2\mu}}\int_{|x-y|\gtrsim |y|}\frac{dx}{|x-y|^{d+2\nu}}\\
& \lesssim & \int\frac{|u(x)|^2dx}{|x|^{2\mu+2\nu}}+\int\frac{|u(y)|^2dy}{|y|^{2\mu+2\nu}}\lesssim  \|\nabla^{\nu+\mu}u\|_{L^2}^2
\eee
and \fref{hardyfract} is proved.
\end{proof}

%%%%%%%%%%%%%%%%%%%%%%%%%%%%%%%%%%%%%%%%%%%%%%%%%%%%%%%%%%%%%%%%%%%%%%%%%%%%%%%%%%%%%%%%%%%%%%%%%%%%%%%%%%%

\section{Linear weighted coercitivity bounds}
\label{coercseciot}
%%%%%%%%%%%%%%%%%%%%%%%%%%%%%%%%%%%%%%%%%%%%%%%%%%%%%%%%%%%%%%%%%%%%%%%%%%%%%%%%%%%%%%%%%%%%%%%%%%%%%%%%%%%

Given $M\geq 1$, we let $\Xi_{M,\pm}$ be given by \fref{defdirectionplus}. We claim suitable weighted coercivity bounds for the linearized operator $\Lt$
$$
\Lt^*=\left(\begin{array}{ll} 0&L_-\\-L_+&0\end{array}\right)
$$
with
$$L_+=-\Delta-pQ^{p-1}, \ \ L_-=-\Delta-Q^{p-1}.$$  
We will use in an essential way the factorization of $L_\pm=A^*_\pm A_\pm$, 
 $$A_\pm u=-\pa_yu+V_\pm u, \ \ A_\pm^*u=\frac{1}{y^{d-1}}\pa_y(y^{d-1}u)+V_\pm u,$$
with
$$
V_+=\pa_y(\log (\Lambda Q)), \ \ V_-=\pa_y(\log Q),
$$ 
and deal first with $A_\pm$ and $A^*_\pm$ separately. 

\subsection{Coercivity of $A^*_\pm$} We start with the weighted coercivity of $A^*_\pm$

\begin{lemma}[Weighted coercitivity for $A^*_\pm$]
 \label{boundweightbis}
 Let $k\in \Bbb R^+ $, then there exists $c_k>0$ such that for all $u\in\mathcal D_{\rm rad}$:
 \be
 \label{coerciveboundbis}
 \int\frac{|A_\pm^*u|^2}{1+y^{4k}}\geq c_k \left[\int\frac{|u|^2}{y^2(1+y^{4k})}+\int\frac{|\pa_yu|^2}{1+y^{4k}}\right].
  \ee
 \end{lemma}

\begin{proof}
 {\bf step 1} Subcoercive bound for $A^*_+$. Let $u\in \mathcal D_{\rm rad}$, we claim the following lower bound:
 \bea
 \label{subcoercabis}
\nonumber   \int\frac{|A_+^*u|^2}{1+y^{4k}}&\geq& c\left[\int \frac{u^2}{y^2(1+y^{4k})}+\int\frac{|\pa_yu|^2}{1+y^{4k}}\right]\\
& -&  \frac{1}{c}\left[u^2(1)+\int \frac{u^2}{1+y^{4k+4}}\right]
 \eea
 for some universal constant $c=c_{p,d,k}>0$. Indeed, recall the definition of $A^*_+$: 
 $$
 A^*_+=\pa_y+\Vt_+, \ \ \Vt_+=\frac{d-1}{y}+V_+
 $$
 where $V_+$ satisfies \fref{behavoir}. Near the origin,
 \bea
 \label{jobrpjhpjhtjohtp}
 \nonumber  &&\int_{y\leq 1}\frac{|A_+^*u|^2}{1+y^{4k}}\gtrsim \int_{y\leq 1}|\pa_yu+\Vt_+u|^2=\int_{y\leq 1}\left[|\pa_yu|^2+\Vt_+^2u^2+2\Vt_+u\pa_y u\right]\\
 \nonumber& = & \int_{y\leq 1}|\pa_yu|^2+\int_{y\leq 1}u^2\left[\Vt_+^2-\frac{1}{y^{d-1}}\pa_y(y^{d-1}\Vt_+)\right]\\
 \nonumber& = & \int_{y\leq 1}|\pa_yu|^2+\int_{y\leq 1}\frac{u^2}{y^2}[(d-1)^2-(d-1)(d-2)]+O\left(\int_{y\leq 1}\frac{u^2}{y}\right)\\
 & \gtrsim & \int_{y\leq 1}|\pa_yu|^2+\int_{y\leq 1}\frac{u^2}{y^2}+O\left(\int_{y\leq 1}u^2\right).
 \eea
 Away from the origin, we estimate from \fref{behavoir}:
 \bee
&& \int_{y\geq 1}\frac{(\pa_yu+\Vt_+u)^2}{y^{4k}}= \int_{y\geq 1}\frac{1}{y^{4k}}\left[\pa_yu+\frac{d-1-\gamma}{y}u+O\left(\frac{u}{y^2}\right)\right]^2\\
& \gtrsim & \int_{y\geq 1}\frac{1}{y^{4k}}\left[\pa_yu+\frac{d-1-\gamma}{y}u\right]^2+O\left(\int_{y\geq 1}\frac{u^2}{y^{4k+4}}\right)\\
& = & \int_{y\geq 1}\frac{1}{y^{4k+2(d-1-\gamma)}}\left|\pa_y(y^{d-1-\gamma}u)\right|^2+O\left(\int_{y\geq 1}\frac{u^2}{y^{4k+4}}\right).
 \eee
Let then $v=y^{d-1-\gamma}u$, $p=2k+(d-1-\gamma)$, then from \fref{boundgamma}: $$2p-(d-2)=4k+2(d-1-\gamma)-(d-2)=4k+d-2\gamma>0,$$ and we may therefore apply Lemma \ref{lemmahardy} in the non-degenerate case to conclude:
\bee
\int_{y\geq 1}\frac{1}{y^{4k+2(d-1-\gamma)}}\left|\pa_y(y^{d-1-\gamma}u)\right|^2&=&\int_{y\geq 1}\frac{1}{y^{2p}}\left|\pa_yv\right|^2\gtrsim \int_{y\geq 1}\frac{v^2}{y^{2p+2}}-cv^2(1)\\
& \gtrsim & \int_{y\geq 1}\frac{u^2}{y^{4k+2}}-cu^2(1).
\eee
This gives the lower bound: $$\int \frac{|A_+^*u|^2}{1+y^{4k}}\geq c\int\frac{u^2}{y^2(1+y^{4k})}-\frac{1}{c}\left[\int\frac{u^2}{1+y^{4k+4}}+u^2(1)\right].$$ On the other hand, there holds the trivial bound from \fref{behavoir}: $$\int \frac{|\pa_yu|^2}{1+y^{4k}}- \int\frac{u^2}{y^2(1+y^{4k})}\lesssim \int \frac{|A_+^*u|^2}{1+y^{4k}}$$ 
and \fref{subcoercabis} follows.\\

 \noindent{\bf step 2} Subcoercive  bound for $A^*_-$. We claim the following lower bound:
 \bea
 \label{subcoercabismoins}
\nonumber   \int\frac{|A_-^*u|^2}{1+y^{4k}}&\geq& c\left[\int \frac{u^2}{y^2(1+y^{4k})}+\int\frac{|\pa_yu|^2}{1+y^{4k}}\right]\\
& -&  \frac{1}{c}\left[u^2(1)+\int \frac{u^2}{1+y^{4k+4}}\right]
 \eea
 for some universal constant $c=c_{p,d,k}>0$. Indeed, recall the definition of $A^*_-$: 
 $$
 A^*_-=\pa_y+\Vt_-, \ \ \Vt_+=\frac{d-1}{y}+V_-
 $$
 where $V_-$ satisfies \fref{behavoirvminus}. Near the origin, we estimate verbatim as in the proof of \fref{jobrpjhpjhtjohtp}:
 $$\int_{y\leq 1}\frac{|A_+^-u|^2}{1+y^{4k}}\gtrsim \int_{y\leq 1}|\pa_yu|^2+\int_{y\leq 1}\frac{u^2}{y^2}+O\left(\int_{y\leq 1}u^2\right).
 $$
 Away from the origin, we estimate from \fref{behavoirvminus}:
 \bee
&& \int_{y\geq 1}\frac{(\pa_yu+\Vt_-u)^2}{y^{4k}}= \int_{y\geq 1}\frac{1}{y^{4k}}\left[\pa_yu+\frac{d-1-\frac2{p-1}}{y}u+O\left(\frac{u}{y^2}\right)\right]^2\\
& \gtrsim & \int_{y\geq 1}\frac{1}{y^{4k}}\left[\pa_yu+\frac{d-1-\frac{2}{p-1}}{y}u\right]^2+O\left(\int_{y\geq 1}\frac{u^2}{y^{4k+4}}\right)\\
& = & \int_{y\geq 1}\frac{1}{y^{4k+2(d-1-\frac{2}{p-1})}}\left|\pa_y(y^{d-1-\frac{2}{p-1}}u)\right|^2+O\left(\int_{y\geq 1}\frac{u^2}{y^{4k+4}}\right).
 \eee
Let $v=y^{d-1-\frac2{p-1}}u$, $q=2k+(d-1-\frac2{p-1})$, then from \fref{boundgamma}: $$2q-(d-2)=4k+2(d-1-\frac2{p-1})-(d-2)>4k+d-2\gamma>0,$$ and we may therefore apply Lemma \ref{lemmahardy} in the non-degenerate case to conclude:
\bee
\int_{y\geq 1}\frac{1}{y^{4k+2(d-1-\gamma)}}\left|\pa_y(y^{d-1-\gamma}u)\right|^2&=&\int_{y\geq 1}\frac{1}{y^{2p}}\left|\pa_yv\right|^2\gtrsim \int_{y\geq 1}\frac{v^2}{y^{2p+2}}-cv^2(1)\\
& \gtrsim & \int_{y\geq 1}\frac{u^2}{y^{4k+2}}-cu^2(1).
\eee
This yields the lower bound: $$\int \frac{|A_-^*u|^2}{1+y^{4k}}\geq c\int\frac{u^2}{y^2(1+y^{4k})}-\frac{1}{c}\left[\int\frac{u^2}{1+y^{4k+4}}+u^2(1)\right].$$ On the other hand, there holds the trivial bound from \fref{behavoirvminus}: $$\int \frac{|\pa_yu|^2}{1+y^{4k}}- \int\frac{u^2}{y^2(1+y^{4k})}\lesssim \int \frac{|A_-^*u|^2}{1+y^{4k}}$$ 
and \fref{subcoercabismoins} follows.\\

\noindent{\bf step 3} Coercivity. We argue by contradiction. Let $M=M(j)>0$ fixed and consider a normalized sequence $u_n\in\matchal D_{\rm rad}$ with
\be
\label{normlaizationbis}
\int\frac{|u_n|^2}{y^2(1+y^{4k})}+\int \frac{|\pa_yu_n|^2}{1+y^{4k}}=1,\ \ \int\frac{|A_\pm^*u_n|^2}{1+y^{4k}}\leq \frac1n.
  \ee
  This implies from the subcoercivity estimates \fref{subcoercabis}, \fref{subcoercabismoins}:
  \be
  \label{noenoenoebis}
  u_n^2(1)+\int \frac{u_n^2}{1+y^{4k+4}}\gtrsim 1.
  \ee
From \fref{normlaizationbis}, the sequence $u_n$ is bounded in $H^1(\e<y< R)$ for all $R,\e>0$. Hence from a standard diagonal extraction argument, there exists $u\in \cap_{R,\e>0}H^1(\e<y<R)$  such that up to a subsequence, 
\be
\label{beibebeiebis}
\forall \e,R>0, \ \ u_n\rightharpoonup u\ \ \mbox{in}\ \ H^1(\e<y<R)
\ee and from the local compactness of one dimensional Sobolev embeddings: $$u_n\to u\ \  \mbox{in}\ \ L^2(\e<y< R), \ \ u_n(1)\to u(1).$$ This implies from  \fref{noenoenoebis}, \fref{normlaizationbis} and the lower semi continuity of norms:
\be
\label{nondgeegubis}
 u^2(1)+\int \frac{u^2}{1+y^{4k+4}}\gtrsim 1, \ \ \int\frac{|u|^2}{y^2(1+y^{4k})}\lesssim 1.
 \ee 
 and thus in particular $u\neq 0$. On the other hand, from \fref{normlaizationbis}, \fref{beibebeiebis}: $$A^*_\pm u=0\ \ \mbox{in}\ \ \Bbb R^*_+$$ and thus from \fref{efinitei}, \fref{efiniteimoins}: $$u=\left\{\begin{array}{ll} \frac{c}{y^{d-1}{\Lambda Q}}\ \ \mbox{for}\ \ A^*_+\\ \frac{c}{y^{d-1}{Q}}\ \ \mbox{for}\ \ A^*_-\end{array}\right.$$ The constant $c$ is non zero from $u\neq 0$, but then since $Q,\Lambda Q$ are smooth at the origin: $$\int_{y\leq 1}\frac{u^2}{y^2}\gtrsim \int_{y\leq 1}\frac{y^{d-1}}{y^{2(d-2)+2}}dy=\int_{y\leq 1}\frac{dy}{y^{d-1}}=+\infty
 $$
 which contradicts the a priori regularity \fref{nondgeegubis} of $u$.
 \end{proof}
 
\subsection{Weighted coercivity of $\Lt $} We now turn to the coercivity of $\Lt$ which we consider in the generic case $\delta_{k_\pm}\neq 0$,
with $k_\pm$ and $\delta_{k_\pm}$ defined in \eqref{defkzero}, \eqref{defkone}.
 We let $\Xi_{M,\pm}$ be given by \fref{defdirectionplus}, \fref{defdirectionminus}.
 
 \begin{lemma}[Weighted coercitivity for $\Lt$, case $\delta_{k_\pm}\neq 0$]
 \label{boundweight}
 Assume $\delta_{k_\pm}\neq 0$.  Let $k\in \Bbb N$. Then:\\
 \noindent  {\em (i) Case $k$ small}: assume $k_+\geq 2$ and let $1\leq k\leq k_+-1$, then there exists $c_k>0$ such that for all $u\in \matchal D_{\rm rad}$, there holds:
  \be
   \label{coercafsd}
   \int \frac{|\Lt u|^2}{1+y^{4k-2}}\geq c_k\left[\int\frac{|\Delta u|^2}{1+y^{4k-2}}+\int\frac{|\pa_yu|^2}{1+y^{4k}}+\frac{|u|^2}{y^2(1+y^{4k})}\right].
   \ee
 \noindent{\em (ii) Case $k$ intermediate}: let $k_+\leq k\leq k_--1$, let $M\geq M(k)$ large enough, then there exists $c_{M,k}>0$ such that for all $u\in \matchal D_{\rm rad}$ satisfying the orthogonality $$(u,\Xi_{M,+})=0,$$ there holds:
 \be
 \label{coercivebounddeux}
   \int \frac{|\Lt u|^2}{1+y^{4k-2}}\geq c_{M,k}\left[\int\frac{|\Delta u|^2}{1+y^{4k-2}}+\int\frac{|\pa_yu|^2}{1+y^{4k}}+\frac{|u|^2}{y^2(1+y^{4k})}\right].
  \ee
   \noindent{\em (ii) Case $k$ large}: let $k\geq k_-$, let $M\geq M(k)$ large enough, then there exists $c_{M,k}>0$ such that for all $u\in \matchal D_{\rm rad}$ satisfying the orthogonality $$(u,\Xi_{M,+})=(u,\Xi_{M,-})=0,$$ there holds:
 \be
 \label{coerciveboundtrois}
   \int \frac{|\Lt u|^2}{1+y^{4k-2}}\geq c_{M,k}\left[\int\frac{|\Delta u|^2}{1+y^{4k-2}}+\int\frac{|\pa_yu|^2}{1+y^{4k}}+\frac{|u|^2}{y^2(1+y^{4k})}\right].
  \ee

 \end{lemma}

 \begin{proof}[Proof of Lemma \ref{boundweight}] 
 \noindent  {\bf step 1} Subcoercive bound for $A_\pm$. Let $k\geq 0$ and $u\in \matchal D_{\rm rad}$. We claim the following lower bound:
 \bea
 \label{subcoercajojoj}
 \nonumber \int\frac{|A_\pm u|^2}{1+y^{4k}}&\geq &c\left[\int \frac{|\pa_yu|^2}{1+y^{4k}}+\int \frac{u^2}{y^2(1+y^{4k})}\right]\\
 &-& \frac{1}{c}\left[u^2(1)+\int \frac{u^2}{1+y^{4k+4}}\right]
 \eea
 for some universal constant $c=c_{p,d,k}>0$. Recall the definition of $A_\pm$: 
 $$ A_\pm=-\pa_y+V_\pm
$$ with $V_\pm$ satisfying \fref{behavoir}, \fref{behavoirvminus}. We estimate near the origin from \fref{esthardythree}:
 \bee
 \int_{y\leq 1}\frac{|A_\pm u|^2}{1+y^{4k}}\gtrsim \int_{y\leq 1}\left[c|\pa_yu|^2-\frac{1}{c}u^2\right]\gtrsim c \int_{y\leq 1}\left[|\pa_yu|^2+\frac{u^2}{y^2}\right]-\frac1c\left[\int_{y\leq 1}u^2+u^2(1)\right].
 \eee
Away from the origin, we estimate from \fref{behavoir}:
 \bee
 && \int_{y\leq 1}\frac{|A_+u|^2}{1+y^{4k}}\gtrsim\int_{y\geq 1}\frac{1}{y^{4k}}\left[\pa_yu+\frac{\gamma}{y}u+O\left(\frac{u}{y^2}\right)\right]^2\\
 &\gtrsim & \int_{y\geq 1}\frac{1}{y^{4k}}\left[\pa_yu+\frac{\gamma}{y}u\right]^2+O\left(\int_{y\geq 1}\frac{u^2}{y^{4k+4}}\right) 
 \eee
  We let $v=y^{\gamma}u$, $2q=4k+2\gamma$. We observe that $$2q-(d-2)=4k+2\gamma-(d-2)=4(k+1-k_+-\dk)\neq 0$$ from $\dkp\neq 0$ and $k\in \Bbb N$, and we may therefore apply Lemma \ref{lemmahardy} in the non-generate case to conclude:
\bee
 \int_{y\geq 1}\frac{|\pa_y(y^\gamma u)|^2}{y^{4k+2\gamma}}&=& \int_{y\geq 1}\frac{|\pa_yv|^2}{y^{2q}}\geq c \int_{y\geq 1}\frac{v^2}{y^{2q+2}}-\frac{1}{c}v^2(1)\\
 & \geq & c\int_{y\geq 1}\frac{u^2}{y^2(1+y^{4k+2})}-\frac{1}{c}u^2(1).
  \eee
  The collection of the above bounds yields the lower bound:
  $$\int\frac{|A_+u|^2}{1+y^{4k}}\geq c\int \frac{u^2}{y^2(1+y^{4k})}- \frac{1}{c}\left[u^2(1)+\int \frac{u^2}{1+y^{4k+4}}\right]$$
 which together with the trivial estimate $$ \int \frac{|\pa_yu|^2}{1+y^{4k}}-\int \frac{u^2}{y^2(1+y^{4k})}\lesssim
 \int\frac{|A_+u|^2}{1+y^{4k}}$$ implies \fref{subcoercajojoj} for $A_+$.\\
Similarily, we estimate away from the origin from \fref{behavoirvminus}:
 \bee
 && \int_{y\geq 1}\frac{|A_-u|^2}{1+y^{4k}}\gtrsim\int_{y\geq 1}\frac{1}{y^{4k}}\left[\pa_yu+\frac{2}{(p-1)}\frac uy+O\left(\frac{u}{y^2}\right)\right]^2\\
 &\gtrsim & \int_{y\geq 1}\frac{1}{y^{4k}}\left[\pa_yu+\frac{2}{p-1}\frac uy\right]^2+O\left(\int_{y\geq 1}\frac{u^2}{y^{4k+4}}\right) 
 \eee
  We let $v=y^{\frac 2{p-1}}u$, $2q=4k+\frac{4}{p-1}$. We observe that $$2q-(d-2)=4k+\frac{4}{p-1}-(d-2)=4(k+1-k_--\dkm)\neq 0$$ from $\dkm\neq 0$ and $k\in \Bbb N$, and we therefore apply Lemma \ref{lemmahardy} in the non generate case to conclude:
\bee
 \int_{y\geq 1}\frac{|\pa_y(y^{\frac{2}{p-1}} u)|^2}{y^{4k+\frac{4}{p-1}}}&=& \int_{y\geq 1}\frac{|\pa_yv|^2}{y^{2q}}\geq c \int_{y\geq 1}\frac{v^2}{y^{2q+2}}-\frac{1}{c}v^2(1)\\
 & \geq & c\int_{y\geq 1}\frac{u^2}{y^2(1+y^{4k+2})}-\frac{1}{c}u^2(1).
  \eee
  The collection of above bounds yields the lower bound:
  $$\int\frac{|A_-u|^2}{1+y^{4k}}\geq c\int \frac{u^2}{y^2(1+y^{4k})}- \frac{1}{c}\left[u^2(1)+\int \frac{u^2}{1+y^{4k+4}}\right]$$
 which together with the trivial estimate $$\int \frac{|\pa_yu|^2}{1+y^{4k}}-\int \frac{u^2}{y^2(1+y^{4k})}\lesssim \int\frac{|A_-u|^2}{1+y^{4k}}$$ implies \fref{subcoercajojoj} for $A_-$.\\
 
\noindent {\bf step 2} Coercivity. We argue by contradiction and let a normalized sequence $u_n\in\mathcal \matchal D_{\rm rad}$ be such that
\be
\label{normlaizationtqoone}
\int \frac{|\pa_yu_n|^2}{1+y^{4k}}+\int \frac{|u_n|^2}{y^2(1+y^{4k})}=1,\ \ 
\int\frac{|\Lt u|^2}{1+y^{4k-2}}\leq \frac1n
  \ee
  and 
  \be
  \label{orthogohg}
  \left|\begin{array}{ll} (u_n,\Xi_{M,+})=0\ \ \mbox{for}\ \ \max\{k_+,1\}\leq k\leq k_--1\\(u_n,\Xi_{M,+})=(u_n,\Xi_{M,-})=0\ \ \mbox{for}\ \ k\geq k_-.\end{array}\right.
  \ee
From Lemma \fref{coerciveboundbis} 
\bee
\int\frac{|\Lt u_n|^2}{1+y^{4k-2}}& =& \int\frac{|A^*_-A_-\Im u_n|^2+|A^*_+A_+\Re u_n|^2}{1+y^{4k-2}}\gtrsim  \int\frac{|A_-\Im u_n|^2+|A_+\Re u_n|^2}{1+y^{4k}}
\eee
and hence the subcoercivity estimate \fref{subcoercajojoj} and \fref{normlaizationtqoone} imply:
  \be
  \label{noenoenoennbr}
  |u_n|^2(1)+\int \frac{|u_n|^2}{1+y^{4k+4}}\gtrsim 1.
  \ee
From \fref{normlaizationtqoone}, the sequence $u_n$ is bounded in $H^1(\e<y< R)$ for all $R,\e>0$. Hence from a standard diagonal extraction argument, there exists $u\in \cap_{R,\e>0}H^1(\e<y< R)$  such that up to a subsequence, 
\be
\label{beibebeiebrrb}
\forall R>0, \ \ u_n\rightharpoonup u\ \ \mbox{in}\ \ H^1(\e<y<R)
\ee and from the local compactness of Sobolev embeddings $$u_n\to u\ \  \mbox{in}\ \ L^2(\e<y<R), \ \ u_n(1)\to u(1).$$ This implies from \fref{noenoenoennbr}, \fref{normlaizationtqoone}:
\be
\label{nondgeegupouet}
 |u|^2(1)+\int \frac{|u|^2}{1+y^{4k+4}}\gtrsim 1,\ \ \int \frac{|u|^2}{y^2(1+y^{4k})}\lesssim 1.
 \ee 
The compact support and regularity of $\Xi_{M\pm}$ allows us to pass to the limit in \fref{orthogohg} and conclude:
\be
\label{prhtnveineo}
 \left|\begin{array}{ll} (u,\Xi_{M,+})=0\ \ \mbox{for}\ \ \max\{k_+,1\}\leq k\leq k_--1\\(u,\Xi_{M,+})=(u,\Xi_{M,-})=0\ \ \mbox{for}\ \ k\geq k_-.\end{array}\right.
 \ee
On the other hand, from \fref{normlaizationtqoone}, \fref{beibebeiebrrb}: $$\Lt u=0\ \ \mbox{on}\ \ \Bbb R^*_+$$ and hence from \fref{kernelh}, \fref{behviorgmmaorigin} and \fref{kernelhmoins}, \fref{behviorgmmaoriginmoins} and the a priori regularity at the origin \fref{nondgeegupouet}: 
 \be
 \label{nceoeonecnoene}
 u=c_+\left|\begin{array}{ll} \Lambda Q\\0\end{array}\right.+c_-\left|\begin{array}{ll} 0\\ Q\end{array}\right.=c_+\Phi_{0,+}+c_-\Phi_{0,-}.
 \ee 
 We now distinguish cases.\\
 \noindent\underline{case $1\leq k\leq k_+-1$.} In this case: 
 $$\int_{y\geq 1}\frac{|\Lambda Q|^2}{y^2(1+y^{4k})}\gtrsim \int_{y\geq 1}\frac{y^{d-1}}{y^2(1+y^{4k})y^{2\gamma}}dy=+\infty$$ from $$1+2\gamma+4k+2-d=1+4(k+1)-4(k_++\dk)\leq 1-4\dk<1.$$ Similarily: $$\int_{y\geq 1}\frac{|Q|^2}{y^2(1+y^{4k})}\gtrsim \int_{y\geq 1}\frac{y^{d-1}}{y^2(1+y^{4k})y^{\frac{4}{p-1}}}dy\gtrsim \int_{y\geq 1}\frac{y^{d-1}}{y^2(1+y^{4k})y^{2\gamma}}dy=+\infty.$$ We conclude from \fref{nceoeonecnoene} and the established regularity \fref{nondgeegupouet} that $u\equiv 0$ which contradicts the non degeneracy \fref{nondgeegupouet}.\\
\noindent\underline{case $\max\{k_+,1\}\leq k\leq k_--1$.} In this case: 
$$\int_{y\geq 1}\frac{|Q|^2}{y^2(1+y^{4k})}\gtrsim \int_{y\geq 1}\frac{y^{d-1}}{y^2(1+y^{4k})y^{\frac{4}{p-1}}}dy=+\infty$$ from 
  $$1+\frac{4}{p-1}+4k+2-d=1+4(k+1)-4(k_-+\dkm)\leq 1-4\dkm<1.$$ Hence from \fref{nceoeonecnoene}, \fref{nondgeegupouet}, $c_-=0$. But then the orthogonality condition \fref{prhtnveineo} and the non degeneracy \fref{nondegeenr} imply $c_+=0$, hence $u\equiv 0$ which contradicts the non degeneracy \fref{nondgeegupouet}.\\
 \noindent\underline{case $k\geq k_-$.} In this case, \fref{nceoeonecnoene}, the orthogonality condition \fref{prhtnveineo} and the relations \fref{nondegeenr}, \fref{nondegeenrbis}, \fref{nondegeenrtrois}, \fref{nondegeenrquatre} imply $c_+=c_-=0$. Hence $u\equiv 0$ which contradicts the non degeneracy \fref{nondgeegupouet}.
 \end{proof}

%%%%%%%%%%%%%%%%%%%%%%%%%%%%%%%%%%%%%%%%%%%%%%%%%%%%%%%%%%%
\subsection{Coercivity of $\Lt^k$}
%%%%%%%%%%%%%%%%%%%%%%%%%%%%%%%%%%%%%%%%%%%%%%%%%%%%%%%%%%%

We are now position to prove the coercivity of $\Lt^k$ under suitable orthogonality conditions. We recall from \fref{estkwplus} that $k_+\geq 1$.

\begin{lemma}[Coercivity of $\Lt^k$, non degenerate case]
\label{propcorc}
Assume $\delta_{k_\pm}\neq 0$.\\
{\em (i) Case $k$ small}: let $0\leq k\leq k_+-1$, then there exists $\delta_k>0$ such that for all $u\in \mathcal D_{\rm rad}$, there holds:
\be
\label{coerciviteuk}
(J\Lt\Lt^ku ,\Lt^ku)\geq c_k\sum_{n=0}^{2k+1}\int\frac{|D^n u|^2}{1+y^{4k+2-2n}}.
\ee
{\em (ii) Case $k$ intermediate}: let $k=k_++j_+\leq k_--1$, $j_+\in \Bbb N$, let $M=M(j_+)$ large enough, there there exists $c_{k,M}>0$ such that for all $u\in \mathcal D_{\rm rad}$ satisfying the orthogonality conditions:
\be
\label{orthoplus}
 (u,(\Lt^*)^{n}\Xi_{M,+})=0, \ \ 0\leq n\leq j_+,
\ee
 there holds:
 \be
 \label{coerciviityley}
(J\Lt\Lt^ku ,\Lt^ku)\geq c_{k,M}\sum_{n=0}^{2k+1}\int\frac{|D^nu|^2}{1+y^{4k+2-2n}}.
 \ee
 {\em (iii) Case $k$ large}: let $k=k_++j_+=k_-+j_-$, $(j_+, j_-)\in \Bbb N^2$, let $M=M(j_+)$ large enough, there there exists $c_{k,M}>0$ such that for all $u\in\mathcal D_{\rm rad}$ satisfying the orthogonality conditions:
\be
\label{orthoappendixminus}
 \left\{\begin{array}{ll} (u,(\Lt^*)^{n}\Xi_{M,+})=0, \ \ 0\leq n\leq j_+\\ (u,(\Lt^*)^{n}\Xi_{M,-})=0, \ \ 0\leq n\leq j_-
 \end{array}\right.
\ee
 there holds:
 \be
 \label{coerciviityleybis}
(J\Lt\Lt^ku ,\Lt^ku)\geq c_{k,M}\sum_{n=0}^{2k+1}\int\frac{|D^nu|^2}{1+y^{4k+2-2n}}.
 \ee
\end{lemma}

\begin{proof}[Proof of Lemma \ref{propcorc}] {\bf step 1} Hardy bound. We first claim: $\forall \delta\geq 0$, 
\bea
\label{cneocnenceone}
\nonumber &&\sum_{n=0}^{2k+1}\int\frac{|D^n\Lt u|^2}{1+y^{4k+2-2n+\delta}}+\int \frac{|u|^2}{1+y^{4k+6+\delta}}\\
&\gtrsim_{\delta,k}&  \sum_{n=0}^{2k+3}\int\frac{|D^nu|^2}{1+y^{4(k+1)+2-2n+\delta}}.
\eea
We argue by induction on k.\\
\noindent\underline{$k=0$}: We infer from the definition of $\Lt$ and the decay $$|D^jW_\pm|\lesssim \frac{1}{1+y^{2+j}}, \ \ j\geq 0$$ the bound:
$$\int \frac{|\Delta u|^2}{1+y^{2+\delta}}\lesssim \int \frac{|\Lt u|^2}{1+y^{2+\delta}}+\int \frac{|u|^2}{1+y^{6+\delta}}.$$ Hence from \fref{hardywehghtgeenral}:
$$\int \frac{|\nabla u|^2}{1+y^{4+\delta}}\lesssim \int \frac{|\Delta u|^2}{1+y^{2+\delta}}+\int \frac{|u|^2}{1+y^{6+\delta}}\lesssim \int \frac{|\Lt u|^2}{1+y^{2+\delta}}+\int \frac{|u|^2}{1+y^{6+\delta}}.$$ This implies:
\bee
\int \frac{|\nabla\Delta u|^2}{1+y^\delta}&\lesssim& \int \frac{|\nabla(-\Delta u-W_{\pm} u)|^2}{1+y^\delta}+ \int \frac{|\Lt u|^2}{1+y^{2+\delta}}+\int \frac{|u|^2}{1+y^{6+\delta}}\\
&\lesssim&  \int \frac{|D\Lt u|^2}{1+y^\delta}+\int \frac{|\Lt u|^2}{1+y^{2+\delta}}+\int \frac{|u|^2}{1+y^{6+\delta}}
\eee and \fref{cneocnenceone} is proved for $k=0$.\\
\noindent\underline{$(\delta+4,k)\to (\delta,k+1)$}: From the induction claim for $(k,\delta+4)$:
\bee
\sum_{n=0}^{2k+1}\int\frac{|D^n\Lt u|^2}{1+y^{4(k+1)+2-2n+\delta}}+\int \frac{|u|^2}{1+y^{4(k+1)+6+\delta}}\gtrsim_{\delta,k} \sum_{n=0}^{2k+3}\int\frac{|D^nu|^2}{1+y^{4(k+2)+2-2n+\delta}}.
\eee
We now estimate from Leibniz:
\bee
&&\int \frac{|D^{2k+4}u|^2}{1+y^{2+\delta}}\lesssim \int \frac{|D^{2k+2}\Lt u|^2}{1+y^{2+\delta}}+\sum_{n=0}^{2k+2}\int\frac{|D^nu|^2}{1+y^{4(k+2)+2-2n+\delta}}\\
&&\int \frac{|D^{2k+5}u|^2}{1+y^{\delta}}\lesssim \int \frac{|D^{2k+3}\Lt u|^2}{1+y^{\delta}}+\sum_{n=0}^{2k+3}\int\frac{|D^nu|^2}{1+y^{4(k+2)+2-2n+\delta}}
\eee
and the conclusion follows.\\

\noindent{\bf step 2} Conclusion. We now prove the claim by induction on $k$.\\
\noindent\underline{\em Initialization $k=0,1$.} For $k=0$, we recall from \fref{positivityh}: $$L_->L_+>0\ \ \mbox{on}\ \ \dot{H}^1$$ and hence from the standard Hardy inequality: 
\be
\label{neionionov}
(J\Lt u,u)=(L_+\Re u,\Re u)+(L_-\Im u,\Im u)\gtrsim \int |\pa_y u|^2+\int \frac{|u|^2}{y^2}.
\ee
Assume that $k_+\geq 2$ and let us prove \fref{coerciviteuk} for $k=1$.  We estimate from \fref{neionionov} and Lemma \ref{boundweight}:
$$
(J\Lt \Lt u,\Lt u)\gtrsim  \int |\pa_y\Lt u|^2+\int \frac{|\Lt u|^2}{1+y^2}\gtrsim \int |\pa_y\Lt u|^2+\int \frac{|\Delta u|^2}{1+y^2}+\int\frac{|\pa_yu|^2}{1+y^{4}}+\int\frac{|u|^2}{1+y^{6}}
$$ 
and hence using the expression for $\Lt$:
$$(J\Lt \Lt u,\Lt u)\gtrsim \int |D^3u|^2+\int \frac{|D^2u|^2}{1+y^2}+\int\frac{|Du|^2}{1+y^{4}}+\frac{|u|^2}{1+y^{6}}.$$

\noindent\underline{\em Induction $k\to k+1\leq k_+-1$.} We assume the claim for $k\geq 0$ and prove it for $k+1\leq k_+-1$. Let $v=\Lt u$, then by induction: $$(J\Lt \Lt^{k+1}u,\Lt^{k+1}u)=(J\Lt \Lt^k v,\Lt^kv)\gtrsim \sum_{n=0}^{2k+1}\int\frac{|D^nv|^2}{1+y^{4k+2-2n}}.$$ Now from Lemma \ref{boundweight}, case $k+1\leq k_+-1$, there holds: 
$$\int\frac{|v|^2}{1+y^{4k+2}}=\int\frac{|\Lt u|^2}{1+y^{4k+2}}\gtrsim \int \frac{|u|^2}{1+y^{4k+6}}.$$ and hence the expected lower bound follows from  \fref{cneocnenceone} with $\delta=0$: $$(J\Lt \Lt^{k+1}u,\Lt^{k+1}u)\gtrsim \sum_{p=0}^{2k+3}\int\frac{|D^pu|^2}{1+y^{4(k+1)+2-2p}}.$$
\noindent\underline{\em Initialization $k=k_+$}. Recall that $k_+<k_-$. Let $u$ satisfy $(u,\Xi_{M,+})=0$, $v=\Lt u$, then from the previous step: $$(J\Lt \Lt^{k}u,\Lt^{k}u)=(J\Lt \Lt^{k-1} v,\Lt^{k-1}v)\gtrsim \sum_{n=0}^{2k-1}\int\frac{|D^nv|^2}{1+y^{4k-2-2n}}.$$ Now from Lemma \ref{boundweight}, case $k_+\leq k\leq k_--1$, $$\int \frac{|v|^2}{1+y^{4k-2}}=\int \frac{|\Lt u|^2}{1+y^{4k-2}}\gtrsim \int \frac{|u|^2}{1+y^{4k+2}}$$ and the conclusion follows from \fref{cneocnenceone} again written for $k-1$.\\ 
\noindent\underline{\em Initialization $k\to k+1\leq k_--1$}. Let $k+1=k_++j_++1$ and $u$ satisfy $$(u,(\Lt^*)^{p}\Xi_{M,+})=0, \ \ 0\leq p\leq j_++1,$$ then $v=\Lt u$ satisfies $$(v,(\Lt^*)^{p}\Xi_{M,+})=0, \ \ 0\leq p\leq j_+,$$ and hence by induction:
$$(J\Lt \Lt^{k+1}u,\Lt^{k+1}u)=(J\Lt \Lt^k v,\Lt^kv)\gtrsim \sum_{n=0}^{2k+1}\int\frac{|D^nv|^2}{1+y^{4k+2-2n}}.$$ Now from Lemma \ref{boundweight}, case $k_+\leq k+1\leq k_--1$, and using $(u,\Xi_{M,+})=0$, there holds: $$\int \frac{|v|^2}{1+y^{4k+2}}=\int \frac{|\Lt u|^2}{1+y^{4k+2}}\gtrsim \int \frac{|u|^2}{1+y^{4k+6}}$$ and the conclusion follows from \fref{cneocnenceone} again.\\  
\noindent\underline{\em Initialization $k=k_-$}. Let $k=k_-=k_++j_+$, let $u$ satisfy $$(u,(\Lt^*)^n\Xi_{M,+})=0,\ \ 0\leq n\leq j_+, \ \ \mbox{and}\ \ (u,\Xi_{M,-})=0.$$ Then $v=\Lt u$ satisfies $$(v,(\Lt^*)^n\Xi_{M,+})=0,\ \ 0\leq p\leq j_+-1$$ and hence from the previous step:
$$(J\Lt \Lt^{k}u,\Lt^{k}u)=(J\Lt \Lt^{k-1} v,\Lt^{k-1}v)\gtrsim \sum_{n=0}^{2k-1}\int\frac{|D^nv|^2}{1+y^{4k-2-2n}}.$$ From Lemma \ref{boundweight}, case $k\geq k_-$, and using $(u,\Xi_{M,\pm})=0$, we have: $$\int \frac{|v|^2}{1+y^{4k-2}}=\int \frac{|\Lt u|^2}{1+y^{4k-2}}\gtrsim \int \frac{|u|^2}{1+y^{4k+2}}$$ and the conclusion follows from \fref{cneocnenceone} again written for $k-1$.\\
\noindent\underline{\em Induction $k\to k+1$}. Let $k+1=k_++j_++1=k_-+j_-+1$ and $u$ satisfy $$(u,(\Lt^*)^{n}\Xi_{M,\pm})=0, \ \ 0\leq n\leq j_\pm+1,$$then $v=\Lt u$ satisfies $$(v,(\Lt^*)^{n}\Xi_{M,\pm})=0, \ \ 0\leq n\leq j_\pm,$$ and hence by induction:
$$(J\Lt \Lt^{k+1}u,\Lt^{k+1}u)=(J\Lt \Lt^k v,\Lt^kv)\gtrsim \sum_{n=0}^{2k+1}\int\frac{|D^nv|^2}{1+y^{4k+2-2n}}.$$ Now from Lemma \ref{boundweight}, case $k\leq k_-$, and using $(u,\Xi_{M,\pm})=0$, there holds: $$\int \frac{|v|^2}{1+y^{4k+2}}=\int \frac{|\Lt u|^2}{1+y^{4k+2}}\gtrsim \int \frac{|u|^2}{1+y^{4k+6}}$$ and the conclusion follows from \fref{cneocnenceone} again.
\end{proof}

%%%%%%%%%%%%%%%%%%%%%%%%%%%%%%%%%%%%%%%%%%%%%%%%%%%%%%%%%%%%%%%%%
%%%%%%%%%%%%%%%%%%%%%%%%%%%%%%%%%%%%%%%%%%%%%%%%%%%%%%%%%%%%%%%%%

\section{Interpolation bounds}

%%%%%%%%%%%%%%%%%%%%%%%%%%%%%%%%%%%%%%%%%%%%%%%%%%%%%%%%%%%%%%%%%
%%%%%%%%%%%%%%%%%%%%%%%%%%%%%%%%%%%%%%%%%%%%%%%%%%%%%%%%%%%%%%%%%

In this appendix we derive some weighted $L^{\infty}$ bounds which are used to control the  lower order terms ($N(\e),L(\e)$) in section \ref{sectionmonoton}. 
They will follow from simple interpolation arguments.
 
\begin{lemma}[$L^{\infty}$ bounds]
\label{lemmainterpolation}
{\em (i) $L^{\infty}$ bound}: 
\be
\label{bebebebeo}
\|\e\|_{L^{\infty}}+\|\nabla^{\frac d2}\e\|_{L^2}\lesssim \|\nabla^{\sigma}\e\|^{1+O\left(\frac 1{L_+}\right)}_{L^2}b_1^{\frac 12\left(\frac d2-\sigma\right)+O\left(\frac 1{L_+}\right)},
\ee
\be
\label{bebebebeogradient}
\|\nabla\e\|_{L^{\infty}}\lesssim \|\nabla^{\sigma}\e\|^{1+O\left(\frac 1{L_+}\right)}_{L^2}b_1^{\frac 12\left(\frac d2+1-\sigma\right)+O\left(\frac 1{L_+}\right)},
\ee
{\em (ii) Weighted $L^{\infty}$ bound}: let $0\leq \delta\ll L_+$, then:
\be
\label{Linfttybound}
\left\|\frac{\e}{1+y^{\delta}}\right\|_{L^{\infty}}\lesssim  \|\nabla^{\sigma}\e\|^{1+O\left(\frac 1{L_+}\right)}_{L^2}b_1^{\frac\delta 2+\frac 12\left(\frac d2-\sigma\right)+O\left(\frac 1{L_+}\right)}.
\ee
{\em (iii) Sobolev interpolation:} Let $\sigma\leq \beta\ll L_+$, then 
\be
\label{interpolationboundbeta}
\|\nabla^\beta\e\|^2_{L^2}\lesssim \|\nabla^\sigma\e\|_{L^2}^{2+O(\frac1{L_+})}b_1^{\beta-\sigma+O(\frac1{L_+})}.
\ee
\end{lemma}

\begin{remark} Interpolation constants in \fref{bebebebeo}, \fref{bebebebeogradient}, \fref{Linfttybound} depend on the bootstrap constant $K(M)$.
\end{remark}
\begin{proof}
{\em Proof of (i)}:  From Sobolev, $$\|\e\|_{L^{\infty}}+\|\nabla^{\frac d2}\e\|_{L^2}\lesssim \|\nabla^{s_+}\e\|^{1-z}_{L^2}\|\nabla^{\sigma}\e\|^{z}_{L^2}$$ with $$z=\frac{s_+-\frac d2}{s_+-\sigma}=1-\frac{1}{2L_+}\left(\frac d2-\sigma\right)+O\left(\frac 1{L_+^2}\right)$$ and thus using \fref{bootnorm}:
$$
\|\e\|_{L^{\infty}}+\|\nabla^{\frac d2}\e\|_{L^2}\lesssim \|\nabla^{\sigma}\e\|^{1+O\left(\frac 1{L_+}\right)}_{L^2}b_1^{\frac 12\left(\frac d2-\sigma\right)+O\left(\frac 1{L_+}\right)}.
$$
Similarily:
$$
\|\nabla \e\|_{L^{\infty}}\lesssim \|\nabla^{s_+}\e\|^{1-z}_{L^2}\|\nabla^{\sigma}\e\|^{z}_{L^2}$$ with $$z=\frac{s_+-\frac d2-1}{s_+-\sigma}=1-\frac{1}{2L_+}\left(\frac d2+1-\sigma\right)+O\left(\frac 1{L_+^2}\right)$$
and thus
$$
\|\nabla\e\|_{L^{\infty}}\lesssim \|\nabla^{\sigma}\e\|^{1+O\left(\frac 1{L_+}\right)}_{L^2}b_1^{\frac 12\left(\frac d2+1-\sigma\right)+O\left(\frac 1{L_+}\right)}.
$$
\noindent {\em Proof of (ii)}: For $y\leq 1$, we have from Sobolev $$\|\e\|_{L^{\infty}(y\leq 1)}\lesssim \|\e\|_{H^{s_+}(y\leq 1)}\lesssim b_1^{L_+}.$$ We estimate from \fref{esthardytwo} with $p=s_+-1$ and \fref{bootnorm}: $$\left\|\frac{\e}{1+y^{s_+-\frac d2}}\right\|_{L^{\infty}(y\geq 1)}^2\lesssim b_1^{2L_+}.$$ We therefore interpolate for $0<\delta\ll L_+$ using \fref{bebebebeo}:
\bee
\left\|\frac{\e}{1+y^{\delta}}\right\|_{L^{\infty}}&\lesssim& A^{s_+-\frac d2-\delta}\left\|\frac{\e}{1+y^{s_+-\frac d2}}\right\|_{L^{\infty}(y\leq A)}+\frac{\|\e\|_{L^{\infty}(y\geq A)}}{A^{\delta}}\\
&\lesssim &(b_1^{L_+})^{\frac{\delta}{s_+-\frac d2-2\delta}}\|\e\|_{L^{\infty}}^{1-\frac{\delta}{s_+-\frac d2-2\delta}}\\
& \lesssim & \|\nabla^{\sigma}\e\|^{1+O\left(\frac 1{L_+}\right)}_{L^2}b_1^{\frac\delta 2+\frac 12\left(\frac d2-\sigma\right)+O\left(\frac 1{L_+}\right)}.
\eee
\noindent{\it Proof of (iii)}. We interpolate $$\|\nabla^\beta\e\|_{L^2}\lesssim \|\nabla^\sigma\e\|_{L^2}^{z_+}\|\nabla^{s_+}\e\|_{L^2}^{1-z_+}$$ with $$1-z_+=\frac{\beta-\sigma}{s_+-\sigma}=\frac{\beta-\sigma}{2L_+}+O\left(\frac1{L_+^2}\right)$$ and hence using \fref{bootnorm}:
$$
\|\nabla^\beta\e\|^2_{L^2}\lesssim \|\nabla^\sigma\e\|_{L^2}^{2+O(\frac1{L_+})}b_1^{\beta-\sigma+O(\frac1{L_+})}.
$$
\end{proof}

%%%%%%%%%%%%%%%%%%%%%%%%%%%%%%%%%%%%%%%%%%%%%%%%%%%%%%%%%%%%%%%%%%%%%%%%%%%%%%%%%%%%%%%%%%%%%%%%%%%%%%%%%%%

%%%%%%%%%%%%%%%%%%%%%%%%%%%%%%%%%%%%%%%%%%%%%%%%%%%%%%%%%%%%%%%%%%%%%%%%%%%%%%%%%%%%%%%%%%%%%%%%%%%%%%%%%%%

\section{Eigenvalues of the linearized operator in self similar variables}
\label{sectionapprnedix}
%%%%%%%%%%%%%%%%%%%%%%%%%%%%%%%%%%%%%%%%%%%%%%%%%%%%%%%%%%%%%%%%%%%%%%%%%%%%%%%%%%%%%%%%%%%%%%%%

We briefly revisit in this section the standard computation of the eigenvalues and eigenvectors of the linearized operator close to the self similar solution: $$H\Phi=H_0\Phi -\left[\frac{1}{p-1}\Phi+\frac 12 r\Phi'\right], \ \ H_0=\left(\begin{array}{ll} 0&H_-\\-H_+&0\end{array}\right)$$ with $$H_+=-\Delta -\frac{pc_{\infty}^{p-1}}{r^2}, \ \ H_-=-\Delta -\frac{c_{\infty}^{p-1}}{r^2}.$$

\noindent {\bf step 1} First set of eigenvalues.\\ 

\noindent{\bf case $\ell=0$}. We let $$\Phi_{0,+}(r)=\left|\begin{array}{ll}\frac{1}{r^{\gamma}}\\0\end{array}\right., \ \ \lambda_0=\frac{1}{p-1}-\frac \gamma 2$$ and compute: 
$$H\Phi_{0,+}=\left|\begin{array}{ll} -\left[\frac{1}{p-1}+\frac 12 r\pa_r\right] r^{-\gamma}\\-H_+r^{-\gamma}\end{array}\right.+=-\l_0\Phi_{0,+}$$
where we used the $\gamma$ equation: $$H_+(r^{-\gamma})=-\frac{\gamma^2-(d-2)\gamma+pc_{\infty^{p-1}}}{r^{\gamma+2}}=0.$$

\noindent{\bf case $\ell\geq 1$}. We let $$\Phi_{\ell,+}=\sum_{k=0}^{\ell}c_kJ^k\left|\begin{array}{ll} r^{2k-\gamma}\\0\end{array}\right., \ \ \lambda_{\ell,+}=\frac{1}{p-1}-\frac \gamma 2+\ell, \ \ c_0=1$$ and compute:
\bee
H\Phi_{\ell,+}+\l_{\ell,+}\Phi_{\ell,+}&=&\sum_{k=0}^\ell c_k H_0J^k\left|\begin{array}{ll} r^{2k-\gamma}\\0\end{array}\right.+\sum_{k=0}^\ell \left\{\l_{\ell,+}-\left[\frac 1{p-1}+\frac 12r\pa_r\right]\right\}c_kJ^k\left|\begin{array}{ll} r^{2k-\gamma}\\0\end{array}\right.\\
& = & \sum_{k=1}^\ell c_k H_0J^k\left|\begin{array}{ll} r^{2k-\gamma}\\0\end{array}\right.+\sum_{k=0}^{\ell-1} \left\{\l_{\ell,+}-\left[\frac 1{p-1}+\frac 12r\pa_r\right]\right\}c_kJ^k\left|\begin{array}{ll} r^{2k-\gamma}\\0\end{array}\right.\\
& = & \sum_{k=0}^{\ell-1}c_{k+1} H_0J^{k+1}\left|\begin{array}{ll} r^{2k+2-\gamma}\\0\end{array}\right.+c_k\left\{\l_{\ell,+}-\left[\frac 1{p-1}+\frac 12r\pa_r\right]\right\}J^k\left|\begin{array}{ll} r^{2k-\gamma}\\0\end{array}\right.
\eee
thanks to the $\gamma$ equation for $k=0$ and the choice of $\l_{\ell,+}$ for $k=\ell$. We now compute for $k=2p$:
\bee
&&c_{k+1} H_0J^{k+1}\left|\begin{array}{ll} r^{2k+2-\gamma}\\0\end{array}\right.+c_k\left\{\l_{\ell,+}-\left[\frac 1{p-1}+\frac 12r\pa_r\right]\right\}J^k\left|\begin{array}{ll} r^{2k-\gamma}\\0\end{array}\right.\\
& = &\left|\begin{array}{ll}(-1)^pc_{k+1}H_-(r^{2k+2-\gamma})+c_k\left\{\l_{\ell,+}-\left[\frac 1{p-1}+\frac 12r\pa_r\right]\right\}r^{2k-\gamma}\\0\end{array}\right.=\left|\begin{array}{ll}c_{k+1}d_k-c_ke_k\\0\end{array}\right.
\eee
and for $k=2p+1$:
\bee
&&c_{k+1} H_0J^{k+1}\left|\begin{array}{ll} r^{2k+2-\gamma}\\0\end{array}\right.+c_k\left\{\l_{\ell,+}-\left[\frac 1{p-1}+\frac 12r\pa_r\right]\right\}J^k\left|\begin{array}{ll} r^{2k-\gamma}\\0\end{array}\right.\\
& = &\left|\begin{array}{ll}0\\(-1)^{p}c_{k+1}H_+(r^{2k+2-\gamma})+(-1)^{p+1}c_k\left\{\l_{\ell,+}-\left[\frac 1{p-1}+\frac 12r\pa_r\right]\right\}r^{2k-\gamma}\end{array}\right.=\left|\begin{array}{ll}0\\c_{k+1}d_k-c_ke_k\end{array}\right.
\eee
and hence the recurrence relation ($d_k\neq 0$)$$c_{k+1}=\frac{e_k}{d_k}c_k, \ \ c_1=0$$ yields an eigenvector.\\

\noindent {\bf step 2} Second set of eigenvalues.\\ 

\noindent{\bf case $\ell=0$}. We let $$\Phi_{0,-}(r)=\left|\begin{array}{ll}0\\\frac{1}{r^{\frac 2{p-1}}}\end{array}\right., \ \ \lambda_{0,-}=0$$ and compute: 
$$H\Phi_{0,-}=\left|\begin{array}{ll} H_-(r^{-\frac{2}{p-1}})\\-\left[\frac{1}{p-1}+\frac 12 r\pa_r\right] r^{-\frac{2}{p-1}}\\\end{array}\right.+=0$$
where we used the definition \fref{selfsimilarsolutions} of $c_{\infty}$:
\bee
H_-(r^{-\frac{2}{p-1}})&=&\frac{1}{r^{\frac 2{p-1}+2}}\left\{-\frac{2}{p-1}\left(\frac{2}{p-1}+1\right)+\frac{2(d-1)}{p-1}-c_\infty^{p-1}\right\}\\
&=&\frac{1}{r^{\frac 2{p-1}+2}}\left\{\frac{2}{p-1}\left(d-2-\frac2{p-1}\right)-c_{\infty}^{p-1}\right\}=0.
\eee

\noindent{\bf case $\ell\geq 1$}. We let $$\Phi_{\ell,-}=\sum_{k=0}^{\ell}c_kJ^k\left|\begin{array}{ll} 0\\r^{2k-\frac{2}{p-1}}\end{array}\right., \ \ \lambda_{\ell,-}=\ell, \ \ c_0=1$$ and compute:
\bee
&&H\Phi_{\ell,-}+\l_{\ell,-}\Phi_{\ell,-}\\
&=&\sum_{k=0}^\ell c_k H_0J^k\left|\begin{array}{ll} 0\\r^{2k-\frac{2}{p-1}}\end{array}\right.+\sum_{k=0}^\ell \left\{\l_{\ell,-}-\left[\frac 1{p-1}+\frac 12r\pa_r\right]\right\}c_kJ^k\left|\begin{array}{ll} 0\\r^{2k-\frac2{p-1}}\end{array}\right.\\
& = & \sum_{k=1}^\ell c_k H_0J^k\left|\begin{array}{ll} 0\\r^{2k-\frac{2}{p-1}}\end{array}\right.+\sum_{k=0}^{\ell-1} \left\{\l_{\ell,-}-\left[\frac 1{p-1}+\frac 12r\pa_r\right]\right\}c_kJ^{k}\left|\begin{array}{ll} 0\\r^{2k-\frac 2{p-1}}\end{array}\right.\\
& = & \sum_{k=0}^{\ell-1}c_{k+1} H_0J^{k+1}\left|\begin{array}{ll} 0\\r^{2k+2-\frac 2{p-1}}\end{array}\right.+c_k\left\{\l_{\ell,-}-\left[\frac 1{p-1}+\frac 12r\pa_r\right]\right\}J^k\left|\begin{array}{ll} 0\\r^{2k-\frac{2}{p-1}}\end{array}\right.
\eee
thanks to the $c_{\infty}$ equation for $k=0$ and the choice of $\l_{\ell,-}$ for $k=\ell$. This as above yields a suitable induction relation on the $c_k$ to create an eigenvector.

\end{appendix}

\frenchspacing
\bibliographystyle{plain}

\end{document}